\documentclass[11pt]{article}  
\usepackage{amsmath}
\usepackage{amsfonts}
\usepackage{amssymb}
\usepackage{amsthm}
\usepackage{euscript}
\usepackage{ifthen}
\usepackage{xifthen}
\usepackage{xcolor}
\usepackage{fullpage}
\usepackage[english]{babel}
\usepackage{mathtools}
\usepackage{hyperref}
\usepackage{authblk}
\hypersetup{colorlinks = true,linkcolor=black,citecolor=black}

\newcommand{\inv}{\mathrm {inv}}

\usepackage{algorithm} 
\usepackage{algorithmicx}
\usepackage[noend]{algpseudocode}

\usepackage{array}

\usepackage{imakeidx}
\makeindex

\newtheorem{theorem}{Theorem}[section]
\newtheorem{lemma}[theorem]{Lemma}
\newtheorem{corollary}[theorem]{Corollary}
\newtheorem{proposition}[theorem]{Proposition}
\newtheorem{remark}[theorem]{Remark}
\newtheorem{definition}[theorem]{Definition}
\newtheorem{assumption}[theorem]{Assumption}
\newtheorem{example}[theorem]{Example}
\numberwithin{equation}{section}
\newcommand{\conv}{\alpha}
\newcommand{\todo}[1]{\ifthenelse{\isempty{#1}}{\textcolor{red}{\textbf{TODO}}}{\textcolor{red}{\textbf{TODO: #1}}}}

\newcommand{\seqi}[1]{c_{#1}}
\newcommand{\seqii}[1]{d_{#1}}

\newcommand{\bPhi}{{{\boldsymbol{\Phi}}}}
\newcommand{\bPsi}{{{\boldsymbol{\Psi}}}}
\newcommand{\wk}[2][]{\ifthenelse{\equal{#1}{}}{\omega_{#2}}{\omega_{#2}^{#1}}}
\newcommand{\twk}[2][]{\ifthenelse{\equal{#1}{}}{\tilde \omega_{#2}}{\tilde \omega_{#2}^{#1}}}
\newcommand{\SW}{\mathfrak{W}}
\newcommand{\sw}[2][]{\ifthenelse{\equal{#1}{}}{\mathfrak{w}_{#2}}{\mathfrak{w}_{#2}^{#1}}}
\newcommand{\lev}{l}
\newcommand{\blev}{\mathbf{l}}

\newcommand{\KW}{K_{\mathfrak{W}}}
\newcommand{\work}{\mathrm{work}}
\newcommand{\pts}{{\rm pts}}
\newcommand{\ext}{{\rm ext}}
\newcommand{\norm}[2][]{\| #2 \|_{#1}}
\newcommand{\normc}[2][]{\left\| #2 \right\|_{#1}}
\newcommand{\normk}[2][]{\bigg\| #2 \bigg\|_{#1}}
\newcommand{\set}[2]{\{#1\,:\,#2\}}
\newcommand{\eps}{\varepsilon}

\newcommand{\cA}{\mathcal A} 
\newcommand{\cB}{\mathcal B} 
\newcommand{\cC}{\mathcal C} 
\newcommand{\cS}{\mathcal S} 
\newcommand{\cE}{\mathcal E} 
\newcommand{\cF}{\mathcal F} 
\newcommand{\cT}{\mathcal T} 
\newcommand{\CF}{\mathcal F}
\newcommand{\cL}{\mathcal L} 
\newcommand{\cM}{\mathcal M} 
\newcommand{\Qq}{\mathcal Q}

\newcommand{\IP}{\mathbb{P}}
\newcommand{\R}{\mathbb{R}}
\newcommand{\IS}{\mathbb{S}} 
\newcommand{\E}{\mathbb{E}}
\newcommand{\N}{\mathbb{N}}
\newcommand{\C}{\mathbb{C}}
\newcommand{\VI}{\mathbf{I}}
\newcommand{\VIml}{\mathbf{I}^{\rm ML}}
\newcommand{\VQml}{\mathbf{Q}^{\rm ML}}
\newcommand{\VQ}{\mathbf{Q}}
\newcommand{\bnul}{{\boldsymbol{0}}}
\newcommand{\dd}{\,\mathrm{d}}
\newcommand{\essinf}{\operatorname{ess\,inf}}
\newcommand{\esssup}{\operatorname{ess\,sup}}
\newcommand{\dup}[2]{\left\langle #1, #2\right\rangle}
\newcommand{\obs}{\mathfrak{d}}
\newcommand{\data}{E}
\newcommand{\ttA}{{\tt A}}
\newcommand{\ttE}{{\tt E}}

\newcommand{\bb}{{\boldsymbol{b}}}
\newcommand{\bc}{{\boldsymbol{c}}}
\newcommand{\bd}{{\boldsymbol{d}}}
\newcommand{\bee}{{\boldsymbol{e}}}
\newcommand{\be}{{\boldsymbol{e}}}

\newcommand{\bk}{{\boldsymbol{k}}}
\newcommand{\bh}{{\boldsymbol{h}}}

\newcommand{\bm}{{\boldsymbol{m}}}
\newcommand{\bn}{{\boldsymbol{n}}}

\newcommand{\bq}{{\boldsymbol{q}}}

\newcommand{\bu}{{\boldsymbol{u}}}
\newcommand{\bv}{{\boldsymbol{v}}}
\newcommand{\boldf}{{\boldsymbol{f}}}
\newcommand{\bs}{{\boldsymbol{s}}}

\newcommand{\bx}{{\boldsymbol{x}}}

\newcommand{\by}{{\boldsymbol{y}}}
\newcommand{\bY}{{\boldsymbol{Y}}}
\newcommand{\bZ}{{\boldsymbol{Z}}}
\newcommand{\bz}{{\boldsymbol{z}}}
\newcommand{\bI}{{\boldsymbol{I}}}
\newcommand{\bL}{{\boldsymbol{L}}}
\newcommand{\bM}{{\boldsymbol{M}}}
\newcommand{\brho}{{\boldsymbol{\rho}}}

\newcommand{\bxi}{{\boldsymbol{\xi}}}
\newcommand{\bnu}{{\boldsymbol{\nu}}}
\newcommand{\bmu}{{\boldsymbol{\mu}}}
\newcommand{\beeta}{{\boldsymbol{\eta}}}
\newcommand{\bvarrho}{{\boldsymbol{\varrho}}}
\newcommand{\bbeta}{{\boldsymbol{\beta}}}
\newcommand{\balpha}{{\boldsymbol{\alpha}}}
\newcommand{\bgamma}{{\boldsymbol{\gamma}}}
\newcommand{\bGamma}{{\boldsymbol{\Gamma}}}
\newcommand{\blambda}{{\boldsymbol{\lambda}}}
\newcommand{\bsigma}{{\boldsymbol{\sigma}}}
\newcommand{\bepsilon}{{\boldsymbol{\epsilon}}}
\newcommand{\bobs}{{\boldsymbol{\obs}}}
\newcommand{\bcalO}{{\boldsymbol{\calO}}}
\newcommand{\boldeta}{{\boldsymbol{\eta}}}

\newcommand{\rev}{\operatorname{ev}} 
\newcommand{\rd}{\dd} 
\newcommand*{\bigtimes}{\mathop{\raisebox{-.5ex}{\hbox{\huge{$\times$}}}}}

\def\mfu{{\mathfrak u }}
\def\mfm{{\mathfrak m }}

\def\EE{\mathbb E}

\def\ZZ{{\mathbb Z}}

\def\RR{{\mathbb R}}

\def\NN{{\mathbb N}}

\def\CC{{\mathbb C}}

\def\NN{{\mathbb N}}
\def\RR{{\mathbb R}}

\def\Vv{{\mathbb P}}

\def\calA{{\mathcal A}}
\def\calB{{\mathcal B}}
\def\calF{{\mathcal F}}

\def\calH{{\mathcal H}}
\def\calN{{\mathcal N}}
\def\calO{{\mathcal O}}
\def\Vv{{\mathcal P}}
\def\Kk{{\mathcal K}}
\def\Ff{{\mathcal F}}
\def\Bb{{\mathcal B}}
\def\TT{{\mathbb T}}

\def\TTd{{\mathbb T}^d}

\def\RRi{{\mathbb R}^\infty}

\def\RRi{{\mathbb R}^\infty}

\def\Bb{{\mathcal B}}
\def\Cc{{\mathcal C}}

\def\Ee{{\mathcal E}}
\def\Ff{{\mathcal F}}
\def\Gg{{\mathcal G}}

\def\Ii{{\mathcal I}}

\def\Kk{{\mathcal K}}

\def\Oo{{\mathcal O}}

\def\Ss{{\mathcal S}}

\def\Uu{{\mathcal U}}
\def\Vv{{\mathcal V}}
\def\Ww{{\mathcal W}}

\def\CC{{\mathbb C}}
\def\ZZ{{\mathbb Z}}
\def\NN{{\mathbb N}}
\def\RR{{\mathbb R}}

\def\FF{{\mathcal F}}
\def\TT{{\mathbb T}}

\def\TTd{{\mathbb T}^d}

\def\supp{\operatorname{supp}}
\def\div{\operatorname{div}}
\def\im{{\rm i}}

\newcommand{\domain}{{D}}
\newcommand{\D}{{\domain}}

\def\Hn{H^s}
\def\Win{W^s_\infty}
\usepackage{color}

 \newcommand{\black}[1]{\textcolor{\black}{#1}}
\newcommand{\KL}{Karhunen-Lo\`{e}ve }
\newcommand{\LC}{L\'{e}vy-Cieselsky }

\date{\today}

\begin{document}

\title{Analyticity and sparsity in uncertainty quantification  \\
  for PDEs with Gaussian random field inputs}

\author[a]{Dinh D\~ung} \affil[a]{Information Technology Institute,
  Vietnam National University, Hanoi
  \protect\\
  144 Xuan Thuy, Cau Giay, Hanoi, Vietnam
  \protect\\
  Email: dinhzung@gmail.com}

\author[b]{Van Kien Nguyen} \affil[b]{Department of Mathematical Analysis,
  University of Transport and Communications \protect\\ No.3 Cau Giay
  Street, Lang Thuong Ward, Dong Da District, Hanoi, Vietnam
  \protect\\
  Email: kiennv@utc.edu.vn}

\author[c]{Christoph Schwab} \affil[c]{Seminar for Applied
  Mathematics, ETH Z\"urich, 8092 Z\"urich, Switzerland \protect\\
  Email: schwab@math.ethz.ch}

\author[d]{Jakob Zech} \affil[d]{Interdisziplin\"ares Zentrum f\"ur
  wissenschaftliches Rechnen,
  \protect\\
  Universit\"at Heidelberg, 69120 Heidelberg, Germany
  \protect\\
  Email: jakob.zech@uni-heidelberg.de}

\date{\today}

\maketitle

\begin{abstract}
  We establish sparsity and summability results for coefficient
  sequences of Wiener-Hermite polynomial chaos expansions of
  countably-parametric solutions of linear elliptic and parabolic
  divergence-form partial differential equations with Gaussian random
  field inputs.

  The novel proof technique developed here is based on analytic
  continuation of parametric solutions into the complex domain.  It
  differs from previous works that used bootstrap arguments and
  induction on the differentiation order of solution derivatives with
  respect to the parameters.  The present holomorphy-based argument
  allows a unified, ``differentiation-free'' proof of sparsity
  (expressed in terms of $\ell^p$-summability or weighted
    $\ell^2$-summability) of sequences of Wiener-Hermite coefficients
  in polynomial chaos expansions in various scales of function spaces.
  The analysis also implies corresponding analyticity and sparsity
  results for posterior densities in Bayesian inverse problems subject
  to Gaussian priors on uncertain inputs from function spaces.

  Our results furthermore yield dimension-independent convergence
  rates of various \emph{constructive} high-dimensional deterministic
  numerical approximation schemes such as single-level and multi-level
  versions of Hermite-Smolyak anisotropic sparse-grid interpolation
  and quadrature in both forward and inverse computational uncertainty
  quantification.
\end{abstract}

\newpage
\
\newpage
\tableofcontents
\newpage

\section{Introduction}
\label{sec:Intro}
Gaussian random fields (GRFs for short) play a fundamental role in the
modelling of spatio-temporal phenomena subject to uncertainty.  In
several broad research areas, particularly, in spatial statistics,
data assimilation, climate modelling and meteorology to name but a
few, GRFs play a pivotal role in mathematical models of physical
phenomena with distributed, uncertain input data.  Accordingly, there
is an extensive literature devoted to mathematical, statistical and
computational aspects of GRFs.  We mention only
\cite{Lifshits95,Janson97,AdlerGeoGRF} and the references there for
mathematical foundations, and \cite{Matern2nd,GineNickl} and the
references there for a statistical perspective on GRFs.

In recent years, the area of \emph{computational uncertainty
  quantification} (UQ for short) has emerged at the interface of the
fields of applied mathematics, numerical analysis, scientific
computing, computational statistics and data assimilation.  Here, a
key topic is the mathematical and numerical analysis of partial
differential equations (PDEs for short) with random field inputs, and in
particular with GRF inputs. The mathematical analysis 
of PDEs with GRF inputs addresses
questions of well-posedness, pathwise and $L^p$-integrability and
regularity in scales of Sobolev and Besov spaces of random solution
ensembles of such PDEs.  The numerical analysis 
focuses on questions of efficient numerical simulation methods of
GRF inputs (see, e.g., \cite{Matern2nd,DietrNewsam,
  Gitt2012,Charrier12,CharrDebu13,bachmayr2018GRFRep,bachmayr2019unified,SteinKrig} and the
references there), and the numerical approximation 
of corresponding PDE solution ensembles, which
arise for GRF inputs. This concerns in particular the efficient
representation of such solution ensembles (see
\cite{HS,BCDM,BCDS,ErnstSprgkTam18,dD21}), 
and the numerical quadrature of corresponding solution fields 
(see, e.g., 
\cite{HS,KSSS17_2112,DHSHMatGRF17,GKNSSS,HS19_2310,HS17_2475,CSAMS2011,ChenlogNQuad2018,dD21}
and the references there).  
Applications include for instance
subsurface flow models (see, e.g., \cite{DietrNewsam,GKSS13_484}) but
also other PDE models for media with uncertain properties (see, e.g.,
\cite{logNMax2018} for electromagnetics).  The careful analysis of
efficient computational sampling of solution families of PDEs subject
to GRF inputs is also a key ingredient in numerical data assimilation,
e.g., in Bayesian inverse problems (BIPs for short); we refer to the
surveys \cite{DashtiStuart17,DshtiLwStrtVossMAP} and the references
therein for a mathematical formulation of BIPs for PDEs subject to
Gaussian prior measures and function space inputs.

In the past few years there have been considerable developments in the
analysis and numerical simulation of PDEs with random field input 
subject to Gaussian
  measures (GMs for short).
The method of choice in many applications for the 
numerical treatment of GMs is Monte-Carlo (MC for short) sampling.  
The (mean-square) convergence rate $1/2$ 
in terms of the number of MC samples is assured
under rather mild conditions (existence of MC samples, and of finite second moments). 
We refer to, e.g.,
  \cite{MR2835612,MR3033013,MR3117512,HQSMLMCMC} and the references there for
  a discussion of MC methods  in this context.
Given the high cost of MC sampling, recent years have seen the advent
of numerical techniques which afford higher convergence orders than
$1/2$, also on infinite-dimensional integration domains.  Like MC,
these techniques are not prone to the so-called curse of
dimensionality. 
Among them are Hermite-Smolyak sparse-grid
interpolation (also referred to as ``stochastic collocation''), see e.g. \cite{ErnstSprgkTam18,dD21,dD-Erratum22},
and 
sparse-grid quadrature \cite{ChenlogNQuad2018,ErnstSprgkTam18,HHMPMS16,dD21,dD-Erratum22}, and 
quasi-Monte Carlo (QMC for short) integration as developed in 
\cite{GKNSSS,RSAMSAT_Bip2017,KSSS17_2112,kaza,HS19_2310,HS17_2475}
and the reference there.

The key condition which emerged as governing the
convergence rates of numerical integration and interpolation methods for a function is
\emph{a sparsity of the coefficients of its Wiener-Hermite polynomial chaos (PC for short) expansion}, 
see, e.g., \cite{HS,bachmayr2018GRFRep}.
Rather than counting the ratio of nonzero
coefficients, the sparsity is quantified by 
\emph{$\ell^p$-summability and/or weighted $\ell^2$-summability} of
these coefficients.
This observation forms the foundation for the current text.

\subsection{An example}
To indicate some of the mathematical issues which are considered in
this book, consider in the interval $\domain = (0,1)$ and in
a probability space $(\Omega,\cA,\IP)$, a GRF
$g:\Omega\times \domain \to \R$ which takes values in $L^\infty(\domain)$.  
That is to say, that the map
$\omega \mapsto g(\omega, \cdot)$ is an element of the Banach space
$L^\infty(\domain)$.  
\emph{Formally}, at this stage, we represent
realizations of the random element $g \in L^\infty(\domain)$ with a
\emph{representation system}
$(\psi_j )_{j=1}^J\subset L^\infty(\domain)$ in
\emph{affine-parametric form}
\begin{equation}\label{eq:AffParg}
 g(\omega,x) = \sum_{j=1}^J y_j(\omega) \psi_j(x) \,;,
\end{equation}
where the coefficients $( y_j )_{j=1}^J$ are assumed to be i.i.d.\
standard normal random variables (RVs for short) and $J$ may be a
finite number or infinity.  Representations such as \eqref{eq:AffParg}
are widely used both in the analysis and in the numerical simulation
of random elements $g$ taking values in a function space.  The
coefficients $y_j(\omega)$ being standard normal RVs, the sum
$\sum_{j=1}^J y_j \psi_j(x)$ may be considered as a \emph{parametric
  deterministic map} $g: \R^J \to L^\infty(\domain)$. The random
element $g(\omega,x)$ in \eqref{eq:AffParg} 
can then be obtained  
by evaluating this deterministic map in random coordinates, i.e., by
sampling it in Gaussian random vectors
$(y_j(\omega))_{j=1}^J\in \R^J$.

Gaussian random elements as inputs for PDEs appear in particular, in
coefficients of diffusion equations.  
Consider, for illustration, in
$\domain$, and for given $f\in L^2(\domain)$, the boundary value
problem: find a random function $u: \Omega \to V$ with
$V := \{w\in H^1(\domain): w(0) = 0 \}$ such that  
\begin{equation}\label{eq:BVP}
  f(x) + \frac{\dd}{\dd x}\left(a(x,\omega) \frac{\dd}{\dd x} u (x,\omega) \right) = 0 \quad \mbox{in} \quad \domain\;,
  \quad 
  a(1,\omega) u'(1,\omega) = \bar f\;.
\end{equation}
Here, $a(x,\omega) = \exp(g(x,\omega))$ with GRF
$g:\Omega \to L^\infty(\domain)$, and $\bar{f}:= F(1)$ with
$$
F(x) := \int_0^x f(\xi)\dd\xi \in V, \quad x\in \domain.
$$
In order to dispense with summability and 
	measurability issues, let us temporarily assume  that the sum in
	\eqref{eq:AffParg} is finite, with $J\in\N$ terms. 
        We
        find that a random solution $u$ of the problem must satisfy
$$
u'(x,\omega) = - \exp(-g(x,\omega)) F(x) , \quad x\in \domain, \omega
\in \Omega \;.
$$
Inserting \eqref{eq:AffParg}, this is equivalent to the
\emph{parametric, deterministic family of solutions}
$u(x,\by):\domain\times \R^J \to \R$ given by
\begin{equation}\label{eq:DetSoly}
  u'(x,\by) = -\exp(-g(x,\by)) F(x) , \quad x\in \domain, \by\in \R^J \;.
\end{equation}
Hence
\begin{equation*}\label{eq:Nrmu}
  \| u'(\cdot,\by) \|_{L^2(\domain)} 
  = 
  \| \exp(-g(\cdot,\by)) F \|_{L^2(\domain)} \;,\quad \by\in \R^J\;,
\end{equation*}
which implies the (sharp) bounds
$$
\| u'(\cdot,\by) \|_{L^2(\domain)} \left\{
  \begin{array}{l}
    \geq 
    \exp(-\| g(\cdot,\by) \|_{L^\infty(\domain)}) \| F \|_{L^2(\domain)} 
    \\
    \leq \exp(\| g(\cdot,\by) \|_{L^\infty(\domain)}) \| F \|_{L^2(\domain)}.
  \end{array}
\right .
$$
Due to the homogeneous Dirichlet condition at $x=0$, up to an absolute
constant the same bounds also hold for $\| u(\cdot,\by) \|_V$.

It is evident from the explicit expression \eqref{eq:DetSoly} and the
upper and lower bounds, that for every parameter $\by\in \R^J$,
the solution $u\in V$ exists.  
However, we can not, in general, 
expect uniform w.r.t.\ $\by\in \R^J$ a-priori estimates, 
also of the higher derivatives, for smoother functions $x\mapsto g(x,\by)$ 
and $x\mapsto f(x)$. 
Therefore, the parametric problem \eqref{eq:BVP}
is nonuniformly elliptic, \cite{Charrier12,HQSMLMCMC}. 
In particular,
also a-priori error bounds for various discretization
schemes will contain this uniformity w.r.t.\ $\by$.  The random
solution will be recovered from \eqref{eq:DetSoly} by inserting for
the coordinates $y_j$ samples of i.i.d.\ standard normal RVs.

This book focuses on developing a regularity theory for
  countably-parametric solution families
  ${ u(\cdot,\by) : \by \in \R^J }$ with a particular emphasis on the
  case $J=\infty$. This allows for arbitrary Gaussian random fields
  $g(\cdot,\omega)$ in \eqref{eq:BVP}.
Naturally, our results also cover the
finite-parametric setting where the number $J$ of random parameters
is finite, but may be very large. 
Then, all constants in our error estimates are
either independent of the parameter dimension $J$
or their dependence of $J$ is explicitly indicated.
Previous works
\cite{BCDS,BCDM,GKNSSS} addressed the $\ell^p$-summability
of the Wiener-Hermite PC expansion coefficients of solution families
$\{u(\cdot,\by): \by\in \R^\infty\}\subset V$ for the forward problem,
based on moment bounds of derivatives of parametric solutions w.r.t.
GM.  Estimates for these coefficients and, in particular, for the
summability, were obtained in \cite{HS,BCDS,BCDM,GKNSSS,HS2}. In these
references, all arguments were based on real-variable, bootstrapping
arguments with respect to $\by$.
 
\subsection{Contributions}
We make the following contributions to 
the area computational UQ for PDEs  with  GRF inputs. 
First, we provide novel proofs of some of the sparsity results in
\cite{HS,BCDM,BCDS} of the infinite-dimensional parametric forward
solution map to PDEs with GRF inputs.  The presently developed proof
technique is based on 
holomorphic continuation and complex variable arguments in order
to bound derivatives of parametric solutions, and their coefficients
in Wiener-Hermite PC expansions.  This is in line with similar
arguments in the so-called ``uniform case'' in
\cite{CoDeSch1,CCS13_783}.  There, the random parameters in the
representation of the input random fields range in compact subsets of
$\RR$.  Unlike in these references, in the present text due to the
Gaussian setup the parameter domain $\RR^\infty$ is not compact.
This entails significant modifications of mathematical arguments as
compared to those in \cite{CoDeSch1,CCS13_783}.

Contrary to the analysis in \cite{BCDS,BCDM,GKNSSS}, where parametric
regularity results were obtained by real-variable arguments combined
with induction-based bootstrapping with respect to the derivative
order, the present text develops derivative-free, complex variable
arguments which allow directly to obtain bounds of the Wiener-Hermite
PC expansion coefficients of the parametric solutions in scales of
Sobolev and Besov spaces in the physical domain $\domain$ in which the
parametric PDE is posed. They also allow to treat in a unified manner
parametric regularity of the solution map in several scales of Sobolev
and Kondrat'ev spaces in the physical domain $\domain$ which is the
topic of Section \ref{sec:KondrReg}, resulting in novel sparsity
results for the solution operators to
linear elliptic and parabolic PDEs with GRF inputs 
in scales of Sobolev and Besov spaces.
We apply the quantified holomorphy of parametric solution families
    to PDEs with GRF inputs and preservation of 
    holomorphy under composition, to problems of Bayesian PDE inversion
    conditional on noise observation data in Section~\ref{sec:BIP}, 
    establishing in particular quantified parametric holomorphy of the
    corresponding Bayesian posterior.

  We construct deterministic sparse-grid interpolation and quadrature
  methods for the parametric solution with convergence rate bounds
  that are free from the curse of dimensionality, and that afford possibly
  high convergence rates, given a sufficient sparsity in the Wiener-Hermite PC
  expansion of the parametric solutions.
  For sampling strategies in deterministic numerical quadrature, 
  our findings show improved convergence rates, 
  as compared to previous results in this area. 
  Additionally, our novel sparsity
  results provided in scales of  function spaces of
  varying spatial regularity enable us
  to construct apriori \emph{multilevel versions of sparse-grid
  interpolation and quadrature}, with corresponding approximation rate
  bounds which are free from the curse of dimensionality, 
  and explicit in terms of the overall number of degrees of freedom. 
  Lastly, and in contrast to previous works, 
  leveraging the preservation of holomorphy under compositions
  with holomorphic maps, our holomorphy-based arguments enable us to
  establish that our algorithms and bounds are applicable to posterior
  distributions in Bayesian inference problems involving GRF or Besov priors,
  as developed in \cite{DashtiStuart17,NicklBIP} and the references there.
\subsection{Scope of results}
\label{sec:IntrScope}
We prove quantified holomorphy of countably-parametric solution
families of linear elliptic and parabolic PDEs.  The parameter range
equals
$\RRi$, corresponding to countably-parametric representations of
GRF input data, taking values in a separable locally convex space, in
particular, Hilbert or Banach space of uncertain input data, endowed
for example with a Gaussian product measure $\gamma$ on $\RR^\infty$.

The results established in this text and the related bounds on
partial derivatives w.r.t.\ the parameters in \KL or \LC expansions of
uncertain GRF inputs imply convergence rate bounds for several
families of computational methods to numerically access these
parametric solution maps.  Importantly, we prove that in terms of 
$n\geq 1$,
    an integer measure of work and memory, 
an approximation accuracy $O(n^{-a})$
for some parameter $a>0$ can be achieved, where the convergence rate
$a$ depends on the approximation process and on the amount of sparsity
in the Wiener-Hermite PC expansion coefficients of the GRF 
under consideration.  
In the terminology of computational complexity, a
prescribed numerical tolerance $\eps>0$ can be reached in work and
memory of order $O(\eps^{-1/a})$.  
\emph{In particular, the convergence rate $a$ and the constant hidden in the Landau
  $O(\cdot)$ symbol do not depend on the dimension of the space of
  active parameters involved in the approximations which we
  construct.}  
The approximations developed in the present text are
\emph{constructive and linear} and can be realized computationally by
\emph{deterministic algorithms} 
of so-called ``stochastic collocation'' or ``sparse-grid'' type.  
Error bounds are proved in
$L^2$-type Bochner spaces with respect to the GM 
$\gamma$ on the input data space of the PDE, 
in natural Hilbert or Banach spaces of
solutions of the PDEs under consideration. 
Here, it is important to notice that 
the sparsity of the Wiener-Hermite PC expansion coefficients used in constructive
linear approximation algorithms and in estimating convergence rates, 
takes the form of
weighted $\ell^2$-summability, but not $\ell^p$-summability 
as in best $n$-term approximations \cite{HS,BCDM,BCDS}. 
Furthermore, 
$\ell^p$-summability results are implied from the corresponding weighted $\ell^2$-summability ones.

All approximation rates for deterministic sampling strategies in the present text 
are free from the so-called \emph{curse of dimensionality}, a terminology
coined apparently by R.E. Bellmann (see \cite{Bellmann}).  
The rates
are in fact only limited by the sparsity of the Wiener-Hermite PC
expansion coefficients of the deterministic, countably-parametric
solution families.  
In particular, dimension-independent convergence
rates $>1/2$ are possible, provided a sufficient Wiener-Hermite PC
expansion coefficient sparsity, that the random inputs feature
sufficient pathwise regularity, and the affine representation system
(being a tight frame on space of admissible input realizations) are
stable in a suitable smoothness scale of inputs. 
\subsection{Structure and content of this text}
\label{sec:IntrStrct}
We briefly describe the structure and content of the present text.

In {\bf Section \ref{S:Prelim}},
we collect known facts from functional analysis and 
GM theory which are required throughout this text.
In particular, we
review constructions and results on GMs on separable Hilbert and
Banach spaces.  
Special focus
will be on constructions via
countable products of univariate GMs on countable products of real
lines.  We also review assorted known results on convergence rates of
Lagrangian finite elements for linear, second order, divergence-form
elliptic PDEs in polytopal domains $\domain$ with Lipschitz boundary
$\partial \domain$.

In {\bf Section \ref{sec:EllPDElogN}}, 
we address the analyticity and sparsity for 
elliptic divergence-form PDEs with log-Gaussian coefficients. 
In Section \ref{S:PbmStat}, we introduce a model linear, second
order elliptic divergence-form PDE with log-Gaussian coefficients, with
variational solutions in the ``energy space'' $H^1_0(\domain)$.  This
equation was investigated with parametric input data in a number of
references in recent years
\cite{CoDeSch,CoDeSch1,HS,CCS13_783,BCDM,BCDS,Dung19,GKNSSS,KSSS17_2112,ZDS19}.
It is considered in this work mainly to develop the holomorphic
approach to establish our mathematical approach to parametric
holomorphy and sparsity of Wiener-Hermite PC expansions of parametric
solutions in a simple setting, and to facilitate comparisons with the
mentioned previous works and results.
 We review known results on its
well-posedness in Section
\ref{S:PbmStat}, and Lipschitz continuous dependence on the input
data in Section \ref{S:LipCont}.  We discuss regularity results for
parametric coefficients in Section \ref{S:Reg}.
Sections \ref{S:RndDat} and \ref{S:ParCoef} describe uncertainty
modelling by placing GMs on sets of admissible, countably parametric
input data, i.e., formalizing mathematically aleatoric uncertainty in
input data.  Here, the Gaussian series introduced in Section
\ref{S:GSer} will be seen to take a key role in converting operator
equations with GRF inputs to infinitely-parametric, deterministic
operator equations.  The Lipschitz continuous dependence of the
solutions on input data from function spaces will imply strong
measurability of corresponding random solutions, and render
well-defined the \emph{uncertainty propagation}, i.e., the
push-forward of the GM on the input data. In Section \ref{S:HolSumSol}, 
we connect the quantified holomorphy of the
parametric, deterministic solution manifold
$\{ u(\by): \by \in \RR^\infty \}$ in the space $H^1_0(\domain)$ with a sparsity
(weighted $\ell^2$-summability and $\ell^p$-summability)
of the coefficients  $(u_\bnu)_{\bnu \in \Ff}$ of the
($ H^1_0(\domain)$-valued) Wiener-Hermite PC expansion.
With 
this
methodology in place, we show in Section \ref{sec:HsReg} how
to obtain holomorphic regularity of the parametric solution family
$\{ u(\by) : \by\in \RRi\}$ in Sobolev spaces $H^s(\domain)$ of
possibly high smoothness order $s \in \NN$ 
and how to derive from here the corresponding sparsity.  
The argument is
self-contained and provides parametric holomorphy for any
differentiation order $s\in \NN$ in a unified way, in domains
$\domain$ of sufficiently high regularity and for sufficiently high
almost sure regularity of coefficient functions.  
In Section
\ref{sec:KondrReg}, we extend these results for linear second order
elliptic differential operators in divergence form in a bounded
polygonal domain $\domain \subset \mathbb{R}^2$.  
Here, corners are
well-known to obstruct 
high almost sure pathwise regularity in the
usual Sobolev and Besov spaces in $\domain$ for
both, PC coefficients and parametric solutions.
Therefore, we develop summability of the  
Wiener-Hermite PC
expansion coefficients $(u_\bnu)_{\bnu \in \Ff}$ of the random
solutions in terms of corner-weighted Sobolev spaces, 
originating with
V.A. Kondrat'ev (see, e.g., \cite{Gr,BLN,MazRoss2010} and the references there).
In Section \ref{Some related results on sparsity}, we briefly recall some known related results  
	\cite{CCS13_783,CoDe,CoDeSch,CoDeSch1,HS, BCM,BCMI,BCDM,BCDS}
	on $\ell^p$-summability and weighted $\ell^2$-summability of the generalized PC expansion coefficients of 
	solutions to parametric divergence-form elliptic PDEs, 
	as well as applications to best $n$-term approximation.

In {\bf Section \ref{sec:SumHolSol}}, 
we investigate sparsity of the Wiener-Hermite PC expansions coefficients of holomorphic functions.
In Section~\ref{S:DefbxdHol}, we introduce a concept of $(\bb,\xi,\delta,X)$-holomorphy of  
parametric  deterministic functions on the parameter domain $\R^\infty$
taking values in a separable Hilbert space $X$.  
This concept is fairly broad and covers a large range of parametric PDEs depending on log-Gaussian distributed data.
In order to extend the results and the approach to bound
Wiener-Hermite PC expansion coefficients via quantified holomorphy
beyond the simple, second order diffusion equation introduced in
Section \ref{sec:EllPDElogN}, we
address sparsity of the Wiener-Hermite PC expansions coefficients of  $(\bb,\xi,\delta,X)$-holomorphic functions.
In Section \ref{sec:bdX}, we show that  composite functions of  a certain type
are $(\bb,\xi,\delta,X)$-holomorphic  under certain conditions.
The significance of such functions is   
that they cover solution operators of a collection of linear, elliptic divergence-form 
PDEs in a unified way along with structurally similar PDEs with log-Gaussian random input data.
This will allow to apply the ensuing results on convergence rates of
deterministic collocation and quadrature algorithms to a wide range
of PDEs with GRF inputs and functionals on their random solutions. 
In Section  \ref{sec:HlExmpl}, we analyze some examples of holomorphic functions which are solutions to certain PDEs, 
including linear elliptic divergence-form PDEs with parametric diffusion coefficient, 
linear parabolic PDEs with parametric coefficient, 
linear elastostatics equations with log-Gaussian modulus of elasticity, 
Maxwell equations with log-Gaussian permittivity. 
	
In {\bf Section~\ref{sec:BIP}}, 
we apply the preceding abstract results on
parametric holomorphy to establish quantified holomorphy of
countably-parametric, posterior densities of corresponding BIPs where
the uncertain input of the forward PDE is a countably-parametric GRF
taking values in a separable Banach space of inputs. 
As an example, we analyze the BIP for  the parametric diffusion coefficient 
of the diffusion equation with parametric log-Gaussian inputs.

In {\bf Section~\ref{sec:StochColl}}, 
we discuss deterministic interpolation  and quadrature
algorithms for approximation and numerical integration of  $(\bb,\xi,\delta,X)$-holomorphic functions.
Such algorithms are necessary for the
approximation of certain statistical quantities (expectations,
statistical moments) of the parametric solutions with respect to a GM
on the parameter space.  
The proposed algorithms are variants and
generalizations of so-called ``stochastic collocation'' or
``sparse-grid'' type approximation, and proved to outperform sampling
methods such as MC methods, under suitable sparsity conditions on coefficients of the
Wiener-Hermite PC expansion of integrands. 
In the quadrature case, they are also known as ``Smolyak quadrature''
methods.  Their common feature is a) the deterministic nature of the
algorithms, and b) the possibility of achieving convergence rates
$>1/2$ independent of the dimension of parameters and therefore the
curse of dimensionality is broken.  They offer, in particular, the
perspective of deterministic numerical approximations for GRFs under
nonlinear pushforwards (being realized via the deterministic
data-to-solution map of the PDE of interest).  The decisive analytic
property to be established are dimension-explicit estimates of
individual Wiener-Hermite PC expansion coefficients of parametric
solutions, and based on these, sharp summability estimates of norms
of the coefficients  of Wiener-Hermite PC expansion of
parametric, deterministic solution families are given. 
In Sections \ref{Smolyak interpolation and quadrature} and \ref{sec:mi}, 
we construct sparse-grid Smolyak-type interpolation and quadrature algorithms.  
In Sections \ref{sec:intrate} and \ref{sec:quadrate}, 
we prove the convergence rates of interpolation and quadrature algorithms for  
$(\bb,\xi,\delta,X)$-holomorphic functions.

{\bf Section~\ref{sec:MLApprox}} 
is devoted to multilevel interpolation and quadrature of parametric holomorphic functions. 
We construct  
deterministic interpolation and quadrature algorithms
for generic $(\bb,\xi,\delta,X)$-holomorphic functions.
For linear second
order elliptic divergence-form PDEs with log-Gaussian coefficients, 
the results on the weighted $\ell^2$-summability of the Wiener-Hermite PC expansion
coefficients of parametric, deterministic solution families with
respect to corner-weighted Sobolev spaces on spatial domain $\domain$ finally
also allow to analyze methods for constructive, deterministic linear
approximations of parametric solution families. Here, a truncation of
Wiener-Hermite PC expansions is combined with approximating the
Wiener-Hermite PC expansion coefficients in the norm of the ``energy
space'' $ H^1_0(\domain)$ of these solutions from finite-dimensional
approximation spaces which are customary in the numerical approximation
of solution instances.  Importantly, \emph{required approximation
accuracies of the Wiener-Hermite PC expansion coefficients $u_\bnu$
will depend on the relative importance of $u_\bnu$ within the
Wiener-Hermite PC expansion}.  This observation gives rise to
\emph{multilevel approximations} where a prescribed overall
accuracy in mean square w.r.t.\ the GM $\gamma$ with respect to
$H^1_0(\domain)$ will be achieved by a $\bnu$-dependent discretization
level in the physical domain. 
Multilevel approximation and integration and the corresponding error estimates will
be developed in this section in an abstract setting: 
Besides $(\bb,\xi,\delta,X)$-holomorphy, 
it is neccessary to require an assumption on the discretization error in the physical domain 
in the form of stronger holomorphy of the approximation error in this discretization. 
A combined assumption for guaranteeing constructive multilevel approximations 
is formulated in Section \ref{sec:SetNot}. 
In Section \ref{sec:MLAlg} we introduce multilevel algorithms for 
interpolation and quadrature of $(\bb,\xi,\delta,X)$-holomorphic functions, 
and discuss work models and choices of discretization levels.
A key for the sparse-grid integration and interpolation approaches is to
efficiently numerically allocate discretization levels to
Wiener-Hermite PC expansion coefficients.  We develop such an approach
in Section \ref{app:mlweight}.  It is based on greedy searches and
suitable thresholding of (suitable norms of) Wiener-Hermite PC
expansion coefficients and on a-priori bounds for these quantities
which are obtained by complex variable arguments.
In Sections \ref{sec:MLInterpol} and \ref{sec:MLQuad}, 
we establish convergence rate bounds of multilevel interpolation and quadrature algorithms 
for  $(\bb,\xi,\delta,X)$-holomorphic functions.
In Section \ref{sec:Approx}, we verify the abstract hypotheses of the
sparse-grid multilevel approximations for
the forward and inverse problems for concrete linear 
elliptic and parabolic PDEs on corner-weighted Sobolev spaces (Kondrat'ev spaces) with log-GRF inputs.
In Section \ref{Multilevel approximation and quadrature in Bochner spaces}, 
we briefly recall some results from  \cite{dD21} (see also, \cite{dD-Erratum22} for some corrections) 
on  linear multilevel (fully discrete) interpolation and quadrature 
in abstract Bochner spaces based on weighted $\ell^2$-summabilities.  
These results are subsequently applied  
to parametric divergence-form elliptic PDEs and to parametric holomorphic functions.
\subsection{Notation and conventions}
\label{sec:IntrNotat}
Additional to the real numbers $\R$, the complex numbers $\C$, and
  the positive integers $\N$, we set $\R_+:=\set{x\in\R}{x\ge 0}$
and $\N_0:=\{0\}\cup\N$.
We denote by $\RR^\infty$ the set of all
sequences $\by = (y_j)_{j\in \NN}$ with $y_j\in \RR$, and
similarly define $\C^\infty$, $\R_+^\infty$ and $\N_0^\infty$.
 Both,
$\RR^\infty$ and $\CC^\infty$, will be understood with the product
topology from $\RR$ and $\CC$, respectively. 
For 
$\balpha$, $\bbeta \in \NN_0^d$,
$d\in\N\cup\{\infty\}$,
the inequality $\bbeta \leq \balpha$ is
understood component-wise, i.e., 
$\bbeta\leq \balpha$ if and only if
$\beta_j\leq \alpha_j$ for all $j$.

Denote by $\Ff$ the
countable set of all sequences of nonnegative integers $\bnu = (\nu_j)_{j \in \NN}$ 
such that $\supp(\bnu)$ is
finite, where $\supp(\bnu) := \{ j\in \NN: \nu_j \ne 0\}$ denotes the
``support'' of the multi-index $\bnu$.  
Similarly, we define
$\supp(\brho)$ of a sequence $\brho \in \RR^\infty_+$.  
For
$\bnu\in \Ff$, and for a sequence $\bb=(b_j)_{j\in \NN}$ of positive
real numbers, the quantities
\begin{equation*}
  \bnu! := \prod_{j \in \NN}\nu_j!\,,\qquad 
  |\bnu|:=\sum_{j \in \NN}\nu_j,
  \qquad\text{and}\qquad 
  \bb^\bnu := \prod_{j \in \NN}b_j^{\nu_j}
\end{equation*}
are finite and well-defined. 

For a multi-index $\balpha \in \NN_0^d$ and a function $u(\bx,\by)$ of
$\bx\in \RR^d$ and parameter sequence $\by\in \RR^{\infty}$ we use the
notation $D^\balpha u(\bx,\by)$ to indicate the partial derivatives
taken with respect to $\bx$.  The partial derivative of order
$\balpha \in \NN_0^\infty$ with respect to $\by$ \emph{of finite total
  order} $|\balpha| = \sum_{j\in \NN} \alpha_j$ is denoted by
$\partial^\balpha u(\bx,\by)$.  
In order to simplify notation, we will systematically suppress the
variable $\bx\in \domain\subset \R^d$ in mathematical expressions,
except when necessary.  For example, instead
$\int_\domain v(\bx)\rd \bx$ we will write $\int_\domain v\rd \bx$,
etc.  For a Banach space $X$, we denote by $X_{\CC}:=X+ \im X$ the
complexification of $X$.  The space $X_\CC$ is also a Banach space
endowed with the (minimal, among several possible equivalent ones, see
\cite{mst99}) norm
$\|x_1+\im x_2\|_{X_\CC}:=\sup_{0\leq t\leq 2\pi}\|x_1\cos t-x_2 \sin
t \|_X$.  The space $X^\infty$ is defined in a similar way as
$\RR^\infty$.

By $\cL(X,Y)$ we denote the vector space of bounded, linear operators
between to Banach spaces $X$ and $Y$. With $\cL_{{\rm is}}(X,Y)$ we
denote the subspace of boundedly invertible, linear operators from $X$
to $Y$.

For a function space $X(\D)$ defined on the domain $\D$, if there is
no ambiguity, when writing the norm of $x\in X(\D)$ we will omit $\D$,
i.e., we write $\|x\|_X$ instead of $\|x\|_{X(\D)}$.

For $0 < p \le \infty$ and a finite or countable index set $J$, we
denote by $\ell^p(J)$ the quasi-normed space of all
$\by = (y_j)_{j\in J}$ with $y_j \in \R$, equipped with the quasi-norm
$\|\by\|_{{\ell^p(J)}}:= \big(\sum_{j \in J}|y_j|^p\big)^{1/p}$ for
$p < \infty$, and $\|\by\|_{{\ell^\infty(J)}}:= \sup_{j \in J}|y_j|$.
Sometimes, we make use of the abbreviation $\ell^p=\ell^p(J)$ in a
particular context if there is no misunderstanding of the
meaning. We denote by $(\bee_j)_{j\in J}$ the standard basis of
  $\ell^2(J)$, i.e., $\bee_j = (e_{j,i})_{i\in J}$ with $e_{j,i} = 1$
  for $i=j$ and $e_{j,i} = 0$ for $i\not=j$.
\newpage

\section{Preliminaries}
\label{S:Prelim}
A key technical ingredient in the analysis of numerical approximations
of PDEs with GRF inputs from function spaces, and of numerical methods
for their efficient numerical treatment are constructions and
numerical approximations of GRFs on real Hilbert and Banach spaces.
Due to their high relevance in many areas of science (theoretical
physics, quantum field theory, spatial and high-dimensional
statistics, etc.), a rich theory has been developed in the past
decades and a large body of literature is available now.  We
recapitulate basic definitions and key results, in particular on GMs,
that are necessary for the ensuing developments.  We do not attempt to
provide a comprehensive survey.  We require the exposition on GMs
 on real-valued Hilbert and Banach spaces, as most PDEs of
interest are formulated for real-valued inputs and solutions.
However, we crucially use in the ensuing sections of this text
analytic continuation of parametric representations to the complex
parameter domain. This is required in order to bring to bear complex
variable methods for derivative-free, sharp bounds on Hermite
expansion coefficients of GRFs.  
Therefore, we develop in our
presentation solvability, well-posedness and regularity for the
PDEs that are subject to GRF inputs 
in Hilbert and Banach spaces of complex-valued fields.

The structure of this section is as follows.  In Section
\ref{S:FinDimGM}, we recapitulate GMs on finite dimensional spaces, in
particular on $\R^d$ and $\C^d$.  In Section \ref{S:GMSepHS}, we
extend GMs to separable Banach spaces.  
Section \ref{S:CamMart} reviews the Cameron-Martin space.  
In Section \ref{sec-Gaussian product measures} 
we recall a notion of Gaussian product measures on a
Cartesian product of locally convex spaces.  
Section \ref{S:GSer} is
devoted to a summary of known representations of a GRF by a Gaussian
series.  
A key object in these and more general spaces is the concept
of \emph{Parseval frame} which we introduce. 
For details, the reader may consult, for example, 
the books \cite{AdlerGeoGRF,Lifshits95,Bogach98}.

In Section \ref{S:FEM} we recapitulate, from
\cite{BaPi79,BNZPolygon,GasMorFESing09}, (known) technical results on
approximation properties of Lagrangian Finite Elements (FEs for short)
in polygonal and polyhedral domains $\domain\subset \R^d$, on regular,
simplicial partitions of $\domain$ with local refinement towards
corners (and, in space dimension $d=3$, towards edges).  These will be
used in Section \ref{sec:StochColl} in conjunction with collocation
approximations in the parameter space of the GRF to build
deterministic numerical approximations of solutions in polygonal and
in polyhedral domains.
\subsection{Finite dimensional Gaussian measures}
\label{S:FinDimGM}
\subsubsection{Univariate Gaussian measures} \index{Gaussian!measure}
\label{S:1dGM}
In dimension $d=1$, for every $\mu,\sigma\in \R$, there holds the
well-known identity
\[
\frac{1}{\sigma \sqrt{2\pi}} \int_{\R} \exp\left(-
\frac{(y-\mu)^2}{2\sigma^2}\right) \rd y = 1\;.
\]
A Borel probability measure $\gamma$ on $\R$ is \emph{Gaussian} if it
is either a Dirac measure $\delta_\mu$ at $\mu\in \R$ or its
density with respect to Lebesgue measure $\lambda$ on $\R$  
is given by
\begin{equation*}\label{eq:GdnsR1}
	\frac{\rd\gamma}{\rd\lambda} = p(\cdot;\mu,\sigma^2)\;,\;\;
	p(\cdot;\mu,\sigma^2) 
	:= 
	y\mapsto \frac{1}{\sigma\sqrt{2\pi}} 
	\exp\left(-\frac{(y-\mu)^2}{2\sigma^2}\right)
	\;.
\end{equation*}
We shall refer to $\mu$ as \emph{mean}, and to $\sigma^2$ as
\emph{variance} of the GM $\gamma$.  The case that
$\gamma = \delta_\mu$ is understood to correspond to $\sigma = 0$.  If
$\sigma > 0$, we shall say that \emph{the GM $\gamma$ is
	nondegenerate}.  
Unless explicitly stated otherwise, we assume GMs to be nondegenerate.

For $\mu=0$ and $\sigma = 1$, we shall refer to the GM $\gamma_1$ as
\emph{the standard GM on $\R$}.  A GM with $\mu = 0$ is called
\emph{centered} (or also \emph{symmetric}).  
For a GM $\gamma$ on $\R$, there holds
\begin{equation*}\label{eq:musigmaR1}
\mu = \int_{\R} y   \rd \gamma(y),\quad \sigma^2 = \int_{\R} (y-\mu)^2 \rd \gamma(y).
\end{equation*}
Let $(\Omega,\calA,\IP)$ be a probability space with sample space
$\Omega$, $\sigma$-fields $\calA$, and probability measure $\IP$.  A
\emph{Gaussian random variable} (``Gaussian RV'' for short) \index{Gaussian!random variable}
$\eta : \Omega\to \R$ is a RV whose law is Gaussian, i.e., it admits a
Gaussian distribution.  If $\eta$ is a Gaussian RV with mean $\mu$ and
variance $\sigma^2$ we write $\eta \sim \mathcal{N}(\mu,\sigma^2)$.

Linear transformations of Gaussian RVs are Gaussian: every Gaussian RV
$\eta$ can be written as $\eta = \sigma\xi + \mu$, where $\xi$ is a
standard Gaussian RV, i.e., a Gaussian RV whose law is a standard GM on $\R$.

The Fourier transformation of a GM $\gamma$ on $\R$ is defined, 
for every $\xi \in \R$, as
\begin{equation*}\label{eq:FTGMR1}
\hat{\gamma_1}(\xi) := \int_{\R} \exp(\im \xi y) \gamma(y)
= \exp\left( \im \mu \xi - \frac{1}{2} \sigma^2 \xi^2 \right) 
\;.
\end{equation*}
We denote by $\Phi$ the distribution function of $\gamma_1$.  
For the standard normal distribution
$$
\Phi(t) = \int_{-\infty}^t p(s;0,1)\rd s \qquad\forall t\in \R.
$$
With the convention $\Phi^{-1}(0) := -\infty$,
$\Phi^{-1}(1) := +\infty$, the inverse function $\Phi^{-1}$ of $\Phi$
is defined on $[0,1]$.
\subsubsection{Multivariate Gaussian measures} \index{Gaussian!measure}
\label{S:ddimGM}
Consider now a finite dimension $d > 1$.  A Borel probability measure
$\gamma$ on $(\R^d,\cB(\R^d))$ is called Gaussian if for every
$f\in \cL(\R^d,\R)$ the measure $\gamma\circ f^{-1}$ is a GM on $\R$,
where as usually, $\cB(\R^d)$ denotes the $\sigma$-field on $\R^d$.
Since $d$ is finite, we may identify $\cL(\R^d,\R)$ with $\R^d$, and
we denote the Euclidean inner product on $\R^d$ by $(\cdot,\cdot)$.
The Fourier transform of a Borel measure $\gamma$ on $\R^d$ is given
by
$$
\hat{\gamma}: \R^d \to \C: \hat{\gamma}(\bxi) = \int_{\R^d}
\exp\left(\im(\bxi,\by)\right) \rd \gamma( \by)\;.
$$
For a GM $\gamma$ on $\R^d$, the Fourier transform
$\hat{\gamma}$ uniquely determines $\gamma$.

\begin{proposition}[{\cite[Proposition 1.2.2]{Bogach98}}]\label{prop:FTGM}
	A Borel probability measure
	$\gamma$ on $\R^d$ is Gaussian iff
	$$
	\hat{\gamma}(\bxi) = \exp\left( \im (\bxi,\bmu)
	-\frac{1}{2}(\boldsymbol{K}\bxi,\bxi)\right), \quad \bxi\in \R^d\;.
	$$
	Here, $\bmu\in \R^d$ and $\boldsymbol{K}\in \R^{d\times d}$ is a
	symmetric positive semidefinite matrix.
	
	We shall say that a GM  $\gamma$ on $\R^d$ has a density with respect
	to Lebesgue measure $\lambda$ on $\R^d$ 
        iff the matrix $\boldsymbol{ K}$ is nondegenerate.  
        Then, this density is given by
	$$
	\frac{\rd\gamma}{\rd\lambda}(\bx): \bx\mapsto
	\frac{1}{\sqrt{(2\pi)^d \det \boldsymbol{ K}}} \exp\left( -\frac{1}{2}
	(\boldsymbol{ K}^{-1} (\bx-\bmu), \bx-\bmu) \right) \;.
	$$
	Furthermore,
	$$
	\bmu = \int_{\R^d} \by \rd \gamma(\by) , \quad \forall \by,\by'\in
	\R^d: (\boldsymbol{ K}\by,\by') = \int_{\R^d}
	(\by,\bx-\bmu)(\by',\bx-\bmu)\rd{\gamma}(\bx)\;.
	$$
	The symmetric linear operator $\cC\in \cL(\R^d,\R^d)$ defined by the
	later relation and represented by the symmetric positive definite
	matrix $\boldsymbol{K}$ is the \emph{covariance operator} associated
	to the GM $\gamma$ on $\R^d$.
\end{proposition}
When we do not need to distinguish between the covariance operator
$\cC$ and the covariance matrix $\boldsymbol{K}$, we simply speak of ``the
covariance'' of a GM $\gamma$.
If a joint probability distribution of RVs $y_1,\ldots,y_d$ is a GM on
$\R^d$ with mean vector $\bmu$ and covariance matrix $\boldsymbol{K}$
we write $(y_1,\ldots,y_d)\sim \mathcal{N}(\bmu,\boldsymbol{ K})$.

In what follows, we use $\gamma_d$ to denote the standard GM on $\R^d$.
Denote by $L^2(\RR^d;\gamma_d)$ the Hilbert space of all
$\gamma_d$-measurable, real-valued functions $f$ on $\R^d$ 
such that the norm
\begin{equation}\nonumber
	\|f\|_{L^2(\RR^d;\gamma_d)} := \left( \int_{\R^d} |f(\by)|^2 \rd \gamma_d(\by)\right)^{1/2} 
\end{equation}
is finite.
The corresponding inner product is denoted by
$(\cdot,\cdot)_{L^2(\RR^d;\gamma_d)}$.
\subsubsection{Hermite polynomials}
\label{S:HerPol} \index{Hermite polynomial}
A key role in the ensuing sparsity analysis of parametric solution
families is taken by Wiener-Hermite PC expansions.  We consider GRF
inputs and, accordingly, will employ polynomial systems on $\RR$ which
are orthogonal with respect to the GM $\gamma_1$ on $\RR$, the
so-called Hermite polynomials, as pioneered for the analysis of GRFs
by N. Wiener in \cite{Nwiener}.  To this end, we recapitulate basic
definitions and properties, in particular the various normalizations
which are met in the literature.  Particular attention will be paid to
estimates for Hermite coefficients of functions which are holomorphic
in a strip, going back to Einar Hille in \cite{EHilleII}.
\begin{definition}\label{def:HermPol}
	For $k\in \N_0$, the normalized probabilistic Hermite polynomial \index{Hermite polynomial!normalized probabilistic $\sim$}
	$H_k$ of degree $k$ on $\R$ is defined by
	\begin{equation}\label{eq:Hk}
		H_k(x) 
		:= 
		\frac{(-1)^k}{\sqrt{k!}} 
		\exp\left(\frac{x^2}{2}\right) \frac{\rd^k}{\rd x^k} \exp\left(-\frac{x^2}{2}\right) .
	\end{equation}
	
	For every multi-degree $\bnu\in \N_0^m$, 
        the $m$-variate Hermite
	polynomial $H_\bnu$ is defined by
	\begin{equation*}\label{eq:Hnu}
		H_\bnu(x_1,\ldots,x_m) := \prod_{j=1}^{m} H_{\nu_j}(x_j),
		\;\; x_j \in \R, \; j=1,\ldots ,m
                \;.
	\end{equation*}
\end{definition}
\begin{remark} {\rm[Normalizations of Hermite
	polynomials and Hermite functions]\label{rmk:HermNorml} 
	
	\begin{enumerate}
		\item  Definition \eqref{eq:Hk} provides for every $k\in \NN_0$ a
		polynomial of degree $k$.  The scaling factor in \eqref{eq:Hk} has
		been chosen to ensure normalization with respect to GM $\gamma_1$, see
		also Lemma \ref{lem:HkProp}, item (i).

		\item  Other normalizations with at times the same notation are used.
		The ``classical'' normalization of $H_k$ we denote by
		$\tilde{H}_k(x)$.  It is defined by (see, e.g., \cite[Page 787]{AS},
		and compare \eqref{eq:Hk} with \cite[Equation (5.5.3)]{szego})
		\[
		\tilde{H}_k(x/\sqrt{2}) := 2^{k/2} \sqrt{k!} H_k(x).
		\]
		
		\noindent
		\item  In \cite{BNT07}, so-called ``\emph{normalized Hermite polynomials}'' are introduced as \index{Hermite polynomial!normalized $\sim$}
		\begin{equation*}\label{eq:NormHermPol}
			\tilde{\tilde{H}}_k(x) 
			:= 
			[\sqrt{\pi} 2^k k!]^{1/2} (-1)^k \exp(x^2) \frac{\rd^k}{\rd x^k} \exp(-x^2)
			\;.
		\end{equation*}
		The system $(\tilde{\tilde{H}}_k)_{k\in \N_0}$ is an 
                orthonormal basis (ONB for short) for the space
		$L^2(\RR,\tilde{\tilde{\gamma}})$ with the weight
		$\tilde{\tilde{\gamma}} = \exp(-x^2) \rd x$, i.e., 
                (compare, e.g., \cite[Eqn. (5.5.1)]{szego})
		$$
		\int_{\RR} \tilde{\tilde{H}}_n(x) \tilde{\tilde{H}}_{n'}(x)
		\exp(-x^2)\rd x = \delta_{nn'}, \;\; n,n'\in \NN_0 \;.
		$$

		\item  With the Hermite polynomials $\tilde{\tilde{H}}_k$, in
		\cite{EHilleII} \emph{Hermite functions} are introduced for
		$k\in \NN_0$ as \index{function!Hermite $\sim$}
		\begin{equation*}\label{eq:HermFncs}
			h_k(x):= \exp(-x^2/2) \tilde{\tilde{H}}_k(x) \;,\quad x\in \RR\;.
		\end{equation*}

		\item  It has been shown in \cite[Theorem 1]{EHilleII} that in order for
		functions $f:\CC \to \CC$ defined in the strip
		$S(\rho) := \{z\in \CC: z=x+\im y, \; x\in \RR,\; |y|< \rho\}$ to
		admit a Fourier-Hermite expansion
		\begin{equation*}\label{eq:FourHermite}
		\sum_{n=0}^\infty f_n h_n(z) ,\qquad 
                f_n := \int_{\RR} f(x)h_n(x) \rd x = \int_{\RR} f(x) \tilde{\tilde{H}}_n(x) \exp(-x^2)\rd x
		\end{equation*}
		which converges to $f(z)$ for $z\in S(\rho)$ a necessary and
		sufficient condition is that a) $f$ is holomorphic in
		$S(\rho)\subset \CC$ and b) for every $0<\rho'<\rho$ there exist a
			finite bound $B(\rho')$ and $\beta$ such that
		\begin{equation*}\label{eq:fHermBd}
			|f(x+\im y)| \leq B(\rho') \exp[-|x|(\beta^2-y^2)^{1/2}]\;,\quad x\in \RR, |y|\leq \rho' \;.
		\end{equation*}
		There is a constant $C(f)>0$ such that for the Fourier-Hermite
		coefficients $f_n$, holds
		\begin{equation*}\label{eq:fnBd}
			|f_n| \leq C\exp(-\rho\sqrt{2n+1}) \quad  \forall n\in \NN_0.
		\end{equation*}
	\end{enumerate}
} \end{remark} 
We state some basic properties of the Hermite polynomials $H_k$
defined in \eqref{eq:Hk}.
\begin{lemma}\label{lem:HkProp}
	The collection $(H_k)_{k\in \NN_0}$ of Hermite polynomials
	\eqref{eq:Hk} in $\R$ has the following properties.
	\begin{enumerate}
		\item $(H_k)_{k\in \NN_0}$ is an ONB of the space
		$L^2(\R;\gamma_1)$.
		\item For every $k\in \NN$ holds:
		$H_k'(x) = \sqrt{k}H_{k-1}(x) = H_k(x) -\sqrt{k+1}H_{k+1}(x)$.
		\item For all $x_1,\ldots ,x_m \in \R$ holds
		\[
		\prod_{i=1}^m\sqrt{k_i!} H_{k_i}(x_i) =
		\frac{\partial^{k_1+\ldots +k_m}}{\partial t_1^{k_1}\ldots
			\partial t_m^{k_m}} \exp\left( \sum_{i=1}^m t_ix_i
		-\frac{1}{2}\sum_{i=1}^m t_i^2 \right)\mid_{t_1 = \ldots = t_m
			= 0}.
		\]
		\item For every $f\in C^\infty(\R)$ such that
		$f^{(k)}\in L^2(\R;\gamma_1)$ for all $k\in \NN_0$ holds
		\[
                 \int_{\R} f(x)H_k(x) \rd\gamma_1(x) = \frac{(-1)^k}{\sqrt{k!}
		\int_{\R} f^{(k)}(x) \rd\gamma_1(x)} \;,
		\]
		and, hence, in $L^2(\R;\gamma_1)$,
		\[
		f = \sum_{k\in \NN_0} \frac{(-1)^k}{\sqrt{k!}}
		(f^{(k)},1)_{L^2(\R;\gamma_1)} H_k \;.
		\]
	\end{enumerate}
\end{lemma}
It follows from item (i) of this lemma in particular that
\begin{equation*}\label{eq:HnuONB}
	\left\{ H_\bnu: \bnu \in \N_0^m \right\} \;\mbox{is an ONB of} \; L^2(\R^m;\gamma_m)\;.
\end{equation*}
Denote for $k\in \N_0$ and $m\in \N$ by $\calH_k$ the space of
$d$-variate Hermite polynomials which are homogeneous of degree $k$,
i.e.,
\[
\calH_k := {\rm span} \left\{ H_\bnu: \bnu \in \N_0^m, |\bnu| = k
\right\} \;.
\]
Then $\calH_k$ (``homogeneous polynomial chaos of degree $k$'' \cite{Nwiener}) 
is a 
closed, linear subspace of $L^2(\R^m;\gamma_m)$ and
\[
L^2(\R^m;\gamma_m) = \bigoplus_{k\in \NN_0} \calH_k
\;\;\mbox{in}\;\;L^2(\R^m;\gamma_m)\;.
\]
\subsection{Gaussian measures on separable locally convex spaces}
\label{S:GMSepHS}
An important mathematical ingredient in a number of applications, in
particular in UQ, Bayesian PDE inversion, risk analysis, but also in
statistical learning theory applied to input-output maps for PDEs, is
the construction of measures on function spaces.  A particular
interest is in GMs on separable on Hilbert or Banach or, more
generally, on locally convex spaces of uncertain input data for PDEs.
Accordingly, we review constructions of such measures, in terms of
suitable \emph{bases of the input spaces}. This implies, in
particular, \emph{separability} of the spaces of admissible PDE inputs
or, at least, the uncertain input data being \emph{a separably-valued}
random element of otherwise nonseparable spaces (such as, e.g.,
$L^\infty(\D)$) of valid inputs for the PDE of interest.

Let $(\Omega, \cA, \mu)$ be a measure space and $1 \le p \leq \infty$.
Recall that the normed space $L^p(\Omega,\mu)$ is defined as
the space of all $\mu$-measurable functions $u$ from $\Omega$ to $\R$
such that the norm
\begin{equation} \nonumber \|u\|_{L^p(\Omega,\mu)} := \
	\left(\int_{\Omega} |u(x)|^p \, \rd \mu(x) \right)^{1/p} < \infty.
\end{equation}
When $p=\infty$ the norm of $u\in L^\infty(\Omega,\mu)$ is given by
	$$
	\|u\|_{L^\infty(\Omega,\mu)} := \underset{x\in\Omega}{\esssup}|u(x)|.
	$$
	If $\Omega \subset \RR^m$ and $\mu$ is the Lebesgue measure, 
        we simply denote these spaces by $L^p(\Omega)$.

Throughout this section, $X$ will denote a real separable and locally
convex space with Borel $\sigma$-field $\cB(X)$ and with dual space
$X^*$.
\begin{example}\label{ex:R^infty}
{\rm	Let $\RRi$ be the linear space of all sequences
	$\by = (y_j)_{j\in \NN}$ with $y_j\in \RR$.  This linear space
	becomes a locally convex space (still denoted by $\RRi$) equipped
	with the topology generated by the countable family of semi-norms
	\[
	p_j(\by):= |y_j|, \quad j \in \N.
	\]
	The locally convex space $\RRi$ is separable and complete and,
	therefore, a Fr\'echet space. 
        However, it is not normable, and hence  not a Banach space.
}\end{example}
	\begin{example} {\rm \label{ex:SepBsp}
		Let $\domain\subset \R^d$ be an open bounded Lipschitz domain.
		\begin{enumerate}
			\item The Banach spaces $C(\overline{\domain})$ and $L^1(\domain)$ are separable.
			\item For $0<s<1$ we denote by $C^s(\domain)$ the space of $s$-H\"older
			continuous functions in $\domain$ equipped with the norm
			and seminorm \index{space!$s$-H\"older continuous $\sim$}
			\begin{equation*}\label{eq:CsNrm}
				\| a \|_{C^s} := \| a \|_{L^\infty} + | a |_{C^s}\;,
				\quad 
				| a |_{C^s} 
                                := \sup_{\bx,\bx'\in \domain, \bx\ne \bx'} \frac{|a(\bx) - a(\bx')|}{|\bx-\bx'|^s} 
                        \;.
			\end{equation*}
			Then the Banach space $C^s(\D)$  is not separable.  
                        A separable subspace is
			\begin{equation*}\label{eq:Cso}
				C^s_\circ(\domain) 
				:= 
				\bigg\{ a \in C^s(\domain): 
				\forall \bx\in \domain \; \lim_{\domain \ni \bx'\to \bx} \frac{|a(\bx)-a(\bx')|}{|\bx-\bx'|^s} = 0 
				\bigg\}
				\;.
			\end{equation*}
		\end{enumerate}
}	\end{example}

We review and present constructions of GMs $\gamma$ on $X$.
\subsubsection{Cylindrical sets}
\label{S:CylSet}
\emph{Cylindrical sets} are subsets of $X$ of the form
\begin{equation*}\label{eq:CylSet}
C = \left\{ x\in X: (l_1(x),\ldots ,l_n(x))\in C_0 : C_0 \in \cB(\R^n), l_i\in X^* \right\},
\;\mbox{for some}\; n\in \N\;.
\end{equation*}
Here, the Borel set $C_0 \in \cB(\R^n)$ is sometimes referred to as
\emph{basis of the cylinder $C$}.  
We denote by $\cE(X)$ the
$\sigma$-field generated by all cylindrical subsets of $X$. 
It is the
smallest $\sigma$-field for which all continuous linear functionals
are measurable.  Evidently then $\cE(X)\subset \cB(X)$, with in
general strict inclusion (see, e.g., \cite[A.3.8]{Bogach98}).  If,
however, $X$ is separable, then $\cE(X)=\cB(X)$ (\cite[Theorem
A.3.7]{Bogach98}).

Sets of the form
$$
\left\{ \by\in \R^\infty : (y_1,\ldots ,y_n)\in B, B \in \cB(\R^n),
n\in \N \right\}
$$
generate $\cB(\R^\infty)$ \cite[Lemma 2.1.1]{Bogach98}, and a set $C$
belongs to $\cB(X)$ iff it is of the form
$$
C = \left\{ x\in X: \, (l_1(x),\ldots ,l_n(x),\ldots ) \in B, \
\mbox{for}\ l_i\in X^*, B \in \cB(\R^\infty) \right\} \;,
$$
(see, e.g., \cite[Lemma 2.1.2]{Bogach98}).
%
\subsubsection{Definition and basic properties of Gaussian measures}
\label{S:DefGMX}
\begin{definition}[{\cite[Definition 2.2.1]{Bogach98}}]\label{def:GMX}
	A probability measure $\gamma$ defined on the
	$\sigma$-field $\cE(X)$ generated by $X^*$ is called Gaussian if,
	for any $f\in X^*$ the induced measure $\gamma\circ f^{-1}$ on $\R$
	is Gaussian.  The measure $\gamma$ is centered or symmetric if all
	measures $\gamma\circ f^{-1}$, $f\in X^*$ are centered.
	
	Let $(\Omega, \cA,\IP)$ be a probability space.  A random field $u$
	taking values in $X$ (recall that throughout, $X$ is a separable
	locally convex space) is a map $u:\Omega\to X$ such that
	$$
	\forall B\in \calB(X):\;\; u^{-1}(B)\in \calA\;.
	$$
	The \emph{law of the random field $u$} is the probability measure
	$\mfm_u$ on $(X,\calB(X))$ which is defined as
	$$
	\mfm_u (B):= \IP(u^{-1}(B)), \quad B \in \calB(X)\;.
	$$
	The random field $u$ is said to be \emph{Gaussian} if its law is a GM \index{Gaussian!random field}
	on $(X,\calB(X))$.
\end{definition}
Images of GMs under continuous affine transformations on $X$ are
Gaussian.
\begin{lemma}[{\cite[Lemma 2.2.2]{Bogach98}}]\label{lem:AffGM} Let $\gamma$ be a GM on $X$ and let $T:X\to Y$ be a linear
	map to another locally convex space $Y$ such that $l\circ T \in X^*$
	for all $l\in Y^*$.  Then $\gamma\circ T^{-1}$ is a GM on $Y$.
	
	This remains true for the affine map $x\mapsto Tx + \mu$ for some
	$\mu\in Y$.
\end{lemma}
The \emph{Fourier transform} of a measure $\mfm$ over $(X,\cB(X))$ is
given by
\begin{equation*}\label{eq:FTmeas}
	\hat{\mfm} : X^*\to \C: f\mapsto \hat{\mfm}(f) := \int_X \exp\left(\im f(x)\right) \rd\mfm( x) \;.
\end{equation*}
\begin{theorem}[{\cite[Theorem 2.2.4]{Bogach98}}]\label{thm:FTGMX}  A measure $\gamma$ on $X$ is Gaussian iff its Fourier
	transform $\hat{\gamma}$ can be expressed with some linear
	functional $L(\cdot )$ on $X^*$ and a symmetric bilinear form
	$B(.,.)$ on $X^*\times X^*$ such that $f\mapsto B(f,f)$ is
	nonnegative as
	\begin{equation}\label{eq:FTGMX}
		\forall f \in X^*:\quad 
		\hat{\gamma}(f) = \exp\left( \im L(f) - \frac{1}{2}B(f,f)\right).
	\end{equation}
\end{theorem}
A GM $\gamma$ on $X$ is therefore characterized by $L$ and $B$.  It
also follows from \eqref{eq:FTGMX} that a GM $\gamma$ on $X$ is
centered iff $\gamma(A) = \gamma(-A)$ for all $A\in \cB(X)$, i.e., iff
$L = 0$ in \eqref{eq:FTGMX}.
\begin{definition}\label{def:MeanCov}
	Let $\mfm$ be a measure on $\cB(X)$ such that
	$X^*\subset L^2(X,\mfm)$.  Then the element $a_\mfm\in (X^*)'$ in
	the algebraic dual $(X^*)'$ defined by
	$$
	{a}_\mfm(f) := \int_X f(x) \rd\mfm(x), \; f\in X^*,
	$$
	is called \emph{mean of $\mfm$}.
	
	The operator $R_\mfm:X^* \to (X^*)'$ defined by
	$$
	R_\mfm(f)(g) := \int_X [f(x) - {a}_\mfm(f)] [g(x) - {a}_\mfm(g)]
	{\rd\mfm( x)}
	$$
	is called \emph{covariance operator} of $\mfm$.  The quadratic form on
	$X^*$ is called \emph{covariance of $\mfm$}.
\end{definition}
When $X$ is a real separable Hilbert space, one can say more.
\begin{definition}[Nuclear operators]\label{def:NucOp}
	Let $H_1$, $H_2$ be real separable Hilbert spaces with the
	norms $\| \circ \|_{H_1}$ and $\| \circ \|_{H_2}$, respectively, and
	with corresponding inner products $(\cdot,\cdot)_{H_i}$, $i=1,2$.
	
	A linear operator $K\in \cL(H_1,H_2)$ is called \emph{nuclear} or
	\emph{trace class} if it can be represented as
	$$
	\forall u\in H_1: \quad Ku = \sum_{k\in \NN} (u, x_{1k})_{H_1} x_{2k}
	\; \mbox{in}\; H_2\;.
	$$
	Here, $( x_{ik})_{k\in \NN}\subset H_i$, $i=1,2$ are such that
	$\sum_{k\in \NN} \| x_{1k} \|_{H_1} \| x_{2k} \|_{H_2} < \infty$.
\end{definition}
We denote by $\cL_1(H_1,H_2)\subset \cL(H_1,H_2)$ the space of all
nuclear operators.  This is a separable Banach space when it is
endowed with \emph{nuclear norm}
\begin{equation*}\label{eq:NucNrm}
	\| K \|_1 
	:= 
	\inf
	\left\{ 
	\sum_{k\in \NN} \| x_{1k} \|_{H_1}  \| x_{2k} \|_{H_2} : Ku = \sum_{k\in \NN} (u, x_{1k})_{H_1} x_{2k}
	\right\}\,.
\end{equation*}
When $X = H_1 = H_2$, we also write $\cL_1(X)$.

\begin{proposition}[{\cite[Theorem 2.3.1]{Bogach98}}]\label{prop:GMHilb}  
	Let $\gamma$ be a GM on a separable Hilbert space $X$ with
	innerproduct $(\cdot,\cdot)_X$, and let $X^*$ denote its dual,
	identified with $X$ via the Riesz isometry.
	
	Then there exist $\mu\in X$ and a symmetric, nonnegative nuclear
	operator $K\in \cL_1(X)$ such that the Fourier transform
	$\hat{\gamma}$ of $\gamma$ is
	\begin{equation}\label{eq:GMHilb}
		\hat{\gamma}:X\to \C: x\mapsto \exp\left( \im (\mu,x)_X - \frac{1}{2} (Kx,x)_X \right) \,.
	\end{equation}
\end{proposition}
\begin{remark} {\rm\label{rmk:GMH}
	Consider that $X$ is a real, separable Hilbert space with
	innerproduct $(\cdot,\cdot)_X$ and assume given a GM $\gamma$ on
	$X$.
	\begin{enumerate}
		\item In \eqref{eq:GMHilb}, $K\in \cL(X)$ and $\mu\in X$ are
		determined by
		$$
		\forall u,v\in X: (\mu,v)_X = \int_X (x,v)_X \rd \gamma(x), \quad
		(Ku,v)_X = \int_X (u,x-\mu)_X (v,x-\mu)_X \rd \gamma(x) \;.
		$$
		The closure of $X=X^*$ in $L^2(X;\gamma)$ then equals the completion
		of $X$ with respect to the norm
		$x\mapsto \| K^{1/2}x\|_X = \sqrt{(Kx,x)_X}$.  Let $(e_n)_{n\in \NN}$
		denote the ONB of $X$ formed by eigenvectors of $K$, with
		corresponding real, non-negative eigenvalues $k_n \in \NN_0$, i.e.,
		$Ke_n = k_ne_n$ for $n=1,2,\ldots $.  Then the completion can be
		identified with the weighted sequence (Hilbert) space
		$$
		\left\{ (x_n)_{n\in \NN} : \sum_{n\in \NN} k_n x_n^2 <\infty \right\}
		\;.
		$$
		The nuclear operator $K$ is the \emph{covariance of the GM $\gamma$ on
			the Hilbert space $X$}.
		\item In coordinates $\by = (y_j)_{j\in \NN} \in \ell^2(\N)$
		associated to the ONB $(e_n)_{n\in \NN}$ of $X$, \eqref{eq:GMHilb}
		takes the form
		$$
		\hat{\gamma}: \ell^2(\N) \to \C: \by \mapsto \exp\left( \im \sum_{n\in
			\NN} a_n y_n - \frac{1}{2} \sum_{n\in \NN} k_n y_n^2 \right)\;.
		$$
		\item Consider $a=0\in X$ and, for finite $n\in \N$, a cylindrical set
		$C = P_n^{-1}(B)$ with $P_n$ denoting the orthogonal projection onto
		$X_n := {\rm span}\{e_j: j=1,\ldots ,n\}\subset X$, and with
		$B\in \cB(X_n)$.  Then
		$$
		\gamma(C) = \int_B \prod_{j=1}^n (2\pi k_j)^{-1/2}
		\exp\left(-\frac{1}{2k_j} y_j^2 \right)\rd y_1\ldots \rd y_n \;.
		$$
	\end{enumerate}
} \end{remark}

For $f\in X^*$ and $x\in X$, one frequently writes the $X^*\times X$
duality pairing as
\begin{equation*}\label{eq:<f,x>}
	f(x) = \langle f,x \rangle\;.
\end{equation*}
With the notation from Definition \ref{def:MeanCov}, the covariance
operator $C_g = R_{\gamma_g}$ in Definition \ref{def:MeanCov} of a
\emph{centered}, Gaussian random vector
$g:(\Omega,\cA;\gamma_g) \to X$ with Gaussian law $\gamma_g$ on a
separable, real Banach space $X$ admits the representations
\begin{equation*}\label{eq:CovKer}
	R_{\gamma_g} = C_g: X^*\to X: C_g \varphi:= \E\langle \varphi,g \rangle g,
	\quad 
	C_g: X^*\times X^*\to \R: (\psi,\varphi) \mapsto \langle \psi, C_g \varphi \rangle\;.
\end{equation*}
%
\subsection{Cameron-Martin space} 
\label{S:CamMart}\index{space!Cameron-Martin $\sim$}
%

Let $X$ be a real separable locally convex space and $\gamma$ a GM on
$\cE(X)$ such that $X^* \subset L^2(X;\gamma)$.  Then, for every
$\varphi\in X^*$, the image measure $\varphi(\gamma)$ is a GM on $\R$.
By \cite[Theorem 3.2.3]{Bogach98}, there exists a unique
$a_\gamma \in X$, the mean of $\gamma$, such that
$$
\forall \varphi\in X^*: \quad \varphi(a_\gamma) = \int_X \varphi(h)
\rd\gamma(h)\;.
$$
Denote by $X^*_\gamma$ the closure of the set
$\{\varphi - \varphi({a_\gamma)}, \, \varphi \in X^*) \}$ embedded
into the normed space $L^2(X;\gamma)$ w.r.t.\  its norm.

The covariance operator, $R_\gamma$, of $\gamma$ is formally given by
\begin{equation}\label{eq:Cov}
	\forall \varphi,\psi\in X^*:\quad 
	\langle R_\gamma \varphi,\psi \rangle 
	=
	\int_X \varphi(h-{a_\gamma})\psi(h-{a_\gamma}) \rd\gamma(h) 
	\;.
\end{equation}
As $X$ is a separable locally convex space, \cite[Theorem
3.2.3]{Bogach98} implies that there is a unique linear operator
$R_\gamma: X^*\to X$ such that \eqref{eq:Cov} holds. We define
$$
\forall \varphi \in X^*: \quad \sigma(\varphi):= \sqrt{\langle
	R_\gamma \varphi, \varphi \rangle} \;.
$$
If $h=R_\gamma \varphi$ for some $\varphi \in X^*$, the map
$h\mapsto \| h \|:=\sigma(\varphi)$ defines a norm on
${\rm range}(R_\gamma)\subset X$.  There holds \cite[Lemma
2.4.1]{Bogach98}
$\| h \| = \| h \|_{H(\gamma)} = \| \varphi \|_{L^2(X;\gamma)}$.

The \emph{Cameron-Martin space}  of the GM $\gamma$ on $X$ is the
completion of the range of $R_\gamma$ in $X$ with respect to the norm
$\| \circ \|$.
The Cameron-Martin space of the GM $\gamma$ on $X$ is denoted by
$H(\gamma)$.  It is also called the reproducing kernel Hilbert space
(RKHS for short) of $\gamma$ on $X$.

By \cite[Theorem 3.2.7]{Bogach98}, $H(\gamma)$ is a
separable Hilbert space, and $H(\gamma)\subset X$ with continuous
embedding, according to \cite[Proposition 2.4.6]{Bogach98}.  In case
that $X\subset Y$ for another Banach space, with continuous and linear
embedding, the Cameron-Martin spaces for $X$ and $Y$ coincide.  For
example, in the context of Remark \ref{rmk:GMH}, item~(i),
$H(\gamma) = K(X^*_\gamma)$.

Being a Hilbert space, introduce an innerproduct
$(\cdot,\cdot)_{H(\gamma)}$ on $H(\gamma)$ compatible with the norm
$\| \circ \|_{H(\gamma)}$ via the parallelogram law.  Then there holds
$$
\forall \varphi\in X^* \; \forall f\in H(\gamma): \quad (f,R_\gamma
\varphi)_{H(\gamma)} = \varphi(f)\;.
$$
Since $H(\gamma)$ is also separable, there is an ONB.
\begin{proposition}[{\cite[Theorem 3.5.10, Corollary 3.5.11]{Bogach98}}]\label{prop:ONBH}
	For a centered GM on a real, separable Banach space $X$ with norm
	$\| \circ \|_X$, there exists an ONB $( e_n )_{n\in \N}$ of the
	Cameron-Martin space $H(\gamma)\subset X$ such that
	$$
	\sum_{n\in \NN} \| e_n \|_X^2 <\infty \;, \qquad \forall \varphi \in
	X^*: \; R_\gamma\varphi = \sum_{n\in \NN} \varphi(e_n)e_n\;.
	$$
\end{proposition}
We remark that Proposition \ref{prop:ONBH} is not true for arbitrary
ONB $( e_n )_{n\in \NN}$ of $H(\gamma)$.

\subsection{Gaussian product measures}\index{Gaussian!product measures}
\label{sec-Gaussian product measures}
We recall a notion of product measures which gives an efficient method
to construct Gaussian measures on a countable Cartesian product of
locally convex spaces.
\begin{definition}[{Product measure, \cite[p.\ 372]{Bogach98}}]\label{def: product of measures} 
	Let $\mu_n$ be probability measures defined on $\sigma$-fields
	$\cB_n$ in locally convex spaces $X_n$.  Put
	$$
	X:= \prod_{n \in \N} X_n.
	$$
	Let $$\cB:= \bigotimes_{n \in \N} \cB_n$$ be the
	$\sigma$-field generated by all the sets of the form
	\begin{equation} \label{prod} B = B_1 \times B_2 \times \ldots
		\times B_n \times X_{n+1} \times X_{n+2}\times \ldots , \
		B_i \in \cB_i.
	\end{equation}
	The product measure $$\mu:= \bigotimes_{n \in \N} \mu_n$$ is
	the probability measure on $\cB$ defined by
	$\mu (B):= \prod_{i=1}^n \mu_i(B_i)$ for the sets $B$ of the
	form \eqref{prod}.
\end{definition}

\begin{example}[{\cite[Example 2.3.8]{Bogach98}}]\label{ex:Gaussian-prod-meas}
{\rm	Let $(\mu_n)_{n \in \N}$ be a sequence of GMs.  
        Then the product measure
	$\mu:= \otimes_{n \in \N} \mu_n$ is a GM on
	$X:= \prod_{n \in \N} X_n$. The Cameron-Martin space $H(\mu)$ of
	$\mu$ is the Hilbert direct sum of spaces $H(\mu_n)$, i.e.,
	\[
	H(\mu) = \left\{h = (h_j)_{j \in \N} \in X: \, h_j \in
	H(\mu_j), \|h\|^2_{H(\mu)} = \sum_{j \in \N} \|h_j\|^2_{H(\mu_j)} \right\}.
	\]
	The space $X^*_\mu$ is the set of all functions of the form
	\[
	\varphi \mapsto \sum_{j \in \N} f_j (\varphi_j), \quad f_j
	\in X^*_{\mu_j}, \quad \sum_{j \in \N} \sigma(f_j)^2 <
	\infty,
	\]
	and
	\[
	a_\mu(f) = \sum_{j \in \N} a_{\mu_j}(f_j), \quad \forall f
	= (f_j)_{j \in \N} \in X^*.
	\]
}
\end{example}
\begin{example}[{\cite[Example 2.3.5]{Bogach98}}]\label{ex:prod-measR^infty}
{\rm 	Denote by $(\gamma_{1,n})_{n \in \NN}$ a sequence of
	standard GMs on $(\R,\cB(\R))$.  Then the product measure
	$$
	\gamma = \bigotimes _{n\in \NN} \gamma_{1,n}
	$$
	is a centered GM on $\R^\infty$.  Furthermore,
	$H(\gamma) = \ell^2(\N)$ and
	$X^*_\gamma \simeq \ell^2(\N)$.  
        If $\mu$ is a GM on $\RRi$, 
        then by a result of Fernique, 
        the measures $\gamma$ and $\mu$ are
	either mutually singular or equivalent \cite[Theorem 2.12.9]{Bogach98}.  
        The locally convex space $\RRi$
	with the product measure $\gamma$ of standard GMs is
	the main parametric domain in the stochastic setting
	of UQ problems for PDEs with GRF inputs considered in
	this text.
}
\end{example}
%
\subsection{Gaussian series}
\label{S:GSer} \index{Gaussian!series}
A key role in the numerical analysis of PDEs with GRF
inputs from separable Banach spaces $\data$ is played by
representing these GRFs in terms of series with respect
to suitable \emph{representation systems}
$(\psi_j)_{j\in \NN}\in \data^\infty$ of $\data$ with
random coefficients.  There arises the question of
admissibility of $(\psi_j)_{j\in \NN}\in \data^\infty$
so as to allow a) to transfer randomness of function
space-valued inputs to a parametric, deterministic
representation (as is customary, for example, in the
transition from nonparametric to parametric models in
statistics) and b) to ensure suitability for numerical
approximation.

Items a) and b) are closely related to the selection of
stable bases for $\data$, with item b) mandating
additional requirements, such as efficient accessibility
for float point computations, quadrature, etc.

We first present an abstract result, Theorem
\ref{thm:AdFrmX} and then, in Sections \ref{S:GSerKL} and
\ref{S:GSerMRes}, we review several concrete
constructions of such series.  We discuss in Sections
\ref{S:GSerKL} and \ref{S:GSerMRes} several
examples, in particular the classical \KL Expansion
\cite{Karhunen46,StWrtMercer} of GRFs taking values in
separable Hilbert space.  All examples will be
admissible in parametrizing GRF input data for PDEs and
of Gaussian priors in the ensuing sparsity and
approximation rate analysis in Section
\ref{sec:EllPDElogN} and the following sections.
\subsubsection{Some abstract results}
\label{S:GserAbsRes}
We place ourselves in the setting of a real separable
locally convex space $X$, with a GM $\gamma$ on $X$, and
with associated Cameron-Martin Hilbert space
$H(\gamma)\subset X$ as introduced in Section
\ref{S:CamMart}.

We first consider expansions of Gaussian random vectors
with respect to orthonormal bases $(e_j)_{j\in \N}$ of
the Cameron-Martin space $H(\gamma)$.  As linear
transformations of GM are Gaussian (see Lemma
\ref{lem:AffGM}), we admit a linear transformation $A$.
\begin{theorem}[{\cite[Theorems. 3.5.1,
		3.5.7, (3.5.4)]{Bogach98}}]\label{thm:GSer}  Let $\gamma$ be a centered
	GM on a real separable locally convex space $X$ with
	Cameron-Martin space $H(\gamma)$ and with \emph{some}
	ONB $(e_j)_{j\in \NN}$ of $H(\gamma)$.  Let further
	denote $( y_j )_{j\in \NN}$ any sequence of
	independent standard Gaussian RVs on a probability
	space $(\Omega,\cA,\IP)$ and let $A\in \cL(H(\gamma))$
	be arbitrary.
	
	Then the Gaussian series
	$$
	\sum_{j\in \NN} y_j(\omega) A e_j
	$$
	converges $\IP$-a.s. in $X$.  The law of its limit is a centered GM
	$\lambda$ with covariance $R_\lambda$ given by
	$$
	R_\lambda(f)(g) = \left( A^* R_\gamma(f), A^*
	R_\gamma(g)\right)_{H(\gamma)} \;.
	$$
	Furthermore, there holds 
	of independent standard Gaussian RVs on a probability space
	$(\Omega,\cA,\IP)$.
	$$
	\int_X f(x) \gamma(\rd x) = \int_{\Omega} f\bigg(\sum_{j\in \NN}
	y_j(\omega) e_j\bigg) \rd \IP(\omega) \;.
	$$
	If $X$ is a real separable Banach space $X$ with norm $\| \circ \|_X$,
	for all sufficiently small constants $c>0$ holds
	$$
	\lim_{n\to\infty} \int_\Omega \exp\Bigg( c\bigg\| \sum_{j=n}^\infty
	y_j(\omega) Ae_j \bigg\|_X^2 \Bigg) \rd \IP(\omega) = 1\,.
	$$
	In particular, for every $p\in [1,\infty)$ we have
	$\big\| \sum_{j=n}^\infty y_j Ae_j \big\|_X^p \to 0$ in
	$L^1(\Omega,\IP)$ as $n\to \infty$.
\end{theorem}

Often, in numerical applications, ensuring orthonormality of the basis
elements could be computationally costly.  It is therefore of some
interest to consider Gaussian series with respect to more general
representation systems $(\psi_j)_{j\in \N}$.  An important notion is
\emph{admissibility} of such systems.
\begin{definition}\label{def:Adm}
	Let $X$ be a real, separable locally convex space, and let
	$g:(\Omega, \cA, \IP)\to X$ be a centered Gaussian random vector
	with law $\gamma_g = \IP_X$.  Let further $(y_j)_{j\in \N}$ be a
	sequence of i.i.d.\ standard real Gaussian RVs
	$y_j \sim \mathcal{N}(0,1)$.
	
	A sequence $(\psi_j)_{j\in \NN}\in X^\infty$ is called
	\emph{admissible for $g$} if
	$$
	\sum_{j\in\NN} y_j \psi_j \;\;\mbox{converges $\IP$-a.s. in }\;
	X\;\;\mbox{and}\;\; g =\sum_{j\in \NN} y_j \psi_j \;.
	$$
\end{definition}
To state the next theorem, we recall the notion of 
\emph{frames in separable Hilbert space} (see, e.g., \cite{HeilBases} 
and the references there for background and theory of frames.  
In the terminology of frame theory, 
Parseval frames correspond to tight frames with frame bounds equal to $1$).
\begin{definition}\label{def:ParsFrH} \index{Parseval frame}
	A sequence $(\psi_j )_{j\in \N}\subset H$ in a real separable
	Hilbert space $H$ with inner product $(\cdot,\cdot)_H$ is a
	\emph{Parseval frame of $H$} if
	$$
	\forall f\in H:\ \ \ f = \sum_{j\in \NN} (\psi_j,f)_H \,\psi_j
	\ \ \mbox{in}\ \ H\;.
	$$
\end{definition}
The following result, from \cite{LPFrame09}, characterizes admissible
affine representation systems for GRFs $u$ taking values in real,
separable Banach spaces $X$.
\begin{theorem}[{\cite[Theorem 1]{LPFrame09}}]\label{thm:AdFrmX} 
We have the following.
	\begin{enumerate}
		\item In a real, separable Banach space $X$ with a centered GM
		$\gamma$ on $X$, a representation system
		$\bPsi = (\psi_j)_{j\in \NN}\in X^\infty$ is \emph{admissible} for
		$\gamma$ iff $\bPsi$ is a Parseval frame for the Cameron-Martin
		space $H(\gamma)\subset X$, i.e.,
		\[
		\forall f\in H(\gamma): \quad \| f \|_{H(\gamma)}^2 = \sum_{j\in
			\NN} |\langle f,\psi_j\rangle|^2.
		\]
		\item Let $u$ denote a GRF taking values in $X$ with law $\gamma$
		and with RKHS $H(\gamma)$.  For a countable collection
		$\bPsi = (\psi_j)_{j\in \NN}\in X^\infty$ the following are
		equivalent:
		\begin{itemize}
			\item[(i)] $\bPsi$ is a Parseval frame of $H(\gamma)$ and
			\item[(ii)] there is a sequence $\by = (y_j)_{j\in \NN}$ of i.i.d
			standard Gaussian RVs $y_j$ such that there holds $\gamma-a.s.$
			the representation
			\[
			u = \sum_{j\in \NN} y_j \psi_j \quad \mbox{in}\quad H(\gamma)
			\;.
			\]
		\end{itemize}
		\item Consider a GRF $u$ taking values in $X$ with law $\gamma$ and
		covariance $R_\gamma\in \cL(X',X)$.  If $R_\gamma = SS'$ with
		$S\in \cL(K,X)$ for some separable Hilbert space $K$, for
		\emph{any} Parseval frame $\bPhi = (\varphi_j)_{j\in \NN}$ of $K$,
		the countable collection
		$\bPsi = S\bPhi = ( S\varphi_j )_{j\in \NN}$ is a Parseval frame
		of the RKHS $H(\gamma)$ of $u$.
	\end{enumerate}
\end{theorem}
The last assertion in the preceding result is \cite[Proposition
1]{LPFrame09}.  It generalizes the observation that for a symmetric
positive definite matrix $\bM$ in $\RR^d$, any factorization
$\bM = \bL\bL^\top$ implies that for $z\sim \calN(0,\bI)$ it holds
$\bL z \sim \calN(0,\bM)$.  The result is useful in building
customized representation systems $\bPsi$ which are frames of a GRF
$u$ with computationally convenient properties in particular
applications.

We review several widely used constructions of Parseval frames.  These
comprise expansions in eigenfunctions of the covariance operator $K$
(referred to also as principal component analysis, or as ``\KL
expansions''), but also ``eigenvalue-free'' multiresolution
constructions (generalizing the classical L\'{e}vy-Cieselski
construction of the Brownian bridge) for various geometric settings,
in particular bounded subdomains of euclidean space, compact manifolds
without boundary etc.  Any of these constructions will be admissible
choices as representation system of the GRF input of PDEs to render these
PDEs parametric-deterministic where, in turn, our parametric regularity
results will apply.
\begin{example}[Brownian bridge]\label{expl:BrBr} 
{\rm	On the bounded time interval $[0,T]$, consider the \emph{Brownian
		bridge} $(B_t)_{t\geq 0}$.  It is defined in terms of a Wiener
	process $(W_t)_{t\geq 0}$ by conditioning as
	\begin{equation}\label{eq:BrBr}
		(B_t)_{0\leq t\leq T} := \big\{ (W_t)_{0\leq t\leq T} | W_T = 0\big\}.
	\end{equation}
	It is a simple example of \emph{kriging} applied to the GRF $W_t$.
	
	The covariance function of the GRF $B_t$ is easily calculated as
	\[
	k_B(s,t) = \EE[B_sB_t] = s(T-t)/T \;\;\mbox{if}\;\; s < t.
	\]
	Various other representations of $B_t$ are
	\[
	B_t = W_t - \frac{t}{T}W_T = \frac{T-t}{\sqrt{T}} W_{t/(T-t)}.
	\]
	The RKHS $H(\gamma)$ corresponding to the GRF $B_t$ is the Sobolev
	space $H^1_0(0,T)$.
}
\end{example}
\subsubsection{\KL expansion}
\label{S:GSerKL} 
\index{expansion!Karhunen-Lo\`eve\ $\sim$}
A widely used representation system in the analysis and computation of
GRFs is the so-called \KL expansion (KL expansion for short) of GRFs,
going back to \cite{Karhunen46}.  
We present main ideas and definitions, 
in a generic setting of \cite{KOPP2018}, see also \cite[Chap. 3.3]{AdlerGeoGRF}.

Let $\cM$ be a compact space with metric
$\rho: \cM \times \cM \to \RR$ and with Borel sigma-algebra
$ \cB = \cB(\cM)$.  
Assume given a Borel measure $\mu$ on $(\cM,\cB)$.
Let further $(\Omega,\cA,\IP)$ be a probability space.  
Examples are
$\cM = \D$ a bounded domain in Euclidean space $\RR^d$, with $\rho$
denoting the Euclidean distance between pairs $(x,x')$ of points in
$\D$, and $\cM$ being a smooth, closed $2$-surface in $\RR^3$, where
$\rho$ is the geodesic distance between pairs of points in $\cM$.

Consider a measurable map
\[
Z: (\cM,\cB)\otimes (\Omega, \cA) \to \RR: (x,\omega)\mapsto
Z_x(\omega)\in \RR
\]
such that for each $x\in \cM$, $Z_x$ is a centered, Gaussian RV.  We
call the collection $(Z_x)_{x\in \cM}$ a \emph{GRF indexed by $\cM$}.

Assume furthermore for all $n\in \NN$, for all $x_1,\ldots,x_n\in \cM$
and for every $\xi_1,\ldots,\xi_n\in \RR$
\[
\sum_{i=1}^n \xi_iZ_{x_i} \;\;\mbox{is a centered Gaussian RV}.
\]
Then the \emph{covariance function}
$$K:\cM\times \cM \to \RR: (x,x')\mapsto K(x,x')$$ 
associated with the centered GRF $(Z_x)_{x\in\cM}$ is defined
pointwise by
\begin{equation*}\label{eq:CovFktK}
	K(x,x') := \EE[Z_{x} Z_{x'} ] \quad x,x'\in \cM\;.
\end{equation*}
Evidently, the covariance function $K:\cM\times \cM\to \RR$
corresponding to a Gaussian RV indexed by $\cM$ is a real-valued,
symmetric, and positive definite function, i.e., there holds
\begin{equation*}\label{eq:PosDefK}
	\forall n\in \NN \;
	\forall (x_j)_{1\leq j \leq n} \in \cM^n, 
	\forall (\xi_j)_{1\leq j \leq n}\in \RR^n:\; 
	\sum_{1\leq i,j\leq n} \xi_i\xi_j K(x_i,x_j) \geq 0 \;.
\end{equation*}
The operator $K\in \cL(L^2(\cM,\mu),L^2(\cM,\mu))$ defined by
\[
\forall f\in L^2(\cM,\mu): \;\; (Kf)(x) := \int_{\cM} K(x,x') f(x')
\rd\mu(x') \quad x\in \cM
\]
is a self-adjoint, compact positive operator on $L^2(\cM,\mu)$.
Furthermore, $K$ is trace-class and
$K(L^2(\cM,\mu)) \subset C(\cM,\RR)$.

The spectral theorem for compact, self-adjoint operators on the
separable Hilbert space $L^2(\cM,\mu)$ ensures the existence of a
sequence $\lambda_1 \geq \lambda_2 \geq \ldots \geq 0$ of real
eigenvalues of $K$ (counted according to multiplicity and accumulating
only at zero) with associated eigenfunctions $\psi_k\in L^2(\cM,\mu)$
normalized in $L^2(\cM,\mu)$, i.e., for all $k\in \NN$ holds
\[
K \psi_k = \lambda_k \psi_k \quad \mbox{in}\;\; L^2(\cM,\mu)\;,
\quad \int_{\cM} \psi_k(x) \psi_\ell(x) \rd\mu(x) =
\delta_{k\ell}\;, \;\; k,\ell \in \NN\;.
\]
Then, there holds $\psi_k \in C(\cM;\RR)$ and the sequence
$(\psi_k)_{k\in \NN}$ is an ONB of $L^2(\cM,\mu)$.  From Mercer's
theorem (see, e.g., \cite{StWrtMercer}), there holds the \emph{Mercer
	expansion}
\[
\forall x,x'\in \cM: \quad K(x,x') = \sum_{k\in \NN} \lambda_k
\psi_k(x) \psi_k(x')
\]
with absolute and uniform convergence on $\cM\times \cM$.
This result implies that
$$
\lim_{m\to\infty} \int_{\cM\times \cM} \left| K(x,x') - \sum_{j=1}^m
\lambda_j \psi_j(x)\psi_j(x') \right|^2 \rd\mu(x) \rd\mu(x') = 0\;.
$$
We denote by $H\subset L^2(\Omega,\IP)$ the $L^2(\cM,\mu)$ closure of
finite linear combinations of $(Z_x)_{x\in \cM}$.  This so-called
\emph{Gaussian} space (e.g. \cite{Janson97}) is a Hilbert space when
equipped with the $L^2(\cM,\mu)$ innerproduct.  Then, the sequence
$(B_k)_{k\in \NN}\subset \RR$ defined by
\begin{equation*}\label{eq:Coeff}
	\forall k\in \NN: \quad 
	B_k(\omega) 
	:= 
	\frac{1}{\sqrt{\lambda_k}} \int_{\cM} Z_x(\omega) \psi_k(x) \rd\mu(x)
	\in H
\end{equation*}
is a sequence of i.i.d, ${N}(0,1)$ RVs.  The expression
\begin{equation}\label{eq:KLSer}
	\tilde{Z}_x(\omega) := \sum_{k\in \NN} \sqrt{\lambda_k} \psi_k(x) B_k(\omega)
\end{equation}
is a modification of $Z_x(\omega)$, i.e., for every $x\in \cM$ holds
that $\IP(\{ Z_x = \tilde{Z}_x \}) = 1$, which is referred to as
\emph{\KL expansion of the GRF} $\{ Z_x: x \in \cM \}$.
\begin{example}
	{\rm[KL expansion of the Brownian
	bridge \eqref{eq:BrBr}]\label{expl:KLBrBr} 
        \index{Brownian bridge}
	On the compact interval $\cM = [0,T]\subset \RR$, 
        the KL expansion of the Brownian bridge is
	$$
	B_t = \sum_{k\in \NN} Z_k \frac{\sqrt{2T}}{k\pi} \sin(k\pi t/T)
	\;,\quad t\in [0,T] \;.
	$$
	Then
	$$H(\gamma) = H^1_0(0,T) = {\rm span}\{ \sin(k\pi t/T) : k\in \NN
	\}.$$
}
\end{example}
In view of GRFs appearing as diffusion coefficients in elliptic and
parabolic PDEs, criteria on their path regularity are of some
interest.  Many such conditions are known and we present some of
these, from \cite[Chapter 3.2, 3.3]{AdlerGeoGRF}.
\begin{proposition}\label{prop:C0Z}
	For any compact set $\cM \subset \RR^d$, if for $\alpha>0$,
	$\eta > \alpha$ and some constant $C>0$ holds
	\begin{equation}\label{eq:PthC0}
		\EE[|Z_{\bx+\bh} - Z_\bx|^\alpha] \leq C \frac{|\bh|^{2d}}{|\log|\bh||^{1+\eta}} \;,
	\end{equation}
	then $$\bx\to Z_\bx(\omega) \in C^0(\cM) \ \ \IP-a.s.$$
\end{proposition}
Choosing $\alpha = 2$ in \eqref{eq:PthC0}, we obtain for $\cM$ such
that $\cM = \overline{\domain}$, where $\domain \subset \RR^d$ is a
bounded Lipschitz domain, the sufficient criterion that there exist
$C>0$, $\eta>2$ with
\begin{equation*}\label{eq:PthC1}
	\forall \bx \in \domain: \quad 
	K(\bx+\bh,\bx+\bh) - K(\bx+\bh,\bx) - K(\bx,\bx+\bh) + K(\bx,\bx) 
	\leq C 
	\frac{|\bh|^{2d}}{|\log|\bh||^{1+\eta}} \;.
\end{equation*}
This is to hold for some $\eta>2$ with the covariance kernel $K$ of
the GRF $Z$, in order to ensure that
$[\bx\mapsto Z_\bx]\in C^1(\overline{\domain})\subset
W^{1}_\infty(\domain)$ $\IP$-a.s., 
see \cite[Theorem 3.2.5, page 49 bottom]{AdlerGeoGRF}.

Further examples of explicit \KL expansions of GRFs can be found in
\cite{LangChSSpher,LangKLCptSpc,KOPP2018} and a statement for
$\IP$-a.s H\"older continuity of GRFs $Z$ on smooth manifolds $\cM$ is
proved in \cite{AndrLang}.
%
\subsubsection{Multiresolution representations of GRFs}
\label{S:GSerMRes}
\KL expansions \eqref{eq:KLSer} provide an important source of concrete
examples of Gaussian series representations of GRFs $u$ in Theorem
\ref{thm:GSer}.  Since KL expansions involve the eigenfunctions of the
covariance operators of the GRF $u$, all terms in these expansions
are, in general, globally supported in the physical domain $\cM$
indexing the GRF $u$.  Often, it is desirable to have Gaussian series
representations of $u$ in Theorem \ref{thm:GSer} where the elements
$(e_n)_{n\in \NN}$ of the representation system are locally supported
in the indexing domain $\cM$.
\begin{example}[{\LC representation of Brownian
		bridge, \cite{CiesielBrBr61}}]\label{exple:LCBrBr} 
                \index{Brownian bridge} \index{expansion!Brownian bridge}
	{\rm
	Consider the Brownian bridge $(B_t)_{0\leq t\leq T}$ from Examples
	\ref{expl:BrBr}, \ref{expl:KLBrBr}.  For $T=1$, it may also be
	represented as Gaussian series (e.g. \cite{CiesielBrBr61})
	\begin{equation*}\label{eq:LCBrBr}
		B_t = 
		\sum_{j\in \NN} \sum_{k=0}^{2^j - 1} Z_{jk} 2^{-j/2} h(2^jt-k) 
		=
		\sum_{j\in \NN} \sum_{k=0}^{2^j - 1} Z_{jk}  \psi_{jk}(t)
		, \quad t\in \cM=[0,1]\;,
	\end{equation*}
	where $$\psi_{jk}(t):= 2^{-j/2} h(2^jt-k),$$ with
	$h(s):= \max\{ 1-2|s-1/2|,0 \}$ denoting the standard, continuous
	piecewise affine ``hat'' function on $(0,1)$.  Here, $\mu$ is the
	Lebesgue measure in $\cM = [0,1]$, and
	$Z_{jk} \sim \mathcal{N}(0,1)$ are i.i.d standard normal RVs.
	
	By suitable reordering of the index pairs $(j,k)$, e.g., via the
	bijection $(j,k) \mapsto j:= 2^j+k$, the representation
	\eqref{exple:LCBrBr} is readily seen to be a special case of Theorem
	\ref{thm:AdFrmX}, item ii).  The corresponding system
	$$ 
	\bPsi = \{ \psi_{jk} : j \in \NN_0, 0\leq k \leq 2^j-1 \}
	$$
	is, in fact, a basis for
	$C_0([0,1]) := \{ v\in C([0,1]): v(0) = v(1) = 0 \}$, the so-called
	\emph{Schauder basis}.
	
	There holds
	$$
	\sum_{j\in \NN} \sum_{k=0}^{2^j-1} 2^{js}|\psi_{jk}(t)| < \infty,
	\quad t\in [0,1]\;,
	$$
	for any $0\leq s < 1/2$.  The functions $\psi_{jk}$ are localized in
	the sense that $|{\rm supp}(\psi_{jk})| = 2^{-j}$ for
	$k=0,1,\ldots ,2^j-1.$
}
\end{example}
Further constructions of such \emph{multiresolution representations}
of GRFs with either Riesz basis or frame properties are available on
polytopal domains $M\subset \RR^d$, (e.g.\ \cite{bachmayr2018GRFRep},
for a needlet multiresolution analysis on the $2$-sphere $\cM = \IS^2$
embedded in $\RR^3$, where $\mu$ in Section \ref{S:GSerKL} can be
chosen as the surface measure see, also, for representation systems by
so-called spherical needlets \cite{NarPetWar06},
\cite{bachmayr2020multilevel}).

We also mention \cite{AyTaqq2003} for optimal approximation rates of
truncated wavelet series approximations of fractional Brownian random
fields, and to \cite{KOPP2018} for corresponding spectral representations.

Multiresolution constructions are also available on data-graphs $M$
(see, e.g., \cite{CoifEtAlDifWav06} and the references there).

\subsubsection{Periodic continuation of a stationary GRF}
Let $(Z_\bx)_{\bx\in \domain}$ be a GRF indexed by
$\domain \subset \RR^d$, where $\D$ is a bounded domain. We aim for
representations of the general form
\begin{equation}\label{eq-ref-k-01}
	Z_{\bx}= \sum_{j\in \NN}\phi_j(\bx) y_j,
\end{equation}
where the $y_j$ are i.i.d.\ $\calN(0,1)$ RVs and the
$(\phi_j)_{j\in \N}$ are a given sequence of functions defined on
$\D$. One natural choice of $\phi_j$ is
$\phi_j=\sqrt{\lambda_j} \psi_j$, where $\psi_j $ and are the
eigen-functions and $\lambda_j$ eigenvalues of the covariance
operator. However, \KL eigenfunctions on $\D$ are typically not
explicitly known and globally supported in the physical domain $\D$. 
One of the strategies for deriving better representations over
$\D$ is to view it as the restriction to $\D$ of a periodic Gaussian
process $Z_\bx^{\ext}$ defined on a suitable larger torus $\TT^d$.

Since $\D$ is bounded, without loss of generality, we may the physical
domain $\D$ to be contained in the box $[-\frac12, \frac12]^d$. We
wish to construct a periodic process $Z_\bx^{\ext}$ on the torus
$\TT^d$ where $\TT=[-\ell,\ell]$ whose restriction of $Z_\bx^{\ext}$
on $\D$ is such that $Z_\bx^{\ext}|_\D=Z_{\bx}$.  As a consequence,
any representation
$$
Z_\bx^{\ext}=\sum_{j\in \NN}y_j\tilde{\phi}_j
$$
yields a representation \eqref{eq-ref-k-01} where
$\phi_j=\tilde{\phi}_j|_{\D}$.

Assume that $(Z_\bx)_{\bx\in \domain}$ is a restriction of a
real-valued, stationary and centered GRF $(Z_\bx)_{\bx\in \R^d}$ on
$\R^d$ whose covariance is given in the form
\begin{equation}\label{eq-rep-k-03}
	\EE[Z_{\bx} Z_{\bx'} ] = \rho(\bx -\bx'), \quad \bx,\bx'\in \R^d,
\end{equation} 
where $\rho$ is a real-valued, even function and its Fourier transform
is a non-negative function.  The extension is feasible provided that
we can find an even and $\TT^d$-periodic function $\rho^{\ext}$ which
agree with $\rho$ over $[-1,1]^d$ such that the Fourier coefficients
$$
c_\bn(\rho^{\ext})=\int_{\TTd}\rho^{\ext}(\bxi) \exp\Big(-\im
\frac{\pi}{\ell}(\bn, \bxi)\Big)\dd \bxi,\qquad \bn\in \ZZ^d
$$
are non-negative.

A natural way of constructing the function $\rho^{\ext}$ is by
truncation and periodization. First one chooses a sufficiently smooth
and even cutoff function $\varphi_\kappa$ such that
$\varphi_\kappa|_{[-1,1]^d}=1$ and $\varphi_\kappa(\bx)=0$ for
$\bx\not \in [-\kappa,\kappa]^d$ where $\kappa=2\ell-1$. Then
$\rho^{\ext}$ is defined as the periodization of the truncation
$\rho \varphi_\kappa$, i.e.,
$$
\rho^{\ext}(\bxi)=\sum_{\bn\in \ZZ^d}(\rho \varphi_\kappa)(\bxi+2\ell
\bn).
$$ 
It is easily seen that $\rho^{\ext}$ agrees with $\rho$ over
$[-1,1]^d$ and
$$
c_\bn(\rho^{\ext})=\widehat{\rho
	\varphi_\kappa}\Big(\frac{\pi}{\ell}\bn\Big).
$$
Therefore $c_\bn(\rho^{\ext})$ is non-negative if we can prove that
$\widehat{\rho \varphi_\kappa}(\bxi)\geq 0$ for $\bxi\in \RR^d$. 
The following result is shown in \cite{bachmayr2018GRFRep}.
\begin{theorem}
	Let $\rho$ be an even function on $\R^d$ 
        such that
	\begin{equation}\label{eq-rep-k04}
		c (1+ |\bxi|^2)^{-s}\leq \hat{\rho}(\bxi)
                \leq  C (1+ |\bxi|^2)^{-r},\qquad \bxi \in \R^d 
	\end{equation}
	for some $s\geq r\geq d/2$ and $0<c\leq C$ and
	$$
	\lim\limits_{R\to +\infty}\int_{|\bx|>R} |\partial^\alpha
	\rho(\bx)|\dd \bx =0,\qquad |\balpha|\leq 2\lceil s\rceil.
	$$
	
	Then for $\kappa$ sufficiently large, there exists
	$\varphi_\kappa$ satisfying $\varphi_\kappa|_{[-1,1]^d}=1$ and
	$\varphi_\kappa(\bx)=0$ for $\bx\not \in [-\kappa,\kappa]^d$
	such that
	\begin{equation*}\label{eq-rep-k02}
		0<\widehat{\rho \varphi_\kappa} (\bxi) \leq C (1+ |\bxi|^2)^{-r},\qquad \bxi \in \R^d.
	\end{equation*}
\end{theorem}
The assertion in Theorem \ref{eq-rep-k04} implies that
$$0<c_\bn(\rho^{\ext}) \leq C (1+ |\bn|^2)^{-r},\qquad \bn\in \ZZ^d.$$
In the following we present an explicit construction of the
function $\varphi_\kappa$ for GRFs with Mat\'ern covariance
\begin{equation*}\label{eq:materndef}
	\rho_{\lambda,\nu}(\bx):=\frac{2^{1-\nu}}{\Gamma(\nu)} \bigg(\frac{\sqrt{2\nu}|\bx|}{\lambda}\bigg)^{\nu}K_\nu\bigg(\frac{\sqrt{2\nu}|\bx|}{\lambda}\bigg)\, ,  
\end{equation*}
where $\lambda>0$, $\nu>0$ and $K_\nu$ is the modified Bessel
functions of the second kind.  
Note that the Mat\'ern
covariances satisfy the assumption \eqref{eq-rep-k04} with $s=r=\nu+d/2$.

Let $P:=2\lceil \nu+\frac{d}{2}\rceil+1$ and $N_P$ be the
cardinal B-spline function with nodes $\{- P,\ldots,-1,0
\}$. For $\kappa>0$ we define the even function
$\varphi\in C^{P-1}(\R)$ by
\begin{equation*}\label{eq:bspline1}
	\varphi(t)=\begin{cases}
		1 & \text{if}\ \  |t|\leq \kappa/2\\[3pt]
		\displaystyle \frac{2 P}{\kappa}\int_{-\infty}^{t+\kappa/2} N_{P}\biggl(\frac{2 P}{\kappa}\xi\biggr)\dd\xi & \text{if}\ \ t\leq -\kappa/2\,.
	\end{cases}
\end{equation*}
It is easy to see that $\varphi(t)=0$ if $|t|\geq \kappa $.
We now define
\begin{equation*}\label{eq:bsplined}
	\varphi_\kappa(\bx):=\varphi(|\bx|).
\end{equation*}
With this choice of $\varphi_\kappa$, we have
$\rho^{\rm ext} = \rho_{\lambda,\nu}$ on $[-1,1]^d$ provided
that $ \ell \geq \frac{\kappa + \sqrt{d}}{2}.  $ The required
size of $\kappa$ is given in the following theorem, see
\cite[Theorem 10]{bachmayr2019unified}.

\begin{theorem}\label{thm:smoothcond}
	For $\varphi_\kappa$ as defined above, there exist constants
	$C_1, C_2 > 0$ such that for any $0<\lambda,\nu<\infty$, 
        we have
	$\widehat{\rho_{\lambda,\nu} \varphi_\kappa} > 0$ 
        provided that
	$\kappa > 1$ 
        and
	\begin{equation*}\label{kappacondition}
        \frac{\kappa}\lambda 
        \geq
         C_1 + C_2 \max\Big\{\nu^{\frac12} ( 1 + |{\ln \nu}|) , \nu^{-\frac12} \Big\}.
	\end{equation*}
\end{theorem}
\begin{remark} {\rm
	The periodic random field $Z_{\bx}^\mathrm{\ext}$ on $\mathbb{T}^d$
	provides a tool for deriving {series expansions} of the original
	random field. In contrast to the \KL eigenfunctions on $D$, which
	are typically not explicitly known, the corresponding eigenfunctions
	$\psi_j^{\ext}$ of the periodic covariance are explicitly known
	trigonometric functions and one has the following \KL expansion for
	the periodized random field:
	\[
	Z_\bx^\mathrm{ext} = \sum_{j\in \NN} y_j
	\sqrt{\lambda^\mathrm{\ext}_j} \,\psi_j^{\ext}, \quad y_j \sim
	\mathcal{N}(0,1)\ \text{ i.i.d.,}
	\]
	with $\lambda^\mathrm{ext}_j$ denoting the eigenvalues of the
	periodized covariance and the $\psi_j^{\ext}$ are normalized in
	$L^2(\mathbb{T}^d)$.  Restricting this expansion back to $D$, one
	obtains an exact expansion of the original random field on $D$
	\begin{equation}\label{nondstdexpansion}
		Z_\bx  = \sum_{j\in \NN} y_j \sqrt{\lambda^\mathrm{\ext}_j} \,\psi_j^{\ext}|_\D,
		\quad y_j \sim \mathcal{N}(0,1)\  \text{ i.i.d.,} 
	\end{equation}
	This provides an alternative to the standard KL expansion of $Z_\bx$
	in terms of eigenvalues $\lambda_j$ and eigenfunctions $\psi_j$
	normalized in $L^2(\D)$. The main difference is that the functions
	$\psi^\mathrm{ext}_j\big|_\D$ in \eqref{nondstdexpansion} are not
	$L^2(\D)$-orthogonal. However, these functions are given explicitly,
	and thus no approximate computation of eigenfunctions is required.

	The KL expansion of $Z_\bx^\ext$ also enables the construction of
	alternative expansions of $Z_\bx$ of the basic form
	\eqref{nondstdexpansion}, but with the spatial functions having
	additional properties.  In \cite{bachmayr2018GRFRep}, wavelet-type
	representations
	\[
	Z_\bx^{\ext} = \sum_{\ell,k} y_{\ell,k} \psi_{\ell, k} , \quad
	y_{\ell,k} \sim \mathcal{N}(0,1)\ \text{ i.i.d.,}
	\]
	are constructed where the functions $\psi_{\ell,k}$ have the same
	multilevel-type localisation as the Meyer wavelets. This feature
	yields improved convergence estimates for tensor Hermite polynomial
	approximations of solutions of random diffusion equations with
	log-Gaussian coefficients .
} \end{remark} 

\subsubsection{Sampling stationary GRFs}
The simulation of GRFs with specified covariance is a fundamental task
in computational statistics with a wide range of applications. In this
section we present an efficient methods for sampling such fields.
Consider a GRF $(Z_\bx)_{\bx\in \D}$ where $\D$ is contained in
$[-1/2,1/2]^d$. Assume that $(Z_\bx)_{\bx\in \domain}$ is a
restriction of a real-valued, stationary and centered GRF
$(Z_\bx)_{\bx\in \R^d}$ on $\R^d$ with covariance given in
\eqref{eq-rep-k-03}. Let $m\in \NN$ and $\bx_1,\ldots,\bx_M$ be
$M=(m+1)^d$ uniform grid points on $[-1/2,1/2]^d$ with grid spacing
$h=1/m$. We wish to obtain samples of the Gaussian RV
$$
\bZ=(Z_{\bx_1},\ldots,Z_{\bx_M})
$$
with covariance matrix
\begin{align} \label{covmatrix}
	\boldsymbol{\Sigma}=[\Sigma_{i,j}]_{i,j=1}^M,\qquad \Sigma_{i,j} =
	\rho(\bx_i - \bx_j) , \quad i,j = 1, \ldots, M.
\end{align}
Since $ \boldsymbol{\Sigma}$ is symmetric positive semidefinite, this
can in principle be done by performing the Cholesky factorisation
$ \boldsymbol{\Sigma} = \boldsymbol{F} \boldsymbol{F}^\top$ with
$ \boldsymbol{F} = \boldsymbol{\Sigma}^{1/2}$, from which the desired
samples are provided by the product $ \boldsymbol{F} \bY$ where
$ \bY \sim \mathcal{N}(0,\bI)$.  However, since $\Sigma$ is large and
dense when $m$ is large, this factorisation is prohibitively
expensive.  
Since the covariance matrix $ \boldsymbol{\Sigma}$ is a
nested block Toeplitz matrix under appropriate ordering, an efficient
approach is to extend $\boldsymbol{\Sigma}$ to a appropriate larger
nested block circulant matrix whose spectral decomposition can be
rapidly computed using FFT.

For any $\ell \geq 1$ we construct a $2\ell$-periodic extension of
$\rho$ as follows
\begin{align*}\label{period}
	\rho^\ext(\bx) = \sum_{\bn \in \ZZ^d} \bigl( \rho \chi_{(-\ell,\ell]^d} \bigr)(\bx + 2\ell \bn),  \quad \bx \in \R^d\,.
\end{align*}
Clearly, $\rho^\ext$ is $2 \ell$-periodic and $\rho^\ext = \rho $ on
$[-1,1]^d$. Denote $\bxi_1,\ldots,\bxi_s$, $s=(2\ell/h)^d$, the
uniform grid points on $[-\ell,\ell]^d$ with grid space $h$. Let
$\bZ^\ext=(Z_{\bxi_1},\ldots,Z_{\bxi_s})$ be the extended GRV with
covariance matrix $\boldsymbol{\Sigma}^{\rm ext}$ whose entries is
given by formula \eqref{covmatrix}, with $\rho$ replaced by
$\rho^\ext$ and $\bx_i$ by $\bxi_i$. Hence $ \boldsymbol{\Sigma}$ is
embedded into the nested circulant matrix
$ \boldsymbol{\Sigma}^{\rm ext}$ which can be diagonalized using FFT
(with log-linear complexity) to provide the spectral decomposition
\begin{equation*}
	\label{spectral_decomp}
	\boldsymbol{\Sigma}^{\rm ext} = \boldsymbol{Q}^{\rm ext} \boldsymbol{\Lambda}^{\rm ext} (\boldsymbol{Q}^{\rm ext})^\top,
\end{equation*}
with $\Lambda^{\rm ext}$ diagonal and containing the eigenvalues
$\lambda^{\rm ext}_j$ of $\boldsymbol{\Sigma}^{\rm ext} $ and
$\boldsymbol{Q}^{\rm ext}$ being a Fourier matrix. Provided that these
eigenvalues are non-negative, the samples of the grid values of $Z$
can be drawn as follows. First we draw a random vector
$(y_j)_{j=1,\ldots,s}$ with $y_j\sim \mathcal{N}(0,1)$ i.i.d., then
compute
\begin{equation*}\label{samples}
	\bZ^\text{ext}  =  \sum_{j=1}^s y_j \sqrt{\lambda^{\rm ext}_j} \, \bq_j
\end{equation*}
using the FFT, with $\bq_j$ the columns of $\boldsymbol{Q}^{\rm ext}$. 
Finally, a sample of $\bZ$ is obtained by extracting from
$\bZ^\text{ext}$ the entries corresponding to the original grid
points.

The above mentioned process is feasible, provided that
$\boldsymbol{\Sigma}$ is positive semidefinite. 
The following theorem
characterizes the condition on $\ell$ for GRF with Mat\'ern covariance
such that $\boldsymbol{\Sigma}^{\rm ext}$ is positive semidefinite,
see \cite{GKNSS}.

\begin{theorem}\label{thm:matern-growth}
	Let $1/2 \leq \nu < \infty$, $\lambda \leq 1$, and
	$h/\lambda \leq e^{-1}$. Then there exist $C_1, C_2 >0$ which may
	depend on $d$ but are independent of $\ell, h, \lambda, \nu$, such
	that $\Sigma^{\rm ext}$ is positive definite if
	\begin{equation*}
		\frac{\ell}{\lambda}  \ \geq  \  C_1\  + \ C_2\,  \nu^{\frac12} \, \log\bigl( \max\big\{ {\lambda}/{h}, \, \nu^{\frac12}\big\} \bigr) \, .
		\label{eq:alphahsmall}
	\end{equation*}
\end{theorem}
\begin{remark} {\rm
	For GRF with Mat\' ern covariances, it is well-known (see,
	e.g. \cite[Corollary 5]{GKNSSS}, \cite[eq.(64)]{BCM}) that the exact
	KL eigenvalues $\lambda_j$ of $Z_\bx$ in $L^2(\D)$ decay with the
	rate $
	\label{kldecay} \lambda_j \ \leq \ C j^{-(1+ 2 \nu/d)} $. It has
	been proved recently in \cite{bachmayr2019unified} that the
	eigenvalue $\lambda^{\rm ext}_j$ maintain this rate of decay up to a
	factor of order $\mathcal{O}(|\!\log h|^\nu)$.
} \end{remark} 
\subsection{Finite element discretization}
\label{S:FEM}
The approximation results and algorithms to be developed in the
present text involve, besides the Wiener-Hermite PC expansions with
respect to Gaussian co-ordinates $\by\in \R^\infty$, also certain
numerical approximations in the physical domain $\domain$.  Due to
their wide use in the numerical solution of elliptic and parabolic
PDEs, we opt for considering standard, primal Lagrangian finite
element  discretizations. We confine the presentation
and analysis to Lipschitz polytopal domains $\domain \subset \R^d$
with principal interest in $d=2$ ($\domain$ is a polygon with straight
sides) and $d=3$ ($\domain$ is a polyhedron with plane faces).  We
confine the presentation to so-called primal FE discretizations in
$\domain$ but hasten to add that with minor extra mathematical effort,
similar results could be developed also for so-called mixed, or dual
FE discretizations (see,e.g., \cite{BBF} and the references there).

In presenting (known) results on finite element method (FEM for short)
convergence rates, we consider separately FEM in polytopal domains
$\domain\subset \RR^d$, $d=1,2,3$, and FEM on smooth $d$-surfaces
$\Gamma\subset \RR^{d+1}$, $d=1,2$. See
\cite{BonDmlwEigFkt,DmlwHighOrderFEM}
\subsubsection{Function spaces}
\label{S:FncSpc}
For a bounded domain $\domain\subset \RR^d$, 
the usual Sobolev function spaces \index{space!Sobolev $\sim$}
of integer order
$s \in \NN_0$ and integrability $q\in [1,\infty]$ 
are denoted by
$W^{s}_q(\domain)$ with the understanding that
$L^q(\domain)=W^{0}_q(\domain)$.  
The norm of $v\in W^{s}_q(\domain)$
is defined by
\[
\|v\|_{ W^{s}_q} := \ \sum_{\balpha \in \ZZ^d_+: |\balpha| \le s}
\|D^\balpha v\|_{L^q} .
\]
Here  $D^\balpha$ denotes the partial weak derivative of order $\balpha$.  We refer to any standard text
such as \cite{Adams2nd} for basic properties of these spaces.
Hilbertian Sobolev spaces are given for $s\in \NN_0$ by
$H^s(\domain) = W^{s}_2(\domain)$, with the usual understanding that
$L^2(\domain) = H^0(\domain)$.

For $s\in \NN$, we call a $C^s$-domain $\domain\subset \RR^d$ a
	bounded domain whose boundary $\partial\domain$ is locally
	parameterized in a finite number of co-ordinate systems as a graph of
	a $C^s$ function.  In a similar way, we shall call
	$\domain\subset \RR^d$ a Lipschitz domain, when $\partial\domain$ is,
	locally, the graph of a Lipschitz function.  We refer to
	\cite{Adams2nd,GilbTr} and the references there or to \cite{Gr}.

We call \emph{polygonal domain} a domain $\domain \subset \RR^2$ that
	is a polygon with Lipschitz boundary $\partial\domain$ (which
	precludes cusps and slits) and with a finite number of straight sides.

Let $\domain\subset \RR^2$ denote an open bounded polygonal domain. 
We introduce in $\domain$ a nonnegative
function $r_\domain: \domain \to \RR_+$ which is smooth in  $\domain$, 
and which coincides for $\bx$ in a vicinity of each
corner $\bc\in\partial\domain$ with the Euclidean distance
$| \bx - \bc |$.

To state elliptic regularity shifts in $\domain$, 
we require
certain corner-weighted Sobolev spaces.  
We require these only for
integrability $q=2$ and for $q=\infty$.

For $s\in \NN_0$ and $\varkappa\in \RR$ we define \index{space!Kondrat'ev $\sim$}
\begin{equation*} \Kk^s_\varkappa(\domain): = \big\{ u: \domain \to
	\CC: \ r_\domain^{|\balpha|-\varkappa}D^\balpha u\in L^2(\domain),
	|\balpha|\leq s \big\}
\end{equation*}
and
\begin{equation*}
	\Ww^s_\infty(\domain):
	=
	\big\{u: \domain\to \CC: \ r_\domain^{|\balpha|}D^\balpha u\in L^\infty(\domain),\
	|\balpha|\leq s \big\}.
\end{equation*}
Here, for $\balpha\in \NN_0^2$ and as before $D^\balpha$ denotes the
partial weak derivative  of order $\balpha$.

The \emph{corner-weighted norms in these spaces} are given by
\begin{equation*}
	\|u\|_{ \Kk^s_\varkappa}
	:=
	\sum_{|\balpha|\leq s}\|r_\domain^{|\balpha|-\varkappa}D^\balpha u\|_{L^2}
	\qquad\text{and}\qquad
	\|u\|_{ \Ww^s_\infty}:=\sum_{|\balpha|\leq s}\|r_\domain^{|\balpha|}D^\balpha u\|_{L^\infty} \,.
\end{equation*}
The function spaces $\Kk^s_\varkappa(\domain)$ and
$\Ww^s_\infty(\domain)$ endowed with these norms are Banach spaces,
and $\Kk^s_\varkappa(\domain)$ are separable Hilbert spaces.  These
corner-weighted Sobolev spaces 
are called Kondrat'ev spaces.

An embedding of these spaces is
$ H^1_0(\domain) \hookrightarrow \Kk^1_0(\domain) $. 
This follows from the
existence of a constant $c(\domain)>0$ such that for every
$\bx \in \domain$ holds
$r_\domain(\bx) \geq c(\domain) {\rm dist}(\bx,\partial\domain)$.
%
\subsubsection{Finite element interpolation}
\label{S:FEIntrp}
In this section, 
we review some results on FE approximations \index{finite element  approximation} in polygonal domains
$\domain$ on locally refined triangulations $\cT$ in $\domain$. 
These
results are in principle known for the standard Sobolev spaces
$H^s(\domain)$ and available in the standard texts
\cite{Brenner,Ciarlet}.  
For spaces with corner weights in polygonal domains
$\domain\subset \RR^2$, such as $ \Kk^s_\varkappa $ and
$ \Ww^s_\infty $, however, which arise in the regularity of the
Wiener-Hermite PC expansion coefficient functions for elliptic PDEs in
corner domains in Section \ref{sec:KondrReg} ahead, we provide
references to corresponding FE approximation rate bounds.

The corresponding FE spaces involve suitable mesh refinement to
compensate for the reduced regularity caused by corner and edge
singularities which occur in solutions to elliptic and parabolic
boundary value problems in these domains.

We define the FE spaces in a polygonal domain 
$\domain\subset \RR^2$ (see \cite{Brenner,Ciarlet} for details). 
Let $\cT$ denote a regular
triangulation of $\overline \domain$, i.e., a partition of
$\overline\domain$ into a finite number $N(\cT)$ of closed,
nondegenerate triangles $T\in \cT$ (i.e., $|T|>0$) such that for any
two $T,T'\in \cT$, the intersection $T\cap T'$ is either empty, a
vertex or an entire edge.  We denote the \emph{meshwidth} of $\cT$ as
$$
h(\cT) := \max\{ h(T): T\in \cT \}, \;\;\mbox{where}\;\; h(T) :=
\mbox{diam}(T)\;.
$$
For $T\in \cT$, denote $\rho(T)$ the diameter of the largest circle
that can be inscribed into $T$.  We say $\cT$ is \emph{$\kappa$
	shape-regular}, if
$$
\forall T\in \cT: \;\; \frac{h(T)}{\rho(T)} \leq \kappa \;.
$$
A sequence $\mathfrak{T} := ( \cT_n)_{n\in \NN}$ is $\kappa$
shape-regular if each $\cT\in \mathfrak{T}$ is $\kappa$ shape-regular,
with one common constant $\kappa>1$ for all $\cT\in \mathfrak{T}$.

In a polygon $\domain$, with a regular, simplicial triangulation
$\cT$, and for a polynomial degree $m\in \NN$, the Lagrangian FE space
$S^m(\domain,\cT)$ of continuous, piecewise polynomial functions of
degree $m$ on $\cT$ is defined as
\begin{equation*}\label{eq:DefSm}
	S^m(\domain,\cT) = \{ v\in H^1(\domain): \forall T\in \cT: v|_{T} \in \IP_m \} \;.
\end{equation*}
Here, $\IP_m := {\rm span} \{ \bx^\balpha: |\balpha| \leq m \}$
denotes the space of polynomials of $\bx\in \RR^2$ of total degree at
most $m$.  We also define
$S^m_0(\domain,\cT) := S^m(\domain,\cT) \cap H^1_0(\domain)$.

The main result on FE approximation rates in a polygon
$\domain\subset \RR^2$ in corner-weighted spaces
$\Kk^s_\kappa(\domain)$ reads as follows.
\begin{proposition}\label{prop:FECorner}
	Consider a bounded polygonal domain $\domain \subset \RR^2$.  Then,
	for every polynomial degree $m\in \NN$, there exists a sequence
	$(\cT_n)_{n\in \NN}$ of $\kappa$ shape-regular, simplicial
	triangulations of $\domain$ such that for every
	$u\in (H^1_0\cap\Kk^{m+1}_\lambda)(\domain)$ for some $\lambda > 0$,
	the FE interpolation error converges at rate $m$.  
        More precisely, there
	exists a constant $C(\domain, \kappa, \lambda, m) > 0$ such that for
	all $\cT\in (\cT_n)_{n\in \NN}$ and for all
	$u\in (H^1_0\cap\Kk^{m+1}_\lambda)(\domain)$ holds
	$$
	\| u - I^m_{\cT}u \|_{H^1} \leq C h(\cT)^m \| u \|_{\Kk^{m+1}_\lambda}
	\;.
	$$
	Equivalently, in terms of the number $n :=\#(\cT)$ 
        of triangles, there holds
	\begin{equation}\label{eq:FEErr2d}
		\| u - I^m_{\cT}u \|_{H^1} \leq C n^{-m/2} \| u \|_{\Kk^{m+1}_\lambda} \;.
	\end{equation}
	Here, $I^m_{\cT}: C^0(\overline{\domain})\to S^m(\domain,\cT)$ denotes
	the nodal, Lagrangian interpolant.  The constant $C>0$ depends on $m$,
	$\domain$ and the shape regularity of $\cT$, but is independent of
	$u$.
\end{proposition}
For a proof of this proposition, we refer, for example, to \cite[Theorems 4.2,
4.4]{BNZPolygon}.

We remark that due to
$\Kk_{\lambda}^2(\domain) \subset C^0(\overline{\domain})$, the nodal
interpolant $ I^m_{\cT} $ in \eqref{eq:FEErr2d} is well-defined.  We
also remark that the triangulations $\cT_n$ need not necessarily be
nested (the constructions in \cite{BaPi79,BNZPolygon} do not provide
nestedness; for a bisection tree construction of $( \cT_n)_{n\in \NN}$
which are nested, such as typically produced by adaptive FE
algorithms, with the error bounds \eqref{eq:FEErr2d}, we refer to
\cite{GasMorFESing09}.

For similar results in polyhedral domains in space dimension $d=3$, we
refer to \cite{BacNisZik3dI,BacNisZik3dII,Li3dhFEM} and to the
references there.

\newpage
\section{Elliptic divergence-form PDEs with log-Gaussian coefficient}
\label{sec:EllPDElogN}
We present a model second order linear divergence-form PDE with
  log-Gaussian input data. We review known results on its
  well-posedness, and Lipschitz continuous dependence on the input
  data.  Particular attention is placed on regularity results in
  polygonal domains $\domain\subset \RR^2$.  
  Here, solutions belong to Kondrat'ev spaces.  
  We discuss regularity results for
  parametric coefficients, and establish in particular parametric
  holomorphy results for the coefficient-to-solution maps.

The outline of this section is as follows.  
In Section \ref{S:PbmStat}, we present the strong and variational forms of the
PDE, its well-posedness and the continuity of the data-to-solution map
in appropriate spaces.  Importantly, we do not aim at the most general
setting, but to ease notation and for simplicity of presentation we
address a rather simple, particular case: in a bounded domain
$\domain$ in Euclidean space $\RR^d$.  All the ensuing derivations
will directly generalize to linear second order elliptic systems.  A
stronger Lipschitz continuous dependence on data result is stated in
Section~\ref{S:LipCont}. 
Higher regularity and fractional regularity of the solution  
provided correspondingly by higher regularity of data are discussed in Section \ref{S:Reg}.

Sections \ref{S:RndDat} and \ref{S:ParCoef} describe uncertainty
modelling by placing GMs on sets of admissible, countably parametric
input data, i.e., formalizing mathematically aleatoric uncertainty in
input data.  Here, the Gaussian series introduced in Section
\ref{S:GSer} will be seen to take a key role in converting operator
equations with GRF inputs to infinitely-parametric, deterministic
operator equations.  The Lipschitz continuous dependence of the
solutions on input data from function spaces will imply strong
measurability of corresponding random solutions, and render
well-defined the \emph{uncertainty propagation}, i.e., the
push-forward of the GM on the input data.

In Sections \ref{S:HolSumSol}--\ref{sec:KondrReg}, we connect quantified holomorphy of the
  parametric, deterministic solution manifold
  $\{ u(\by): \by \in \RR^\infty \}$ with sparsity of the
  coefficients $(\norm[H]{u_\bnu})_{\bnu \in \Ff}$ of Wiener-Hermite PC
  expansion as elements of certain Sobolev spaces: 
  We
  start with the case $H=H_0^1(\domain)$ in Section \ref{S:HolSumSol}
  and subsequently discuss higher regularity $H=H^s(\domain)$,
  $s\in\N$, in Section \ref{sec:HsReg} and finally $H$ being a
  Kondrat'ev space on a bounded polygonal domain
  $\domain \subset \mathbb{R}^2$ in Section \ref{sec:KondrReg}.
\subsection{Statement of the problem and well-posedness}
\label{S:PbmStat}
In a bounded Lipschitz domain $\domain \subset \RR^d$
$(d=1,2\ \text{or}\ 3)$, consider the linear second order  elliptic PDE in divergence-form  \index{PDE!linear $\sim$}
\begin{equation}\label{PDE}
  P_a u := \left\{ \begin{array}{l} -  \div(a(\bx) \nabla u(\bx)) \\ \tau_0(u) \end{array} \right\}
  =
  \left\{\begin{array}{c} f(\bx) \;\;\mbox{in}\;\;\domain, \\ 0 \;\; \mbox{on} \;\;\partial \domain\;. 
         \end{array} 
       \right.
     \end{equation}
     Here, $\tau_0 : H^1(\domain)\to H^{1/2}(\partial\domain)$
     denotes the trace map.  With the notation $V:=H_0^1(\domain)$ and
     $V^*=H^{-1}(\domain)$, for any $f\in V^*$, by the Lax-Milgram
     lemma the weak formulation given by
     \begin{equation}\label{vf}
       u\in V:\;\;
       \int_\D a\nabla u\cdot \nabla v\, \rd \bx 
       =  
       \langle f,v \rangle_{V^*,V} \,,\qquad v\in V,
     \end{equation} 
     admits a unique solution $u \in V$ whenever the coefficient $a$
     satisfies the ellipticity assumption
     \begin{equation}\label{PDE-epllipticity}
       0 < a_{\min}: = \underset{\bx\in\D}{\essinf}\, a(\bx) \leq a_{\max} = \| a \|_{L^\infty} < \infty 
       \;.
     \end{equation}
     With $\| v \|_V := \| \nabla v \|_{L^2}$ denoting the norm of
     $v\in V$, there holds the a-priori estimate
     \begin{equation} \label{V-estimate} \|u\|_V\leq
       \frac{\|f\|_{V^*}}{a_{\min}}\,.
     \end{equation} 
     In particular, with
$$L^\infty_+(\domain) := \big\{ a\in L^\infty(\domain): a_{\min}>0\big\},$$
the \emph{data-to-solution operator}
\begin{equation}\label{eq:SolOp}
  \cS:L^\infty_+(\domain) \times V^* \to V: (a,f)\mapsto u
\end{equation}
is continuous.
\subsection{Lipschitz continuous dependence}
\label{S:LipCont}
The continuity \eqref{eq:SolOp} of the data-to-solution map $\cS$
allows to infer already strong measurability of solutions of
\eqref{PDE} with respect to random coefficients $a$.  For purposes of
stable numerical approximation, we will be interested in quantitative
bounds of the effect of perturbations of the coefficient $a$ in
\eqref{vf} and of the source term data $f$ on the solution
$u = \cS(a,f)$.  Mere continuity of $\cS$ as a map from
$L^\infty_+(\domain) \times V^*$ to $V = H^1_0(\domain)$ will not be
sufficient to this end.  To quantify the impact of uncertainty in the
coefficient $a$ on the solution $u\in V$, local H\"older or,
preferably, Lipschitz continuity of the map $\cS$ is required, at
least locally, close to nominal values of the data $(a,f)$.

To this end, consider given $a_1,a_2\in L^\infty_+(\domain)$,
$f_1,f_2\in L^2(\domain)\subset V^*$ with corresponding unique
solutions $u_i = \cS(a_i,f_i)\in V$, $i=1,2$.
\begin{proposition}\label{prop:LipS}
  In a bounded Lipschitz domain $\domain\subset \R^d$, for given data
  bounds $r_a,r_f\in (0,\infty)$, there exist constants $c_a$ and
  $c_f$ such that for every $a_i\in L^\infty_+(\domain)$ with
  $\| \log(a_i) \|_{L^\infty} \leq r_a$, and for every
  $f_i\in L^2(\domain)$ with $\| f_i \|_{L^2} \leq r_f$, $i=1,2$, it
  holds
  \begin{equation}\label{eq:LipBasic}
    \| u_1 - u_2 \|_V 
    \leq 
    \frac{c_P}{a_{1,\min}\wedge a_{2,\min}} \| f_1 - f_2 \|_{L^2} 
    +
    \frac{\| f_1 \|_{V^*} \vee \| f_2 \|_{V^*}}{a_{1,\min} a_{2,\min}} 
    \| a_1-a_2 \|_{L^\infty}
    \;.
  \end{equation}
  Therefore
  \begin{equation}\label{eq:LipCont}
    \| \cS(a_1,f_1) - \cS(a_2,f_2) \|_{V} 
    \leq 
    c_a  \| a_1-a_2 \|_{L^\infty} 
    + 
    c_f \| f_1-f_2 \|_{L^2}\;,
  \end{equation}
  and
  \begin{equation}\label{eq:LogLip}
    \| \cS(a_1,f_1) - \cS(a_2,f_2) \|_{V} 
    \leq 
    \tilde{c}_a \| \log(a_1) - \log(a_2) \|_{L^\infty} 
    +
    c_f \| f_1-f_2 \|_{L^2} \;.
  \end{equation}
  Here, we may take $c_f = c_P \exp(r_a)$, $c_a = c_P r_f\exp(2r_a)$
  and $\tilde{c}_a = c_P r_f \exp(3r_a)$.  The constant
  $c_P = c(\domain)>0$ denotes the $V-L^2(\domain)$ Poincar\'e
  constant of $\domain$.
\end{proposition}
The bounds \eqref{eq:LipCont} and \eqref{eq:LogLip} follow from the
continuous dependence estimates in \cite{BchMVK2019} by elementary
manipulations.  For a proof (in a slightly more general setting), we
also refer to Section \ref{sec:pdc} ahead.
\subsection{Regularity of the solution}
\label{S:Reg}
It is well known that weak solutions $u\in V$ of the linear elliptic
boundary value problem (BVP for short) \eqref{PDE} admit higher
regularity for more regular data (i.e., coefficient $a(\bx)$, source
term $f(\bx)$ and domain $\domain$).  Standard references for
corresponding results are \cite{Gr,GilbTr}. The proofs in these
references cover general, linear elliptic PDEs, with possibly
matrix-valued coefficients, and aim at sharp results on the Sobolev
and H\"older regularity of solutions, in terms of corresponding
regularity of coefficients, source term and boundar $\partial\domain$.
In order to handle the dependence of solutions on random field and
parametric coefficients in a quantitative manner, we develop presently
self-contained, straightforward arguments for solution regularity of
\eqref{PDE}.

Here is a first regularity statement, which will be used in several
places subsequently.
To state it, we denote by $W$ the normed space of all functions
$v \in V$ such that $\Delta v\in L^2(\domain)$. 
The norm in $W$ is defined by
\begin{equation*} \label{W-norm} \| v \|_W := \| \Delta v \|_{L^2}.
\end{equation*}
The map $v\mapsto \|v\|_W$ is indeed a norm on $W$ due to the
homogeneous Dirichlet boundary condition of $v\in V $:
$\| v \|_W = 0$ implies that $v$ is harmonic in $\domain$, and
$v\in V $ 
implies that the trace of $v$ on
$\partial \domain$ vanishes, 
whence $v=0$ in $\domain$ by the maximum principle.
\begin{proposition}\label{prop:LipCoef}
  Consider the boundary value problem \eqref{PDE} in a bounded domain
  $\domain$ with Lipschitz boundary, and with
  $a\in W^{1}_\infty(\domain)$, $f\in L^2(\domain)$.  Then the weak
  solution $u\in V$ of \eqref{PDE} belongs to the space $W$ and there
  holds the a-priori estimate
  \begin{equation} \label{W-estimate} \|u\|_W \ \le \
    \frac{1}{a_{\min}}\left(\| f \|_{L^2} + \|f \|_{V^*}\frac{\|\nabla
        a\|_{L^\infty}}{a_{\min}}\right) \le \frac{c}{a_{\min}}\left(
      1 + \frac{\|\nabla a\|_{L^\infty}}{a_{\min}}\right) \| f
    \|_{L^2}\;,
  \end{equation}
  where $a_{\min} = \min\{ a(\bx) : \bx\in \overline{\domain} \}$.
\end{proposition}
\begin{proof}
  That $u\in V$ belongs to $W$ is verified by observing that under
  these assumptions, there holds
  \begin{equation} \label{[Delta u=]} - a\Delta u \ = \ f + \nabla a
    \cdot \nabla u \;\;\mbox{in the sense of}\;\; L^2(\D)\;.
  \end{equation}
  The first bound \eqref{W-estimate} follows by elementary argument
  using \eqref{V-estimate}, the second bound by an application of the
  $L^2(\domain)$-$V^*$ Poincar\'{e} inequality in $\domain$.
\end{proof}
\begin{remark} {\rm\label{rmk:KW}
  The relevance of the space $W$ stems from the relation to the
  corner-weighted Kondrat'ev spaces $\Kk^m_\kappa(\domain)$ which were
  introduced in Section \ref{S:FncSpc}.  When the domain
  $\domain\subset \RR^2$ is a polygon with straight sides, in the
  presently considered homogeneous Dirichlet boundary conditions on
  all of $\partial\domain$, it holds that
  $W \subset \Kk^2_\kappa(\domain)$ with continuous injection provided
  that $|\kappa|<\pi/\omega$ where $0<\omega<2\pi$ is the largest
  interior opening angle at the vertices of $\domain$.  Membership of
  $u$ in $\Kk^2_\kappa(\domain)$ in turn implies optimal approximation
  rates for standard, Lagrangian FE approximations in $\domain$ with
  suitable, corner-refined triangulations in $\domain$, see
  Proposition \ref{prop:FECorner}.
} \end{remark} 
\begin{remark} {\rm\label{rmk:Dconvex}
  If the physical domain $\domain$ is convex or of type $C^{1,1}$,
  then $u\in W$ implies that $u\in (H^2\cap H^1_0)(\D)$ and
  \eqref{W-estimate} gives rise to an $H^2$ a-priori estimate (see,
  e.g., \cite[Theorem 2.2.2.3]{Gr}).
} \end{remark} 
The regularity in Proposition \ref{prop:LipCoef} is adequate for
diffusion coefficients $a(\bx)$ which are Lipschitz continuous in
$\domain$, which is essentially (up to modification)
$W^{1}_\infty(\domain) \simeq C^{0,1}(\domain)$.  In view of our
interest in admitting diffusion coefficients which are (realizations
of) GRF (see Section \ref{S:RndDat}), it is clear from Example
\ref{exple:LCBrBr} that relevant GRF models may exhibit mere H\"older
path regularity.

The H\"older spaces $C^s(\domain)$ on Lipschitz domains $\domain$ can
be obtained as interpolation spaces, via the so-called $K$-method of
function space interpolation which we briefly recapitulate (see, e.g.,
\cite[Chapter 1.3]{Triebel95}, \cite{BerghLof}).  Two Banach spaces
$A_0,A_1$ with continuous embedding $A_1 \hookrightarrow A_0$ with
respective norms $\| \circ \|_{A_i}$, $i=0,1$, constitute an
interpolation couple.  For $0<s<1$, the \emph{interpolation space}
$[A_0,A_1]_{s,q}$ of smoothness order $s$ with fine index
$q\in [1,\infty]$ is defined via the $K$-functional: for $a\in A_0$,
this functional is given by
\begin{equation}\label{eq:Kfct}
  K(a,t;A_0,A_1) := \inf_{a_1\in A_1} \{ \|a - a_1\|_{A_0} + t \| a_1 \|_{A_1}\}\;,\quad 
  t > 0 \;.
\end{equation}
For $0<s<1$ the intermediate, ``interpolation'' space of order $s$ and
fine index $q$ is denoted by $[A_0,A_1]_{s,q}$.  It is the set of
functions $a\in A_0$ such that the quantity
\begin{equation} \label{eq:KIntNorm} \| a \|_{[A_0,A_1]_{s,q}} :=
  \left\{
    \begin{array}{lr} 
      \left( \int_0^\infty (t^{-s} K(a,t,A_0,A_1))^q \frac{\rd t}{t} \right)^{1/q} \;,&  1 \leq q < \infty,
      \\[1ex]
      \sup_{t>0} t^{-s} K(a,t,A_0,A_1) \;, & q = \infty
    \end{array}
  \right. 
\end{equation}
is finite.  When the $A_i$ are Banach spaces, the sets
$[A_0,A_1]_{s,q}$ are Banach spaces with norm given by
\eqref{eq:KIntNorm}.  In particular (see, e.g., \cite[Lemma
7.36]{Adams2nd}), in the bounded Lipschitz domain $\domain$
\begin{equation}\label{eq:CsLinfW1inf}
  C^s(\domain) = [L^\infty(\domain), W^{1}_\infty(\domain)]_{s,\infty},\quad 0<s<1 \;.
\end{equation}
With the spaces $V:=H^1_0(\domain)$ and $W\subset V$, we define the
(non-separable, non-reflexive) Banach space
\begin{equation}\label{eq:Ws}
  W^s := [V,W]_{s,\infty}\;,\quad 0<s<1 \;.
\end{equation}
Then there holds the following generalization of \eqref{W-estimate}.
\begin{proposition}\label{prop:Ws}
  For a bounded Lipschitz domain $\domain \subset \RR^d$, $d\geq 2$,
  for every $f\in L^2(\domain)$ and $a\in C^s(\domain)$ for some
  $0<s<1$ with
  $$a_{\min} = \min\{ a(\bx) : \bx \in \overline{\domain} \}>0,$$ the
  solution $u\in V$ of \eqref{PDE}, \eqref{vf} belongs to $W^s$, and
  there exists a constant $c(s,\domain)$ such that
  \begin{equation}\label{eq:Ws-estimate}
    \| u \|_{W^s} 
    \leq 
    \frac{c}{a_{\min}} \left(1+\|a\|_{C^s}^{1/s} a_{\min}^{-1/s} \right)\| f \|_{L^2}\,.
  \end{equation}
\end{proposition}
\begin{proof}
  The estimate follows from the a-priori bounds for $s=0$ and $s=1$,
  i.e., \eqref{V-estimate} and \eqref{W-estimate}, by interpolation
  with the Lipschitz continuity \eqref{eq:LipBasic} of the solution
  operator.

  Let $a\in C^s(\domain)$ with $a_{\min} > 0$ be given.  From
  \eqref{eq:CsLinfW1inf}, for every $\delta>0$ exists
  $a_\delta \in W^{1,\infty}(\domain)$ with
$$
\| a - a_\delta \|_{C^0} \leq C\delta^s \| a \|_{C^s}\;, \quad \|
a_\delta \|_{W^{1}_\infty} \leq C \delta^{s-1} \| a \|_{C^s} \;.
$$
From
$$
\min_{\bx\in \domain} a_\delta(\bx) \geq \min_{\bx\in \domain} a(\bx)
- \| a - a_\delta \|_{C^0} \geq a_{\min} - C\delta^s \| a \|_{C^s}
$$
follows for
$ 0 < \delta \leq 2^{-1/s} \left\| a/a_{\min} \right\|_{C^s}^{-1/s},$
that
$$ \min_{\bx\in \domain} a_\delta(\bx) \geq a_{\min}/2\;.$$
For such $\delta$ and for $f\in L^2(\domain)$, \eqref{PDE} with
$a_\delta$ admits a unique solution $u_\delta \in V$ and from
\eqref{W-estimate}
$$
\| u_\delta \|_W \leq \frac{2c}{a_{\min}}\left( 1 + \frac{\|\nabla
    a_\delta\|_{L^\infty}}{a_{\min}}\right) \| f \|_{L^2}\;.
$$
From \eqref{eq:LipBasic} (with $f_1=f_2=f$) we find
$$
\|u-u_\delta \|_V \leq \frac{2c}{a_{\min}^2} \| a - a_\delta
\|_{L^\infty} \| f \|_{L^2} \leq C \frac{\delta^s}{a_{\min}^2} \| a
\|_{C^s} \| f \|_{L^2} \;.
$$
This implies in \eqref{eq:Kfct} that for some constant $C>0$
(depending only on $\domain$ and on $s$)
\begin{equation}\label{eq:tbound}
  K(u,t,V,W) 
  \leq  
  \frac{C}{a_{\min}} 
  \left( \delta^s A_s + t\left(1 + \delta^{s-1} A_s\right) \right)\| f \|_{L^2}
  \;,\;\;t>0
\end{equation}
where we have set
$A_s := \big\| \frac{a}{a_{\min}} \big\|_{C^s} \in [1,\infty)$.

To complete the proof, by \eqref{eq:Ws} we bound
$\| u \|_{W^s} = \sup_{t>0}t^{-s}K(u,t,V,W)$.  To this end, it
suffices to bound $K(u,t,V,W)$ for $0<t<1$.  Given such $t$, we choose
in the bound \eqref{eq:tbound} $\delta = t\delta_0 \in (0,\delta_0)$
with $\delta_0 := 2^{-1/s} A_s^{-1/s}$.  This yields
$$
\delta^s A_s + t\left(1 + \delta^{s-1} A_s\right) = t^s
\left(\delta_0^s A_s + t^{1-s} + \delta_0^{s-1} A_s \right) = t^s
\left(2^{-1} + t^{1-s} + 2^{-(s-1)/s} A_s^{1-(s-1)/s} \right)
$$
and we obtain for $0<t<1$ the bound
$$
t^{-s} K(u,t,V,W) \leq \frac{C}{a_{\min}} \left( 2 + 2^{-(s-1)/s}
  A_s^{1/s} \right) \| f \|_{L^2}\;.
$$
Adjusting the value of the constant $C$, we arrive at
\eqref{eq:Ws-estimate}.
\end{proof}
\subsection{Random input data}
\label{S:RndDat}
We are in particular interested in the input data $a$ and $f$ of
 the elliptic divergence-form PDE \eqref{PDE} being not precisely known.  The Lipschitz continuous
data-dependence in Proposition \ref{prop:LipS} of the variational
solution $u\in V$ of \eqref{PDE} will ensure that 
small variations in the data $(a,f)\in L^\infty_+(\domain) \times V^*$ 
imply corresponding small changes in the (unique)
solution $u\in V$.  
A natural paradigm is to model uncertain data probabilistically.  
To this end, we work with a base probability space
$(\Omega,\cA,\IP)$.  Given a known right hand side
$f\in L^2(\domain)$, and uncertain diffusion coefficient
$a \in E \subseteq L^\infty_+(\domain)$, where $E$ denotes a suitable
subset of $L^\infty_+(\domain)$ of admissible diffusion coefficients,
we model the function $a$ or $\log a$ as RVs taking values in a
subset $E$ of $L^\infty(\domain)$.  
We will assume the random data
$a$ to be separably-valued, more precisely, the set $E$ of admissible random
data will almost surely belong to a subset of a 
separable subspace of $L^\infty(\domain)$. 
See \cite[Chap. 2.6]{Bogach98} for details on separable-valuedness.
Separability of $E$ is natural from the point of view of numerical
approximation of (samples of) random input $a$ and simplifies many
technicalities in the mathematical description; 
we refer in particular to the construction of GMs on $E$ in
Sections~\ref{S:GMSepHS}--\ref{S:GSer}.  
One valid choice for the
space of admissible input data $E$ consists in
$E = C(\overline{\domain})\cap L^\infty_+(\domain)$. 
In the log-Gaussian models to be analyzed subsequently,
$E\subset L^\infty_+(\domain)$ will be ensured by modelling $\log(a)$
as a GRF, i.e., we assume the probability measure $\IP$ to be such
that the law of $\log(a)$ is a GM on $L^\infty(\domain)$ which charges
$E$, so that the random element
$\log(a(\cdot,\omega))\in L^\infty_+(\domain)$ $\IP$-a.s.. 
This,
in turn, implies with the well-posedness result in Section
\ref{S:PbmStat} that there exists a unique random solution
$u(\omega) = \cS(a,f)\in V$ $\IP$-a.s.. Furthermore, the Lipschitz
continuity \eqref{eq:LogLip} then implies that the corresponding map
$\omega\mapsto u(\omega)$ is a composition of the measurable map
$\omega\mapsto \log(a(\cdot,\omega))$ with the Lipschitz continuous
deterministic data-to-solution map $\cS$, hence strongly measurable,
and thus a RV on $(\Omega,\cA,\IP)$ taking values in $V$.
\subsection{Parametric deterministic coefficient}
\label{S:ParCoef}
A key step in the deterministic numerical approximation of 
the elliptic divergence-form PDE \eqref{PDE} with log-Gaussian random inputs (i.e.,
$\log (a)$ is a GRF on a suitable locally convex space $E$ of
admissible input data) is to place a GM on $E$ and to describe the
realizations of GRF $b$ in terms of affine-parametric representations
discussed in Section \ref{S:GSer}.  In Section \ref{S:ParDetPbm}, we
briefly describe this and in doing so extend a-priori estimates to
this resulting deterministic parametric version of elliptic PDE
\eqref{PDE}.  Subsequently, in Section \ref{S:CplxParExt}, we show
that the resulting, countably-parametric, linear elliptic problem
admits an extension to certain complex parameter domains, while still
remaining well-posed.
\subsubsection{Deterministic countably parametric elliptic PDEs}
\label{S:ParDetPbm}
Placing a Gaussian probability measure on the random inputs $\log(a)$
to the elliptic divergence-form PDE \eqref{PDE} can be achieved via Gaussian
series as discussed in Section \ref{S:GSer}.  Affine-parametric
representations which are admissible in the sense of Definition
\ref{def:Adm} of the random input $\log(a)$ of \eqref{PDE}, 
subject to a Gaussian law on the corresponding input locally convex space $E$,
render the elliptic divergence-form PDE \eqref{PDE} with random inputs a deterministic
parametric elliptic PDE.  More precisely, $b := \log(a)$ will depend on
the sequence $\by = (y_j)_{j\in \NN}$ of parameters from the parameter
space $\RR^\infty$.  Accordingly, we consider parametric diffusion
coefficients $a=a(\by)$, where
$$
\by=(y_j)_{j \in \NN} \in U.
$$
Here and throughout the rest of this book we make use of the notation
$$
U:= \RR^\infty.
$$
We develop the
holomorphy-based analysis of parametric regularity and 
Wiener-Hermite PC expansion coefficient
sparsity for the model parametric linear second order elliptic
divergence-form PDE with so-called ``log-affine coefficients'' \index{coefficient!log-Gaussian $\sim$}
\begin{equation}\label{SPDE}
  -\div\big(\exp(b(\by)) \nabla u(\by)\big)=f
  \quad \mbox{in}\quad \domain\;,
  \quad u(\by)|_{\partial \domain} = 0\;,
\end{equation}
i.e., $$a(\by) =\exp(b(\by)).$$ Here, the coefficient
$b(\by)=\log(a(\by))$ is assumed to be affine-parametric
\begin{equation}\label{eq:CoeffAffin}
  b(\by)= \sum_{j \in \NN} y_j\psi_j(\bx) 
  \;, \quad 
  \bx\in \domain \;,\quad \by\in U\;.
\end{equation}
We assume that $\psi_j \in E \subset L^\infty(\D)$ for every
$j\in \NN$.  For any $\by \in U$ such that $b(\by)\in L^\infty(\D)$,
by \eqref{V-estimate} we have the estimate
\begin{equation} \label{eq:uApriori} \|u(\by)\|_V \leq \|f\|_{V^*} \|
  a(\by)^{-1}\|_{L^\infty} \leq \exp(\|b(\by)\|_{L^\infty})\|f\|_{V^*}
  \,.
\end{equation}
For every $\by \in U$ satisfying $b(\by)\in L^\infty(\D)$, the
variational form \eqref{vf} of \eqref{SPDE} gives rise to the
\emph{parametric energy norm} $\| v \|_{a(\by)}$ on $V$
which is defined by

\begin{equation*}\label{eq:ParEn}
  \| v \|_{a(\by)}^2 := \int_{\domain}  a(\by) | \nabla v |^2 \rd\bx \;, \;\;
  v\in V .
\end{equation*}

The norms $ \| \circ \|_{a(\by)}$ and $\| \circ \|_V$ are equivalent
on $V$ but not uniformly w.r.t.\ $\by$. It holds
\begin{equation}\label{eq:Norms}
  \exp(-\| b(\by) \|_{L^\infty} )\| v \|_V^2 
  \leq 
  \| v \|_{a(\by)}^2 
  \leq 
  \exp( \| b(\by) \|_{L^\infty}) \| v \|_V^2,
  \quad 
  v\in V  \;. 
\end{equation}

\subsubsection{Probabilistic setting}
\label{S:ParmProb}
In a probabilistic setting, the parameter sequence $\by$ is chosen as
a sequence of i.i.d.\  standard Gaussian RVs $\calN(0,1)$ and
$(\psi_j)_{j \in \NN}$ a given sequence of functions in the Banach
space $L^\infty(\domain)$ to which we refer as \emph{representation
  system} of the uncertain input.  We then treat \eqref{SPDE} as the
stochastic linear second order elliptic divergence-form PDE with
so-called ``log-Gaussian coefficients''.  
We refer to
Section~\ref{S:GSer} for the construction of GMs based on affine
representation systems $(\psi_j)_{j \in \NN}$.  
Due to
$L^\infty(\domain)$ being non-separable, we consider GRFs $b(\by)$ 
which take values in separable subspaces $E \subset L^\infty(\domain)$, 
such as $E=C^0(\overline{\domain})$.

The probability space $(\Omega,\cA,\IP)$ from Section \ref{S:RndDat}
on the parametric solutions $\{ u(\by) : \by\in U\} $ is chosen as
$(U,\cB(U);\gamma)$.  
Here and throughout the rest of this book, we
make use of the notation: $\cB(U)$ 
is the
$\sigma$-field on the locally convex space $U$ generated by cylinders
of Borel sets on $\RR$, and $\gamma$ is the product measure of the
standard GM $\gamma_1$ on $\R$ 
(see the definition in Example \ref{ex:prod-measR^infty}). 
We shall refer to $\gamma$ as the \emph{standard GM on $U$}.

It follows from the a-priori estimate \eqref{eq:uApriori} that for
$f\in V^*$ the parametric elliptic diffusion problem \eqref{SPDE} admits a
unique solution for parameters $\by$ in the set
\begin{equation}\label{eq:U0}
  U_0 := \{ \by \in U: b(\by)\in L^\infty(\domain)\} \;.
\end{equation}
The measure $\gamma(U_0)$ of the set $U_0 \subset U$ depends on the
structure of $\by\mapsto b(\by)$.  The following sufficient condition on
the representation system $(\psi_j)_{j\in \NN}$ will be assumed
throughout.

\noindent
\begin{assumption}\label{ass:Ass1}
  For every $j\in \NN$, $\psi_j\in L^\infty(\domain)$, and there
  exists a positive sequence $(\lambda_j)_{j\in \NN}$ such that
  $\big(\exp(-\lambda_j^2)\big)_{j\in \NN}\in \ell^1(\NN)$ and the
  series $\sum_{j\in \NN}\lambda_j|\psi_j|$ converges in
  $L^\infty(\domain)$.
\end{assumption}
For the statement of the next result, we recall a notion of Bochner spaces.  
For a measure space $(\Omega, \cA, \mu)$
let $X$ a Banach space and $1 \le p < \infty$.  
Then
the Bochner space $L^p(\Omega,X;\mu)$ is defined as the space of all
strongly $\mu$-measurable mappings $u$ from $\Omega$ to $X$ 
such that the norm \index{space!Bochner $\sim$}
\begin{equation} \label{eq-Bochner} \|u\|_{L^p(\Omega,X;\mu)} := \
  \left(\int_{\Omega} \|u(\by)\|_X^p \, \rd \mu (\by) \right)^{1/p} 
  < \infty.
\end{equation}
In particular, when $(\Omega, \cA, \mu) = (U,\cB(U);\gamma)$, $X$ is
separable and $p=2$, the hilbertian space $L^2(U,X;\gamma)$ 
is one of the most important for the problems considered in this book.  

The following result was shown in \cite[Theorem 2.2]{BCDM}.
\begin{proposition}\label{prop:Meas1}
  Under Assumption \ref{ass:Ass1}, the set $U_0$ has full GM, i.e.,
  $\gamma(U_0) = 1$.  For all $k\in \NN$ there holds, with
  $\EE(\cdot)$ denoting expectation with respect to $\gamma$,
$$
\EE\left( \exp(k\| b(\cdot) \|_{L^\infty}) \right) < \infty \;.
$$
The solution family $\{ u(\by): \by \in U_0 \}$ of the parametric
elliptic boundary value problem \eqref{SPDE} is in $L^k(U,V;\gamma)$ 
for every finite $k\in \NN$.
\end{proposition}
%
\subsubsection{Deterministic complex-parametric elliptic PDEs}
\label{S:CplxParExt}
Towards the aim of establishing sparsity of Wiener-Hermite PC
expansions of the parametric solutions $\{ u(\by): \by\in {U_0}\}$ of
\eqref{SPDE}, we extend the deterministic parametric elliptic problem
\eqref{SPDE} from real-valued to complex-valued parameters.

Formally, 
replacing $\by=(y_j)_{j \in \NN}\in U$ in the coefficient $a(\by)$ by
$\bz=(z_j)_{j \in \NN}=(y_j+\im\xi_j)_{j \in \NN}\in \CC^\infty$, 
the real part of $a(\bz)$ is
\begin{equation} \label{Re(a)} \mathfrak{R}[a(\bz)] =
  \exp\Bigg({\sum_{j \in \NN} y_j\psi_j(\bx)}\Bigg) \cos\Bigg(\sum_{j
    \in \NN} \xi_j\psi_j(\bx)\Bigg)\,.
\end{equation}
We find that $\mathfrak{R}[a(\bz)]>0$ if
$$
\Bigg\|\sum_{j \in \NN} \xi_j\psi_j \Bigg\|_{L^\infty} <
\frac{\pi}{2}.
$$
This observation and Proposition \ref{prop:Meas1} motivate the study
of the analytic continuation of the solution map $\by \mapsto u(\by)$
to $\bz \mapsto u(\bz)$ for complex parameters
$\bz = (z_j)_{j \in \NN}$ by formally replacing the parameter $y_j$ by
$z_j$ in the definition of the parametric coefficient $a$, where each
$z_j$ lies in the strip
\begin{equation} \label{eq:DefSjrho} \mathcal{S}_j (\brho):= \{ z_j\in
  \CC\,: |\mathfrak{Im}z_j| < \rho_j\}
\end{equation}
and where $\rho_j>0$ and
$\brho=(\rho_j)_{j \in \NN} \in (0,\infty)^\infty$ is any sequence of
positive numbers such that
\begin{equation*} \label{k-01} \Bigg\|\sum_{j\in \NN} \rho_j
  |\psi_j|\Bigg\|_{L^\infty} < \frac{\pi}{2}\,.
\end{equation*} 
\subsection{Analyticity and sparsity}
\label{S:HolSumSol}
We address the analyticity (holomorphy) of the parametric solutions
$\{ u(\by): \by\in {U_0}\}$.  We analyze the sparsity by estimating,
in particular, the size of the domains of holomorphy to which the
parametric solutions can be extended.  We also treat the weighted
$\ell^2$-summability and $\ell^p$-summability  (sparsity) 
for the series of Wiener-Hermite the PC expansion coefficients $(u_\bnu)_{\bnu\in\CF}$ of $u(\by)$.
\subsubsection{Parametric holomorphy}
\label{sec:HolProp}
In this section we establish holomorphic parametric dependence $u$ on
$a$ and on $f$ as in \cite{CoDeSch1} by verifying complex
differentiability of a suitable complex-parametric extension of
$\by\mapsto u(\by)$.  We observe that the Lax-Milgram theory can be
extended to the case where the coefficient function $a$ is
complex-valued.  In this case, $V := H^1_0(\domain,\CC)$ in \eqref{vf}
and the ellipticity assumption \eqref{PDE-epllipticity} is extended to
the complex domain as
\begin{equation}\label{eq:PDEellinC}
  0 < \rho(a) := \underset{\bx\in\D}{\essinf}\,\Re(a(\bx))
  \leq |a(\bx)|
  \leq \| a \|_{L^\infty}
  <\infty,\qquad \bx\in \domain.
\end{equation}
Under this condition, there exists a unique variational solution
$u\in V$ of \eqref{PDE} and for this solution, the
estimate \eqref{V-estimate} remains valid, i.e.,
\begin{equation} \label{eq:bound-comp} \|u\|_V \leq
  \frac{\|f\|_{V^*}}{\rho(a)} \,.
\end{equation} 
Let $\brho=(\rho_j)_{j\in \NN} \in [0,\infty)^\infty$ be a sequence of
non-negative numbers and assume that $\mfu \subseteq \supp(\brho)$ is
finite.  Define
\begin{equation}\label{eq:Snubrho}
  \mathcal{S}_\mfu (\brho) :=\bigtimes_{j\in \mfu} \mathcal{S}_j(\brho) \,,
\end{equation}
where the strip $ \mathcal{S}_j (\brho) $ is given in
\eqref{eq:DefSjrho}.  For $\by\in U$, put
\[
  \mathcal{S}_\mfu (\by,\brho) := \big\{(z_j)_{j\in \NN}: z_j \in
  \mathcal{S}_j(\brho)\ \text{if}\ j\in \mfu\ \text{and}\ z_j=y_j \
  \text{if}\ j\not \in \mfu \big\}.
\]
\begin{proposition}\label{prop:holoh1}
  Let the sequence $\brho=(\rho_j)_{j\in \NN}\in [0,\infty)^\infty$
  satisfy
  \begin{equation}\label{eq:leqkappa}
    \Bigg\|\sum_{j \in \NN} \rho_j |\psi_j | \Bigg\|_{L^\infty} 
    \leq 
    \kappa < \frac{\pi}{2}\,.
  \end{equation}
  Let $\by_0=(y_{0,1},y_{0,2},\ldots) \in U$ be such that $b(\by_0)$
  belongs to $L^\infty(\D)$, and let $\mfu\subseteq \supp(\brho)$ be a
  finite set.

  Then the solution $u$ of the variational form of \eqref{SPDE} is
  holomorphic on $ \mathcal{S}_\mfu (\brho) $ as a function of the
  parameters
  $\bz_\mfu=(z_j)_{j \in \NN} \in \mathcal{S}_\mfu (\by_0,\brho)$
  taking values in $V$ with $z_j = y_{0,j}$ for $j\not \in \mfu$ held
  fixed.
\end{proposition}
\begin{proof}
  Let $N\in \NN$.  We denote
  \begin{equation}\label{eq:SUN}
    \mathcal{S}_{\mfu,N} (\brho) 
    := 
    \big\{ (y_j+\im \xi_j)_{j\in \mfu}\in \mathcal{S}_\mfu (\brho): |y_j-y_{0,j}|<N\big\}\,.
  \end{equation}
  For
  $\bz_{\mfu} = (y_j+\im \xi_j)_{j\in \NN}\in \mathcal{S}_{\mfu}
  (\by_0,\brho)$ with
  $(y_j+\im \xi_j)_{j\in \mfu}\in \mathcal{S}_{\mfu,N} (\brho)$
  we have
  \begin{equation*}
    \begin{split} 
      \label{eq:DefM}
      \Bigg\| \sum_{j \in \NN}y_j\psi_j\Bigg\|_{L^\infty} &\leq
      \|b(\by_0)\|_{L^\infty} + \Bigg\| \sum_{j\in
        \mfu}|(y-y_{0,j})\psi_j |\Bigg\|_{L^\infty}
      \\
      & \leq \|b(\by_0)\|_{L^\infty} + N\Bigg\| \sum_{j\in
        \mfu}|\psi_j |\Bigg\|_{L^\infty} =: M <\infty
    \end{split}
  \end{equation*} 
  and
  \begin{equation*}
    \Bigg\| \sum_{j\in \mfu}\xi_j \psi_j\Bigg\|_{L^\infty} 
    \leq 
    \Bigg\| \sum_{j\in \mfu}|\rho_j \psi_j|\Bigg\|_{L^\infty} 
    \leq 
    \kappa\,.
  \end{equation*}
  Consequently, we obtain from \eqref{Re(a)}
  \begin{equation}\label{eq:R-1}
    \rho(a( \bz_{\mfu}))
    \geq  \exp\Bigg(-\Bigg\| \sum_{j \in \NN}y_j\psi_j\Bigg\|_{L^\infty}\Bigg)
    \cos\Bigg(\Bigg\| \sum_{j\in \mfu}\xi_j \psi_j\Bigg\|_{L^\infty} \Bigg)\geq 
    \exp(-M) \cos\kappa 
  \end{equation}
  for all $ \bz_{\mfu}\in \mathcal{S}_{\mfu} (\by_0,\brho)$ with
  $(y_j+\im \xi_j)_{j\in \mfu}\in \mathcal{S}_{\mfu,N} (\brho)$.  From
  this and the analyticity of exponential functions we conclude that
  the map $\bz_\mfu\to u(\bz_{\mfu})$ is holomorphic on the set
  $\mathcal{S}_{\mfu,N} (\brho)$, see \cite[Pages 22, 23]{CoDe}.
  Since $N$ is arbitrary we deduce that the map
  $\bz_\mfu\to u(\bz_{\mfu})$ is holomorphic on
  $\mathcal{S}_{\mfu} (\brho)$.
\end{proof}
The analytic continuation of the parametric solutions
$\{u(\by): \by\in U\}$ to $\mathcal{S}_{\mfu} (\brho)$ leads 
to a result on parametric $V$-regularity. 
\begin{lemma}\label{lem:estV}
  Let $\brho= (\rho_j)_{j \in \NN}$ be a non-negative sequence
  satisfying \eqref{eq:leqkappa}.  Let $\by \in U$ with
  $b(\by)\in L^\infty(\D)$ and $\bnu\in \FF$ such that
  $\supp(\bnu)\subseteq \supp(\brho)$.  Then we have
  \begin{equation*}
    \|\partial^{\bnu}u(\by)\|_V 
    \leq
    C_0\frac{\bnu!}{\brho^\bnu}
    \exp\big( \|b(\by)\|_{L^\infty} \big)  ,
  \end{equation*}
  where $C_0=e^\kappa (\cos\kappa)^{-1}\|f\|_{V^*}$.
\end{lemma}
\begin{proof}Let $\bnu\in \FF$ such that
  $\supp(\bnu)\subseteq \supp(\brho)$.  Denote $\mfu=\supp(\bnu)$.
  For fixed variable $y_j$ with $j\not \in \mfu$, the map
  $\mathcal{S}_\mfu (\by,\brho) \ni \bz_{\mfu}\to u(\bz_\mfu)$ is
  holomorphic on the domain $\mathcal{S}_\mfu(\by,\kappa'\brho)$ where
  $\kappa <\kappa \kappa'<\pi/2$, see Proposition \ref{prop:holoh1}.
  Applying Cauchy's integral formula gives
  \begin{equation*}
    \partial^{\bnu}u(\by) 
    =
    \frac{\bnu!}{(2\pi i)^{|\mfu|}}
    \int_{\mathcal{C}_{\by,\mfu}(\brho)} 
    \frac{u(\bz_\mfu) }{\prod_{j\in \mfu}  (z_j-y_j)^{\nu_j+1}}\prod_{j\in \mfu}\rd z_j,
  \end{equation*}
  where
  \begin{equation} \label{eq:C-rho} \mathcal{C}_{\by, \mfu}(\brho) :=
    \bigtimes_{j\in \mfu} \mathcal{C}_{\by,j}( \brho)\,,\qquad
    \mathcal{C}_{\by,j} ( \brho) := \big\{ z_j \in \CC:
    |z_j-y_j|=\rho_j\big\}\,.
  \end{equation}
  This leads to
  \begin{equation} \label{u-y}
    \begin{split} 
      \|\partial^{\bnu}u(\by)\|_{V} & \leq \frac{\bnu!}{\brho^\bnu}
      \sup_{z_\mfu\in \mathcal{C}_\mfu(\by,\brho)}
      \|u(\bz_\mfu)\|_{V}\,
    \end{split}
  \end{equation} 
  with
  \begin{equation} \label{eq:C-rho-y} \mathcal{C}_\mfu(\by,\brho)
    =\big\{(z_j)_{j\in \NN} \in \mathcal{S}_\mfu (\by,\brho): \
    (z_j)_{j\in \mfu}\in \mathcal{C}_{\by,\mfu}(\brho) \big\}\,.
  \end{equation} 
  Notice that for
  $\bz_\mfu=(z_j)_{j\in \NN} \in \mathcal{C}_\mfu(\by,\brho)$ we can
  write $z_j = y_j + \eta_j + \im\xi_j \in \Cc_{\by,j}(\brho)$ with
  $|\eta_j | \le \rho_j$, $|\xi_j| \le \rho_j$ if $j\in \mfu$ and
  $\eta_j = \xi_j=0$ if $j\not \in \mfu$.  By denoting
  $\beeta=(\eta_j)_{j\in \NN}$ and $\bxi=(\xi_j)_{j\in \NN}$ we see
  that $\|b(\beeta)\|_{L^\infty}\leq \kappa$ and
  $\|b(\bxi)\|_{L^\infty}\leq \kappa$.  Hence we deduce from
  \eqref{eq:bound-comp} that
  \begin{equation*}
    \begin{split} 
      \|u(\bz_\mfu)\|_{V} & \leq \frac{\exp\big(
        \|b(\by+\beeta)\|_{L^\infty}\big)
      }{\cos\big(\|b(\bxi)\|_{L^\infty} \big) } \|f\|_{V^*} \leq
      \frac{ \exp\big( \kappa+\|b(\by)\|_{L^\infty} \big) }{\cos
        \kappa }\|f\|_{V^*} \,.
    \end{split}
  \end{equation*}
  Inserting this into \eqref{u-y} we obtain the desired estimate.
\end{proof}
\subsubsection{Sparsity of Wiener-Hermite PC expansion coefficients}
\label{sec:SumHermCoef}
In this section, we will exploit the analyticity of $u$ to prove a
weighted $\ell^2$-summability result for the $V$-norms of the
coefficients in the Wiener-Hermite PC expansion of the solution map
$\by\to u(\by)$.  Our analysis yields the same $\ell^p$-summability
result as in the papers \cite{BCDM,BCDS} in the case $\psi_j$ have
arbitrary supports.  In this case, our result implies that the
$\ell^p$-summability of $(\|u_\bnu\|_V)_{\Ff}$ for $0<p\leq 1$ (the
sparsity of parametric solutions) follows from the
$\ell^p$-summability of the sequence
$(j^\alpha \|\psi_j\|_{L^\infty})_{j\in \NN}$ for some $\alpha>1/2$
which is an improvement over the condition
$(j\|\psi_j\|_{L^\infty})_{j\in \NN} \in \ell^p(\NN)$ in \cite{HS},
see \cite[Section 6.3]{BCDM}.  In the case of disjoint or finitely
overlapping supports our analysis obtains a weaker result compared to
\cite{BCDM,BCDS}.  As observed in \cite{CoDeSch1}, one advantage of
establishing sparsity of Wiener-Hermite PC expansion coefficients via
holomorphy rather than by successive differentiation is that it allows
to derive, in a unified way, summability bounds for the 
coefficients of Wiener-Hermite PC expansion whose size is measured in
scales of Sobolev and Besov spaces in the domain $\domain$.  
Using
real-variable arguments as, e.g., in \cite{BCDM,BCDS}, establishing
sparsity of parametric solutions in Besov spaces in $\domain$ of
higher smoothness seems to require more involved technical and
notational developments, according to \cite[Comment on Page 2157]{BCDS}.

The
parametric solution $\{u(\by): \by\in U\}$ of \eqref{SPDE} belongs to
the space $L^2(U,V;\gamma)$ or more generally,
$L^2(U,(H^{1+s}\cap H^1_0)(\domain);\gamma)$ for $s$-order of extra
differentiability provided by higher data regularity.  We recall from
Section \ref{S:HerPol} the normalized probabilistic Hermite
polynomials $(H_k)_{k \in \NN_0}$.  Every $u \in L^2(U,X;\gamma)$
admits the \emph{Wiener-Hermite PC expansion} \index{expansion!Wiener-Hermite PC $\sim$}
\begin{equation}
  \sum_{\bnu\in \Ff} u_\bnu H_\bnu(\by), 
  \label{hermite}
\end{equation}
where for $\bnu \in \Ff$,
\begin{equation*}
  H_\bnu(\by)=\prod _{j \in \NN}H_{\nu_j}(y_j),\quad 
  \label{hermite-polynomial}
\end{equation*}
and
\begin{equation*}
  u_\bnu:=\int_U u(\by)\,H_\bnu(\by)\, \rd\gamma (\by)
  \label{hermite-coeff}
\end{equation*}
are called \emph{Wiener-Hermite PC expansion coefficients}.  Notice
that $(H_\bnu)_{\bnu \in \Ff}$ forms an ONB of $L^2(U;\gamma)$.

For every $u\in L^2(U,X;\gamma)$, there holds the Parseval-type
identity
\begin{equation}\label{eq:ParsevH}
  \| u \|^2_{L^2(U,X;\gamma)} 
  = 
  \sum_{\bnu\in \cF} \| u_\bnu \|_X^2 \;,\quad 
  u\in L^2(U,X;\gamma)\;.
\end{equation}

The error of approximation of the parametric solution
$\{u(\by): \by\in U\}$ of \eqref{SPDE} will be measured in the Bochner
space $L^2(U,V;\gamma)$.  A basic role in this approximation is taken
by the Wiener-Hermite PC expansion \eqref{hermite} of $u$ in the space
$L^2(U,V;\gamma)$.

For a finite set $\Lambda \subset \cF$, we denote by
$u_\Lambda = \sum_{\bnu\in \Lambda} u_\bnu$ the corresponding partial
sum of the Wiener-Hermite PC expansion \eqref{hermite}.  It follows
from \eqref{eq:ParsevH} that
\begin{equation*}\label{eq:L2Error}
  \| u - u_\Lambda \|_{L^2(U,V;\gamma)}^2 
  = 
  \sum_{\bnu\in \cF\backslash \Lambda} \| u_\bnu \|_V^2
  \;.
\end{equation*}
Therefore, summability results of the coefficients
$(\| u_\bnu \|_V )_{\bnu\in \cF}$ imply convergence rate estimates of
finitely truncated expansions $u_{\Lambda_n}$ for suitable sequences
$( \Lambda_n )_{n\in \NN}$ of sets of $n$ indices $\bnu$
(see \cite{HS,BCDM, dD21}).  
We next recapitulate some weighted summability results for Wiener-Hermite expansions.

For $r\in \NN$ and a sequence of nonnegative numbers
$\bvarrho=(\varrho_j)_{j\in \NN}$, 
we define the \emph{Wiener-Hermite weights}
\begin{equation} 
\label{beta} 
\beta_\bnu(r,\bvarrho) :=
\sum_{\|\bnu'\|_{\ell^\infty}\leq r} \binom{\bnu}{\bnu'} \bvarrho^{2\bnu'} 
= \prod_{j \in \NN}\Bigg(\sum_{\ell=0}^{r}\binom{\nu_j}{\ell}\varrho_j^{2\ell}\Bigg), \ \ \bnu \in \Ff\,.
\end{equation} 
The following identity was proved in \cite[Theorem 3.3]{BCDM}. 
For convenience to the reader, we present the proof from that paper.

\begin{lemma}\label{lem:equal}
  Let Assumption \ref{ass:Ass1} hold.  Let $r\in \NN$ and
  $\bvarrho=(\varrho_j)_{j\in \NN}$ be a sequence
  of nonnegative numbers.
  Then
  \begin{equation} \label{eq:equal-V} \sum_{\bnu\in
      \Ff}\beta_\bnu(r,\bvarrho)\|u_\bnu\|_V^2 =
    \sum_{\|\bnu\|_{\ell^\infty}\leq r} \frac{\bvarrho^{2\bnu}}{\bnu!}
    \int_U\| \partial^\bnu u(\by)\|_V^2\rd\gamma(\by)\,.
  \end{equation} 
\end{lemma}
\begin{proof} 
	Recall that  $p(y):=p(y,0,1)=-\frac{1}{\sqrt{2\pi}}\exp(-y^2/2)$ 
is the density function of the standard GM on $\RR$. 
Let $\mu\in \NN$. 
For a sufficiently smooth, univariate function $v\in L^2(\R;\gamma)$,  
from
$
H_\nu(y)=\frac{(-1)^\nu}{\sqrt{\nu!}}\frac{p^{(\nu)}(y)}{p(y)}
$ we have for $\nu\geq \mu$
\begin{align*}
	v_\nu&:=\int_{\RR} v(y)H_\nu(y) p(y)\dd y =  \frac{(-1)^\nu}{\sqrt{\nu!}} \int_{\R} v(y)p^{(\nu)}(y)\dd y
	\\
	&= \frac{(-1)^{\nu-\mu}}{\sqrt{\nu!}} \int_{\R} v^{(\mu)}(y)p^{(\nu-\mu)}(y) \dd y 
         =\sqrt{\frac{(\nu-\mu)!}{\nu!}} \int_{\R} v^{(\mu)}(y) H_{\nu-\mu}(y)p(y)\dd y.
\end{align*}
Hence
$$
\sqrt{\frac{\nu!}{\mu!(\nu-\mu)!}}v_\nu= \sqrt{\frac{1}{\mu!}}\int_{\R} v^{(\mu)}(y) H_{\nu-\mu}(y)\dd \gamma(y).
$$
By Parseval's identity, we have
$$
\frac{1}{\mu!} \int_{\R}|v^{(\mu)}(y)|^2\dd \gamma(y) = \sum_{\nu\geq \mu}\frac{\nu!}{\mu!(\nu-\mu)!}|v_\nu|^2 = \sum_{\nu\in \NN_0}\binom{\nu}{\mu}|v_\nu|^2\,,
$$
where we use the convention $\binom{\nu}{\mu}=0$ if $\mu>\nu$.

For multi-indices and for $u\in L^2(U,V;\gamma)$,
if $\bmu \leq \bnu$, applying the above argument in coordinate-wise 
for the coefficients
$$
u_{\bnu} =\sqrt{\frac{(\bnu-\bmu)!}{\bnu!}} \int_U\partial^{\bmu}u(\by) H_{\bnu-\bmu}(\by)\dd\gamma(\by)
$$
we get
$$
\frac{1}{\bmu!} \int_{U}\|\partial^{\bmu}u(\by)\|_V^2\dd\gamma(\by)=\sum_{\bnu\geq \bmu}\frac{\bnu!}{(\bmu!(\bnu-\bmu)!)}\|u_{\bnu}\|_V^2 = \sum_{\bnu\in \FF}\binom{\bnu}{\bmu}\|u_{\bnu}\|_V^2.
$$
Multiplying both sides by $\bvarrho^{2\bmu}$ 
and summing over $\bmu$ with $\|\bmu\|_{\ell^\infty}\leq r$,
we obtain
$$
\sum_{\|\bmu\|_{\ell^\infty}\leq r}\frac{\bvarrho^{2\bmu}}{\bmu!} \int_{U}\|\partial^{\bmu}u(\by)\|_V^2\dd\gamma(\by)=\sum_{\|\bmu\|_{\ell^\infty}\leq r} \sum_{\bnu\in \FF}\binom{\bnu}{\bmu}\bvarrho^{2\bmu}\|u_{\bnu}\|_V^2=\sum_{\bnu\in
	\Ff}\beta_\bnu(r,\bvarrho)\|u_\bnu\|_V^2.
$$
\end{proof}

We recall a summability property of the sequence
$(\beta_\bnu(r,\bvarrho)^{-1})_{j\in \NN}$ 
and its proof, given in \cite[Lemma 5.1]{BCDM}.

\begin{lemma}\label{lem:beta-summability}
  Let $0 < p < \infty$ and $q:= \frac{2p}{2-p}$.  Let
  $\bvarrho=(\varrho_j)_{j\in \NN}\in [0,\infty)^\infty$ be a sequence
  of positive numbers such that
		$$
		(\varrho_j^{-1})_{j\in \NN}\in \ell^q(\NN).
		$$ 
		Then for any $r \in \NN$ such that $\frac{2}{r+1} <
                p$, the family $(\beta_{\bnu}(r,\bvarrho))_{\bnu \in \FF}$ 
                defined in \eqref{beta} for this $r$ satisfies
		\begin{equation} \label{ineq: beta} \sum_{\bnu\in \FF}
                  \beta_{\bnu}(r,\bvarrho)^{-q/2}<\infty\,.
		\end{equation} 
              \end{lemma}
      
        \begin{proof}
First we have the decomposition
    $$
    \sum_{\bnu\in \FF}b_{\bnu}(r,\bvarrho)^{-q/2} = \sum_{\bnu\in \FF}\prod_{j\in \NN}\bigg(\sum_{\ell=0}^r\binom{\nu_j}{\ell}\varrho_j^{2\ell}\bigg)^{-q/2}=\prod_{j\in \NN}\sum_{n\in \NN_0}\bigg(\sum_{\ell=0}^r \binom{n}{\ell}\varrho_j^{2\ell}\bigg)^{-q/2}.
    $$
   For each $j\in \NN$ we have
   \begin{equation}  \label{eq-lemma3.11-a} 
\sum_{n\in \NN_0}\bigg(\sum_{\ell=0}^r \binom{n}{\ell}\varrho_j^{2\ell}\bigg)^{-q/2}
\leq 
\sum_{n\in \NN_0} \bigg[\binom{n}{\min\{n,r\}}\varrho_j^{2\min\{n,r\}}\bigg]^{-q/2} 
= 
\sum_{n=0}^{r-1} \varrho_j^{-nq}+C_{r,q}\varrho_j^{-rq},
\end{equation}
where
$$
C_{r,q}:= \sum_{n= r}^{+\infty}\binom{n}{r}^{-q/2} = (r!)^{q/2}\sum_{n\in \NN_0}\big[(n+1)\ldots(n+r)\big]^{-q/2}.
$$
Since $\lim\limits_{n\to +\infty}\frac{(n+1)\ldots(n+r)}{n^r}=1$, 
we find that $ C_{r,q}$ is finite if and only if $q>2/r$. 
This is equivalent to $\frac{2}{r+1} < p$. 
From the assumption 
$ (\varrho_j^{-1})_{j\in \NN}\in \ell^q(\NN) $ 
we find some $J>1$ such that $\varrho_j>1$ for all $j>J$. 
This implies $\varrho_j^{-nq}\leq \varrho_j^{-q}$ for $n=1,\ldots,r$ and $j>J$. 
Therefore, one can bound the right side of \eqref{eq-lemma3.11-a} 
by $1+( C_{r,q}+r-1)\varrho_j^{-q}$. 
Hence we obtain
\begin{align*}
 \sum_{\bnu\in \FF}b_{\bnu}(r,\bvarrho)^{-q/2}
 &\leq C   \prod_{j>J}\big[1+( C_{r,q}+r-1)\varrho_j^{-q}\big]
 \\
 &\leq C\prod_{j>J}\exp\Big(( C_{r,q}+r-1)\varrho_j^{-q}\Big)
 \\
 &\leq C\exp\Big(( C_{r,q}+r-1)\|(\varrho_j^{-1})_{j\in \NN}\|_{\ell^q}^q\Big)
\end{align*}
which is finite since 	$(\varrho_j^{-1})_{j\in \NN}\in \ell^q(\NN)$.
\end{proof}

In what follows, we denote by  $(\bee_j)_{j\in \NN}$ the standard basis of $\ell^2(\NN)$, i.e., $\bee_j = (e_{j,i})_{i\in \NN}$ with $e_{j,i} = 1$ for $i=j$ and $e_{j,i} = 0$ for $i\not=j$.
The following lemma was obtained in \cite[Lemma 7.1, Theorem 7.2]{CoDeSch} and \cite[Lemma 3.17]{CoDe}.
              \begin{lemma}\label{lem:alpha-summability}
		Let $\balpha=(\alpha_j)_{j\in \NN}$ be a sequence of
                nonnegative numbers. Then we have the following.
		\begin{itemize}
		\item[{\rm (i)}] 
                  For $0 < p <\infty$, the family
                  $(\balpha^\bnu)_{\bnu \in \FF}$ belongs to
                  $\ell^p(\FF)$ if and only if
                  $\|\balpha\|_{\ell^p} < \infty$ and
                  $\|\balpha\|_{\ell^\infty} < 1$.
		\item[{\rm (ii)}] 
                  For $0 < p \le 1$, the family
                  $(\balpha^\bnu|\bnu|!/\bnu!)_{\bnu \in \FF}$ belongs
                  to $\ell^p(\FF)$ if and only if
                  $\|\balpha\|_{\ell^p} < \infty$ and
                  $\|\balpha\|_{\ell^1} < 1$.
		\end{itemize}
              \end{lemma}
\begin{proof} 
 {\it Step 1.} 
We prove the first statement. 
Assume that $\|\balpha\|_{\ell^\infty} < 1$. 
Then we have
\begin{align*}
\sum_{\bnu\in \FF}\balpha^{\bnu p}&=\prod_{j\in \NN}\sum_{n\in \NN_0} \alpha_j^{pn}=\prod_{j\in \NN}\frac{1}{1-\alpha_j^p}\\
&=
 \prod_{j\in \NN}\bigg(1+\frac{\alpha_j^p}{1-\alpha_j^p}\bigg) \leq \prod_{j\in \NN}\exp\bigg(\frac{\alpha_j^p}{1-\alpha_j^p}\bigg) \\
 &\leq \prod_{j\in \NN}\exp\bigg(\frac{1}{1-\|\balpha\|_{\ell^\infty}^p}\alpha_j^p\bigg)=\exp\bigg(\frac{1}{1-\|\balpha\|_{\ell^\infty}^p}\|\balpha\|_{\ell^p}^p\bigg)\,,
\end{align*}
where in the last equality we have used $(\balpha^\bnu)_{\bnu \in \FF}\in\ell^p(\FF)$. 

Since the sequence $(\alpha_j)_{j\in \N}=(\balpha^{\bee_j})_{j\in \NN}$ 
is a subsequence of $(\alpha^{\bnu})_{\bnu\in \FF}$, 
$(\balpha^\bnu)_{\bnu \in \FF}\in \ell^p(\FF)$ 
implies $ \balpha $ belong to $\ell^p(\NN)$. 
Moreover we have for any $j\geq 1$
$$
\sum_{n\in \NN_0}\alpha_j^{np}=\sum_{n\in \NN_0}\balpha^{n\bee_j p} \leq  \sum_{\bnu\in \FF}\balpha^{\bnu p} <\infty.
$$
From this we have $\alpha_j^p<1$ which implies $\alpha_j<1$ for all $j\in \NN$. 
Since $\balpha \in \ell^p(\NN)$ it is easily seen that $\|\balpha\|_{\ell^\infty}<1$.

{\it Step 2.} We prove the second statement. We observe that 
$$
\sum_{\bnu\in \FF}\frac{|\bnu|}{\bnu!}\balpha^\bnu =\sum_{k\in \NN_0}\sum_{|\bnu|=k}\frac{|\bnu|}{\bnu!}\balpha^\bnu =\sum_{k\in \NN_0}\bigg(\sum_{j\in \NN}\alpha_j\bigg)^k.
$$
From this we deduce that $(\balpha^\bnu|\bnu|!/\bnu!)_{\bnu \in \FF}$ belongs
to $\ell^1(\FF)$ if and only if $\balpha\in \ell^1(\NN)$ and
$\|\balpha\|_{\ell^1} < 1$.

Suppose that 
$$
(\balpha^\bnu|\bnu|!/\bnu!)_{\bnu \in \FF}\in \ell^p(\FF)
$$
for some $p\in (0,1)$. As in Step 1, the sequence $(\alpha_j)_{j\in \N}=(\balpha^{\bee_j})_{j\in \NN}$ and $(\alpha_j^n)_{n\in \NN_0}=(\balpha^{n\bee_j})_{n\in \NN_0}$ are subsequences of $(\balpha^\bnu)_{\bnu\in \FF}$. Therefore   $(\balpha^\bnu|\bnu|!/\bnu!)_{\bnu \in \FF}$ belongs
to $\ell^p(\FF)$ implies that
$\|\balpha\|_{\ell^p} < \infty$ and
$\|\balpha\|_{\ell^1} < 1$. 

Conversely, assume that $\|\balpha\|_{\ell^p} < \infty$ 
and $\|\balpha\|_{\ell^1} < 1$. 
We put $\delta:=1-\|\balpha\|_{\ell^1}>0$ and $\eta:=\frac{\delta}{3}$. 
Take $J$ large enough such that $\sum_{j>J}\alpha_j^p\leq \eta$. 
We define the sequence $\bc$ and $\bd$ by
$$
c_j=(1+\eta)\alpha_j,\qquad d_j=\frac{1}{1+\eta} 
$$
if $j\leq J$ and
$$
c_j=\alpha_j^p,\qquad d_j=\alpha_j^{1-p}
$$
if $j>J$. 
By this construction we have $\alpha_j=c_jd_j$ for all $j\in \NN$. 
For the sequence $\bc$ we have
$$
\|\bc\|_{\ell^1}\leq (1+\eta)\|\balpha\|_{\ell^1}+\sum_{j>J}\alpha_j^p \leq (1+\eta)(1-\delta)+\eta < 1-\eta. 
$$
Next we show that $\|\bd\|_{\ell^\infty}<1$. Indeed, for $1\leq j\leq J$ we have $d_j=\frac{1}{1+\eta}<1$ and for $j>J$ we have
$$
d_j =(\alpha_j^p)^{(1-p)/p}\leq \eta^{(1-p)/p}<1.
$$
Moreover since $d_j^{(p/(1-p))}=\alpha_j^p$ for $j>J$ we have $\bd \in \ell^{p/(1-p)}(\N)$. Now we get from H\"older's inequality
$$
\sum_{\bnu\in \FF}\bigg(\frac{|\bnu|}{\bnu!}\balpha^\bnu \bigg)^p= \sum_{\bnu\in \FF}  \bigg(\frac{|\bnu|}{\bnu!}\bc^\bnu \bigg)^p\bd^{p\bnu}\leq \bigg(\sum_{\bnu\in \FF}\frac{|\bnu|}{\bnu!}\bc^\bnu \bigg)^p\bigg(\sum_{\bnu\in \FF}\bd^{\bnu p/(1-p)}\bigg)^{1-p}.
$$
Observe that the first factor on the right side is finite since $\bc\in \ell^1(\NN)$ and $\|\bc\|_{\ell^1}<1$. 
Applying the first statement, the second factor on the right side is finite, whence
$(\balpha^\bnu|\bnu|!/\bnu!)_{\bnu \in \FF} \in \ell^p(\FF)$.
\end{proof}
With these sequence summability results at hand, we are now in position to 
formulate Wiener-Hermite summation results for parametric solution families
of PDEs with log-Gaussian random field data.

\begin{theorem} [General case] \label{thm:s=1} 
                Let
                Assumption \ref{ass:Ass1} hold and assume that
                $\bvarrho=(\varrho_j)_{j\in \NN}\in [0,\infty)^\infty$
                is a sequence satisfying
                $(\varrho_j^{-1})_{j \in \NN}\in \ell^q(\NN)$ for some
                $0 < q < \infty$.  Assume that, for each
                $\bnu\in \FF$, there exists a sequence
                $\brho_\bnu= (\rho_{\bnu,j})_{j \in \NN}\in
                [0,\infty)^\infty$ such that
                $\supp(\bnu)\subseteq \supp(\brho_\bnu)$,
                \begin{equation} \label{assumption: theorem 3.1}
                  \sup_{\bnu\in \FF} \Bigg\| \sum_{j\in
                    \NN}\rho_{\bnu,j}|\psi_j|\Bigg\|_{L^\infty}\leq
                  \kappa <\frac{\pi}{2}, \qquad \text{and} \qquad
                  \sum_{\|\bnu\|_{\ell^\infty}\leq r}
                  \frac{\bnu!\bvarrho^{2\bnu}}{\brho_\bnu^{2\bnu}}
                  <\infty
                \end{equation}  
                with $r\in \NN$, $r>2/q$.
                Then
                \begin{equation} \label{eq:beta-u} 
                	\sum_{\bnu\in\FF}\beta_\bnu(r,\bvarrho)\|u_\bnu\|_{V}^2 <\infty
                	\ \ \ with \ \ \ \big(\beta_\bnu(r,\bvarrho)^{-1/2}\big)_{\bnu\in\FF} \in \ell^q(\FF).
                \end{equation}
                Furthermore,
                $$
                (\|u_\bnu\|_{V})_{\bnu\in\FF}\in \ell^p(\FF) \ \ \  with
                \ \ \ \frac{1}{p}=\frac{1}{q}+\frac{1}{2}.
                $$
              \end{theorem}
              \begin{proof}
                By Proposition
                \ref{prop:Meas1} Assumption \ref{ass:Ass1} implies
                that $b(\by)$ belongs to $L^\infty(\D)$ for
                $\gamma$-a.e.  $\by\in U$ and
                $\mathbb{E}(\exp(k\|b(\by)\|_{L^\infty}))$ is finite
                for all $k\in [0,\infty)$.
	
                For $\by\in U$ such that $b(\by)\in L^\infty(\D)$ and
                $\bnu \in \FF$ with $\mfu=\supp(\bnu)$, the solution
                $u$ of \eqref{SPDE} is holomorphic in
                $\mathcal{S}_\mfu (\brho_\bnu)$, see Proposition \ref{prop:holoh1}.  
                This, \eqref{assumption: theorem 3.1} and 
                Lemmata \ref{lem:estV} and \ref{lem:equal} 
                yield that
                \begin{equation*}
                  \begin{split} 
                  	\sum_{\bnu\in\FF}\beta_\bnu(r,\bvarrho)\|u_\bnu\|_{V}^2
                  	&=
                    \sum_{\|\bnu\|_{\ell^\infty}\leq r}
                    \frac{\bvarrho^{2\bnu}}{\bnu!} \int_U\|
                    \partial^\bnu u(\by)\|_V^2\rd\gamma(\by) 
                    \\
                    & \leq
                    C_0^2 \sum_{\|\bnu\|_{\ell^\infty}\leq r}
                    \frac{\bnu!\bvarrho^{2\bnu}}{\brho_\bnu^{2\bnu}}
                    \EE \left( \exp\big( 2 \|b(\by)\|_{L^\infty }
                      \big) \right) <\infty.
                  \end{split}
                \end{equation*} 	
                Since $r>\frac{2}{q}$ and
                $ (\varrho_j^{-1})_{j \in \NN} \in \ell^q(\NN)$, by
                Lemma \ref{lem:beta-summability} the family
                $\big(\beta_\bnu(r,\bvarrho)^{-1/2}\big)_{\bnu\in\FF}$ belongs to $\ell^q(\FF)$. 
                The relation \eqref{eq:beta-u} is proven. 
                
                From \eqref{eq:beta-u}, by H\"older's inequality we get that
                \begin{equation*}
                  \sum_{\bnu\in\FF}\| u_\bnu\|_V^p 
                  \leq 
                  \Bigg( \sum_{\bnu\in \FF}\beta_\bnu(r,\bvarrho)\|u_\bnu\|_V^2\Bigg)^{p/2} 
                  \Bigg(\sum_{\bnu\in \FF} \beta_{\bnu}(r, \bvarrho)^{-q/2} \Bigg)^{1-p/2} 
                  <\infty\,.
                \end{equation*}
              \end{proof}

              \begin{corollary}[The case of global
                supports] \label{cor:global} Assume that there exists
                a sequence of positive numbers
                $\blambda= (\lambda_j)_{j \in \NN}$ such that
	$$
	\big(\lambda_j \| \psi_j \|_{L^\infty}\big)_{j\in \NN} \in
        \ell^1(\NN) \ \ \mbox{and} \ \ (\lambda_j^{-1})_{j \in \NN}\in
        \ell^q(\NN),$$ for some $0 < q < \infty$.  Then we have
        $(\|u_\bnu\|_{V})_{\bnu\in\FF}\in \ell^p(\FF)$ with
        $\frac{1}{p}=\frac{1}{q}+\frac{1}{2}$.
      \end{corollary}
      \begin{proof}
	Let $\bnu\in \FF$.  We define the sequence
        $\brho_\bnu = (\rho_{\bnu,j})_{j \in \NN}$ by
        $\rho_{\bnu,j} := \frac{ \bnu_j }{|\bnu|\|\psi_j
          \|_{L^\infty}} $ for $j\in \supp(\bnu)$ and
        $\rho_{\bnu,j}=0$ if $j\not\in \supp(\bnu)$ and choose
        $\bvarrho=\tau\blambda$, $\tau$ is an appropriate positive
        constant.  It is obvious that
	\begin{equation*}
          \sup_{\bnu\in \FF}
          \Bigg\|
          \sum_{j\in \NN}\rho_{\bnu,j}\big| \psi_j\big|\Bigg\|_{L^\infty}\leq 1.
	\end{equation*}
	We first show that Assumption \ref{ass:Ass1} is satisfied for
        the sequence $\blambda'= (\lambda'_j)_{j \in \NN}$ with
        $\lambda_j':= \lambda_j^{1/2}$ by a similar argument as in
        \cite[Remark 2.5]{BCDM}.  From the assumption
        $ (\lambda_j^{-1})_{j \in \NN}\in \ell^q(\NN)$ we derive that
        up to a nondecreasing rearrangement,
        $\lambda'_j \geq C j^{1/(2q)}$ for some $C > 0$.  Therefore,
        $\big(\exp(-{\lambda'}_j^2)\big)_{j\in \NN}\in \ell^1(\NN)$.
        The convergence in $L^\infty(\D)$ of
        $\sum_{j\in \NN}\lambda'_j|\psi_j|$ can be proved as follows.
        \begin{equation*}
          \Bigg\|\sum_{j\in \NN}\lambda'_j|\psi_j|\Bigg\|_{L^\infty}
          \le \sup_{j \in \NN} \lambda_j^{-1/2}\sum_{j\in \NN}\lambda_j\|\psi_j\|_{L^\infty}
          < \infty.
        \end{equation*}
	
	With $r>2/q$ we have
	\begin{equation} \label{sum-estimate}
          \begin{split}
            \sum_{\|\bnu\|_{\ell^\infty}\leq r}
            \frac{\bnu!\bvarrho^{2\bnu}}{\brho_\bnu^{2\bnu}} & \leq
            \sum_{\|\bnu\|_{\ell^\infty}\leq r} \frac{
              |\bnu|^{2|\bnu|}}{\bnu^{2\bnu}}\prod_{j\in
              \supp(\bnu)}\big(\tau\sqrt{r!}\lambda_j \| \psi_j
            \|_{L^\infty}\big) ^{2\nu_j}
            \\
            & \leq \Bigg( \sum_{\|\bnu\|_{\ell^\infty}\leq r} \frac{
              |\bnu|^{|\bnu|}}{\bnu^{\bnu}}\prod_{j\in
              \supp(\bnu)}\big(\tau\sqrt{r!}\lambda_j \| \psi_j
            \|_{L^\infty}\big) ^{\nu_j}\Bigg)^{2}
            \\
            & \leq \Bigg( \sum_{\|\bnu\|_{\ell^\infty}\leq r} \frac{
              |\bnu|!}{\bnu!}\prod_{j\in \supp(\bnu)}\big(e
            \tau\sqrt{r!}\lambda_j \| \psi_j \|_{L^\infty}\big)
            ^{\nu_j}\Bigg)^{2}.
          \end{split}
	\end{equation}
        In the last step we used the inequality
		$$
		\frac{ |\bnu|^{|\bnu|}}{\bnu^{\bnu}} \le
                \frac{e^{|\bnu|} |\bnu|!}{\bnu!},
		$$
		which is immediately derived from the inequalities
                $m! \le m^m \le e^m m!$.  Since
                $ \big(\tau\sqrt{r!}\lambda_j \| \psi_j
                \|_{L^\infty}\big)_{j\in \NN} \in \ell^1(\NN) $, we
                can choose a positive number $\tau$ so that
		$$\big\| \big(e\tau\sqrt{r!}\lambda_j  \|\psi_j  \|_{L^\infty}\big)_{j\in \NN}\big\|_{\ell^1}<1.$$
		This implies by Lemma \ref{lem:alpha-summability}(ii)
                that the last sum in \eqref{sum-estimate} is finite.
                Applying Theorem \ref{thm:s=1} the desired result
                follows.
	
              \end{proof}
              \begin{corollary}[The case of disjoint supports]
                \label{cor:local}
                Assuming $\psi_j\in L^\infty(\D)$ for all $j\in \NN$
                with disjoint supports and, furthermore, that there
                exists a sequence of positive numbers
                $\blambda= (\lambda_j)_{j \in \NN}$ such that
	$$ \big(\lambda_j  \| \psi_j  \|_{L^\infty}\big)_{j\in \NN} \in \ell^2(\NN) \ \   
	{\rm and} \ \ (\lambda_j^{-1})_{j \in \NN}\in \ell^q(\NN),$$
        for some $0 < q < \infty.$ Then
        $(\|u_\bnu\|_{V})_{\bnu\in\FF}\in \ell^p(\FF)$ with
        $\frac{1}{p}=\frac{1}{q}+\frac{1}{2}$.
      \end{corollary}
      \begin{proof}
	Fix $\bnu\in \FF$, arbitrary.  For this $\bnu$ we define the
        sequence $\brho_\bnu = (\rho_j)_{j \in \NN}$ by
        $\rho_j := \frac{1 }{ \| \psi_j \|_{L^\infty}}$ for $j\in \NN$
        and $\bvarrho=\tau\blambda$, where a positive number $\tau$
        will be chosen later on.  It is clear that
	\begin{equation*}
          \Bigg\|	
          \sum_{j\in \NN}\rho_j|\psi_j|\Bigg\|_{L^\infty}\leq 1.
	\end{equation*}
	Since $ \big(\lambda_j \rho_j^{-1}\big)_{j\in \NN} \in
        \ell^2(\NN) $ and
        $(\lambda_j^{-1})_{j \in \NN}\in \ell^q(\NN)$, by H\"older's
        inequality we get
        $(\rho_j^{-1})_{j\in \NN}\in \ell^{q_0}(\NN)$ with
        $\frac{1}{q_0}=\frac{1}{2}+\frac{1}{q}$.  Hence, similarly to
        the proof of Corollary \ref {cor:global}, we can show that
        Assumption~\ref{ass:Ass1} holds for the sequence
        $\blambda'= (\lambda'_j)_{j \in \NN}$ with
        $\lambda_j':= \lambda_j^{1/2}$.  In addition, with $r>2/q$ we
        have by Lemma \ref{lem:alpha-summability}(i)
	\begin{equation*}
          \sum_{\|\bnu\|_{\ell^\infty}\leq r} \frac{\bnu!\bvarrho^{2\bnu}}{\brho_\bnu^{2\bnu}} 
          \leq  
          \sum_{\|\bnu\|_{\ell^\infty}\leq r} \Bigg(\prod_{j\in \supp(\bnu)}\big(\tau\sqrt{r!}\lambda_j\|\psi_j  
          \|_{L^\infty}\big)^{2\nu_j}  \Bigg) < \infty, 
	\end{equation*}
	since by the condition
        $ \big(\tau\sqrt{r!}\lambda_j \|\psi_j
        \|_{L^\infty}\big)_{j\in \NN} \in \ell^2(\NN)$ a positive
        number $\tau$ can be chosen so that
        $\sup_{j\in \NN}(\tau\sqrt{r!} \lambda_j \| \psi_j
        \|_{L^\infty})<1$.  
       Finally, we apply Theorem \ref{thm:s=1} to
        obtain the desired results.
	
      \end{proof}

\begin{remark} {\rm
  We comment on the situation when there exists
  $\brho= (\rho_j)_{j \in \NN}\in (0,\infty)^\infty$ such that
  \begin{equation*}
    \Bigg\|\sum_{j \in \NN} \rho_j |\psi_j | \Bigg\|_{L^\infty} =\kappa <\frac{\pi}{2}
  \end{equation*}
  and $ (\rho_j^{-1})_{j \in \NN}\in \ell^{q_0}(\NN)$ for some
  $0 < q_0 < \infty$ as given in \cite[Theorem 1.2]{BCDM}. We choose
  $\bvarrho=(\varrho_j)_{j \in \NN}$ by
	$$\varrho_j=\rho_j^{1-q_0/2}\frac{1}{\sqrt{r!} \norm[\ell^{q_0}]{ (\rho_j^{-1})_{j \in \NN}}^{q_0/2}}$$ 
	and $\brho_\bnu = (\rho_j)_{j \in \NN}$.  Then we obtain
        $(\varrho_j^{-1})_{j\in \NN}\in \ell^{q_0/(1-q_0/2)}(\NN)$ and
	$$
	\sum_{\|\bnu\|_{\ell^\infty}\leq r}
        \frac{\bnu!\bvarrho^{2\bnu}}{\brho_\bnu^{2\bnu}} =
        \sum_{\|\bnu\|_{\ell^\infty}\leq r} \bnu! \prod_{j\in
          \supp\bnu} \Bigg(\frac{\rho_j^{-q_0}}{r!\norm[\ell^{q_0}]{
            (\rho_j^{-1})_{j \in \NN}}^{q_0} }\Bigg)^{\nu_j} < \infty.
	$$
	This implies $(\|u_\bnu\|_{V})_{\bnu\in\FF}\in \ell^p(\FF)$
        with $p=q_0$.
      } \end{remark} 

\begin{remark} {\rm
	\label{rmk:weighted-summ} 
  The $\ell^p$-summability
  $(\|u_\bnu\|_{V})_{\bnu\in\FF}\in \ell^p(\FF)$ proven in Theorem
  \ref{thm:s=1}, has been used in establishing the convergence rate of
  the best $n$-term approximation of the solution $u$ to the
  parametric elliptic PDE \eqref{SPDE} \cite{BCDM}. However, such a
  property cannot   be used for estimating convergence rates of
  high-dimensional deterministic numerical approximation
  \emph{constructive} schemes such as single-level and multi-level
  versions of anisotropic sparse-grid Hermite-Smolyak interpolation
  and quadrature in Section \ref{sec:MLApprox}. 
  In the latter situation, the  weighted
  $\ell^2$-summability presented in Theorem~\ref{thm:s=1}
	\begin{equation} \nonumber 
		\sum_{\bnu\in\FF}\beta_\bnu(r,\bvarrho)\|u_\bnu\|_{V}^2 <\infty
		\ \ \ \text{with} \ \ \ \big(\beta_\bnu(r,\bvarrho)^{-1/2}\big)_{\bnu\in\FF} \in \ell^q(\FF)
	\end{equation}
and its generalization 
	\begin{equation} \label{weighted-summ}
	\sum_{\bnu\in\cF} (\sigma_\bnu \|u_\bnu\|_X)^2 <\infty \quad 
		\ \ \ \text{with} \ \ \ \big(p_{\bnu}(\tau,\lambda)\sigma_\bnu^{-1}\big)_{\bnu\in\FF} \in \ell^q(\FF)
	\end{equation}
for a Hilbert space $X$ are efficiently applied,
where $0 < q < \infty$, $\lambda, \tau \ge 0$,  
$ \big(\sigma_\bnu\big)_{\bnu\in\FF}$ is a family of positive numbers and 
	\begin{equation*} %
		p_\bnu(\tau, \lambda) := \prod_{j \in \NN} (1 + \lambda \nu_j)^\tau, \quad \bnu \in \FF.
	\end{equation*}
Weighted summability properties such as \eqref{weighted-summ} 
have been employed in
\cite{ChenlogNQuad2018,dD21,dD-Erratum22,ErnstSprgkTam18} for particular 
questions in quadrature and interpolation with respect to GMs.

In Sections 
\ref{sec:StochColl} and \ref{sec:MLApprox}, 
we will see that weighted $\ell^2$-summabilities of the form \eqref{weighted-summ} 
play a basic role in constructing approximation algorithms of 
sparse-grid interpolation and quadrature and in establishing their convergence rates. 
} \end{remark} 
\subsection{Parametric $H^s(\domain)$-analyticity and sparsity}
\label{sec:HsReg}
Whereas the previous results were, in principle, already known from
the real-variable analyses in \cite{BTNT,BCDM,BCDS}, in this and the
subsequent sections, we prove via analytic continuation the sparsity 
of the Wiener-Hermite PC expansion coefficients of the parametric solutions
of \eqref{SPDE} with log-Gaussian coefficient $a(\by) = \exp(b(\by))$
when the Wiener-Hermite PC expansion coefficients of the parametric
solution family $\{ u(\by): \by\in U\}$ are measured in higher Sobolev
norms.  In Section \ref{sec:KondrReg} we shall establish corresponding
results when the physical domain $\domain$ is a plane Lipschitz
polygon whose sides are analytic arcs.
\subsubsection{$H^s(\domain)$-analyticity}
\label{sec:HsAnalyt}
As $H^s(\domain)$ regularity in $\domain$ is relevant in particular in
conjunction with Galerkin discretization in $\domain$ by continuous,
piecewise polynomial, Lagrangian FEM, we review the elementary
regularity results from Section \ref{S:Reg}. 
To state them, we recall the Sobolev
spaces \index{space!Sobolev $\sim$} $ \Hn(\domain)$, and $\Win(\domain)$ of functions $v$ on
$\domain$ for $s \in \NN_0$, equipped with the respective norms
\[
  \|v\|_{\Hn} := \ \sum_{\bk \in \ZZ^d_+: |\bk| \le s} \|D^\bk
  v\|_{L^2}, \qquad \|v\|_{\Win} := \ \sum_{\bk \in \ZZ^d_+: |\bk| \le
    s}\|D^\bk v\|_{L^\infty}.
\]
With these definitions $H^0(\domain) = L^2(\D)$ and
$W^0_\infty(\D) = L^\infty(\D)$.  We recall from Section \ref{S:Reg}
that we identify $L^2(\D)$ with its own dual, so that the space
$H^{-1}(\D)$ is defined as the dual of $H^1_0(\D)$ with respect to the
pivot space $L^2(\D)$.

\begin{lemma} \label{lemma[regularity]} Let $s \in \NN$ and $\domain$
  be a bounded domain in $\RR^d$ with either $C^\infty$-boundary or
  with convex $C^{s-1}$-boundary.  Assume that there holds the
  ellipticity condition \eqref{eq:PDEellinC},
  $a \in W^{s-1}_\infty(\D)$ and $f \in H^{s - 2}(\D)$.  Then the
  solution $u$ of \eqref{PDE} belongs to $H^s(\domain)$ and there
  holds
  \begin{equation} \label{|u|_{Hn}<1} \|u\|_{\Hn} \ \le \
    \begin{cases} 
      \frac{\|f\|_{H^{-1}}}{\rho(a)} & \ s = 1,
      \\[1ex]
      \frac{C_{d,s}}{\rho(a)} \big(\|f\|_{H^{s - 2}} + \|a\|_{W^{s -
          1}_\infty}\|u\|_ {H^{s - 1}}\big) &\ s > 1,
    \end{cases}
  \end{equation}
  with $C_{d,s}$ depending on $d,s$, and $\rho(a)$ given as in \eqref{eq:PDEellinC}.
\end{lemma}
\begin{proof}
  Defining, for $s\in \NN$, $H^s_0(\D):= (\Hn \cap H^1_0)(\D)$, since
  $\domain$ is a bounded domain in $\RR^d$ with either
  $C^\infty$-boundary or convex $C^{s-1}$-boundary, we have the
  following norm equivalence
  \begin{equation} \label{H-norm-equivalence} \|v\|_{\Hn} \ \asymp \
    \begin{cases}
      \|v\|_{H^1_0 }, & s = 1,\\[1ex]
      \|\Delta v\|_{H^{s-2} }, & s > 1,
    \end{cases}
    \qquad \forall v \in H^s_0,
  \end{equation}
  see \cite[Theorem 2.5.1.1]{Gr}.  The lemma for the case $s =1$ and
  $s=2$ is given in \eqref{V-estimate} and \eqref{W-estimate}.  We
  prove the case $s > 2$ by induction on $s$.  Suppose that the
  assertion holds true for all $s' < s$.  We will prove it for $s$.
  Let a $\bk \in \ZZ^d_+$ with $|\bk| = s - 2$ be given.
  Differentiating both sides of \eqref{[Delta u=]} and applying the
  Leibniz rule of multivariate differentiation we obtain
  \begin{equation} \nonumber - \sum_{0 \le \bk' \le \bk}
    \binom{\bk}{\bk'}D^{\bk'} a D^{\bk-\bk'} \big(\Delta u\big) \ = \
    D^\bk f + \sum_{0 \le \bk' \le \bk} \binom{\bk}{\bk'} \big( \nabla
    D^{\bk'} a , \nabla D^{\bk-\bk'} u \big),
  \end{equation}
  see also \cite[Lemma 4.3]{BCDS}.  Hence,
  \begin{equation*} \label{derivative} - a \,D^\bk\Delta u \, = \,
    D^\bk f + \sum_{0 \le \bk' \le \bk} \binom{\bk}{\bk'} \big( \nabla
    D^{\bk'} a , \nabla D^{\bk-\bk'} u \big) + \sum_{0 \le \bk' \le
      \bk, \, \bk' \not=0} \binom{\bk}{\bk'}D^{\bk'} a D^{\bk-\bk'}
    \Delta u.
  \end{equation*}
  Taking the $L^2$-norm of both sides, by the ellipticity condition
  \eqref{PDE-epllipticity} we derive the inequality
  \begin{equation} \nonumber \rho(a) \, \| \Delta u \|_{H^{s-2}} \ \le
    \ C'_{d,s}\,\big( \|f\|_{H^{s-2}} + \|a\|_{W^{s -
        1}_\infty}\,\|u\|_{H^{s-1}} + \|a\|_{W^{s -
        2}_\infty}\,\|\Delta u\|_{H^{s-3}}\big)
  \end{equation}
  which yields \eqref{|u|_{Hn}<1} due to \eqref{H-norm-equivalence}
  and the inequality
  \[
    \|a\|_{W^{s - 2}_\infty}\,\|\Delta u\|_{H^{s-3}} \le \|a\|_{W^{s -
        1}_\infty}\,\|u\|_{H^{s-1}},
  \] where $C'_{d,s}$ is a constant depending on $d, s$ only. By
  induction, this proves that $u$ belongs to $H^s$.

\end{proof}

\begin{corollary} \label{cor:regularity} Let $s \in \NN$ and $\domain$
  be a bounded domain in $\RR^d$ with either $C^\infty$-boundary or
  convex $C^{s-1}$-boundary.  Assume that there holds the ellipticity
  condition \eqref{PDE-epllipticity}, $a \in W^{s-1}_\infty(\D)$ and
  $f \in H^{s - 2}(\D)$. Then the solution $u$ of \eqref{PDE} belongs
  to $H^s(\D)$ and there holds the estimate
  \begin{equation*} \label{|u|_{Hn}<} \|u\|_{\Hn} \ \le \
    \frac{\|f\|_{H^{s - 2}}}{\rho(a)}
    \begin{cases}
      1, & \ s = 1, \\[1ex]
      C_{d,s}\Big(1 + \frac{\|a\|_{W^{s -
            1}_\infty}}{{\rho(a)}}\Big)^{s-1}, & \ s > 1,
    \end{cases}
  \end{equation*}
  where $C_{d,s}$ is a constant depending on $d, s$ only, 
  and $\rho(a)$ is given as in \eqref{eq:PDEellinC}.
\end{corollary}

We need the following lemma.
\begin{lemma}\label{a-nu}
  Let $s\in \NN$ and assume that $b(\by)$ belongs to
  $ W^{s}_\infty(\D)$.  Then we have
  \begin{equation*}
    \|a(\by)\|_{W^{s}_\infty} 
    \leq 	
    C\|a(\by)\|_{L^\infty} \big(1+ \|  b(\by)\|_{W^{s}_\infty}\big)^{s}\,,
  \end{equation*}
  where the constant $C$ depends on $s$ and $m$ but is independent of
  $\by$.
\end{lemma}
\begin{proof}
  For $\balpha=(\alpha_1,\ldots,\alpha_d)\in \NN_0^d$ with
  $1\leq |\balpha|\leq s$, we observe that for $\alpha_j>0$ the
  product rule implies
  \begin{equation}\label{eq:Leibniz}
    D^\balpha a(\by) = D^{\balpha-\bee_j} \big[a(\by) D^{\bee_j}b(\by) \big]
    =   \sum_{0\leq \bgamma
      \leq \balpha-\bee_j}\binom{\balpha-\bee_j}{\bgamma}
    D^{\balpha-\bgamma}  b(\by) D^{\bgamma }a(\by)  \,.
  \end{equation}
 Here, we recall that $(\bee_j)_{j=1}^d$ is the standard basis of $\R^d$. Taking norms on both sides, we can estimate
  \begin{equation*}
    \begin{split} 
      \| D^\balpha a(\by)\|_{L^\infty} & =\big\| D^{\balpha-\bee_j}
      \big[a(\by) D^{\bee_j}b(\by) \big] \big\|_{L^\infty}
      \\
      &\leq \sum_{0\leq \bgamma \leq
        \balpha-\bee_j}\binom{\balpha-\bee_j}{\bgamma} \|
      D^{\balpha-\bgamma} b(\by) \|_{L^\infty} \| D^{\bgamma
      }a(\by)\|_{L^\infty}
      \\
      & \leq C \Bigg( \sum_{0\leq \bgamma \leq \balpha-\bee_j} \|
      D^{\bgamma} a(\by) \|_{L^\infty}\Bigg) \Bigg( \sum_{|\bk|\leq
        s}\|D^\bk b(\by)\|_{L^\infty}\Bigg) \,.
    \end{split}
  \end{equation*}
Similarly, each
  term $\| D^{\bgamma} a(\by) \|_{L^\infty}$ with $|\bgamma|>0$ can be
  estimated
  \begin{equation*}
    \| D^{\bgamma}  a(\by) \|_{L^\infty}\leq  C \Bigg( \sum_{0\leq \bgamma' 
      \leq  \bgamma-\bee_j} \| D^{\bgamma'}  a(\by) \|_{L^\infty}\Bigg) \Bigg( \sum_{|\bk|\leq s}\|D^\bk b(\by)\|_{L^\infty}\Bigg) \,
  \end{equation*}
  if $\gamma_j>0$.  This implies
  \begin{equation*}
    \begin{split} 
      \| D^\balpha a(\by)\|_{L^\infty} \leq C \|a(\by)\|_{L^\infty}
      \Bigg(1+ \sum_{|\bk|\leq s}\|D^\bk
      b(\by)\|_{L^\infty}\Bigg)^{|\balpha|} \,,
    \end{split}
  \end{equation*}
  for $1\leq |\balpha|\leq s$. Summing up these terms with
  $\|a(\by)\|_{L^\infty}$ we obtain the desired estimate.

\end{proof}
\begin{proposition}\label{prop3}
  Let $s \in \NN$ and $\domain$ be a bounded domain in $\RR^d$ with
  either $C^\infty$-boundary or convex $C^{s-1}$-boundary.  Assume
  that \eqref{eq:leqkappa} holds and all the functions $\psi_j$ belong
  to $ W^{s-1}_\infty(\D)$.  Let $\mfu \subseteq \supp(\brho)$ be a
  finite set and let $\by_0=(y_{0,1},y_{0,2},\ldots) \in U$ be such
  that $b(\by_0)$ belongs to $ W^{s-1}_\infty(\D)$.  Then the solution
  $u$ of \eqref{SPDE} is holomorphic in $\mathcal{S}_\mfu (\brho) $ as
  a function in variables
  $\bz_{\mfu}=(z_j)_{j \in \NN}\in \mathcal{S}_\mfu (\by_0,\brho) $
  taking values in $\Hn(\D)$ where $z_j = y_{0,j}$ for
  $j\not \in \mfu$ held fixed\,.
\end{proposition}
\begin{proof} Let $\mathcal{S}_{\mfu,N} (\brho)$ be given in
  \eqref{eq:SUN} and
  $\bz_{\mfu} = (y_j+\im \xi_j)_{j\in \NN}\in \mathcal{S}_{\mfu}
  (\by_0,\brho)$ with
  $(y_j+\im \xi_j)_{j\in \mfu}\in \mathcal{S}_{\mfu,N} (\brho)$.
  Then we have from Corollary \ref{cor:regularity}
  \begin{equation*}
    \begin{split} 
      \|u(\bz_{\mfu})\|_{\Hn} & \leq C \rho(a(\bz_{\mfu})) \Big(1 +
      \rho(a(\bz_{\mfu})) \|a(\bz_{\mfu})\|_{W^{s-1}_\infty}
      \Big)^{s-1} \,.
    \end{split}
  \end{equation*}
  Using Lemma \ref{a-nu} we find
  \begin{equation*}
    \begin{split} 
      \|a(\bz_{\mfu})\|_{W_\infty^{s-1}} & \leq C
      \|a(\bz_{\mfu})\|_{L^\infty}\big(1+ \|
      b(\bz_{\mfu})\|_{W_\infty^{s-1}}\big)^{s-1}
      \\
      & \leq C \|a(\bz_{\mfu})\|_{L^\infty}\Bigg(1+\|
      b(\by_0)\|_{W_\infty^{s-1}} + \Bigg\| \sum_{j\in \mfu
      }(y_j-y_{0,j}
      +\im\xi_j)\psi_j\Bigg\|_{W_\infty^{s-1}}\Bigg)^{s-1}
      \\
      & \leq C \|a(\bz_{\mfu})\|_{L^\infty}\Bigg(1+\|
      b(\by_0)\|_{W_\infty^{s-1}} + \sum_{j\in \mfu
      }(N+\rho_j)\|\psi_j \|_{W_\infty^{s-1}}\Bigg)^{s-1}
    \end{split}
  \end{equation*}
  and
  \begin{equation} \label{eq:azu0} \|a(\bz_\mfu)\|_{L^\infty} \leq
    \exp\Bigg(\|b(\by_0)\|_{L^\infty}+\Bigg\|\sum_{j\in \mfu}
    (y_j-y_{0,j}+\im\xi_j)\psi_j\Bigg\|_{L^\infty} \Bigg) <\infty.
  \end{equation} 
  From this and \eqref{eq:R-1} we obtain
  \begin{equation*}
    \|u(\bz_{\mfu})\|_{\Hn}  \leq C <\infty
  \end{equation*}
  which implies the map $\bz_\mfu\to u(\bz_{\mfu})$ 
  is holomorphic on the set $\mathcal{S}_{\mfu,N} (\brho)$ 
  as a consequence of  \cite[Lemma 2.2]{BLN}. 
  For more details we refer the reader to \cite[Examples 1.2.38 and 1.2.39]{JZdiss}. 
  Since $N$ is arbitrary we conclude that
  the map $\bz_\mfu\to u(\bz_{\mfu})$ is holomorphic on $\mathcal{S}_{\mfu} (\brho)$.
\end{proof}
\subsubsection{Sparsity of Wiener-Hermite PC expansion coefficients}
\label{sec:Summ}
For sparsity of $H^s$-norms of Wiener-Hermite PC expansion
coefficients we need the following assumption.

\noindent
\begin{assumption}\label{ass:Ass2}
  Let $s\in \NN$.  For every $j\in \NN$,
  $\psi_j \in W^{s-1}_\infty(\D)$ and there exists a positive sequence
  $(\lambda_j)_{j\in \NN}$ such that
  $\big(\exp(-\lambda_j^2)\big)_{j\in \NN}\in \ell^1(\NN)$ and the
  series
$$\sum_{j\in \NN}\lambda_j|D^{\balpha}\psi_j|$$ 
converges in $L^\infty(\D)$ for all $\balpha\in \NN_0^d$ with
$|\balpha|\leq s-1$.
\end{assumption}

As a consequence of \cite[Theorem 2.2]{BCDM} we have the following
\begin{lemma}\label{lem:s-a.e}
  Let Assumption \ref{ass:Ass2} hold.  Then the set
  $U_{s-1} := \{ \by\in U: b(\by)\in W_\infty^{s-1}(\domain) \}$ has
  full measure, i.e., $\gamma(U_{s-1}) = 1$.  Furthermore,
  $\mathbb{E}(\exp(k\|b(\cdot)\|_{W_\infty^{s-1}}))$ is finite for all
  $k\in [0,\infty).$
\end{lemma}
The $H^s$-analytic continuation of the parametric solutions
$\{u(\by): \by\in U\}$ to $\mathcal{S}_{\mfu} (\brho)$ leads to the
following result on parametric $H^s$-regularity.
\begin{lemma} \label{lemma:hs-regularity} Let $\domain \subset \RR^d$
  be a bounded domain with either $C^\infty$-boundary or convex
  $C^{s-1}$-boundary.  Assume that for each $\bnu\in \FF$, there
  exists a sequence
  $\brho_\bnu= (\rho_{\bnu,j})_{j \in \NN}\in [0,\infty)^\infty$ such
  that $\supp(\bnu)\subseteq \supp(\brho_\bnu)$, and such that
  \begin{equation*} \label{eq:Dkpsisum-lemma} \sup_{\bnu\in \FF}
    \sum_{|\balpha|\leq s-1} \Bigg\| \sum_{j\in
      \NN}\rho_{\bnu,j}|D^{\balpha}\psi_j|\Bigg\|_{L^\infty} \leq
    \kappa <\frac{\pi}{2}.
  \end{equation*}
  Then we have
  \begin{equation} \label{ineq1-lemma} \|\partial^{\bnu}u(\by)\|_{\Hn}
    \leq C \frac{\bnu!}{\brho_\bnu^\bnu}
    \exp\big(\|b(\by)\|_{L^\infty}\big) \Big\{1+
    \exp(2\|b(\by)\|_{L^\infty}) \big(1+\|
    b(\by)\|_{W_\infty^{s-1}}\big)^{s-1} \Big\}^{s-1},
  \end{equation}
  where $C$ is a constant depending on $\kappa$, $d$, $s$ only.
\end{lemma}
\begin{proof}
  Let $\bnu\in \FF$ with $\mfu=\supp(\bnu)$ and $\by\in U$ such that
  $b(\by)\in W_\infty^{s-1}(\D)$.  Let furthermore
  $\Cc_{\by,\mfu}(\brho_\bnu)$ and $\Cc_\mfu(\by,\brho_\bnu)$ be given
  as in \eqref{eq:C-rho} and \eqref{eq:C-rho-y}.  Using Cauchy's
  formula as in the proof of Lemma \ref{lem:estV} we obtain
  \begin{equation} \label{ineq1-proof} \|\partial^{\bnu}u(\by)\|_{\Hn}
    \leq \frac{\bnu!}{\brho_\bnu^\bnu} \sup_{\bz_{\mfu}\in
      C_\mfu(\by,\brho_\bnu)} \|u(\bz_{\mfu})\|_{\Hn} \,.
  \end{equation}
  For $\bz_\mfu=(z_j)_{j\in \NN} \in C_\mfu(\by,\brho_\bnu)$ we can
  write $z_j = y_j + \eta_j + \im\xi_j \in \Cc_{\by,j}(\brho_{\bnu})$
  with $|\eta_j | \le \rho_{\bnu,j}$ and $|\xi_j| \le \rho_{\bnu,j}$
  for $j\in \mfu$ and hence we get
  \begin{equation*}
    \begin{split} 
      \|D^\balpha b(\bz_\mfu)\|_{L^\infty} & = \Bigg\|
      D^\balpha\Big(b(\by)+\sum_{j\in \mfu} ( \eta_j+\im\xi_j)\psi_j
      \Big)\Bigg\|_{L^\infty}
      \\
      & \leq \|D^\balpha b(\by)\|_{L^\infty} +  \sqrt{2}\,\Bigg\| \sum_{j\in \mfu} \rho_{\bnu,j} |D^\balpha \psi_j|\Bigg\|_{L^\infty}\\
      &\leq \|D^\balpha b(\by)\|_{L^\infty} +\kappa\sqrt{2}\,.
    \end{split}
  \end{equation*}
  In addition we have
  \begin{equation} \label{eq:R-11} \frac{1}{\rho(a( \bz_{\mfu}))} \leq
    \frac{\exp(\| b(\by+ \sum_{j\in \mfu}\eta_j
      \psi_j\|_{L^\infty})}{\cos(\|\sum_{j\in \mfu}\xi_j
      \psi_j\|_{L^\infty})} \leq \frac{ \exp\big( \kappa +
      \|b(\by)\|_{L^\infty}\big)}{\cos\kappa}\,
  \end{equation} 
  and
  \begin{equation} \label{eq:azu} \|a(\bz_\mfu)\|_{L^\infty} =
    \Bigg\|\exp\Bigg(b(\by)+\sum_{j\in \mfu} (
    \eta_j+\im\xi_j)\psi_j\Bigg)\Bigg\|_{L^\infty} \leq
    e^{\kappa\sqrt{2}} \exp(\|b(\by)\|_{L^\infty})\,.
  \end{equation} 
  Consequently, we can bound
  \begin{equation*}
    \begin{split}
      \|a(\bz_\mfu)\|_{W^{s-1}_\infty} & \leq C
      \|a(\bz_\mfu)\|_{L^\infty}\big(1+\|
      b(\bz_{\mfu})\|_{W_\infty^{s-1}}\big)^{s-1}
      \\
      & \leq C\exp(\|b(\by)\|_{L^\infty}) \big(1+\|
      b(\by)\|_{W_\infty^{s-1}}\big)^{s-1}\,.
    \end{split}
  \end{equation*}
  Now Corollary \ref{cor:regularity} implies the inequality
  \begin{equation} \label{ineq:u(z_u)} \sup_{\bz_{\mfu}\in
      C_\mfu(\by,\brho_\bnu)} \|u(\bz_\mfu)\|_{\Hn} \leq C
    \exp\big(\|b(\by)\|_{L^\infty}\big) \Big\{1+
    \exp(2\|b(\by)\|_{L^\infty}) \big(1+\|
    b(\by)\|_{W_\infty^{s-1}}\big)^{s-1} \Big\}^{s-1},
  \end{equation}
  which together with \eqref{ineq1-proof} proves the lemma.
\end{proof}

We are now in position to formulate sparsity results 
for the $H^s$-norms of Wiener-Hermite PC expansion coefficients of the solution $u$.

\begin{theorem}[General case] \label{thm:hs-sum} Let $s,r \in \NN$ and
  $\domain \subset \RR^d$ denote a bounded domain with either
  $C^\infty$-boundary or convex $C^{s-1}$-boundary.  Let further
  Assumption \ref{ass:Ass2} hold, assume that $f \in H^{s - 2}(\D)$,
  and assume given a sequence
  $\bvarrho= (\varrho_j)_{j \in \NN}\subset (0,\infty)^\infty$ that
  satisfies $(\varrho_j^{-1})_{j \in \NN}\in \ell^q(\NN)$ for some
  $0 < q < \infty$.  Assume in addition that, for each $\bnu\in \FF$,
  there exists a sequence
  $\brho_\bnu= (\rho_{\bnu,j})_{j \in \NN}\in [0,\infty)^\infty$ such
  that $\supp(\bnu)\subseteq \supp(\brho_\bnu)$, and such that, with
  $r>2/q$,
  \begin{equation*} \label{eq:Dkpsisum} \sup_{\bnu\in \FF}
    \sum_{|\balpha|\leq s-1} \Bigg\| \sum_{j\in
      \NN}\rho_{\bnu,j}|D^{\balpha}\psi_j|\Bigg\|_{L^\infty} \leq
    \kappa <\frac{\pi}{2}, \quad \text{and}\quad
    \sum_{\|\bnu\|_{\ell^\infty}\leq r}
    \frac{\bnu!\bvarrho^{2\bnu}}{\brho_\bnu^{2\bnu}} <\infty.
  \end{equation*}
  Then there holds, with $\beta_\bnu(r, \bvarrho)$ as in \eqref{beta},
		\begin{equation} \label{ineq:sum<infty} 
			\sum_{\bnu\in\FF}\beta_\bnu(r,\bvarrho)\|u_\bnu\|_{\Hn}^2 <\infty
			\ \ \ with \ \ \ \big(\beta_\bnu(r,\bvarrho)^{-1/2}\big)_{\bnu\in\FF} \in \ell^q(\FF)\,.
		\end{equation}
 Furthermore,
	$$
	(\|u_\bnu\|_{\Hn})_{\bnu\in\FF}\in \ell^p(\FF) \ \ \  with
	\ \ \ \frac{1}{p}=\frac{1}{q}+\frac{1}{2}.
	$$

\end{theorem}
\begin{proof}
  Arguing as in the proof of \cite[ Theorem 3.3]{BCDM} we obtain that
  for any $r\in \NN$ there holds following generalization of the
  Parseval-type identity
  \begin{equation} \label{eq:equal-H^s} \sum_{\bnu\in
      \Ff}\beta_\bnu(r,\bvarrho)\|u_\bnu\|_{H^s}^2 =
    \sum_{\|\bnu\|_{\ell^\infty}\leq r} \frac{\bvarrho^{2\bnu}}{\bnu!}
    \int_U\| \partial^\bnu u(\by)\|_{H^s}^2\rd\gamma(\by)\,.
  \end{equation} 
By \eqref{ineq:u(z_u)}, Lemma \ref{lem:s-a.e} 
and H\"older's inequality 
we derive that
  \[
    \int_U \bigg( \sup_{\bz_{\mfu}\in C_\mfu(\by,\brho_\bnu)}
    \|u(\bz_{\mfu})\|_{\Hn} \bigg)^2 \rd\gamma(\by) \leq C
  \]
  and in particular, $\mathbb{E}(\|u(\by)\|_{H^s}^k)$ is finite for
  all $k\in [0,\infty)$.  
  Now \eqref{eq:equal-H^s}, Lemma \ref{lemma:hs-regularity} 
  and our assumption give
  \begin{equation} \nonumber
    \begin{split}
    	\sum_{\bnu\in\FF}\beta_\bnu(r,\bvarrho)\|u_\bnu\|_{\Hn}^2 
    	&=
      \sum_{\|\bnu\|_{\ell^\infty}\leq r}
      \frac{\bvarrho^{2\bnu}}{\bnu!} \int_U\| \partial^\bnu
      u(\by)\|_{H^s}^2\rd\gamma(\by) 
      \\
      &\leq
      C^2\sum_{\|\bnu\|_{\ell^\infty}\leq r}
      \frac{\bnu!\bvarrho^{2\bnu}}{\brho_\bnu^{2\bnu}}\int_U
      \rd\gamma(\by)
      = C^2 \sum_{\|\bnu\|_{\ell^\infty}\leq r}
      \frac{\bnu!\bvarrho^{2\bnu}}{\brho_\bnu^{2\bnu}} <\infty,
    \end{split}
  \end{equation} 
  where $C$ is the constant in \eqref{ineq1-lemma}.  
  As in the proof of Theorem \ref{thm:s=1}, by
  Lemma \ref{lem:beta-summability} the family
 $\big(\beta_\bnu(r,\bvarrho)^{-1/2}\big)_{\bnu\in\FF}$ belongs to $\ell^q(\FF)$.  
 The relation \eqref{ineq:sum<infty} is proven. 
 
  The assertion
  $(\|u_\bnu\|_{\Hn})_{\bnu\in\FF}\in \ell^p(\FF)$ can be proved in
  the same way as in the proof of Theorem~\ref{thm:s=1}.
\end{proof}
Similarly to Corollaries \ref{cor:global} and \ref{cor:local} from
Theorem~\ref{thm:hs-sum} we obtain
\begin{corollary}[The case of global supports]
  Let $s \in \NN$ and $\domain \subset \RR^d$ denote a bounded domain
  with either $C^\infty$-boundary or convex $C^{s-1}$-boundary.
  Assume that for all $j\in \NN$ holds $\psi_j\in W^{s-1}_\infty(\D)$,
  and that $f\in H^{s-2}(\D)$.  Assume further that there exists a
  sequence of positive numbers $\blambda= (\lambda_j)_{j \in \NN}$
  such that
	$$
	\big(\lambda_j \| \psi_j \|_{W_\infty^{s-1}}\big)_{j\in \NN}
        \in \ell^1(\NN) \ \ and \ \ (\lambda_j^{-1})_{j \in \NN}\in
        \ell^q(\NN),$$ for some $0 < q < \infty$.  
  Then we have
        $(\|u_\bnu\|_{H^s})_{\bnu\in\FF}\in \ell^p(\FF)$ 
  with
        $\frac{1}{p}=\frac{1}{q}+\frac{1}{2}$.
\end{corollary}

\begin{corollary}[The case of disjoint supports]
  Let $s \in \NN$ and $\domain \subset \RR^d$ denote a bounded domain
  with either $C^\infty$-boundary or convex $C^{s-1}$-boundary.
  Assume that $f\in H^{s-2}(\D)$ and for all $j\in \NN$ holds
  $\psi_j\in W^{s-1}_\infty(\D)$ with disjoint supports.  Assume
  further that there exists a sequence of positive numbers
  $\blambda= (\lambda_j)_{j \in \NN}$ such that
		$$
                \big(\lambda_j \| \psi_j
                \|_{W_\infty^{s-1}}\big)_{j\in \NN} \in \ell^2(\NN) \
                \ and \ \ (\lambda_j^{-1})_{j \in \NN}\in
                \ell^q(\NN),$$ for some $0 < q < \infty$.  

Then
                $(\|u_\bnu\|_{H^s})_{\bnu\in\FF}\in \ell^p(\FF)$ 
with
                $\frac{1}{p}=\frac{1}{q}+\frac{1}{2}$.
\end{corollary}
              \subsection{Parametric Kondrat'ev analyticity and sparsity}
              \label{sec:KondrReg}
              In the previous section, we investigated the weighted
              $\ell^2$-summability and $\ell^p$-summability of
              Wiener-Hermite PC expansion coefficients of parametric
              solutions measured in the standard Sobolev spaces
              $H^s(\domain)$.  We assumed that $\domain\subset
              \RR^d$ %
              with boundary $\partial\domain$ of sufficient smoothness
              (depending on $s$).  In this section we consider in
              space dimension $d=2$ the case when the physical domain
              $\domain$ is a polygonal domain.  In such domains,
              elliptic regularity shift results and shift theorems in
              $\domain$ hold in Kondrat'ev spaces which are corner-weighted Sobolev spaces.  
              We
              refer to \cite{Gr,MazRoss2010} and the references there
              for an extensive survey.

              To state corresponding results for the log-Gaussian
              parametric elliptic problems, we first review
              definitions of the weighted Sobolev spaces of Kondrat'ev
              type and results from \cite{BLN} on the holomorphy of
              parametric solutions in weighted Kondrat'ev spaces in
              polygonal domains $\domain$.  Then, we establish
              summability results of the coefficients of
              Wiener-Hermite PC expansions of the parametric solutions
              in Kondrat'ev spaces.  
              FE approximation results
              for Wiener-Hermite PC expansion coefficient functions
              which are in these spaces were provided in Section
              \ref{S:FEM}.
              \subsubsection{Parametric $K^s_{\varkappa}(\D)$-holomorphy}
              \label{sec:KondrAn}
              We recall the Kondrat'ev spaces \index{space!Kondrat'ev $\sim$} in a bounded
              polygonal domain $\domain$ introduced in Section
              \ref{S:FncSpc}: for $s\in \NN_0$ and $\varkappa\in \RR$,
              \begin{equation*}
                \Kk^s_\varkappa(\domain)
                :=
                \big\{
                u: \domain \to \CC: 
                \ r_\domain^{|\balpha|-\varkappa}D^\balpha u\in L^2(\domain), 
                |\balpha|\leq s \big\}
              \end{equation*}
              and
              \begin{equation*}
                \Ww^s_\infty(\domain)
                := 
                \big\{u: \domain\to \CC: \ r_\domain^{|\balpha|}D^\balpha u\in L^\infty(\domain),\ 
                |\balpha|\leq s \big\}.
              \end{equation*}
              The weighted Sobolev norms in these spaces are given in
              Section \ref{S:FncSpc}.
              \begin{lemma}\label{kon-lem-1} 
                Let $s\in \NN_0$.
                Assume that $\by\in U$ is such that
                $b(\by) \in \Ww^{s}_\infty(\D)$.  

                Then
                \begin{equation*}
                  \|a(\by)\|_{\Ww^{s}_\infty} 
                  \leq 
                  C\|a(\by)\|_{L^\infty} \big(1+ \|  b(\by)\|_{\Ww^s_\infty}\big)^{s}\,,
                \end{equation*}
                where the constant $C$ depends on $s$ and $m$.
              \end{lemma}
              \begin{proof}
                The proof proceeds along the lines of the proof of Lemma \ref{a-nu}.  
                Let
                $\balpha=(\alpha_1,\ldots,\alpha_d)\in \NN_0^d$ with
                $1\leq |\balpha|\leq s$ and  recall that $(\bee_j)_{j=1}^d$ is the standard basis of $\R^d$. Assuming that $\alpha_j>0$
                we have \eqref{eq:Leibniz}.  We apply corner-weighted
                norms to both sides of \eqref{eq:Leibniz}.  This
                implies
                \begin{equation*}
                  \begin{split} 
                    \| r_\domain^{|\balpha|} D^\balpha
                    a(\by)\|_{L^\infty} & =\big\| D^{\balpha-\bee_j}
                    \big[a(\by) D^{\bee_j}b(\by) \big]
                    \big\|_{L^\infty}
                    \\
                    &\leq \sum_{0\leq \bgamma \leq
                      \balpha-\bee_j}\binom{\balpha-\bee_j}{\bgamma}
                    \| r_\domain^{|\balpha-\bgamma|}
                    D^{\balpha-\bgamma} b(\by) \|_{L^\infty} \|
                    r_\domain^{|\bgamma|} D^{\bgamma
                    }a(\by)\|_{L^\infty}
                    \\
                    & \leq C \Bigg( \sum_{0\leq \bgamma \leq
                      \balpha-\bee_j} \| r_\domain^{|\bgamma|}
                    D^{\bgamma} a(\by) \|_{L^\infty}\Bigg) \Bigg(
                    \sum_{|\bk|\leq s}\|r_\domain^{|\bk|} D^\bk
                    b(\by)\|_{L^\infty}\Bigg)
                    \\
                    & = C \Bigg( \sum_{0\leq \bgamma \leq
                      \balpha-\bee_j} \| r_\domain^{|\bgamma|}
                    D^{\bgamma} a(\by) \|_{L^\infty}\Bigg)\|
                    b(\by)\|_{\Ww^s_\infty}\,.
                  \end{split}
                \end{equation*}
                Similarly, if $\gamma_j>0$, each term
                $\| r_\domain^{|\bgamma|} D^{\bgamma}
                a(\by)\|_{L^\infty}$ with $|\bgamma|>0$ can be
                estimated
                \begin{equation*}
                  \| r_\domain^{|\bgamma|} D^{\bgamma}  a(\by) \|_{L^\infty}
                  \leq  
                  C 
                  \Bigg( 
                  \sum_{0\leq \bgamma' \leq \bgamma-\bee_j} 
                  \| r_\domain^{|\bgamma'|} D^{\bgamma'} a(\by)\|_{L^\infty}
                  \Bigg) 
                  \| b(\by)\|_{\Ww^s_\infty} 
                  .
                \end{equation*}
                This implies
                \begin{equation*}
                  \begin{split} 
                    \| r_\domain^{|\balpha|} D^\balpha
                    a(\by)\|_{L^\infty} \leq C \|a(\by)\|_{L^\infty}
                    \big(1+ \|
                    b(\by)\|_{\Ww^s_\infty}\big)^{|\balpha|} \,,
                  \end{split}
                \end{equation*}
                for $1\leq |\balpha|\leq s$.  This finishes the proof.

              \end{proof}

              We recall the following result from \cite[Theorem
              1]{BLN}.
              \begin{theorem}\label{thm:bacuta}
                Let $\domain \subset \R^2$ be a polygonal domain,
                $\eta_0>0$, $s\in \NN$ and $N_s=2^{s+1}-s-2$. 
                Let
                $a\in L^\infty(\domain,\CC)$.

                Then there exist $\tau$ and $C_s$ with the following
                property: for any $a\in \Ww^{s-1}_\infty(\D)$ and for
                any $\varkappa \in \R$ such that
$$|\varkappa|<\eta:=\min\{\eta_0,\tau^{-1}\|a\|_{L^\infty}^{-1}\rho(a) \},$$
the operator $P_a$ defined in \eqref{PDE} induces an isomorphism
\begin{equation*}
  P_a: \Kk_{\varkappa+1}^{s}(\D) \cap \{ u|_{\partial \domain }=0\} \to \Kk_{\varkappa-1}^{s-2}(\D)
\end{equation*}
such that $P_a^{-1}$ depends analytically on the coefficients $a$ and
has norm
\begin{equation*}
  \|P_a^{-1}\| 
  \leq 
  C_s \big(\rho(a)-\tau|\varkappa|\|a\|_{L^\infty}\big)^{-N_s-1}\|a\|_{\Ww_\infty^{s-1}}^{N_s}\,.
\end{equation*}
The bound of $\tau$ and $C_s$ depends only on $s$, 
$\domain$ and $\eta_0$.
\end{theorem}

Applying this result to our setting, we obtain the following
parametric regularity.
\begin{theorem}\label{thm:Kond}
  Suppose $\eta_0>0$, $\psi_j \in \Ww^{s-1}_\infty(\D)$ for all
  $j\in \NN$ and that \eqref{eq:leqkappa} holds.  Let
  $\mfu\subseteq \supp(\brho)$ be a finite set.  Let further
  $\by_0=(y_{0,1},y_{0,2},\ldots) \in U$ be such that $b(\by_0)$
  belongs to $\Ww^{s-1}_\infty(\D)$.  We denote
  \begin{equation*}
    \vartheta :
    = 
    \inf_{\bz_\mfu \in \mathcal{S}_\mfu (\by_0,\brho) } \rho\big(a(\bz_\mfu)\big)\|a(\bz_\mfu)\|_{L^\infty}^{-1}\,.
  \end{equation*}
  Let $\tau > 0$ be as given in Theorem \ref{thm:bacuta}.

  Then there exists a positive constant $C_s$ such that for
  $\varkappa \in \R$ with
  $|\varkappa|\leq \min\{\eta_0, \tau^{-1}\vartheta/2 \},$ and for
  $f\in \Kk^{s-2}_{\varkappa-1}(\D)$, the solution $u$ of \eqref{SPDE}
  is holomorphic in the cylinder $ \mathcal{S}_\mfu (\brho) $ as a
  function in variables
  $\bz_\mfu=(z_j)_{j \in \NN}\in \mathcal{S}_\mfu (\by_0,\brho)$
  taking values in $\Kk_{\varkappa+1}^s(\D)\cap V$, where
  $z_j = y_{0,j}$ for $j\not \in \mfu$ held fixed.  Furthermore, we
  have the estimate
  \begin{equation*}
    \|u(\bz_\mfu)\|_{\Kk_{\varkappa+1}^s} 
    \leq 
    C_s \frac{1}{ \big(\rho( a(\bz_\mfu) \big)^{N_{s}+1}} \|a(\bz_\mfu)\|_{\Ww^{s-1}_\infty}^{N_s}\,.
  \end{equation*}
\end{theorem}

\begin{proof} 
  Observe first that for the parametric coefficient $a(\bz_\mfu)$, the
  conditions of Proposition \ref{prop:holoh1} are satisfied.

  Thus, the solution $u$ is holomorphic in $\mathcal{S}_\mfu (\brho)$
  as a $V$-valued map in variables
  $\bz_\mfu=(z_j)_{j \in \NN}\in \mathcal{S}_\mfu (\by_0,\brho)$.  We
  assume that $\vartheta>0$.  Let $\mathcal{S}_{\mfu,N} (\brho)$ be
  given in \eqref{eq:SUN} and
  $\bz_{\mfu} = (y_j+\im \xi_j)_{j\in \NN}\in \mathcal{S}_{\mfu}
  (\by_0,\brho)$ with
  $(y_j+\im \xi_j)_{j\in \mfu}\in \mathcal{S}_{\mfu,N} (\brho)$.
  From Lemma \ref{kon-lem-1} we have
  \begin{equation*}
    \|a(\bz_\mfu)\|_{\Ww^{s-1}_\infty}
    \leq 	C\|a(\bz_\mfu)\|_{L^\infty} \Big(1+\| b(\bz_\mfu)\|_{\Ww^{s-1}_\infty}\Big)^{s-1}\,.
  \end{equation*}
  Furthermore
  \begin{equation*} \label{eq:bzu}
    \begin{split} 
      \| b(\bz_\mfu)\|_{\Ww^{s-1}_\infty} &= \sum_{|\balpha|\leq
        s-1}\Bigg\|r_\domain^{|\balpha|}\sum_{j\in \NN}(y_j+\im
      \xi_j)D^{\balpha}\psi_j\Bigg\|_{L^\infty}
      \\
      & \leq \sum_{j\in \mfu}(|y_j-y_{0,j}|+\rho_j)\|
      \psi_j\|_{\Ww^{s-1}_\infty}+\| b(\by_0)\|_{\Ww^{s-1}_\infty} <
      \infty\,.
    \end{split}
  \end{equation*} 
  This together with \eqref{eq:azu0} implies
  $ \|a(\bz_\mfu)\|_{\Ww^{s-1}_\infty} \leq C $. From the condition of
  $\varkappa$ we infer $ |\varkappa| \tau \leq {\vartheta}/{2} $ which
  leads to
  \begin{equation*}
    \tau |\varkappa| \|a(\bz_\mfu)\|_{L^\infty} \leq \rho(a(\bz_\mfu))/2.
  \end{equation*} 
  As a consequence we obtain
  \begin{equation*}
    \big(\rho(a(\bz_\mfu))-\tau |\varkappa|\|a(\bz)\|_{L^\infty}\big)^{-1} 
    \leq 
    \frac{1}{ \rho( a(\bz_\mfu) }\,.
  \end{equation*}
  Since the function $\exp$ is analytic in
  $\mathcal{S}_{\mfu,N} (\brho)$, the assertion follows for the case
  $\vartheta>0$ by applying Theorem \ref{thm:bacuta}.  In addition,
  for
  $\bz_{\mfu} = (z_j)_{j\in \NN}\in \mathcal{S}_{\mfu} (\by_0,\brho)$
  with $(z_j)_{j\in \mfu}\in \mathcal{S}_{\mfu,N} (\brho)$, we have
  \begin{equation*}
    \rho(a(\bz_\mfu))\|a(\bz_\mfu)\|_{L^\infty}^{-1} 
    \geq C>0,
  \end{equation*}
  From this we conclude that $u$ is holomorphic in the cylinder
  $\mathcal{S}_{\mfu,N} (\brho)$ as a $\Kk_{1}^s(\D)\cap V$-valued
  map, by again Theorem \ref{thm:bacuta}.  
  This completes the proof.
\end{proof}
\begin{remark} {\rm \label{rmk:kappa} The value of $\vartheta$ depends on
  the system $(\psi_j)_{j\in \NN}$.  Assume that $\psi_j=j^{-\alpha}$
  for some $\alpha>1$.  Then for any $\by\in U$, $\brho$ satisfying
  \eqref{eq:leqkappa}, and finite set $\mfu\subset\supp(\brho)$ we
  have
  \begin{equation*}
    \vartheta = \inf_{\bz_\mfu \in \mathcal{S}_\mfu (\by,\brho) } \frac{\Re[\exp(\sum_{j\in \NN}(y_j+\im \xi_j)j^{-\alpha})]}{\exp(\sum_{j\in \NN}y_jj^{-\alpha})} \geq \cos\kappa\,.
  \end{equation*}
  We consider another case when there exists some $\psi_j$ such that
  $\psi_j\geq C>0$ in an open set $\Omega$ in $\domain$ and
  $\|\exp(y_j \psi_j)\|_{L^\infty}\geq 1$ for all $y_j\leq 0$.  With
  $\by_0=(\ldots,0,y_j,0,...)$ and $v_0\in C_0^\infty(\Omega)$ we have
  in this case
  \begin{equation*}
    \vartheta 
    \leq \rho(\exp(y_j\psi_j)) \to 0 
    \quad 
    \text{when}\quad y_j\to -\infty\,.
  \end{equation*}
  Hence, only for $\varkappa=0$ is satisfied Theorem \ref{thm:Kond} in
  this situation.

Due to this observation, for Kondrat'ev regularity we consider only the
case $\varkappa=0$.  In Section \ref{S:DiffPolyg}, we will present a
stronger regularity result for  a
polygonal domain $D \subset \R^2$.
} \end{remark} 
\begin{lemma}\label{lem:Kond:01}
  Let $\bnu\in \FF$, $f\in \Kk_{-1}^{s-2}(\D)$, and assume that
  $\psi_j \in \Ww^{s-1}_\infty(\D)$ for $j\in \NN$.  Let $\by \in U$
  with $b(\by)\in \Ww^{s-1}_\infty(\D)$.  Assume further that there
  exists a non-negative sequence
  $\brho_\bnu=(\rho_{\bnu,j})_{j\in \NN}$ such that
  $\supp(\bnu)\subset \supp(\brho_\bnu)$ and
  \begin{equation} \label{kon-02} \sum_{|\balpha|\leq s-1}\Bigg\|
    \sum_{j\in
      \NN}\rho_{\bnu,j}|r_\domain^{|\balpha|}D^{\balpha}\psi_j|\Bigg\|_{L^\infty}\leq
    \kappa <\frac{\pi}{2}\,.
  \end{equation} 	
  Then we have the estimate
  \begin{equation*}
    \|\partial^{\bnu}u(\by)\|_{\Kk^s_1} 
    \leq C
    \frac{\bnu! } {\brho_\bnu^\bnu} 
    \big( \exp\big(\|b(\by)\|_{L^\infty }\big)^{2N_s+1}  \Big(1+\|  b(\by)\|_{\Ww^{s-1}_\infty} \Big)^{(s-1)N_s}.
  \end{equation*} 
\end{lemma}

\begin{proof} Let $\bnu\in \FF$ with $\mfu=\supp(\bnu)$.  By our
  assumption, it is clear that (with $\balpha=0$ in \eqref{kon-02})
  \begin{equation*}
    \Bigg\|\sum_{j \in \NN} \rho_{\bnu,j} |\psi_j | \Bigg\|_{L^\infty}  \leq \kappa <\frac{\pi}{2}\,.
  \end{equation*}
  Consequently, if we fix the variable $y_j$ with $j\not \in \mfu$,
  the function $u$ of \eqref{SPDE} is holomorphic on the domain
  $\mathcal{S}_\mfu(\brho_\bnu)$, see Theorem \ref{thm:Kond}.  Hence,
  applying Cauchy's formula gives that
  \begin{equation*}
    \begin{split} 
      \|\partial^{\bnu}u(\by)\|_{\Kk^s_1} & \leq
      \frac{\bnu!}{\brho_\bnu^\bnu} \sup_{\bz_\mfu\in
        \mathcal{C}_\mfu(\by,\brho_\bnu)} \|u(\bz_\mfu)\|_{\Kk^s_1}
      \\
      & \leq C \frac{\bnu!}{\brho_\bnu^\bnu} \sup_{\bz_\mfu\in
        \mathcal{C}_\mfu(\by,\brho_\bnu)} \frac{1}{ \big(\rho(
        a(\bz_\mfu)) \big)^{N_{s}+1}}
      \|a(\bz_\mfu)\|_{\Ww^{s-1}_\infty}^{N_s},
    \end{split}
  \end{equation*}
  where $C_\mfu(\by,\brho_\bnu)$ is given as in
  \eqref{eq:C-rho-y}. 
  Notice that for
  $\bz_\mfu=(z_j)_{j\in \NN} \in \mathcal{C}_\mfu(\by,\brho_\bnu)$, we
  can write
  $z_j = y_j + \eta_j + \im\xi_j \in \Cc_{\by,j}(\brho_\bnu)$ with
  $|\eta_j | \le \rho_{\bnu,j}$ and $|\xi_j| \le \rho_{\bnu,j}$ for
  $j\in \mfu$.  
  Hence, by \eqref{eq:R-11}, \eqref{eq:azu} and
  \begin{equation*}
    \begin{split} 
      \|a(\bz_\mfu)\|_{\Ww^{s-1}_\infty} &\leq
      C\|a(\bz_\mfu)\|_{L^\infty} \Big(1+\|
      b(\bz_\mfu)\|_{\Ww^{s-1}_\infty}\Big)^{s-1}
      \\
      &= C\exp(\|b(\by)\|_{L^\infty}) \Bigg[1+ \sum_{|\balpha|\leq
        s-1}\Bigg\|r_\domain^{|\balpha|} \sum_{j\in
        \NN}(y_j+\eta_j+\im \xi_j)D^{\balpha}\psi_j\Bigg\|_{L^\infty}
      \Bigg]^{s-1}
      \\
      &= C\exp(\|b(\by)\|_{L^\infty}) \Bigg[1+ \sum_{|\balpha|\leq
        s-1}\Bigg\|2\sum_{j\in \mfu} \rho_{\bnu,j}|
      r_\domain^{|\balpha|} D^{\balpha} \psi_j |\Bigg\|_{L^\infty}+\|
      b(\by)\|_{\Ww^{s-1}_\infty} \Bigg]^{s-1}
      \\
      &\leq C\exp(\|b(\by)\|_{L^\infty}) \Big(1+2\kappa+\|
      b(\by)\|_{\Ww^{s-1}_\infty} \Big)^{s-1}\,,
    \end{split}
  \end{equation*}
  we obtain the desired result.
\end{proof}
\subsubsection{Summability of $K^s_{\varkappa}$-norms of  Wiener-Hermite PC expansion coefficients}
\label{sec:KmbetaSumm}
To establish weighted $\ell^2$-summability and $\ell^p$-summability of
$K^s_{\varkappa}$-norms of Wiener-Hermite PC expansion coefficients we
need the following assumption.
\begin{assumption}\label{ass:Ass3}
  Let $s\in \NN$.  
  All functions $\psi_j$ belong to
  $\Ww^{s-1}_\infty(\D)$ and there exists a positive sequence
  $(\lambda_j)_{j\in \NN}$ such that
  $\big(\exp(-\lambda_j^2)\big)_{j\in \NN}\in \ell^1(\NN)$ 
  and the series
$$\sum_{j\in \NN}\lambda_j\left|r_\domain^{|\balpha|}D^{\balpha}\psi_j\right|$$ 
converges in $L^\infty(\D)$ for all $\balpha\in \NN_0^d$ with $|\balpha|\leq s-1$.
\end{assumption}

\begin{lemma}\label{lem:Kond:02}
  Suppose that Assumption \ref{ass:Ass3} holds.  
  Then $b(\by)$ belongs to
  $\Ww^{s-1}_\infty(\D)$ $\gamma-a.e.\ \by\in U$.  
  Furthermore,
  $\mathbb{E}(\exp(k\|b(\by)\|_{\Ww_\infty^{s-1}}))$ 
  is finite for all $k\in [0,\infty)$.
\end{lemma}
\begin{proof}
  Under Assumption \ref{ass:Ass3}, by \cite[Theorem 2.2.]{BCDM} we
  infer that for $\balpha\in \NN_0^d$, $|\balpha|\leq s-1$, the
  sequence
  $$\left(\sum_{j=1}^N y_j r_\domain^{|\balpha|}
    D^\balpha\psi_j\right)_{N\in\NN}$$ converges to some
  $\psi_\balpha$ in $L^\infty$ for $\gamma-a.e.\ \by\in U$ and
  $\mathbb{E}(\exp(k\|\psi_\balpha(\by)\|_{L^\infty}))$ is finite for
  all $k\in [0,\infty)$.  Hence, for $\gamma-a.e.\ \by\in U$, the
  sequence $\big(\sum_{j=1}^N y_j \psi_j\big)_{N\in\NN}$ is a Cauchy
  sequence in $\Ww^{s-1}_\infty(\D)$.  Since $\Ww^{s-1}_\infty(\D)$ is
  a Banach space, the statement follows.
\end{proof}

\begin{theorem} [General case] \label{kon-thm-ge} Let $s \in \NN$,
  $s\geq 2$ and $\domain$ be a bounded curvilinear polygonal domain.
  Let $f\in \Kk_{-1}^{s-2}(\D)$ and Assumption \ref{ass:Ass3} hold.
  Assume there exists a sequence
$$\bvarrho= (\varrho_j)_{j \in \NN}\in (0,\infty)^\infty \ \
with \ (\varrho_j^{-1})_{j \in \NN}\in \ell^q(\NN)$$ for some
$0 < q < \infty$.  
Assume furthermore that, for each $\bnu\in \FF$, there exists a sequence
$\brho_\bnu:= (\rho_{\bnu,j})_{j \in \NN}\in [0,\infty)^\infty$ 
such that 
$\supp(\bnu)\subset \supp(\brho_\bnu)$,
\begin{equation*}
  \sup_{\bnu\in \FF}	
  \sum_{|\balpha|\leq s-1}
  \Bigg\|	\sum_{j\in \NN}\rho_{\bnu,j}|r_\domain^{|\balpha|}D^{\balpha}\psi_j|\Bigg\|_{L^\infty}\leq \kappa <\frac{\pi}{2},
  \quad \text{and}\quad 	
  \sum_{\|\bnu\|_{\ell^\infty}\leq r}  \frac{\bnu!\bvarrho^{2\bnu}}{\brho_\bnu^{2\bnu}} <\infty\,
\end{equation*}  
with $r\in \NN$, $r>2/q$.

Then
	\begin{equation} \nonumber 
		\sum_{\bnu\in\FF}\beta_\bnu(r,\bvarrho)\|u_\bnu\|_{\Kk^s_1}^2 <\infty
		\ \ \ with \ \ \ \big(\beta_\bnu(r,\bvarrho)^{-1/2}\big)_{\bnu\in\FF} \in \ell^q(\FF),
	\end{equation}
where 
$\beta_\bnu(r,\bvarrho)$ 
is given in \eqref{beta}.
Furthermore,
$$
(\|u_\bnu\|_{\Kk^s_1})_{\bnu\in\FF}\in \ell^p(\FF) \ \ \  with
\ \ \ \frac{1}{p}=\frac{1}{q}+\frac{1}{2}.
$$
\end{theorem}

\begin{proof}
  For each $\bnu \in \FF$ with $\mfu=\supp(\bnu)$ and $\by\in U$ such
  that $b(\by)\in \Ww^{s-1}_\infty(\D)$, Assumption \ref{ass:Ass3}
  implies that the solution $u$ of \eqref{SPDE} is holomorphic in
  $\mathcal{S}_\mfu (\brho_\bnu)$ as a $\Kk_1^s(\D)\cap V$-valued map,
  see Theorem \ref{thm:Kond}.

  We obtain from Lemmata \ref{lem:Kond:01} and \ref{lem:Kond:02}
  \begin{equation*}
    \begin{split} 
      \int_U\| \partial^\bnu u(\by)\|_{\Kk^s_1}^2\rd\gamma(\by) & \leq
      C \frac{\bnu!}{\brho_\bnu^{2\bnu}} \int_U\big(
      \exp\big(\|b(\by)\|_{L^\infty }\big)^{4N_s+2} \Big(1+\|
      b(\by)\|_{\Ww^{s-1}_\infty} \Big)^{2(s-1)N_s}\rd\gamma(\by)
      \\
      & \leq C \frac{\bnu!}{\brho_\bnu^{2\bnu}} <\infty\,.
    \end{split}
  \end{equation*} 
  This leads to
  \begin{equation*}
    \begin{split} 
      \sum_{\bnu\in \FF}\beta_\bnu(r,\bvarrho)\|u_\bnu\|_{\Kk^s_1}^2
      &= \sum_{\|\bnu\|_{\ell^\infty}\leq r}
      \frac{\bvarrho^{2\bnu}}{\bnu!} \int_U\| \partial^\bnu
      u(\by)\|_{\Kk^s_1}^2\rd\gamma(\by) \leq C
      \sum_{\|\bnu\|_{\ell^\infty}\leq r}
      \frac{\bnu!\bvarrho^{2\bnu}}{\brho_\bnu^{2\bnu}} <\infty\,.
    \end{split}
  \end{equation*}  
  The rest of the proof follows similarly to the proof of Theorem
  \ref{thm:s=1}.
\end{proof}

Similarly to Corollaries \ref{cor:global} and \ref{cor:local} 
from Theorem~\ref{kon-thm-ge} we obtain
\begin{corollary}[The case of global supports]\label{kon-global}
Let $s \in \NN$, $s\geq 2$ and $\domain$ be a bounded curvilinear,
polygonal domain. 
Assume that for all $j\in \NN$ holds
$\psi_j\in\Ww^{s-1}_\infty(\D)$, and that $f\in \Kk_{-1}^{s-2}(\D)$.
Assume further that there exists a sequence of positive numbers
$\blambda= (\lambda_j)_{j \in \NN}$ such that
$$
\big(\lambda_j \| \psi_j \|_{\Ww_\infty^{s-1}}\big)_{j\in \NN} \in
\ell^1(\NN) \ \ and \ \ (\lambda_j^{-1})_{j \in \NN}\in \ell^q(\NN),$$
for some $0 < q < \infty$. 
Then we have
$(\|u_\bnu\|_{\Kk^s_1})_{\bnu\in\FF}\in \ell^p(\FF)$ 
with
$\frac{1}{p}=\frac{1}{q}+\frac{1}{2}$.
\end{corollary}
\begin{corollary}[The case of disjoint supports]
  Let $s \in \NN$, $s\geq 2$ and $\domain\subset \RR^d$ with
  $d \geq 2$ be a bounded curvilinear polygonal domain.  Assume that
  all the functions $\psi_j$ belong to $\Ww^{s-1}_\infty(\D)$ and have
  disjoint supports.  Assume further that $f\in \Kk_{-1}^{s-2}(\D)$
  and that there exists a sequence of positive numbers
  $\blambda= (\lambda_j)_{j \in \NN}$ such that
$$
\big(\lambda_j \| \psi_j \|_{\Ww_\infty^{s-1}}\big)_{j\in \NN} \in
\ell^2(\NN) \ \ and \ \ (\lambda_j^{-1})_{j \in \NN}\in \ell^q(\NN),$$
for some $0 < q < \infty$.  Then
$(\|u_\bnu\|_{\Kk^s_1})_{\bnu\in\FF}\in \ell^p(\FF)$ with
$\frac{1}{p}=\frac{1}{q}+\frac{1}{2}$.
\end{corollary}

\subsection{Bibliographical remarks}	
\label{Some related results on sparsity}

In this section, 
we briefly recall some known related results in previous works on $\ell^p$-summability and on
weighted $\ell^2$-summability of the generalized PC expansion coefficients of solutions 
to parametric divergence-form elliptic PDEs \eqref{SPDE}, 
as well as some applications to best $n$-term approximation.

A basic role in the approximation and numerical integration 
for parametric divergence-form elliptic PDEs \eqref{SPDE} are generalized PC expansions 
for the dependence on the parametric variables.  
In \cite{CCS13_783,CoDe,CoDeSch,CoDeSch1}, based on the conditions
$\big(\|\psi_j\|_{W_\infty^1}\big)_{j \in \NN} \in \ell^p(\NN)$ for some $0 < p <1$ 
on the affine expansion  \index{coefficient!affine $\sim$}
\begin{equation} \label{affine}
	a(\by)= \bar a + \sum_{j = 1}^\infty y_j\psi_j, \ \ \ \by\in [0,1]^\infty,
\end{equation}
the authors have proven $\ell^p$-summability of the coefficients  
in a Taylor or Legendre PC expansion and 
hence proposed adaptive best  $n$-term rate optimal
approximation methods of Galerkin and collocation type 
by choosing 
a set of  $n$ largest estimated terms in these expansions. 
To derive a fully discrete approximation,  
the best $n$-term approximants are then discretized by finite element methods.  
Some results on convergence rates of Galerkin approximation were proven in \cite{HS} 
for the log-Gaussian expansion \eqref{eq:CoeffAffin}, 
based on the summability
$\big(j\|\psi_j\|_{W_\infty^1}\big)_{j \in \NN} \in \ell^p(\NN)$ for some $0 < p <1$. 
However, in these papers possible local 
support properties of the component functions $\psi_j$ were not taken into account.

A different approach to studying summability   that takes into account the 
support properties has been recently proposed in \cite{BCM} for the affine-parametric case, 
in \cite{BCDM} for the log-exponential, parametric case, 
and in \cite{BCDS} for extension of both cases to 
higher-order Sobolev norms of the corresponding generalized PC expansion coefficients.  
This approach leads to significant improvements
on the results on $\ell^p$-summability and therefore, 
on best $n$-term semi-discrete and fully discrete approximations  
when the functions $\psi_j$ have limited overlap, 
such as splines, finite elements or compactly supported wavelet bases. 
These approximation results  provide a benchmark for convergence rates. 

We present some results from \cite{BCDM}  and  \cite{BCDS} on $\ell^p$-summability and  
weighted  $\ell^2$-summability of the  
Wiener-Hermite PC expansion coefficients of the solution to 
the parametric divergence-form elliptic PDEs \eqref{SPDE}--\eqref{eq:CoeffAffin} 
which were proven by real-variable bootstrapping arguments. 

For convenience, 
we use the conventions: 
$$
W^1:= V, \ \ W^2:=W, \ \ H^{-1}(\D):= V', \ \ H^0(\D):= L^2(\D), \ \ W^{0,\infty}(\D):= L^\infty(\D),
$$  
where we recall
$
W:=\{v\in V\; : \; \Delta v\in L^2(\D)\}\;,
$
is the space
equipped with the norm
$
\|v\|_{W}:=\|\Delta v\|_{L^2}.
$
The following theorem and lemma were proven 
in \cite{BCDM} for $i = 1$ and in  \cite{BCDS} for $i=2$. 
\begin{theorem}\label{thm[ell_2summability]}
Let $i=1,2$. 	
Assume that the right side $f$ in \eqref{SPDE}
belongs to $H^{i-2}(\D)$, that the domain $\D$ has $C^{i-1,1}$ smoothness,
that all functions $\psi_j$ belong to $W^{i-1,\infty}(\D)$.
Assume that there exist a number $0<q_i<\infty$ and 
a sequence $\bvarrho_i=(\varrho_{i;j}) _{j \in \NN}$ 
of positive numbers such that $(\varrho_{i;j}^{-1}) _{j \in \NN}\in \ell^{q_i}(\NN)$ 
and
	\begin{equation} \label{assumption[support-properties]}
		\sup_{|\alpha|\leq i-1} 
		\left\| \sum _{j \in \NN} \varrho_{i;j} |D^\alpha \psi_j| \right\|_{L^\infty} 
		<\infty.
	\end{equation}
	Then we have that for any $r \in \NN$,
	\begin{equation} \label{sigma_r,s}
		\sum_{\bnu\in\cF} (\sigma_{i;\bnu} \|u_\bnu\|_{W^i})^2 <\infty \quad  \text{and} \quad 
		(\sigma_{i;\bnu}^{-1})_{\bnu \in \cF} \in \ell^{q_i}(\cF),
	\end{equation}
	where
	\begin{equation} \label{sigma_r,s^2}
		\sigma_{i;\bnu}^2:=\sum_{\|\bnu'\|_{\ell^\infty}\leq r}{\bnu\choose \bnu'} \bvarrho_i^{2\bnu'}.
	\end{equation}
	Furthermore,
$$
(\|u_\bnu\|_{W^i})_{\bnu\in\FF}\in \ell^{2q_i/(2+q_i)}(\FF).
$$ 
\end{theorem}

Notice that the assumption \eqref{assumption[support-properties]} 
which give the weighted $\ell^2$-summability  \eqref{sigma_r,s}, 
already reflects the support properties of the component functions $\psi_j$.

For $\tau, \lambda \ge 0$, we define the family
\begin{equation} \label{[p_s]}
	p_\bnu(\tau, \lambda) := \prod_{j \in \NN} (1 + \lambda \nu_j)^\tau, \quad \bnu \in \FF,
\end{equation}
with the abbreviation $p_\bnu(\tau):= p_\bnu(\tau, 1)$.  

We make use of the following notation
\begin{equation} \label{cF_s}
	\cF_1 :=  \cF, \quad  \cF_2 := \{\bnu \in \cF: \nu_j \not= 1,  \ j \in \NN \}.
\end{equation}

\begin{lemma} \label{lemma[bcdm]}
	Let $0 < q <\infty$, $s =1,2$ and	$\tau, \lambda \ge 0$. Let  $\brho=(\rho_j) _{j \in \NN}$ be a sequence of  positive numbers such $(\rho_j^{-1}) _{j \in \NN}$ belongs to $\ell^q(\NN)$.  For $r \in \NN$, define  the family $(\sigma_\bnu)_{\bnu \in \cF}$  by
	\begin{equation} \nonumber
		\sigma_\bnu^2  :=   \sum_{\|\bnu'\|_{\ell^\infty} \le r}\binom{\bnu}{\bnu'}\brho^{2\bnu'}.
	\end{equation}	
	Then for any  $r > \frac{2s (\tau + 1)}{q}$, we have
	\begin{equation} \nonumber
		\sum_{\bnu \in \cF_s} p_\bnu(\tau,\lambda) \sigma_\bnu^{-q/s} < \infty.
	\end{equation}
\end{lemma}

This lemma has been proven in \cite[Lemma 5.3]{dD21}. Observe that for $s=1$ and $\tau = 0$, an equivalent formulation of Lemma \ref{lemma[bcdm]}   is  Lemma \ref{lem:beta-summability}. 

Theorem \ref{thm[ell_2summability]} and Lemma \ref{lemma[bcdm]} directly imply the following corollary.

\begin{corollary}\label{corollary[ell_2summability]-FF_s}
Under the assumptions of Theorem \ref{thm[ell_2summability]}, let 	
$s =1,2$ and $\tau, \lambda \ge 0$. 
Then we have that for any $r > \frac{2s (\tau + 1)}{q}$,
	\begin{equation} \label{sigma_r,s2}
		\sum_{\bnu\in\cF_s} (\sigma_{i;\bnu} \|u_\bnu\|_{W^i})^2 <\infty \quad  \text{and} \quad 
		(p_\bnu(\tau, \lambda) \sigma_{i;\bnu}^{-1})_{\bnu \in \cF_s} \in \ell^{q_i/s}(\cF_s).
	\end{equation}
\end{corollary}

As commented in Section \ref{sec:SumHermCoef}, in the case of disjoint or finitely
overlapping supports the results on sparsity of Theorem~\ref {thm[ell_2summability]} and 
Corollary~\ref{corollary[ell_2summability]-FF_s} are stronger than those in 
Sections~\ref{sec:SumHermCoef} and \ref{sec:Summ}. 
They play a basic role in best $n$-term approximation \cite{BCDM,BCDS} 
and 
linear approximation and quadrature  \cite{dD21} (see also \cite{dD-Erratum22}) 
of the solution  to  the parametric divergence-form elliptic PDEs \eqref{SPDE}--\eqref{eq:CoeffAffin}. 
\newpage
\section{Sparsity for holomorphic functions}
\label{sec:SumHolSol}
In Section \ref{sec:EllPDElogN} we introduced a concept of
holomorphic extensions of countably-parametric families
$\{ u(\by) : \by \in U \}\subset V$ in the separable Hilbert space $V$
with respect to the parameter $\by$ 
into the Cartesian product $\mathcal{S}_\mfu (\brho)$ of strips in the complex domain
(cp. \eqref{eq:Snubrho}).  
We now introduce a refinement which is
required for the ensuing results on rates of numerical approximation and integration
of such families, based on sparsity (weighted $\ell^2$-summability)
and of Wiener-Hermite PC expansions of $\{ u(\by) : \by \in U \}$:
\emph{quantified parametric holomorphy} of (complex extensions of)
the parametric families $\{ u(\by) : \by \in U \}\subset X$ for a
separable Hilbert space X.  Section \ref{S:DefbxdHol} presents a
definition of quantified holomorphy of families 
$\{ u(\by) : \by \in U \}$ 
and
discusses the sparsity of the  Wiener-Hermite PC expansion coefficients of these families.
In Section~\ref{sec:bdX}, 
we present the notion $(\bb,\xi,\delta,X)$-holomorphy of composite functions.
In Section  \ref{sec:HlExmpl}, 
we analyze some examples of holomorphic functions which are solutions to certain PDEs.

There are two basic
steps in the approximations which we consider: 
\medskip \newline
\noindent
(i) We truncate the countably-parametric family
$\{ u(\by) : \by \in U \}\subset X$ to a finite number $N\in \NN$ of
parameters. 
This step, which is sometimes also referred to as
``dimension-truncation'', of course implicitly depends on the
enumeration of the coordinates $y_j \in \by$.  
\emph{We assume throughout that this numbering is fixed by the indexing of the
  Parseval frame in Theorem \ref{thm:AdFrmX}} which frame is used as
affine representation system to parametrize the uncertain input
$a = \exp(b)$ of the PDE of interest. We emphasize that the finite
dimension $N \in \NN$ of the truncated parametric Wiener-Hermite PC
expansion is a discretization parameter, and we will be interested in
quantitative bounds on the error incurred by restricting
$\{ u(\by) : \by \in U \}\subset X$ to Wiener-Hermite PC expansions of
the first $N$ active variables only. 
We denote these restrictions by
$\{ u_N(\by):\by\in U \}$.  
\medskip 
\newline
\noindent
(ii) The coefficients $u_\bnu\in X$ of the resulting,
finite-parametric Wiener-Hermite PC expansion, can not be computed
exactly, but must be numerically approximated.  As is done in
stochastic collocation and stochastic Galerkin algorithms, we seek
numerical approximations of $u_\bnu$ in suitable, finite-dimensional
subspaces $X_l\subset X$.  
Assuming the collection
$(X_l )_{l\in \NN}\subset X$ to be dense in $X$, any prescribed
tolerance $\eps>0$ of approximation of $u_N(\by)$ in
$L^2(U, X;\gamma)$ can be met. 
For notational convenience, we also set $X_0 = \{ 0\}$. 
\medskip 
\newline 
In computational practice, however, 
given a target accuracy $\eps\in (0,1]$,
one searches an allocation of
$l:\calF\times (0,1]\to \NN: (\bnu,\eps) \mapsto l(\bnu,\eps)$ 
of discretization levels along the ``active''  Wiener-Hermite PC expansion 
coefficients which ensures that the prescribed tolerance $\eps\in (0,1]$ is met
with possibly minimal ``computational budget''.
We propose and
analyze the \emph{a-priori construction of an allocation} $l$ which
ensures convergence rates of the corresponding collocation
approximations which are independent of $N$ 
(i.e.\ they are free from the ``curse of dimensionality''). 
These approximations are based on ``stochastic
collocation'', i.e.\ on sampling the parametric family
$\{ u(\by) : \by \in U \}\subset V$ in a collection of deterministic
Gaussian coordinates in $U$.  
We prove, subsequently,
dimension-independent convergence rates of the sparse collocation
w.r.t.\ $\by\in U$ and w.r.t.\ the subspaces $X_l\subset X$ realize
convergence rates which are free from the curse of dimensionality.
These rates depend only on the summability (resp. sparsity) of the
coefficients of the norm of the Wiener-Hermite PC expansion
of the parametric family $\{ u(\by) : \by \in U \}$ 
with respect to $\by$.
\subsection{$(\bb,\xi,\delta,X)$-Holomorphy and sparsity}
\label{S:DefbxdHol}

We introduce the concept of ``$(\bb,\xi,\delta,X)$-holomorphic
functions'', \index{function!holomorphic $\sim$}which constitutes a subset of $L^2(U,X;\gamma)$. 
As such
these functions are typically not pointwise well defined for each $\by\in U$.  
In order to still define a suitable form of pointwise
function evaluations to be used for numerical algorithms such as sampling at $\by\in U$ 
for ``stochastic collocation'' or for quadrature in ``stochastic Galerkin'' algorithms, 
we 
define them as $L^2(U,X;\gamma)$ limits of certain smooth (pointwise
defined) functions, 
cp.~Remark.~\ref{rmk:defu} and Example \ref{ex:u} ahead.

For $N\in\N$ and $\bvarrho=(\varrho_j)_{j=1}^N\in (0,\infty)^N$ set
(cp. \eqref{eq:Snubrho})
\begin{equation}
  \label{eq:Sjrho}
  \Ss(\bvarrho) := \set{\bz\in \C^N}{|\mathfrak{Im}z_j| < \varrho_j~\forall j}\qquad\text{and}\qquad
  \Bb(\bvarrho) := \set{\bz\in\C^N}{|z_j|<\varrho_j~\forall j}.
\end{equation}
\begin{definition}[($\bb,\xi,\delta,X$)-Holomorphy]
  \label{def:bdXHol}
  Let $X$ be a complex, separable Hilbert space,
  $\bb=(b_j)_{j\in\N} \in (0,\infty)^\infty$ and $\xi>0$, $\delta>0$.

  For $N\in\N$, $\bvarrho\in (0,\infty)^N$ is 
  called \emph{$(\bb,\xi)$-admissible} if
  \begin{equation}\label{eq:adm}
    \sum_{j=1}^N b_j\varrho_j\leq \xi\,.
  \end{equation}
  
  A function $u\in L^2(U,X;\gamma)$ is called
  \emph{$(\bb,\xi,\delta,X)$-holomorphic} if
  \begin{enumerate}
  \item\label{item:hol} for every $N\in\N$ there exists
    $u_N:\R^N\to X$, which, for every $(\bb,\xi)$-admissible
    $\bvarrho\in (0,\infty)^N$, admits a holomorphic extension
    (denoted again by $u_N$) from $\Ss(\bvarrho)\to X$; furthermore,
    for all $N<M$
    \begin{equation}\label{eq:un=um}
      u_N(y_1,\dots,y_N)=u_M(y_1,\dots,y_N,0,\dots,0)\qquad\forall (y_j)_{j=1}^N\in\R^N,
    \end{equation}

  \item\label{item:varphi} for every $N\in\N$ there exists
    $\varphi_N:\R^N\to\R_+$ such that
    $\norm[L^2(\R^N;\gamma_N)]{\varphi_N}\le\delta$ and
    \begin{equation*} \label{ineq[phi]}
      \sup_{\substack{\bvarrho\in(0,\infty)^N\\
          \text{is $(\bb,\xi)$-adm.}}}~\sup_{\bz\in
        \Bb(\bvarrho)}\norm[X]{u_N(\by+\bz)}\le
      \varphi_N(\by)\qquad\forall\by\in\R^N,
    \end{equation*}
  \item\label{item:vN} with $\tilde u_N:U\to X$ defined by
    $\tilde u_N(\by) :=u_N(y_1,\dots,y_N)$ for $\by\in U$ it holds
    \begin{equation*}
      \lim_{N\to\infty}\norm[L^2(U,X;\gamma)]{u-\tilde u_N}=0.
    \end{equation*}
  \end{enumerate}
\end{definition}
We interpret the definition of
$(\bb,\xi,\delta,X)$-holomorphy in the following remarks.
\begin{remark} {\rm\label{rmk:bdexpl}
  While the numerical value of
  $\xi>0$ in Definition \ref{def:bdXHol} of
  $(\bb,\xi,\delta,X)$-holomorphy is of minor importance in the
  definition, the sequence $\bb$ and the constant
  $\delta$ will crucially influence the magnitude of our upper bounds
  of the Wiener-Hermite PC expansion coefficients: The stronger the
  decay of
  $\bb$, the larger we can choose the elements of the sequence
  $\bvarrho$, so that
  $\bvarrho$ satisfies \eqref{eq:adm}. Hence stronger decay of
  $\bb$ indicates larger domains of holomorphic extension. The
  constant
  $\delta$ is an upper bound of these extensions in the sense of item
  \ref{item:varphi}.  
  Importantly, the decay of
  $\bb$ will determine the \emph{sparsity} of the Wiener-Hermite PC
  expansion coefficients, while decreasing
  $\delta$ by a factor will roughly speaking translate to a decrease
  of all coefficients by the same \emph{factor}.
} \end{remark} 

\begin{remark} {\rm
  Since $u_N\in L^2(\R^N,X;\gamma_N)$, the function $\tilde
  u_N$ in item \ref{item:vN} belongs to
  $L^2(U,X;\gamma)$ by Fubini's theorem.
} \end{remark} 

\begin{remark} {\rm\label{rmk:defu} [Evaluation of countably-parametric functions]
  In the following sections, for arbitrary $N\in\N$ and
  $(y_j)_{j=1}^N\in\R^N$ we will write
  \begin{equation}\label{eq:udef}
    u(y_1,\dots,y_N,0,0,\dots):=
    u_N(y_1,\dots,y_N).
  \end{equation}
  This is well-defined due to \eqref{eq:un=um}.  
  Note however
  that %
  \eqref{eq:udef} should be considered as an abuse of notation, since
  pointwise evaluations of functions $u\in L^2(U,X;\gamma)$ are in
  general not
  well-defined. 
} \end{remark} 

\begin{remark} {\rm\label{rmk:separable} 
 The assumption of $X$ being separable is not necessary in Definition~\ref{def:bdXHol}:
 Every function $u_N:\R^N\to X$ as in
 Definition~\ref{def:bdXHol} is continuous since it allows a
 holomorphic extension. 
 Hence,
  $$A_{N,n}:=\set{u_N((y_j)_{j=1}^N)}{y_j\in [-n,n]~\forall j}\subseteq X$$ 
  is compact and thus there is a countable set $X_{N,n}\subseteq X$
  which is dense in $A_{N,n}$ for every $N$, $n\in\N$.  Then
  $\bigcup_{n\in\N} A_{N,n}$ is contained in the (separable) closed
  span $\tilde X$ of
  $$
  \bigcup_{N,n\in\N} X_{N,n}\subseteq X.
  $$ Since
  $\tilde u_N\in L^2(U,\tilde X;\gamma)$ for every $N\in\N$ we also
  have $$u=\lim_{N\to\infty}u_N\in L^2(U,\tilde X;\gamma).$$ Hence,
  $u$ is separably valued.}
 \end{remark} 
\begin{lemma}\label{holo-lem1}
  Let $u$ be $(\bb,\xi,\delta,X)$-holomorphic, let $N\in\N$ and
  $0<\kappa<\xi<\infty$.  Let $u_N$, $\varphi_N$ be as in
  Definition~\ref{def:bdXHol}.  Then with $\bb_N=(b_j)_{j=1}^N$ it
  holds for every $\bnu\in\N_0^N$
  \begin{equation*}
    \|\partial^{\bnu}u_N(\by)\|_X 
    \leq
    \frac{\bnu!|\bnu|^{|\bnu|}\bb_N^\bnu}{\kappa^{|\bnu|}\bnu^{\bnu} } 
    \varphi_N(\by)\qquad\forall\by\in\R^N.
  \end{equation*}
\end{lemma}
\begin{proof}
  For $\bnu\in \N_0^N$ fixed we choose $\bvarrho=(\varrho_j)_{j=1}^N$
  with $\varrho_j= \kappa \frac{\nu_j}{|\bnu|b_j}$ for
  $j\in \supp(\bnu)$ and $\varrho_j=\frac{\xi -\kappa}{Nb_j}$ for
  $j\not \in \supp(\bnu)$.  Then
  \begin{equation*}
    \sum_{j=1}^N
    \varrho_jb_j=\kappa \sum_{j\in \supp(\bnu)} \frac{\nu_j}{|\bnu|}
    +\sum_{j\not\in \supp(\bnu)}\frac{\xi-\kappa}{N}\leq \xi.
  \end{equation*}
  Hence $\bvarrho$ is $(\bb,\xi)$-admissible, i.e.\ there exists a
  holomorphic extension $u_N:\Ss(\bvarrho)\to X$ as in
  Definition~\ref{def:bdXHol} \ref{item:hol}-\ref{item:varphi}.
  Applying Cauchy's integral formula as in the proof of Lemma
  \ref{lem:estV} we obtain the desired estimate.
\end{proof}

Let us recall the following. 
Let again $X$ be a separable Hilbert
space and $u\in L^2(U,X;\gamma)$. 
Then
$$
L^2(U,X;\gamma) = L^2(U;\gamma)\otimes X
$$
with Hilbertian tensor product, 
and $u$ can be represented in a 
Wiener-Hermite PC expansion \index{expansion!Wiener-Hermite PC $\sim$}
\begin{equation}\label{eq:uHermiteExp}
  u=\sum_{\bnu\in \FF} u_\bnu H_\bnu,
\end{equation}
where
\begin{equation*}
  u_\bnu=\int_U u(\by) H_\bnu(\by) \rd \gamma(\by)
\end{equation*}
are the Wiener-Hermite PC expansion coefficients.  
Also, there holds the Parseval-type identity
\begin{equation*}\label{eq:PCParseval}
  \|u\|_{L^2(U,X;\gamma)}^2=\sum_{\bnu \in \FF}
  \|u_\bnu\|_X^2\,.
\end{equation*}
When $u$ is $(\bb,\xi,\delta,X)$-holomorphic, then we have
for the functions $u_N:\R^N\to X$ in Definition~\ref{def:bdXHol}
\begin{equation*}
  u_N=\sum_{\bnu\in\N_0^N} u_{N,\bnu} H_\bnu,
\end{equation*}
where
\begin{equation*}
  u_{N,\bnu}=\int_{\R^N} u_N(\by) H_\bnu(\by) \rd \gamma_N(\by).
\end{equation*}

	In an analogous manner  to \eqref{beta}, for $r\in \NN$ and a finite sequence
of nonnegative numbers
$\bvarrho_N=(\varrho_j)_{j=1}^N$, we define
\begin{equation} \label{betaN} 
	\beta_\bnu(r,\bvarrho_N) :=
	\sum_{\bnu' \in \N_0^N: \ \|\bnu'\|_{\ell^\infty}\leq r} \binom{\bnu}{\bnu'}
	\bvarrho^{2\bnu'} = \prod_{j=1}^N\Bigg(\sum_{\ell=0}^{r}\binom{\nu_j}{\ell}\varrho_j^{2\ell}\Bigg),
	\ \ \ \bnu \in \N_0^N.
\end{equation} 

\begin{lemma}\label{holo-lem2}
  Let $u$ be $(\bb,\xi,\delta,X)$-holomorphic, 
  let $N\in\N$ 
  and let
  $\bvarrho_N= (\varrho_j)_{j=1}^N\in [0,\infty)^N$. 

 Then, for any fixed $r\in\N$, there holds the identity
  \begin{equation} \label{general} \sum_{\bnu\in
    \N_0^N}\beta_\bnu(r, \bvarrho_N )\|u_{N,\bnu}\|_{X}^2 
    =
    \sum_{\set{\bnu\in\N_0^N}{\|\bnu\|_{\ell^\infty}\leq r}}
    \frac{\bvarrho_N^{2\bnu}}{\bnu!}  
     \int_{\R^N}\| \partial^\bnu u_N(\by)\|_{X}^2\rd\gamma_N(\by). 
    \end{equation}
\end{lemma}
\begin{proof}
  From Lemma \ref{holo-lem1}, for any $\bnu\in \N_0^N$, we have with
  $\bb_N=(b_j)_{j=1}^N$
  \begin{align}\label{general-1}
    \int_{\R^N} \|\partial^\bnu u_N(\by)\|_X^2\rd\gamma_N(\by)
    &\le \int_{\R^N}
      \Big|\frac{\bnu!|\bnu|^{|\bnu|}\bb_N^\bnu}{\kappa^{|\bnu|}\bnu^{\bnu}} \varphi_N(\by) \Big|^2 \rd\gamma_N(\by) 
      \nonumber\\
    &= \Big(\frac{\bnu!|\bnu|^{|\bnu|}\bb_N^\bnu}{\kappa^{|\bnu|}\bnu^{\bnu}}\Big)^2\int_{\R^N} \big| 
      \varphi_N(\by) \big|^2 \rd\gamma_N(\by) 
      < \infty
  \end{align}
  by our assumption.  This condition allows us to integrate by parts
  as in the proof of \cite[Theorem 3.3]{BCDM}.  Following the argument
  there we obtain \eqref{general}.
\end{proof}

\begin{theorem} \label{thm:bdHolSum} 
Let $u$ be
$(\bb,\xi,\delta,X)$-holomorphic for some $\bb\in \ell^p(\N)$ 
and some $p\in (0,1)$. Let $r\in\N$.  

Then, with
 \begin{equation} \label{varrho_j}
  \varrho_j:=b_j^{p-1}\frac{\xi}{4\sqrt{r!} \norm[\ell^p]{\bb}}, \ \
  j\in\N,
\end{equation}
and  $\bvarrho_N=(\varrho_j)_{j=1}^N$ it holds for all $N\in\N$,
\begin{equation}\label{eq:generalN}
  \sum_{\bnu\in \N_0^N}\beta_\bnu(r,\bvarrho_N)\norm[X]{u_{N,\bnu}}^2\le
  \delta^2C(\bb) <\infty
  \ \ \ with \ \ \  
  \norm[\ell^{p/(1-p)}(\N_0^N)]{\beta_\bnu(r,\bvarrho_N)^{-1/2}} \le C'(\bb,\xi) < \infty
\end{equation}
for some constants $C(\bb)$  and $C'(\bb,\xi)$ depending on $\bb$ and $\xi$,
but independent of $\delta$ and $N\in\N$.

Furthermore, for every $N\in\N$ and every $q>0$ there holds
$$(\norm[X]{u_{N,\bnu}})_{\bnu\in\N_0^N}\in \ell^{q}(\N_0^N).$$
If $q\geq \frac{2p}{2-p}$ then there exists a constant $C>0$ 
such that for all $N\in \N$ holds
$$
\big\|(\norm[X]{u_{N,\bnu}})_{\bnu}\big\|_{\ell^{q}(\N_0^N)}\leq C <
\infty\;.
$$ 
\end{theorem}
\begin{proof}
  We have
  $$\sum_{j\in\N}\varrho_jb_j=\frac{\xi}{4\sqrt{r!}
    \norm[\ell^p]{\bb}}\sum_{j\in\N}b_j^p<\infty,$$ and
  $(\varrho_j^{-1})_{j\in\N}\in\ell^{p/(1-p)}(\N)$.  Set
  $\kappa:=\xi/2\in (0,\xi)$. Inserting \eqref{general-1} into
  \eqref{general} 
  we obtain with $\bvarrho_N=(\varrho_j)_{j=1}^N$
  \begin{align*}
    \sum_{\bnu\in \N_0^N}\beta_\bnu(r,\bvarrho_N)\norm[X]{u_{N,\bnu}}^2
    &\le \delta^{2}\sum_{\set{\bnu\in\N_0^N}{\norm[\ell^\infty]{\bnu}\le r}} \left(\frac{(\bnu!)^{1/2}|\bnu|^{|\bnu|} \bvarrho_N^{\bnu} \bb_N^{\bnu}}{\kappa^{|\bnu|}\bnu^{\bnu}}\right)^2\nonumber\\
    &\le \delta^{2}\sum_{\set{\bnu\in\N_0^N}{\norm[\ell^\infty]{\bnu}\le r}} \left(\frac{|\bnu|^{|\bnu|} \prod_{j=1}^N\Big(\frac{b_j^p}{2\norm[\ell^p]{\bb}}\Big)^{\nu_j}}{\bnu^\bnu}\right)^2,
  \end{align*}
  where we used $(\varrho_jb_j)^2=b_j^{2p}\kappa/(2(r!))$ and the
  bound
  $$\int_{\R^N}\varphi_N(\by)^2\dd\gamma_N(\by)\le \delta^{2}$$
  from Definition~\ref{def:bdXHol} \ref{item:varphi}.  With
  $\tilde b_j:= b_j^p/(2\norm[\ell^p]{\bb} )$ the last term is bounded
  independent of $N$ by $\delta^2 C(\bb)$ with
  \begin{equation*}
    C(\bb):= \left(\sum_{\bnu\in\CF} \frac{|\bnu|^{|\bnu|}}{\bnu^\bnu} \tilde \bb^\bnu\right)^{1/2},
  \end{equation*}
  since the $\ell^1$-norm is an upper bound of the $\ell^2$-norm.  As
  is well-known, the latter quantity is finite due to
  $\norm[\ell^1]{\tilde \bb}<1$, see, e.g., the argument in \cite[Page
  61]{CoDe}.

  Now introduce $\tilde \varrho_{N,j}:=\varrho_j$ if $j\le N$ and
  $\tilde \varrho_{N,j}:=\exp(j)$ otherwise. For any $q>0$ we then
  have $(\tilde \varrho_{N,j}^{-1})_{j\in\N}\in\ell^{q}(\N)$ and by 
  Lemma \ref{lem:beta-summability}
  this implies
  $$(\beta_\bnu(r,\tilde\bvarrho_N)^{-1})_{\bnu\in\CF}\in\ell^{q/2}(\CF)$$
  as long as $r>2/q$. Using
  $\beta_{\bnu}(r,\tilde\bvarrho_N)=\beta_{(\nu_j)_{j=1}^N}(r,\bvarrho_N)$
  for all $\bnu\in\CF$ with $\supp\bnu\subseteq \{1,\dots,N\}$ we
  conclude
  $$(\beta_{\bnu}(r,\bvarrho_N)^{-1})_{\bnu\in\N_0^N}\in \ell^{q/2}(\N_0^N)$$ 
  for any $q>0$.  Now fix $q>0$ (and $2/q<r\in\N$).  
  Then, by
  H\"older's inequality with $s:=2(q/2)/(1+q/2)$, there holds
  \begin{align*}
    \sum_{\bnu\in\N_0^N}\norm[X]{u_{N,\bnu}}^{s}
    &= \sum_{\bnu\in\N_0^N}\norm[X]{u_{N,\bnu}}^{s}\beta_{\bnu}(r,\bvarrho_N)^{\frac s 2}\beta_{\bnu}(r,\bvarrho_N)^{-\frac s 2}\nonumber\\
    &\le \Bigg(\sum_{\bnu\in\N_0^N}\norm[X]{u_{N,\bnu}}^{2}\beta_{\bnu}(r,\bvarrho_N) \Bigg)^{\frac s 2}\Bigg(\sum_{\bnu\in\N_0^N}\beta_{\bnu}(r,\bvarrho_N)^{\frac{s}{2-s}} \Bigg)^{\frac{2-s}{2}},
  \end{align*}
  which is finite since $s/(2-s)=q/2$.  
  Thus we have shown
  $$
  \forall q>0, N \in \N: \quad
  (\norm[X]{u_{N,\bnu}})_{\bnu\in\N_0^N}\in\ell^{q/(1+q/2)}(\N_0^N)
  \;.
  $$
  
  Finally, due to $(\varrho_j^{-1})_{j\in\N} \in\ell^{p/(1-p)}(\N)$,
  Lemma \ref{lem:beta-summability} for all $N\in \NN$ it holds
  $$(\beta_\bnu(r,\bvarrho_N)^{-1})_{\bnu\in\N_0^N}\in\ell^{p/(2(1-p))}(\N_0^N)$$ 
  and there exists a constant $C'(\bb,\xi)$ such that for all $N\in \NN$ it holds
  $$
  \big\|(\beta_\bnu(r,\bvarrho_N)^{-1})_{\bnu}\big\|_{\ell^{p/(2(1-p))}(\N_0^N)}
  \leq C'(\bb,\xi) < \infty \;.
  $$ 
  This completes the proofs of \eqref{eq:generalN} and of the last statement.
\end{proof}

The following result states the sparsity of Wiener-Hermite PC
expansion coefficients of $(\bb,\xi,\delta, X)$-holomorphic maps.
\begin{theorem}\label{thm:bdXSum}
  Under the assumptions of Theorem~\ref{thm:bdHolSum} it holds
  	\begin{equation} \label{eq:general}
  		\sum_{\bnu\in\FF}\beta_\bnu(r,\bvarrho)\|u_\bnu\|_X^2 
  		\le
  		\delta^2C(\bb)<\infty
  		\ \ \ with \ \ \ \big(\beta_\bnu(r,\bvarrho)^{-1/2}\big)_{\bnu\in\FF} \in \ell^{p/(1-p)}(\FF),
  	\end{equation}
  	 where  $C(\bb)$ is the same constant as in Theorem~\ref{thm:bdHolSum} and 
  	  $ \beta_\bnu(r,\bvarrho)$ is given in \eqref{beta}.
  	Furthermore,
  	$$
  	(\|u_\bnu\|_X)_{\bnu\in\FF}\in \ell^{2p/(2-p)}(\FF).
  	$$ 
\end{theorem}
\begin{proof}
  Let $\tilde u_N\in L^2(U,X;\gamma)$ be as in
  Definition~\ref{def:bdXHol} and for $\bnu\in\CF$ denote by
  $$
  \tilde u_{N,\bnu}:= \int_U \tilde u_N(\by) H_\bnu(\by)\dd\gamma(\by)
  \in X
  $$
  the Wiener-Hermite PC expansion coefficient.  By Fubini's theorem
  $$\tilde u_{N,\bnu}=\int_U u_N((y_j)_{j=1}^N)
  \prod_{j=1}^NH_{\nu_j}(y_j)\dd\gamma_N((y_j)_{j=1}^N)=u_{N,(\nu_j)_{j=1}^N}$$
  for every $\bnu\in\CF$ with $\supp\bnu\subseteq\{1,\dots,N\}$.
  Furthermore, since $\tilde u_N$ is independent of the variables
  $(y_j)_{j=N+1}^\infty$ we have $\tilde u_{N,\bnu}=0$ whenever
  $\supp\bnu\subsetneq \{1,\dots,N\}$. 
  Therefore Theorem~\ref{thm:bdHolSum} implies
  \begin{equation*}
    \sum_{\bnu\in\CF}\beta_\bnu(r,\bvarrho) \norm[X]{\tilde u_{N,\bnu}}^2\le
    \frac{C(\bb)}{\delta^2}\qquad\forall N\in\N.
  \end{equation*}

  Now fix an arbitrary, finite set $\Lambda\subset\CF$. 
  Because of
  $\tilde u_N\to u\in L^2(U,X;\gamma)$ it holds
  $$\lim_{N\to\infty} \tilde u_{N,\bnu}=u_{\bnu}$$ for all
  $\bnu\in\CF$. 
  Therefore
  \begin{equation*}
    \sum_{\bnu\in\Lambda}
    \beta_\bnu(r,\bvarrho) \norm[X]{u_{\bnu}}^2
    =\lim_{N\to\infty}\sum_{\bnu\in\Lambda}    
    \beta_\bnu(r,\bvarrho) \norm[X]{\tilde u_{N,\bnu}}^2\le \frac{C(\bb)}{\delta^2}.
  \end{equation*}
  Since $\Lambda\subset\CF$ was arbitrary, this shows that
  \begin{equation} \nonumber
  	\sum_{\bnu\in\FF}\beta_\bnu(r,\bvarrho)\|u_\bnu\|_X^2 
  	\le
  	\delta^2C(\bb)<\infty.
  \end{equation}
  Finally, due to $\bb\in\ell^p(\NN)$, 
  with
  $$\varrho_j=b_j^{p-1}\frac{\xi}{4\sqrt{r!} \norm[\ell^p]{\bb}}$$
  as in Theorem~\ref{thm:bdHolSum} we have
  $(\varrho_j^{-1})_{j\in\N} \in \ell^{p/(1-p)}(\N)$.  
  By  Lemma \ref{lem:beta-summability}, 
  it holds 
  \begin{equation} \label{beta_nu-summability}
  (\beta_\bnu(r,\bvarrho)^{-1/2})_{\bnu\in\CF}\in\ell^{p/(1-p)}(\CF).
  \end{equation}
The relation \eqref{eq:general} is proven.
H\"older's inequality can be used to show that \eqref{beta_nu-summability} gives
$$(\norm[X]{u_\bnu})_{\bnu\in\CF}\in\ell^{2p/(2-p)}(\CF)$$
(by a similar calculation as at the end of the proof of
Theorem~\ref{thm:bdHolSum} with $q=p/(1-p)$).
\end{proof}

\begin{remark} {\rm
\label{rmk:bestN}
We establish the convergence rate of best $n$-term approximation of $(\bb,\xi,\delta,X)$-holomorphic functions 
     based on the  $\ell^p$-summability.
Let $u$ be
$(\bb,\xi,\delta,X)$-holomorphic for some $\bb\in \ell^p(\N)$ and
some $p\in (0,1)$ as in Theorem \ref{thm:bdHolSum}. 
By Theorem \ref{thm:bdXSum} we then have
$(\norm[X]{u_\bnu})_{\bnu\in\cF}\in\ell^{\frac{2p}{2-p}}$.
    
Let $\Lambda_n\subseteq \cF$ be a set of cardinality $n\in\N$
containing $n$ multiindices $\bnu\in\cF$ such that
$\norm[X]{u_\bmu}\le\norm[X]{u_\bnu}$ whenever $\bnu\in\Lambda_n$
and $\bmu\notin\Lambda_n$. 
Then, by Theorem \ref{thm:bdXSum}, 
for the truncated
the Wiener-Hermite PC expansion 
we have the error bound
    \begin{equation*}
      \normc[L^2(U,X;\gamma)]{u(\by)-\sum_{\bnu\in\Lambda_n}u_\bnu H_\bnu(\by)}^2
      =\sum_{\bnu\in\cF\backslash\Lambda_n}\norm[X]{u_\bnu}^2
      \le \sup_{\bnu\in\cF\backslash\Lambda_n}\norm[X]{u_\bnu}^{2-\frac{2p}{2-p}}
      \sum_{\bmu\in\cF\backslash\Lambda_n}\norm[X]{u_\bmu}^{\frac{2p}{2-p}}.
    \end{equation*}
For a nonnegative, monotonically decreasing sequence
$(x_j)_{j\in\N}\in\ell^q(\N)$ with $q>0$ we have
$$
x_n^q\le \frac{1}{n}\sum_{j=1}^nx_j^q
$$ 
and thus
$$
x_n\le n^{-\frac{1}{q}}\norm[\ell^q(\N)]{(x_j)_{j\in\N}}.
$$
With $q=\frac{2p}{2-p}$ this implies
\begin{equation*}
      \left(\sup_{\bnu\in\cF\backslash\Lambda_n}\norm[X]{u_\bnu}\right)^{2-\frac{2p}{2-p}}\le \left(n^{-\frac{2-p}{2p}}\left(\sum_{\bnu\in\cF}\norm[X]{u_\bnu}^{\frac{2p}{2-p}}\right)^{\frac{2-p}{2p}}\right)^{2-\frac{2p}{2-p}}= \Oo(n^{-\frac{2}{p}+2}).
\end{equation*}
Hence, by truncating the 
Wiener-Hermite PC expansion 
\eqref{eq:uHermiteExp}
after $n$ largest terms yields the
best $n$-term convergence rate
\begin{equation}\label{eq:bestNgeneral}
\normc[L^2(U,X;\gamma)]{u(\by)-\sum_{\bnu\in\Lambda_n}u_\bnu H_\bnu(\by)}
= 
\Oo(n^{-\frac{1}{p}+1})\qquad\text{as }n\to\infty.
\end{equation}
} \end{remark} 
\subsection{$(\bb,\xi,\delta,X)$-Holomorphy  of composite functions}
\label{sec:bdX}
We now show that certain composite functions of the type
\begin{equation} \label{composite-function}
u(\by)=\Uu\bigg(\exp\bigg(\sum_{j\in\N}y_j\psi_j\bigg)\bigg)
\end{equation}
are $(\bb,\xi,\delta,X)$-holomorphic under certain conditions.

The significance of such functions is the following: if we think for
example of $\Uu$ as the solution operator $\cS$ in \eqref{eq:SolOp} (for a fixed $f$)
which maps the diffusion coefficient $a\in L^\infty(\D)$ to the
solution $\Uu(a)\in H_0^1(\D)$ of an elliptic PDE 
on some domain $\D\subseteq\R^d$, then
$\Uu\big(\exp\big(\sum_{j\in\N}y_j\psi_j\big)\big)$ is exactly the
parametric solution discussed in Sections \ref{S:PbmStat}--\ref{S:HolSumSol}. 
We explain this in more
detail in Section~\ref{sec:pdc}. 
The presently developed, abstract setting allows, however, 
to \emph{consider $\Uu$ as a solution operator
of other, structurally similar PDEs with log-Gaussian random input data}. 
Furthermore, if $\Gg$ is another map 
with suitable holomorphy properties, 
the composition
$\Gg\big(\Uu\big(\exp\big(\sum_{j\in\N}y_j\psi_j\big)\big)\big)$ 
is again of the general type
$\tilde \Uu\big(\exp\big(\sum_{j\in\N}y_j\psi_j\big)\big)$ 
with
$\tilde\Uu=\Gg\circ\Uu$.

This will allow to apply {the ensuing results on convergence rates of
  deterministic collocation and quadrature algorithms to a wide range
  of PDEs with GRF inputs and functionals on their random
  solutions. As a particular case in point, we apply} our results to
posterior densities in Bayesian inversion, as we explain subsequently
in Section~\ref{sec:BIP}.  As a result, the concept of
$(\bb,\xi,\delta,X)$-holomorphy is fairly broad and covers a large
range of parametric PDEs depending on log-Gaussian distributed
data. 

To formalize all of this, \emph{we now provide sufficient conditions
on the solution operator $\Uu$ and the sequence $(\psi_j)_{j\in\N}$
guaranteeing $(\bb,\xi,\delta,X)$-holomorphy}.
Let $d\in\N$, $\domain \subseteq\R^d$ be an open set and ${\data}$ a complex
Banach space which is continuously embedded into
$L^\infty(\domain;\C)$, and finally let $X$ be another complex Banach
space. Additionally, suppose that there exists $C_{\data}>0$ such
that for all $\psi_1, \psi_2 \in {\data}$ and some $m\in\N$
\begin{equation}\label{eq:Znorm2}
  \norm[{\data}]{\exp(\psi_1)-\exp(\psi_2)} 
  \leq 
  C_{\data}\norm[{\data}]{\psi_1-\psi_2} 
  \max\Big\{\exp\big(m\norm[{\data}]{\psi_1}\big);\, 
  \exp\big(m\norm[{\data}]{\psi_2} \big) 
  \Big\}.
\end{equation}
 
This inequality covers in particular the Sobolev spaces
$W^{k}_\infty(\domain;\C)$, $k\in\N_0$, on bounded Lipschitz domains
$\D\subseteq\R^d$, but also the Kondrat'ev spaces
$\Ww_\infty^k(\domain;\C)$ on polygonal domains $\D\subseteq \R^2$,
cp.~Lemma \ref{kon-lem-1}. 

For a function $\psi \in {\data}\subseteq L^\infty(\domain;\C)$ 
we will
write $\Re(\psi)\in L^\infty(\domain;\R)\subseteq L^\infty(\domain;\C)$
to denote its real part and
$\Im({\psi})\in L^\infty(\domain ;\R)\subseteq L^\infty(\domain ;\C)$ its
imaginary part so that $\psi=\Re(\psi)+\im\Im(\psi)$. 
Recall that the quantity $\rho(a)$ is defined  in \eqref{eq:PDEellinC} 
for $a \in L^\infty(\domain;\C)$.

\begin{theorem}\label{thm:bdX}
Let $0<\delta<\delta_{\rm max}$, $K>0$, $\eta>0$ and $m\in\N$. 
Let the inequality \eqref{eq:Znorm2} hold for the space $E$.
Assume that for an open set $O\subseteq {\data}$ containing
$$
\set{\exp(\psi)}{\psi\in {\data},~\norm[{\data}]{\Im(\psi)}\le \eta},
$$ 
it holds
\begin{enumerate}
  \item\label{item:uhol} $\Uu:O \to X$ is holomorphic,
  \item\label{item:norma} for all $a\in O$
    \begin{equation*}
      \norm[X]{\Uu(a)}\le \delta \left(\frac{1+\norm[{\data}]{a}}{\min\{1,{\rho(a)}\}}\right)^m,
    \end{equation*}
  \item\label{item:loclip} for all $a$, $b\in O$
	$$
	\norm[X]{\Uu(a)-\Uu(b)} \le K
        \left(\frac{1+\max\{\norm[{\data}]{a},\norm[{\data}]{b}\}}{\min\{1,{\rho(a)},{\rho(b)}\}}\right)^m
        \norm[{\data}]{a-b},
	$$
      \item\label{item:psi}
        $(\psi_j)_{j\in\N}\subseteq {\data}\cap L^\infty(\domain)$ and
        with $b_j:=\norm[{\data}]{\psi_j}$ it holds
        $\bb\in\ell^1(\N)$.
\end{enumerate}
  
      Then there exists $\xi>0$ and for every $\delta_{\rm max}>0$
      there exists $\tilde C$ depending on $\bb$, $\delta_{\rm max}$,
      $C_{\data}$ and $m$ but independent of
      $\delta\in (0,\delta_{\rm max})$, such that with
      \begin{equation*}\label{eq:u}
        u_N\left((y_j)_{j=1}^N\right)=\Uu\Bigg(\exp\bigg(\sum_{j=1}^N y_j\psi_j\bigg)\Bigg)\qquad
        \forall (y_j)_{j=1}^N\in\R^N,
      \end{equation*}
      and $\tilde u_N(\by)=u_N(y_1,\dots,y_N)$ for $\by\in U$, the
      function
  $$
  u:=\lim_{N\to\infty}\tilde u_N\in L^2(U,X;\gamma)
  $$
  is well-defined and $(\bb,\xi,\delta \tilde C,X)$-holomorphic.
\end{theorem}

\begin{proof}
  {\bf Step 1.} 
  Choosing $\psi_2 \equiv 0$ in
 \eqref{eq:Znorm2} with $\psi_1 = \psi$, we obtain
 \begin{equation}\label{eq:Znorm}
 	\norm[{\data}]{\exp(\psi)} 
 	\leq C_{\data}'   \exp\big((m+1)\norm[{\data}]{\psi}\big). 
 \end{equation}
 for some positive constant $C_{\data}'$.  Indeed,
 \begin{equation}\nonumber
 	\begin{aligned}
 		\norm[{\data}]{\exp(\psi_1)} & \leq \|1\|_{\data} +
 		C_{\data}\norm[{\data}]{\psi_1}
 		\exp\big(m\|1\|_{\data}+m\norm[{\data}]{\psi_1}\big)
 		\\
 		& \leq C_{\data}'\big(1+\norm[{\data}]{\psi_1}\big)
 		\exp\big(m\norm[{\data}]{\psi_1}\big)
 		\\
 		&\leq C_{\data}' \exp\big((m+1)\norm[{\data}]{\psi_1}\big).
 	\end{aligned}
 \end{equation}
  
  We show that 
  $u_N\in L^2(\R^N,X;\gamma_N)$ for every $N\in\N$.  
  To this end we
  recall that for any $s>0$ (see, e.g., \cite[Appendix B]{HS16_656},
  \cite[(38)]{BCDM} for a proof)
  \begin{equation}\label{eq:intexp}
    \int_\R \exp(s|y|)\dd \gamma_1(y) \le \exp\bigg(\frac{s^2}{2}+\frac{\sqrt{2}s}{\pi}\bigg).
  \end{equation}
  Since ${\data}$ is continuously embedded into
  $L^\infty(\domain;\C)$, there exists $C_0>0$ such that
  \begin{align*}
    \|\psi\|_{L^\infty(\D)}\leq C_0 \|\psi\|_{\data} \qquad \forall \psi \in \data.
  \end{align*}	
  Using \ref{item:norma}, \eqref{eq:Znorm}, and
  \begin{align*}
    \frac{1}{\essinf_{\bx \in \domain}  
    \big(\exp\big(\sum_{j=1}^Ny_j\psi_j(\bx)\big)\big) }
    &
      \leq
      \bigg\|\exp\Big(-\sum_{j=1}^Ny_j\psi_j\Big)\bigg\|_{L^\infty}
    \\
    &\leq  \exp\bigg(\bigg\|\sum_{j=1}^Ny_j\psi_j\bigg\|_{L^\infty}\bigg) \leq  \exp\bigg(C_0\sum_{j=1}^N|y_j| \big\|\psi_j\big\|_{\data}\bigg),
  \end{align*}
  we obtain the bound
  \begin{align*}
    \norm[X]{u_N(\by)}
    &
      \le \delta\bigg(1+ \normk[{\data}]{\exp\bigg(\sum_{j=1}^Ny_j\psi_j\bigg)}\bigg)^m \exp\bigg(C_0m\sum_{j=1}^N|y_j| \big\|\psi_j\big\|_{\data}\bigg)
    \\
    & \le   \delta\Bigg(1+ C_{\data}'
      \exp\bigg((m+1)\sum_{j=1}^N|y_j|\normc[{\data}]{\psi_j}\bigg)\Bigg)^m \exp\bigg(C_0m\sum_{j=1}^N |y_j|\big\|\psi_j\big\|_{\data}\bigg)
    \\
    & \leq  C_1 \exp\bigg((2+C_0)m^2\sum_{j=1}^N|y_j|\normc[{\data}]{\psi_j}\bigg)
  \end{align*}
  for some constant $C_1>0$ depending on $\delta, C_{\data}$ and $m$.
  Hence, by \eqref{eq:intexp} we have
  \begin{align*}
        \int_{\R^N}\norm[X]{u_N(\by)}^2 \dd\gamma_N(\by)
    &
      \le C_1 \int_{\R^N}  \exp\bigg((2+C_0)m^2)\sum_{j=1}^N |y_j|\norm[{\data}]{\psi_j}\bigg)\dd\gamma_N(\by)\\
      &\le C_1  \exp\Bigg(\frac{(2+C_0)^2m^4}{2}\sum_{j=1}^Nb_j^2+\frac{\sqrt{2}(2+C_0)m^2}{\pi}\sum_{j=1}^Nb_j \Bigg) <\infty.
  \end{align*}
  {\bf Step 2.} We show that $(\tilde u_N)_{N\in\N}$ which is defined
  as $\tilde u_N(\by):=u_N(y_1,\dots,y_N)$ for $\by\in U$, is a Cauchy
  sequence in $L^2(U,X;\gamma)$. For any $N<M$ by \ref{item:loclip}
  \begin{align*}
    \norm[L^2(U,X;\gamma)]{\tilde u_M-\tilde u_N}^2
    &=\int_U \normk[X]{\Uu\bigg(\exp\Big(\sum_{j=1}^My_j\psi_j\Big)\bigg)
      -\Uu\bigg(\exp\Big(\sum_{j=1}^Ny_j\psi_j\Big)\bigg)}^2\dd\gamma(\by)\nonumber
    \\
    &\quad\le K \int_U \Bigg[ \Bigg(1+\normk[{\data}]{\exp\Big(\sum_{j=1}^My_j \psi_j\Big)}+\normk[{\data}]{\exp\Big(\sum_{j=1}^Ny_j \psi_j\Big)}\Bigg)^m 
    \\
    &  \quad \times \exp\bigg(C_0m\sum_{j=1}^M|y_j| \big\|\psi_j\big\|_{\data}\bigg) \cdot \normk[{\data}]{\exp\Big(\sum_{j=1}^M y_j\psi_j\Big)-\exp\Big(\sum_{j=1}^N y_j\psi_j\Big)}\Bigg]\dd\gamma(\by).
  \end{align*}
  Using \eqref{eq:Znorm} again we can estimate
  \begin{align*}
    &\norm[L^2(U,X;\gamma)]{\tilde u_M-\tilde u_N}^2
      \le K \int_U \Bigg[\Bigg(1+ 2C_{\data}'
      \exp\bigg((m+1)\sum_{j=1}^M|y_j|\normc[{\data}]{\psi_j}\bigg)\Bigg)^m
    \\
    & \times \exp\bigg(C_0m\sum_{j=1}^M|y_j| \big\|\psi_j\big\|_{\data}\bigg)  \cdot \normk[{\data}]{\exp\Big(\sum_{j=1}^M y_j\psi_j\Big)-\exp\Big(\sum_{j=1}^N y_j\psi_j\Big)}\Bigg]\dd\gamma(\by)
    \\
    &\leq C_2\int_U \Bigg[ \exp\bigg((2+C_0)m^2)\sum_{j=1}^M|y_j|\normc[{\data}]{\psi_j}\bigg) \normk[{\data}]{\exp\Big(\sum_{j=1}^M y_j\psi_j\Big)-\exp\Big(\sum_{j=1}^N y_j\psi_j\Big)}\Bigg]\dd\gamma(\by)\,
  \end{align*}
  for $C_2>0$ depending only on $K,C_{\data}$, and $m$.  Now,
  employing \eqref{eq:Znorm2} we obtain
  \begin{align*}
    \normk[{\data}]{\exp\Big(\sum_{j=1}^M y_j\psi_j\Big)-\exp\Big(\sum_{j=1}^N y_j\psi_j\Big)}
    & \le C_{\data} \sum_{j=N+1}^M |y_j|\norm[{\data}]{\psi_j}\exp\Bigg(m\sum_{j=1}^M|y_j|\norm[{\data}]{\psi_j} \Bigg)
    \\
    & \le C_{\data}\sum_{j=N+1}^M |y_j|\norm[{\data}]{\psi_j} \exp\bigg(m^2\sum_{j=1}^M|y_j|\norm[{\data}]{\psi_j}  \bigg)   .
  \end{align*}
  Therefore for a constant $C_3$ depending on $C_{\data}$ and $\delta$
  (but independent of $N$), using $|y_j|\le\exp(|y_j|)$,
  \begin{align*}
    &   \norm[L^2(U,X;\gamma)]{\tilde u_M-\tilde u_N}^2 
    \\
    &\le
      C_3 \sum_{j=N+1}^M\norm[{\data}]{\psi_j}\int_{\R^M}|y_j|\exp\bigg((3+C_0)m^2
      \sum_{i=1}^M|y_i|\norm[{\data}]{\psi_i}\bigg)  \dd\gamma_M((y_i)_{i=1}^M)
      \nonumber\\
    &\le   C_3 \sum_{j=N+1}^M\norm[{\data}]{\psi_j}\int_{\R^M}
      \exp\bigg(|y_j|+(3+C_0)m^2\sum_{i=1}^M|y_i|\norm[{\data}]{\psi_i}\bigg) \dd\gamma_M((y_i)_{i=1}^M)
      \nonumber\\
    &\le C_3 \bigg(\sum_{j=N+1}^M b_j\bigg)     
      \Bigg( \exp\bigg(\frac{1}{2}+\frac{\sqrt{2}}{\pi} + \frac{(3+C_0)^2m^4}{2}\sum_{i=1}^Mb_j^2  
      + \frac{ \sqrt{2}(3+C_0)m^2}{\pi}\sum_{j=1}^Mb_j \bigg)\Bigg),
  \end{align*}
  where we used \eqref{eq:intexp} and in the last inequality.  Since
  $\bb\in\ell^1(\N)$ the last term is bounded by
  $ C_4 \Big(\sum_{j=N+1}^\infty b_j \Big)$ for a constant $C_4$
  depending on $C_{\data}$, $K$ and $\bb$ but independent of $N$, $M$.
  Due to $\bb\in\ell^1(\N)$, it also
  holds $$\sum_{j=N+1}^{\infty}b_j\to 0 \ \ {\rm as} \ \ N\to\infty.$$
  Since $N<M$ are arbitrary, we have shown that
  $(\tilde u_N)_{N\in\N}$ is a Cauchy sequence in the Banach space
  $L^2(U,X;\gamma)$.  This implies that there is a
  function $$u:=\lim_{N\to\infty}\tilde u_N\in L^2(U,X;\gamma).$$
	
  {\bf Step 3.} To show that $u$ is $(\bb,\xi,\delta \tilde{C},X)$
  holomorphic, we provide constants $\xi>0$ and $\tilde{C}>0$
  independent of $\delta$ so that $u_N$ admits holomorphic extensions
  as in Definition~\ref{def:bdXHol}.  This concludes the proof.

  Let $\xi:=\pi/(4C_0)$.  Fix $N\in\N$ and assume
	$$\sum_{j=1}^N b_j\varrho_j<\xi$$ 
	(i.e.\ $(\varrho_j)_{j=1}^N$ is $(\bb,\delta_1)$-admissible).
        Then for $z_j=y_j+\im\zeta_j \in\C$ such that
        $|\Im(z_j)|=|\zeta_j|<\varrho_j$ for all $j$,
	\begin{equation*}
          \rho\Bigg(\exp\Big(\sum_{j=1}^N z_j\psi_j(\bx) \Big)\Bigg)
          =
          \underset{\bx\in\D}{\essinf}\,\Bigg(\exp\Big(\sum_{j=1}^N y_j\psi_j(\bx) \Big)\Bigg)
          \cos\Bigg(\sum_{j=1}^N \zeta_j\psi_j(\bx)\Bigg).
	\end{equation*}
	Due to
	\begin{equation*}
          \underset{\bx\in\D}{\esssup} \,\Bigg|\sum_{j=1}^N\zeta_j\psi_j(\bx)\Bigg|
          \le \sum_{j=1}^N \varrho_j\norm[L^\infty]{\psi_j}
          \le \sum_{j=1}^N C_0\varrho_j\norm[V]{\psi_j}
          = \sum_{j=1}^N C_0\varrho_j b_j
          \le \frac{\pi}{4},
	\end{equation*}
	we obtain for such $(z_j)_{j=1}^N$
	\begin{equation}\label{eq:essinfx}
          \rho\Bigg(\exp\Big(\sum_{j=1}^N z_j\psi_j(\bx) \Big)\Bigg)
          \ge \exp\Bigg(-\sum_{j=1}^N|y_j|\norm[L^\infty]{\psi_j} \Bigg)
          \cos\Bigg(\frac \pi 4 \Bigg)>0.
	\end{equation}
	This shows that for every
        $\bvarrho=(\varrho_j)_{j=1}^N\in (0,\infty)^N$ such that
        $\sum_{j=1}^N b_j\varrho_j<\xi$, it holds
	$$\sum_{j=1}^Nz_j\psi_j\in O \quad \forall \bz\in \Ss(\bvarrho).$$  
        Since
	$\Uu:O\to X$ is holomorphic, the function
	$$ u_N\left((y_j)_{j=1}^N\right)=\Uu\Bigg(\exp\bigg(\sum_{j=1}^N y_j\psi_j\bigg)\Bigg)$$ can be
	holomorphically extended to arguments
        $(z_j)_{j=1}^N\in \Ss(\bvarrho)$.

        Finally we fix again $N\in\N$ and provide a function
        $\varphi_N\in L^2(U;\gamma)$ as in
        Definition~\ref{def:bdXHol}.  Fix $\by\in\R^N$ and
        $\bz\in \Bb_\bvarrho$ and set
	$$a:=\sum_{j=1}^N(y_j+z_j)\psi_j.$$ 
	By \ref{item:norma}, \eqref{eq:essinfx} and because
        $b_j=\norm[{\data}]{\psi_j}$ and
	$$\sum_{j=1}^N b_j\varrho_j\le\xi,$$ we have that
	\begin{align*}
          \norm[X]{u_N((y_j+z_j)_{j=1}^N)} &\le \delta\left(\frac{1+\norm[{\data}]{a}}{\min\{1,\rho(a) \}} \right)^m
                                             \nonumber\\
                                           &\le \delta \left(  \frac{1+C_{\data}'  
                                             \exp\big((m+1)\sum_{j=1}^N(|y_j|+|z_j|)\norm[{\data}]{\psi_j}\big) }{\exp(-C_0\sum_{j=1}^N(|y_j|+|z_j|)\norm[{\data}]{\psi_j})\cos(\frac{\pi}{4})}\right)^m 
                                             \nonumber\\
                                           &\le \delta  \left(\frac{1+C_{\data}'\exp\big((m+1)
                                             \sum_{j=1}^N |y_j|b_j\big)\exp((m+1)\xi)}{\exp(-C_0\sum_{j=1}^N|y_j|b_j)\exp(-C_0\xi)\cos(\frac{\pi}{4})}\right)^m
                                             \nonumber\\                                     
                                           &\le \delta L \exp\bigg( (2+C_0)m^2\sum_{j=1}^N|y_j|b_j\bigg)
        \end{align*}
	for some $L$ depending only on $C_{\data},C_0$ and $m$.  Let
        us define the last quantity as
        $\varphi_N\left((y_j)_{j=1}^N\right)$.  Then by
        \eqref{eq:intexp} and because $\gamma_N$ is a probability
        measure on $\R^N$,
	\begin{align*}
          \norm[L^2(\R^N;\gamma_N)]{\varphi_N}
          &\le \delta L  \exp\bigg(\sum_{j=1}^N
            \frac{(2+C_0)^2m^4b_j^2}{2}+(2+C_0)m^2\frac{\sqrt{2}b_j}{\pi} \bigg) \nonumber\\
          &\le \delta L  \exp\bigg(\sum_{j\in \N}
            \frac{(2+C_0)^2m^4b_j^2}{2}+(2+C_0)m\frac{\sqrt{2}b_j}{\pi} \bigg) 
          \\   
          & \le \delta \tilde C(\bb,C_0,C_\data,m),
	\end{align*}
	for some constant $\tilde C(\bb,C_0,C_\data,m) \in (0,\infty)$
        because $\bb\in\ell^1(\N)$.  In all, we have shown that $u$
        satisfies $(\bb,\xi,\delta \tilde{C},X)$-holomorphy as in
        Definition~\ref{def:bdXHol}.
      \end{proof}
      \subsection{Examples of holomorphic data-to-solution maps}
      \label{sec:HlExmpl}
      We revisit the example of linear elliptic divergence-form PDE
      with diffusion coefficient introduced in
      Section~\ref{sec:EllPDElogN}. 
      Its coefficient-to-solution map
      $S$ from \eqref{eq:SolOp} for a fixed $f \in X'$, 
      gives rise to parametric maps which
      are parametric-holomorphic. This kind of function will, on the
      one hand, arise as generic model of Banach-space valued
      uncertain inputs of PDEs, and on the other hand as model of
      solution manifolds of PDEs.  The connection is made through
      preservation of holomorphy under composition with inversion of
      boundedly invertible differential operators.

      Let $f\in X'$ be given.  
      If $A(a)\in \cL_{{\rm is}}(X,X')$ is an
      isomorphism depending (locally) holomorphically on $a\in \data$,
      then 
      $$ 
            \Uu:E\to X: a\mapsto (\inv \circ A(a))f
      $$ 
      is also locally holomorphic as a function of $a\in {\data}$. 
      Here $\inv$ denotes the inversion map.  
      This is a consequence of the fact that the 
      $\inv: \cL_{{\rm is}}(X,X') \to \cL_{{\rm is}}(X',X)$ is
      holomorphic, see e.g.\ \cite[Example 1.2.38]{JZdiss}.  
      This
      argument can be used to show that the solution operator
      corresponding to the solution of certain PDEs is holomorphic in
      the parameter.  We informally discuss this for some parametric
      PDEs and refer to \cite[Chapter 1 and 5]{JZdiss} for more
      details.
      \subsubsection{Linear elliptic divergence-form PDE with parametric diffusion coefficient}\label{sec:pdc}
      Let us again consider the model linear elliptic PDE \index{PDE!linear elliptic $\sim$}
      \begin{equation}\label{eq:elliptic0}
        - \div(a \nabla \Uu(a)) =f \;\;\text{in }\D \;,\quad 
        \Uu(a) =0 \;\; \text{on }\partial \D
      \end{equation}
      where $d\in\N$, $\D\subseteq\R^d$ is a bounded Lipschitz domain,
      $X:=H_0^1(\D;\C)$, $f\in H^{-1}(\D;\C):=(H_0^1(\D;\C))'$ and
      $a\in {\data}:=L^\infty(\D;\C)$.  Then the solution operator
      $\Uu:O\to X$ maps the coefficient function $a$ to the weak solution $\Uu(a)$, 
      where
  $$
  O:= \set{a\in L^\infty(\D;\C)}{\rho(a) >0}, 
  $$
  with $\rho(a)$ defined in \eqref{eq:PDEellinC} 
  for $a \in L^\infty(\domain;\C)$. 
  With $A(a)$ denoting the differential operator
  $- \div(a \nabla \cdot)\in L(X,X')$ we can also write
  $\Uu(a)=A(a)^{-1}f$. 
  We now check assumptions
  \ref{item:uhol}--\ref{item:loclip} of Theorem \ref{thm:bdX}.
  \begin{enumerate}
  \item As mentioned above, complex Fr\'echet differentiability (i.e.\
    holomorphy) of $\Uu:O\to X$ is satisfied because the operation of
    inversion of linear operators is holomorphic on the set of
    boundedly invertible linear operators, $A$ depends boundedly and
    linearly (thus holomorphically) on $a$, and therefore, the map
    $$a\mapsto A(a)^{-1}f=\Uu(a)$$ is a composition of holomorphic
    functions.  We refer once more to \cite[Example 1.2.38]{JZdiss}
    for more details.
  \item For $a\in O$, it holds
    \begin{equation*}
      \norm[X]{\Uu(a)}^2 \, \rho(a)
      \le 
      \left|\int_\D\nabla\Uu(a)^\top a \overline{\nabla\Uu(a)}\dd \bx \right|
      =
      \left|\dup{f}{\overline{\Uu(a)}}\right| 
      \le \norm[X']{f}\norm[X]{\Uu(a)}.
    \end{equation*}
Here $\dup{\cdot}{\cdot}$ denotes the dual product between $X'$ and $X$.
This gives the usual a-priori bound
    \begin{equation}\label{eq:apriori}
      \norm[X]{\Uu(a)}\le \frac{\norm[X']{f}}{\rho(a)}.
    \end{equation}
  \item For $a$, $b\in O$ and with $w:=\Uu(a)-\Uu(b)$, we have that
    \begin{align*}
      \frac{\norm[X]{w}^2}{\rho(a)}
      &\le\left| \int_\D \nabla w^\top a \overline{\nabla w}\dd \bx\right|\nonumber\\
      &= \left|\int_\D \nabla \Uu(a)^\top a \overline{\nabla w}\dd \bx
        - \int_\D \nabla \Uu(b)^\top b \overline{\nabla w}\dd \bx
        -\int_\D \nabla \Uu(b)^\top (a-b) \overline{\nabla w}\dd \bx
        \right|\nonumber\\
      &\le \norm[X]{\Uu(b)} \norm[X]{w}\norm[{\data}]{a-b}\nonumber\\
      &\le \frac{\norm[X']{f}}{\rho(b)}
        \norm[X]{w}\norm[{\data}]{a-b},
    \end{align*}
    and thus
    \begin{equation}\label{eq:lipschitz}
      \norm[X]{\Uu(a)-\Uu(b)}\le \norm[X']{f}
      \frac{\norm[{\data}]{a}}{\rho(b)}
      \norm[{\data}]{a-b}.
    \end{equation}
  \end{enumerate}
  Hence, if
  $(\psi_j)_{j\in\N}\subset {\data}$ such that with
  $b_j:=\norm[{\data}]{\psi_j}$ it holds $\bb\in\ell^1(\N)$, 
  then the solution
$$
u(\by)=\lim_{N\to\infty}\Uu\left(\exp\left(\sum_{j=1}^N
    y_j\psi_j\right)\right) \in L^2(U,X;\gamma)
$$
is well-defined and $(\bb,\xi,\delta,X)$-holomorphic by Theorem \ref{thm:bdX}.

This example can easily be generalized to spaces of higher-regularity,
e.g.,  if $\D\subseteq\R^d$ is a bounded $C^{s-1}$ domain for some
$s\in\N$, $s\ge 2$, then we may set
$X:= H_0^1(\D;\C)\cap H^{s}(\D;\C)$ and ${\data}:=W^{s}_\infty(\D;\C)$
and repeat the above calculation.
\subsubsection{Linear parabolic PDE with parametric coefficient} 
\label{sec:LinParPDE}
Let $0<T<\infty$ denote a finite time-horizon and let $\D$ be a
bounded domain with Lipschitz boundary $\partial \D$ in $\RR^d$.  We
define $I:=(0,T)$ and consider the initial boundary value problem
(IBVP for short) for the linear parabolic PDE \index{PDE!linear parabolic $\sim$}
\begin{equation}\label{eq:parabolic}
  \begin{cases}
    \frac{\partial u(t,\bx)}{\partial t} - \div\big(a(\bx)\nabla
    u(t,\bx)\big)=f(t,\bx), \qquad (t,\bx)\in I\times \D,
    \\
    u|_{\partial \D\times I}=0,
    \\
    u|_{t=0}=u_0(\bx).
  \end{cases}
\end{equation}
In this section, we prove that the solution to this problem satisfies
the assumptions of Theorem~\ref{thm:bdX} for certain spaces $E$ and $X$. 
We first review results on
the existence and uniqueness of solutions to the equation
\eqref{eq:parabolic}. 
We refer to \cite{SchS} and the references
there for proofs and more detailed discussion.

We denote $V:=H_0^1(\D;\C)$ and $V':=H^{-1}(\D;\C)$.  
The parabolic
IBPV given by equation \eqref{eq:parabolic} is a well-posed operator
equation in the intersection space of Bochner spaces (e.g.\
\cite[Appendix]{SchS}, and e.g.\ \cite{Wloka,Evan10} for the
definition of spaces)
\begin{equation*}
  X := L^2(I,V)\cap H^1(I,V') = \big(L^2(I)\otimes V \big)\cap \big(H^1(I)\otimes V'\big)
\end{equation*}
equipped with the sum norm
\begin{equation*}
  \|u\|_{X} := \Big(\|u\|_{L^2(I,V)}^2 +\|u\|_{H^1(I,V')}^2 \Big)^{1/2},\qquad u\in X,
\end{equation*}
where
$$
\|u\|_{L^2(I,V)}^2 = \int_I \|u(t, \cdot) \|_V^2\, \rd t\, ,
$$
and
$$
\|u\|_{H^1(I,V')}^2 = \int_I \| \partial_t u(t, \cdot) \|_{V'}^2\, \rd
t\,.
$$
To state a space-time variational formulation and to specify the data
space for \eqref{eq:parabolic}, we introduce the test-function space
\begin{equation*}\label{eq:ParIBVPY}
  Y 
  = L^2(I,V)\times L^2(\domain) 
  = \big(L^2(I)\otimes V \big) \times L^2(\domain)
\end{equation*}
which we endow with the norm
\begin{equation*}
  \|v\|_{Y}
  =
  \Big(\|v_1\|_{L^2(I,V)}^2 + \|v_2\|_{L^2(\domain)}^2\Big)^{1/2},\qquad v=(v_1,v_2)\in Y\,.
\end{equation*}
Given a time-independent diffusion coefficient $a\in L^\infty(\D;\C)$
and $(f,u_0)\in Y'$, the continuous sesquilinear and antilinear forms
corresponding to the parabolic problem \eqref{eq:parabolic} reads for
$u\in X$ and $v=(v_1,v_2)\in Y$ as
\begin{equation*}
  \begin{split}
    B(u,v;a)& := \int_I \int_\D \partial_t u\,\overline{v_1}\rd \bx
    \rd t + \int_I\int_\D a\nabla u \cdot \overline{\nabla v_1} \rd
    \bx \rd t + \int_\D u_0\,\overline{v_2 } \rd \bx
  \end{split}
\end{equation*}
and
\begin{equation*}
  \begin{split}
    L(v) := \ \int_I \big\langle f(t,\cdot),v_1(t,\cdot) \big\rangle
    \rd t + \int_\D u_0 \,\overline{v_2 } \rd \bx,
  \end{split}
\end{equation*}
where $\langle \cdot,\cdot \rangle$ is the anti-duality pairing
between $V'$ and $V$.  Then the space-time variational formulation of
equation \eqref{eq:parabolic} is: Find $\Uu(a)\in X$ such that
\begin{equation}\label{para-weak}
  B(\Uu(a),v;a)=L(v),\quad  \forall v\in Y\,.
\end{equation}
The existence and uniqueness of solution to the equation
\eqref{para-weak} was proved in \cite{SchS} which reads as follows.
\begin{proposition}\label{prop:parabolic}
  Assume that $(f,u_0)\in Y'$ and that
  \begin{equation}\label{eq:para-uni}
    0 < \rho(a) := \underset{\bx\in\D}{\essinf}\,\Re(a(\bx))
    \leq |a(\bx)|
    \leq \|a\|_{L^\infty}
    <\infty,\qquad \bx\in \domain
  \end{equation}
  Then the parabolic operator $\mathcal{B}\in \mathcal{L}(X,Y')$
  defined by $$(\mathcal{B}u)(v)=B(u,v;a),$$ 
  is an isomorphism and
  $\mathcal{B}^{-1}:Y \to X$ has the norm
  \begin{equation*}
    \|\mathcal{B}^{-1}\|  \leq \frac{1}{\beta(a)}, 
  \end{equation*}
  where
  \begin{equation*}
    \beta(a) := \frac{\min\big(\rho(a) \|a\|_{L^\infty}^{-2}, \rho(a) \big)}{\sqrt{2\max(\rho(a)^{-2},1)+\vartheta^2}}
    \qquad\text{and}\qquad 
    \vartheta := \sup_{w \not=0, w\in X} \frac{\| w(0,\cdot)\|_{L^2(\domain)}}{\|w\|_{X}}\,.
  \end{equation*}
  The constant $\vartheta$ depends only on $T$.
\end{proposition}
The data space for the equation \eqref{eq:parabolic} for complex-valued data
is $E:= L^\infty(\domain,\CC)$.  
With the set of admissible diffusion coefficients in the data space
$$
O:=\set{a\in L^\infty(\D,\C)} {\rho(a) >0},
$$ 
from the above proposition we immediately deduce that for given
$(f,u_0)\in Y'$, the map 
$$
\Uu: O \to X: a \ \mapsto \ \Uu(a)
$$ is
well-defined.

Furthermore, there holds the a-priori estimate
\begin{equation}\label{eq:para-bound}
  \|\Uu(a)\|_{X} 
  \leq 
  \frac{1}{\beta(a)}\Big(\|f\|_{L^2(I,V')}^2 + \|u_0\|_{L^2}^2\Big)^{1/2}\,.
\end{equation}
This bound is a consequence of the following result which states that
the data-to-solution map $ a \to \Uu(a)$ is locally Lipschitz
continuous.
\begin{lemma} \label{lem:para-Lip} 
Let $(f,u_0)\in Y'$.  
Assume that
$\Uu(a)$ and $\Uu(b)$ be solutions to \eqref{para-weak} with
coefficients $a$, $b$ satisfying \eqref{eq:para-uni}, respectively.

Then, with the function $\beta( \cdot)$ in variable $a$ as in
Proposition \ref{prop:parabolic}, we have
  \begin{equation*}
    \|\Uu(a)-\Uu(b)\|_{X} 
    \leq 
    \frac{1}{\beta(a)\beta(b)}\|a-b\|_{L_\infty}\Big(\|f\|_{L^2(I,V')}^2 + \|u_0\|_{L^2}^2\Big)^{1/2}\,.
  \end{equation*}
\end{lemma}
\begin{proof}
  From \eqref{para-weak} we find that for $w := \Uu(a)-\Uu(b)$,
  \begin{equation*}
    \begin{split} 
      \int_I \int_\D \partial_t w\,\overline{v_1 } \rd \bx\, \rd t & +
      \int_I\int_\D a \nabla w \cdot \overline{\nabla v_1 } \, \rd \bx
      \, \rd t + \int_\D w\big|_{t=0}\overline{v_2} \rd \bx
      \\
      & = -\int_I \int_\D \big(a-b\big) \nabla \Uu(b)\cdot
      \overline{\nabla v_1} \, \rd \bx \, \rd t\,.
    \end{split}
  \end{equation*}
  This is a parabolic equation in the variational form with
  $(\tilde{f},0)\in Y'$ where $\tilde{f}: L^2(I,V) \to \C$ is given by
  \begin{equation*}
    \tilde{f}(v_1) 
    :=  
    -\int_I \int_\D  \big(a-b\big) \nabla \Uu(b)\cdot  \overline{\nabla v_1} \, \rd \bx \, \rd t\,, \qquad v_1\in L^2(I,V).
  \end{equation*}
  Now applying Proposition \ref{prop:parabolic} we find
  \begin{equation} \label{eq:ua-ua} \|\Uu(a)-\Uu(b)\|_{X} \leq
    \frac{\|\tilde{f}\|_{L^2(I,V')}}{\beta(a)}.
  \end{equation}
  We also have
  \begin{equation*}
    \begin{split}
      \|\tilde{f}\|_{L^2(I,V')} = \sup_{\|v_1\|_{L^2(I,V)}=1} |
      \tilde{f}(v_1) | & \leq \|a-b\|_{L_\infty} \|\Uu(b)\|_{L^2(I,V)}
      \|v_1\|_{L^2(I,V)}
      \\
      & \leq \|a-b\|_{L_\infty}
      \frac{1}{\beta(b)}\Big(\|f\|_{L^2(I,V')}^2 +
      \|u_0\|_{L^2}^2\Big)^{1/2},
    \end{split}
  \end{equation*}
  where in the last estimate we used again Proposition
  \ref{prop:parabolic}. Inserting this into \eqref{eq:ua-ua} we obtain
  the desired result.  \hfill
\end{proof}
We are now in position to verify the assumptions
\ref{item:uhol}--\ref{item:loclip} of Theorem \ref{thm:bdX} for the
data-to-solution map $a\mapsto \Uu(a)$ to the equation
\eqref{eq:parabolic}.
\begin{enumerate}
\item For the first condition, it has been shown that the weak
  solution to the linear parabolic PDEs \eqref{eq:parabolic} depends
  holomorphically on the data $a\in O$ by the Ladyzhenskaya-Babu\v
  ska-Brezzi theorem in Hilbert spaces over $\C$, see e.g.\
  \cite[Pages 26, 27]{CoDe}.
\item Let $a\in O$.  Using the elementary estimate $a+b\leq ab$ with
  $a,b\geq 2$, we get
  \begin{equation*}
    \begin{split} 
      \sqrt{2\max(\rho(a)^{-2},1)+\vartheta^2} & \leq
      \sqrt{2\max(\rho(a)^{-2},1)+\max(\vartheta^2,2)}
      \\
      &\leq \sqrt{2\max(\rho(a)^{-2},1) \max(\vartheta^2,2)} \leq
      \max(\vartheta\sqrt{2},2) ({\rho(a)}^{-1}+1)\,.
    \end{split}
  \end{equation*}
  Hence, from \eqref{eq:para-bound} we can bound
  \begin{equation} \label{eq-1.beta(a)}
    \begin{split} 
      \|\Uu(a)\|_{X} \leq \frac{C_0(\rho(a)^{-1}+1)}{\min\big(\rho(a)
        \|a\|_{L^\infty}^{-2}, \rho(a)\big)} & =
      \frac{C_0(1+\rho(a))}{\rho(a)^2\min\big(\|a\|_{L^\infty}^{-2},
        1\big)}
      \\
      &\leq
      \frac{C_0(1+\|a\|_{L^\infty})\|a\|_{L^\infty}^2}{\min\big(\rho(a)^{4},
        1\big)} \leq
      C_0\bigg(\frac{1+\|a\|_{L^\infty}}{\min\big(\rho(a),
        1\big)}\bigg)^4\,,
    \end{split}
  \end{equation}
  where
$$C_0= \max(\vartheta\sqrt{2},2) \big(\|f\|_{L^2(I,V')}^2 + \|u_0\|_{L^2}^2\big)^{1/2}.$$ 
\item The third assumption follows from Lemma \ref{lem:para-Lip} and
  the part (ii), i.e., for $a,b\in O$ holds
  \begin{equation}\label{eq-ua-ub}
    \|\Uu(a)-\Uu(b)\|_{X} 
    \leq 
    C \bigg(\frac{1+\|a\|_{L^\infty}}{\min\big(\rho(a), 1\big)}\bigg)^4
    \bigg(\frac{1+\|b\|_{L^\infty}}{\min\big(\rho(b), 1\big)}\bigg)^4 \|a-b\|_{L_\infty}, 
  \end{equation}
  for some $C>0$ depending on $f$, $u_0$ and $T$.
\end{enumerate}

In conclusion,  if
$(\psi_j)_{j\in\N}\subset L^\infty(\D)$ such that with
$b_j:=\norm[{L^\infty}]{\psi_j}$ it holds $\bb\in\ell^1(\N)$, then the
solution
\begin{equation} \label{u(by)}
u(\by)=\lim_{N\to\infty}\Uu\left(\exp\left(\sum_{j=1}^N
y_j\psi_j\right)\right) 
\end{equation}
belonging to $L^2(U,X;\gamma)$ 
is well-defined and $(\bb,\xi,\delta,X)$-holomorphic by Theorem \ref{thm:bdX}.

We continue studying the holomorphy of the solution map to the
equation \eqref{eq:parabolic} in function space of higher-regularity.
Denote by $H^1(I,L^2(\domain))$ the space of all functions
$v(t,\bx)\in L^2(I,L^2(\domain))$ such that the norm
$$
\|v\|_{H^1(I,L^2)} := \Big(\|v\|_{L(I,L^2)}^2 + \|\partial_t
v\|_{L^2(I,L^2)}^2 \Big)^{1/2}
$$
is finite.  We put
$$
Z := L^2(I,W)\cap H^1(I,L^2(\domain)), \quad W:=\big\{ v\in V: \
\Delta v\in L^2(\domain)\big\},
$$
and
$$
\|v\|_Z := \Big( \|v\|_{H^1(I,L^2)}^2 + \|v\|_{L^2(I,W)}^2\Big)^{1/2}.
$$
In the following the constant $C$ and $C'$ may change their values
from line to line.
\begin{lemma} \label{lem:Z1} Assume that
  $a\in W^{1}_\infty(\domain)\cap O$ and $f\in L^2(I,L^2(\domain))$
  and $u_0\in V$.  Suppose further that $\Uu(a)\in X$ is the weak
  solution to the equation \eqref{eq:parabolic}.  Then
  $\Uu(a)\in L^2(I,W)\cap H^1(I,L^2(\domain))$.  Furthermore,
  \begin{equation*}
    \begin{aligned} 
      \| \partial_t \Uu(a) \|_{L^2(I,L^2)} \leq \bigg(\frac{1+\|
        a\|_{W^{1}_\infty}}{\min(\rho(a),1)} \bigg)^4 \big(
      \|u_0\|_{V} + \|f\|_{L^2(I,L^2)} \big)^{1/2},
    \end{aligned}
  \end{equation*}
  and
  \begin{align*}
    \| \Delta \Uu(a) \|_{L^2(I,L^2)}
    &\leq C	  \bigg(\frac{1+\| a\|_{W^{1}_\infty}}{\min({\rho(a)},1)} \bigg)^5
      \big(\|u_0\|_{V}^2 + \|f\|_{L^2(I,L^2)}^2\big)^{1/2}, 
  \end{align*}
  where $C>0$ independent of $f$ and $u_0$.  Therefore,
$$
\|\Uu(a)\|_Z \leq C \bigg(\frac{1+\|
  a\|_{W^{1}_\infty}}{\min({\rho(a)},1)} \bigg)^5 \big(\|u_0\|_{V}^2 +
\|f\|_{L^2(I,L^2)}^2\big)^{1/2} .
$$
\end{lemma}
\begin{proof}
  The argument follows along the lines of, e.g., \cite[Section
  7.1.3]{Evan10} by separation of variables.  Let
  $(\omega_k)_{k\in \NN}\subset V$ be an orthogonal basis which is
  orthonormal basis of $L^2(\domain)$, {[eigenbasis in polygon
    generally not smooth]}, see, e.g.\ \cite[Page 353]{Evan10}.  Let
  further, for $m\in \N$,
$$
\Uu_m(a)=\sum_{k=1}^m d_m^k(t)\omega_k\in V_m
$$ 
be a Galerkin approximation to $\Uu(a)$ on
$V_m:={\rm span}\{\omega_k,\ k=1,\ldots,m\}$.

Then we have
$$\partial_t\Uu_m(a) =\sum_{k=1}^m \frac{d}{dt}d_m^k (t)\omega_k\in V_m.$$
Multiplying both sides with $\partial_t\Uu_m(a)$ we get
\begin{equation*}
  \begin{aligned} 
    \int_\D \partial_t \Uu_m(a) \partial_t\overline{ \Uu_m(a) } \rd
    \bx + \int_\D a\nabla \Uu_m(a) \cdot \partial_t \overline{\nabla
      \Uu_m(a)} \, \rd \bx = \int_\D f \partial_t \overline{ \Uu_m(a)
    } \, \rd \bx.
  \end{aligned}
\end{equation*}
The conjugate equation is given by
\begin{equation*}
  \begin{aligned} 
    \int_\D \partial_t \Uu_m(a) \partial_t \overline{ \Uu_m(a) } \rd
    \bx & + \int_\D \overline{a}\, \overline{\nabla \Uu_m(a) }\cdot
    \partial_t\nabla \Uu_m(a) \, \rd \bx = \int_\D \bar{f} \partial_t
    { \Uu_m(a) } \, \rd \bx.
  \end{aligned}
\end{equation*}
Consequently we obtain
\begin{equation*}
  \begin{aligned} 
    2\| \partial_t \Uu_m(a) \|_{L^2}^2 & + \frac{\rd}{\rd t} \int_\D
    \Re(a) |\nabla \Uu_m(a) |^2 \rd \bx = \int_\D f \partial_t
    \overline{ \Uu_m(a) } \, \rd \bx+\int_\D \bar{f} \partial_t {
      \Uu_m(a) } \, \rd \bx.
  \end{aligned}
\end{equation*}
Integrating both sides with respect to $t$ on $I$ and using the
Cauchy-Schwarz inequality we arrive at
\begin{equation*}
  \begin{aligned} 
    2\| \partial_t \Uu_m(a) \|_{L^2(I,L^2)}^2 & + \int_\D \Re(a)
    \big|\nabla \Uu_m(a) \big|_{t=T}\big|^2 \rd \bx
    \\
    &\leq \int_\D \Re(a) \big|\nabla \Uu_m(a)\big|_{t=0}\big|^2 \rd
    \bx + \|f\|_{L^2(I,L^2)}^2 + \| \partial_t \Uu_m(a)
    \|_{L^2(I,L^2)}^2,
  \end{aligned}
\end{equation*}
which implies
\begin{equation}\label{eq-partial-t-u}
  \begin{aligned} 
    \| \partial_t \Uu_m(a) \|_{L^2(I,L^2)}^2 & \leq \int_\D \Re(a)
    \big|\nabla \Uu_m(a)\big|_{t=0}\big|^2 \rd \bx +
    \|f\|_{L^2(I,L^2)}^2
    \\
    & \leq \|a\|_{L^\infty} \big\|\nabla
    \Uu_m(a)\big|_{t=0}\big\|_{L^2}^2 + \|f\|_{L^2(I,L^2)}^2
    \\
    &\leq \|a\|_{L^\infty} \|u_0\|_{V}^2 + \|f\|_{L^2(I,L^2)}^2 ,
  \end{aligned}
\end{equation}
where we used the bounds
$\|\nabla \Uu_m(a)\big|_{t=0}\|_{L^2} \leq \|u_0\|_{V}$, see
\cite[Page 362]{Evan10}.

Passing to limits we deduce that
\begin{align*}
  \|  \partial_t   \Uu(a) \|_{L^2(I,L^2)} 
  &
    \leq  \big(\|a\|_{L^\infty} \|u_0\|_{V}^2  + \|f\|_{L^2(I,L^2)}^2\big)^{1/2} 
  \\
  &\leq 
    (\|a\|_{L^\infty} + 1)^{1/2}\big( \|u_0\|_{V}^2 + \|f\|_{L^2(I,L^2)}^2\big)^{1/2}
  \\
  &
    \leq 
    C\bigg(\frac{1+\|a\|_{W^{1}_\infty}}{\min({\rho(a)},1)} \bigg)^4 
    \big(      \|u_0\|_{V}^2    + \|f\|_{L^2(I,L^2)}^2\big)^{1/2} .
\end{align*}
We also have from \eqref{eq:para-bound} and \eqref{eq-1.beta(a)} that
\begin{equation}\label{eq-u-L2L2}
  \begin{aligned}
    \| \Uu(a) \|_{L^2(I,L^2)} \leq C\| \Uu(a) \|_{L^2(I,V)} & \leq
    \frac{C}{\beta(a)} \big(\|f\|_{L^2(I,V')}^2 +
    \|u_0\|_{L^2}^2\big)^{1/2}
    \\
    & \leq \frac{C}{\beta(a)} \big( \|u_0\|_{V}^2 +
    \|f\|_{L^2(I,L^2)}^2\big)^{1/2}
    \\
    & \leq C\bigg(\frac{1+\|a\|_{L^\infty}}{\min({\rho(a)},1)}
    \bigg)^4 \big( \|u_0\|_{V}^2 + \|f\|_{L^2(I,L^2)}^2\big)^{1/2}
    \\
    & \leq C\bigg(\frac{1+\|a\|_{W^{1}_\infty}}{\min({\rho(a)},1)}
    \bigg)^4 \big( \|u_0\|_{V}^2 + \|f\|_{L^2(I,L^2)}^2\big)^{1/2} .
  \end{aligned}
\end{equation}
We now estimate $\| \Delta \Uu(a) \|_{L^2(I,L^2)}$.  From the identity
(valid in $L^2(I,L^2(\domain))$)
\begin{align*}
  - \Delta \Uu(a) = \frac{1}{a}\big[\nabla a \cdot \nabla  \Uu(a) + f -\partial_t \Uu(a) \big],
\end{align*}
and \eqref{eq-partial-t-u}, \eqref{eq-u-L2L2} we obtain that
\begin{align*}
  \| \Delta \Uu(a) \|_{L^2(I,L^2)}
  &
    \leq  
    \frac{1}{\rho(a)}\bigg[\| a\|_{W^{1}_\infty} \|   \Uu(a) \|_{L^2(I,V)}
    + 
    \| f\|_{L^2(I,L^2)}  + \|\partial_t \Uu(a)\|_{L^2(I,L^2)} \bigg]
  \\
  &
    \leq 
    C\frac{\| a\|_{W^{1}_\infty}}{{\rho(a)}} 
    \bigg(\frac{1+\|a\|_{L^\infty}}{\min({\rho(a)},1)}\bigg)^4\big(\|u_0\|_{V}^2 
    + 
    \|f\|_{L^2(I,L^2)}^2\big)^{1/2} 
  \\
  &
    \leq C	  \bigg(\frac{1+\| a\|_{W^{1}_\infty}}{\min({\rho(a)},1)} \bigg)^5
    \big( \|u_0\|_{V}^2 + \|f\|_{L^2(I,L^2)}^2\big)^{1/2}, 
\end{align*}
with $C>0$ independent of $f$ and $u_0$.  Combining this and
\eqref{eq-partial-t-u}, \eqref{eq-u-L2L2}, the desired result follows.
\end{proof}
\begin{lemma} \label{lem-para-Z2} Assume $f\in L^2(I,L^2(\domain))$
  and $u_0\in V$.  Let $\Uu(a)$ and $\Uu(b)$ be the solutions to
  \eqref{para-weak} with $a,b\in W^{1}_\infty(\domain)\cap O$,
  respectively.  Then we have
  \begin{equation*}
    \begin{aligned} 
      \|\Uu(a)-\Uu(b)\|_Z & \leq C' \bigg(\frac{1+\|
        a\|_{W^{1}_\infty}}{\min({\rho(a)},1)} \bigg)^5
      \bigg(\frac{1+\| b\|_{W^{1}_\infty}}{\min({\rho(b)},1)} \bigg)^5
      \|a-b\|_{W^{1}_\infty},
    \end{aligned}
  \end{equation*}
  with $C'>0$ depending on $f$ and $u_0$.
\end{lemma}
\begin{proof}
  Denote $w := \Uu(a)-\Uu(b)$.  Then $w$ is the solution to the
  equation
  \begin{equation} \label{eq-equation-w}
    \begin{cases}
      {\partial_t w - \div\big(a\nabla w\big)= \nabla (a-b)\cdot\nabla
        \Uu(b) + (a-b)\Delta \Uu(b), }
      \\
      w|_{\partial \D\times I}=0,
      \\
      w|_{t=0}=0.
    \end{cases}
  \end{equation}
  Hence
$$
-\Delta w = \frac{1}{a}\bigg[\nabla a\cdot \nabla w + \nabla
(a-b)\cdot\nabla \Uu(b) + (a-b)\Delta \Uu(b) -\partial_t w \bigg]
$$
which leads to
\begin{equation*}\label{eq:DeltaParab}
  \begin{aligned} 
    \| \Delta w \|_{L^2(I,L^2)} &\leq \frac{1}{{\rho(a)}}\bigg[
    \|a\|_{W^{1}_\infty} \|w\|_{L^2(I,V)} + \|\partial_t
    w\|_{L^2(I,L^2)}
    \\
    &\ \ + \|a-b\|_{W^{1}_\infty } \big(\|\Uu(b)\|_{L^2(I,W)}
    +\|\Uu(b)\|_{L^2(I,V)}\big) \bigg].
  \end{aligned}
\end{equation*}
Lemma \ref{lem:Z1} gives that
\begin{align*}
  \|\partial_t w\|_{L^2(I,L^2)} 
  &
    \leq \bigg(\frac{1+\| a\|_{W^{1}_\infty}}{\min({\rho(a)},1)} \bigg)^5\big(\|\nabla (a-b)\cdot\nabla \Uu(b)\|_{L^2(I,L^2)}^2 + \|(a-b)\Delta \Uu(b)\|_{L^2(I,L^2)}^2\big)^{1/2}
  \\
  &
    \leq \bigg(\frac{1+\| a\|_{W^{1}_\infty}}{\min({\rho(a)},1)} \bigg)^5 \|a-b\|_{W^{1}_\infty} \big(\|\Uu(b)\|_{L^2(I,W)}^2 +\|\Uu(b)\|_{L^2(I,V)}^2\big)^{1/2},
\end{align*}
and
\begin{align*}
  \|\Uu(b)\|_{L^2(I,W)} +\|\Uu(b)\|_{L^2(I,V)} 
  &
    \leq  C	  \bigg(\frac{1+\| b\|_{{W^{1}_\infty}}}{\min({\rho(b)},1)} \bigg)^5\big(      \|u_0\|_{V}^2    + \|f\|_{L^2(I,L^2)}^2\big)^{1/2}, 
\end{align*}
which implies
\begin{align*}
  & \|a-b\|_{W^{1}_\infty} \big(\|\Uu(b)\|_{L^2(I,W)} +\|\Uu(b)\|_{L^2(I,V)}\big) +	\|\partial_t w\|_{L^2(I,L^2)} 
  \\
  &
    \leq C' \bigg(\frac{1+\| a\|_{W^{1}_\infty}}{\min({\rho(a)},1)} \bigg)^5  \bigg(\frac{1+\| b\|_{{W^{1}_\infty}}}{\min({\rho(b)},1)} \bigg)^5 \|a-b\|_{W^{1}_\infty} 
\end{align*}
We also have
$$ \|w\|_{L^2(I,V)} \leq  \frac{1}{\beta(a)\beta(b)}\|a-b\|_{L_\infty}
\leq C' \bigg(\frac{1+\|a\|_{L^\infty}}{\min({\rho(a)},1)} \bigg)^4
\bigg(\frac{1+\|b\|_{L^\infty}}{\min({\rho(b)},1)}
\bigg)^4\|a-b\|_{W^{1}_\infty},
$$
see \eqref{eq-ua-ub}.  Hence
\begin{equation}\label{eq-delta-w}
  \begin{aligned} 
    \| \Delta w \|_{L^2(I,L^2)} &\leq C' \bigg(\frac{1+\|
      a\|_{{W^{1}_\infty}}}{\min({\rho(a)},1)} \bigg)^5
    \bigg(\frac{1+\| b\|_{{W^{1}_\infty}}}{\min({\rho(b)},1)} \bigg)^5
    \|a-b\|_{{W^{1}_\infty}}.
  \end{aligned}
\end{equation}
Since the terms $ \|\partial_t w\|_{L^2(I,L^2)} $ and
$ \|w\|_{L^2(I,L^2)} $ are also bounded by the right side of
\eqref{eq-delta-w}, we arrive at
\begin{equation*}
  \begin{aligned} 
    \|w\|_Z & = \Big( \| \Delta w \|_{L^2(I,L^2)}^2 + \|\partial_t
    w\|_{L^2(I,L^2)}^2 + \|w\|_{L^2(I,L^2)}^2 \Big)^{1/2}
    \\
    & \leq C' \bigg(\frac{1+\| a\|_{{W^{1}_\infty}}}{\min(\rho(a),1)}
    \bigg)^5 \bigg(\frac{1+\| b\|_{{W^{1}_\infty}}}{\min(\rho(b),1)}
    \bigg)^5 \|a-b\|_{{W^{1}_\infty} }
  \end{aligned}
\end{equation*}
which is the claim.
\end{proof}
From Lemma \ref{lem-para-Z2}, by the same argument as in the proof of
\cite[Proposition 4.5]{HS2} we can verify that the solution map
$a\mapsto \Uu(a)$ from ${W^{1}_\infty}(\domain)\cap O$ to $Z$ is
holomorphic.  If we assume further that
$(\psi_j)_{j\in\N}\subseteq {W^{1}_\infty}(\domain)$ and with
$b_j:=\norm[{W^{1}_\infty}]{\psi_j}$, it holds $\bb\in\ell^1(\N)$ and
all the conditions in Theorem \ref{thm:bdX} are satisfied.  Therefore,
$u(\by)$ 
given by the formula \eqref{u(by)} 
is $(\bb,\xi,\delta,Z)$-holomorphic with appropriate $\xi$
and $\delta$.
\begin{remark} {\rm
  \label{rmk:CompCond}
  For $s>1$, let $$Z^s := \bigcap_{k=0}^s H^k(I,H^{2s-2k}(\domain))$$
  with the norm
$$
\|v\|_{Z^s} = \Bigg(\sum_{k=0}^{s} \bigg\|\frac{\rd^{k}v}{\rd t^{k}}
\bigg\|_{L^2(I,H^{2s-2k})}^2\Bigg)^{1/2}.
$$
Assume that $a\in W^{2s-1}_\infty(\domain)\cap O$.  At present we do
not know whether the solution map $a \mapsto \Uu(a)$ from
$W^{2s-1}_\infty(\domain) \cap O$ to $Z^s$ is holomorphic.  To obtain
the holomorphy of the solution map, we need a result similar to that
in Lemma \ref{lem-para-Z2}.  In order for this to hold, higher-order
regularity and compatibility of the data for equation
\eqref{eq-equation-w} is required, i.e,
$$
g_0=0\in V ,\quad g_1=h(0)-Lg_0 \in V,\ldots,
g_s=\frac{\rd^{s-1}h}{\rd t^{s-1}}(0)-Lg_{s-1}\in V,
$$
where
$$h=\nabla (a-b)\cdot\nabla \Uu(b) + (a-b)\Delta \Uu(b), \quad L= \partial_t \cdot  - \div\big(a\nabla \cdot \big).$$
See e.g.\ \cite[Theorem 27.2]{Wloka}.  It is known that without such
compatibility, the solution will develop spatial singularities at the
corners and edges of $\domain$, and temporal singularities as
$t\downarrow 0$; see e.g.\ \cite{KozRossPara}.

In general the compatibility condition does not hold when we only
assume that
$$
u_0\in H^{2s-1}(\domain)\cap V \ \ {\rm and} \ \ \frac{\rd ^k f}{\rd t
  ^k}\in L^2(I,H^{2s-2k-2}(\domain))
$$ 
for $k=0,\ldots,s-1$.
} \end{remark} 
%
\subsubsection{Linear elastostatics with log-Gaussian modulus of elasticity}
\label{sec:LinElast}
We illustrate the foregoing abstract setting of
Section~\ref{S:DefbxdHol} for another class of boundary value
problems.  In computational mechanics, one is interested in the
numerical approximation of deformations of elastic bodies. 
We refer to e.g.\ \cite{Truesdell} for an accessible exposition of the
mathematical foundations and assumptions.  
In \emph{linearized elastostatics} one is concerned with small (in a suitable sense, see
\cite{Truesdell} for details) deformations. 

We consider an elastic body occupying the domain
$\domain\subset \R^d$, $d=2,3$ (the physically relevant case naturally
is $d=3$, we include $d=2$ to cover the so-called model of
``plane-strain'' which is widely used in engineering, and has
governing equations with the same mathematical structure).  In the
linear theory, small deformations of the elastic body occupying
$\domain$, subject to, e.g., body forces $\boldf: \domain \to \R^d$
such as gravity are modeled in terms of the \emph{displacement field}
$\bu:\domain \to \R^d$, describing the displacement of a
\emph{material point} $\bx\in \domain$ (see \cite{Truesdell} for a
discussion of axiomatics related to this mathematical concept).
Importantly, unlike the scale model problem considered up to this
point, modeling now involves vector fields of data (e.g., $\boldf$)
and solution (i.e., $\bu$).

{ Governing equations for the mathematical model of linearly elastic
  deformation, subject to homogeneous Dirichlet boundary conditions on
  $\partial\domain$, read: to find $\bu:\domain\to \R^d$ such that \index{Linear elastostatics}
  \begin{equation}\label{eq:LinElPDE}
    \begin{array}{rcl}
      {\rm div}\bsigma[\bu] + \boldf &=& 0 \quad\mbox{in}\;\;\domain \;,
      \\
      \bu & = & 0  \quad \mbox{on} \;\;\partial\domain\;.
    \end{array}
  \end{equation}
  Here $\bsigma:\domain \to \R^{d\times d}_{{\rm sym}}$ is a symmetric
  matrix function, the so-called \emph{stress tensor}.  It depends on
  the displacement field $u$ via the so-called (linearized)
  \emph{strain tensor}
  $\bepsilon[\bu]:\D\to \R^{d\times d}_{{\rm sym}}$, which is given by
  \begin{equation}\label{eq:LinElStrain}
    \bepsilon[\bu] := \frac{1}{2}\left({\rm grad}\bu + ({\rm grad}\bu)^\top\right)\;,
    \;\;
    (\bepsilon[\bu])_{ij} := \frac{1}{2}(\partial_j u_i + \partial_i u_j)\;, i,j=1,...,d\;.
  \end{equation}
}

{ In the linearized theory, the tensors $\bsigma$ and $\bepsilon$ in
  \eqref{eq:LinElPDE}, \eqref{eq:LinElStrain} are related by the
  linear constitutive stress-strain relation (``Hooke's law'')
  \begin{equation}\label{eq:Hooke}
    \bsigma = \ttA \bepsilon \;.
  \end{equation}
  In \eqref{eq:Hooke}, $\ttA$ is a fourth order tensor field, i.e.
 $$\ttA = \{ \ttA_{ijkl} : i,j,k,l=1,...,d\},$$ 
 with certain symmetries that must hold among its $d^4$ components
 independent of the particular material constituting the elastic body
 (see, e.g., \cite{Truesdell} for details).  Thus, \eqref{eq:Hooke}
 reads in components as $\sigma_{ij} = \ttA_{ijkl} \epsilon_{kl}$ with
 summation over repeated indices implied.  Let us now fix $d=3$.
 Symmetry implies that $\epsilon$ and $\sigma$ are characterized by
 $6$ components.  If, in addition, the material constituting the
 elastic body is \emph{isotropic}, the tensor $\ttA$ can in fact be
 characterized by only two independent coefficient functions. We adopt
 here the \emph{Poisson ratio}, denoted $\nu$, and the modulus of
 elasticity $\ttE$.  With these two parameters, the stress-strain law
 \eqref{eq:Hooke} can be expressed in the component form
 \begin{equation}\label{eq:HookeComp}
   \left(\begin{array}{c} \sigma_{11} \\ \sigma_{22} \\ \sigma_{33} \\ \sigma_{12} \\ \sigma_{13} \\ \sigma_{23} \end{array} \right)
   =
   \frac{\ttE}{(1+\nu)(1-2\nu)} 
   \left(\begin{array}{cccccc} 
           1-\nu & \nu & \nu & 0 & 0 & 0 
           \\
           \nu & 1-\nu & \nu & 0 & 0 & 0
           \\
           \nu & \nu & 1-\nu & 0 & 0 & 0
           \\
           0 & 0 & 0 & 1-2\nu & 0 & 0 
           \\
           0 & 0 & 0 & 0& 1-2\nu & 0
           \\
           0 & 0 & 0 & 0 & 0 & 1-2\nu 
         \end{array}
       \right) 
       \left(\begin{array}{c} \epsilon_{11} \\ \epsilon_{22} \\ \epsilon_{33} \\ 
               \epsilon_{12} \\ \epsilon_{13} \\ \epsilon_{23} 
             \end{array} \right)
           \;.
         \end{equation}
         We see from \eqref{eq:HookeComp} that for isotropic elastic
         materials, the tensor $\ttA$ is proportional to the modulus
         $\ttE > 0$, with the Poisson ratio $\nu\in [0,1/2)$.  We
         remark that for common materials, $\nu\uparrow 1/2$ arises in
         the so-called \emph{incompressible limit}. In that case,
         \eqref{eq:LinElPDE} can be described by the Stokes equations.
       }

       { With the constitutive law \eqref{eq:Hooke}, we may cast the
         governing equation \eqref{eq:LinElPDE} into the so-called
         ``primal'', or ``displacement-formulation'': find
         $\bu:\domain\to \R^d$ such that
         \begin{equation}\label{eq:LinElDispl}
           -{\rm div}(\ttA\bepsilon[\bu]) = \boldsymbol{f}  \quad \mbox{in}\;\;\domain\;,
           \qquad \bu|_{\partial\domain} = 0\;.
         \end{equation}
         This form is structurally identical to the scalar diffusion
         problem \eqref{PDE}.  }

       { Accordingly, we fix $\nu\in [0,1/2)$ and model uncertainty in
         the elastic modulus $\ttE > 0$ in \eqref{eq:HookeComp} by a
         log-Gaussian random field
         \begin{equation}\label{eq:YoungGRF}
           \ttE(\by)(\bx) := \exp(b(\by))(\bx)\;, \quad \bx\in \domain\;, \;\; \by\in U\;.
         \end{equation}
         Here, $b(\by)$ is a Gaussian series representation of the GRF
         $b(Y(\omega))$ as discussed in Section~\ref{S:GSer}.  The
         log-Gaussian ansatz $\ttE = \exp(b)$ ensures
$$
E_{\min}(\by) := {\rm ess}\inf_{\bx\in \domain} \ttE(\by)(\bx) > 0
\qquad \mbox{$\gamma$-a.e.} \ \by\in U\;,
$$
i.e., the $\gamma$-almost sure positivity of (realizations of) the
elastic modulus $\ttE$.  Denoting the $3\times 3$ matrix relating the
stress and strain components in \eqref{eq:HookeComp} also by $\ttA$
(this slight abuse of notation should, however, not cause confusion in
the following), we record that for $0\leq \nu < 1/2$, the matrix
$\ttA$ is invertible:
\begin{equation}\label{eq:Ainv}
  \ttA^{-1} 
  =
  \frac{1}{\ttE} 
  \left(\begin{array}{cccccc}
          1 & -\nu & -\nu & 0&0&0 
          \\
          -\nu & 1 & -\nu & 0&0&0 
          \\
          -\nu & -\nu & 1  & 0&0&0 
          \\
          0& 0& 0& 1+\nu & 0 & 0 
          \\
          0& 0& 0& 0 & 1+\nu & 0 
          \\
          0& 0& 0& 0 & 0 & 1+\nu 
        \end{array}
      \right)
      \;.
    \end{equation}
    It readily follows from this explicit expression that due
    to $$\ttE^{-1}(\by)(\bx) = \exp(-b(\by)(\bx)),$$ by the
    Gerschgorin theorem invertibility holds for $\gamma$-a.e.
    $\by\in U$.  Also, the components of $\ttA^{-1}$ are GRFs (which
    are, however, fully correlated for deterministic $\nu$).  }

  { Occasionally, instead of the constants $\ttE$ and $\nu$, one finds
    the (equivalent) so-called \emph{Lam\'{e}-constants} $\lambda$,
    $\mu$.  They are related to $\ttE$ and $\nu$ by
    \begin{equation}\label{eq:LameEnu}
      \lambda = \frac{\ttE \nu}{(1+\nu)(1-2\nu)}\;,\quad 
      \mu     = \frac{\ttE}{2(1+\nu)}\;.
    \end{equation}
    For GRF models \eqref{eq:YoungGRF} of $\ttE$, \eqref{eq:LameEnu}
    shows that for each fixed $\nu \in (0,1/2)$, also the
    Lam\'{e}-constants are GRFs \emph{which are fully correlated}.
    This implies, in particular, that ``large'' realizations of the
    GRF \eqref{eq:YoungGRF} do not cause so-called ``volume locking''
    in the equilibrium equation \eqref{eq:LinElPDE}: this effect is
    related to the elastic material described by the constitutive
    equation \eqref{eq:Hooke} being nearly
    incompressible. Incompressibility here arises as either
    $\nu\uparrow 1/2$ at fixed $\ttE$ or, equivalently, as
    $\lambda\to \infty$ at fixed $\mu$.  }

  { Parametric weak solutions of \eqref{eq:LinElDispl} with
    \eqref{eq:YoungGRF} are within the scope of the abstract theory
    developed up to this point.  To see this, we provide a variational
    formulation of \eqref{eq:LinElDispl}.  Assuming for convenience
    homogeneous Dirichlet boundary conditions, we multiply
    \eqref{eq:LinElDispl} by a test displacement field $\bv\in X:=V^d$
    with $V:=H_0^1(\D)$, and integrate by parts, to obtain the weak
    formulation: find $\bu\in X$ such that, for all $\bv\in X$ holds
    (in the matrix-vector notation \eqref{eq:HookeComp})
    \begin{equation}\label{eq:LinElVar}
      \int_{\domain} \bepsilon[\bv] \cdot \ttA \bepsilon [\bu] \rd \bx 
      =
      2\mu (\bepsilon[\bu],\bepsilon[\bv]) + \lambda(\div \bu, \div \bv) = (\boldf,\bv) \;.
    \end{equation}
    The variational form \eqref{eq:LinElVar} suggests that, as
    $\lambda\to\infty$ for fixed $\mu$, the ``volume-preservation''
    constraint $\| \div \bu \|_{L^2} = 0$ is imposed for $\bv=\bu$ in
    \eqref{eq:LinElVar}.  }

  Unique solvability of \eqref{eq:LinElVar} follows upon verifying
  coercivity of the corresponding bilinear form on the left-hand side
  of \eqref{eq:LinElVar}.  It follows from \eqref{eq:HookeComp} and
  \eqref{eq:Ainv} that
  \begin{equation*}\label{eq:bileps}
    \forall \bv\in H^1(\domain)^d: \quad 
    \ttE c_{\min}(\nu) \| \bepsilon[\bv] \|^2_{L^2} 
    \leq 
    \int_{\domain} \bepsilon[\bv]  \cdot \ttA \bepsilon[\bv] \rd \bx
    \leq 
    \ttE c_{\max}(\nu) \| \bepsilon[\bv] \|^2_{L^2}
    \;.
  \end{equation*}
  Here, the constants $c_{\min}, c_{\max}$ are positive and bounded
  for $0<\nu<1/2$ and independent of $\ttE$.

  For the log-Gaussian model \eqref{eq:YoungGRF} of the elastic
  modulus $\ttE$, the relations \eqref{eq:LameEnu} show in particular,
  that \emph{the volume-locking effect arises as in the deterministic
    setting only if $\nu \simeq 1/2$, independent of the realization
    of $\ttE(\by)$}.  Let us consider well-posedness of the
  variational formulation \eqref{eq:LinElVar}, for log-Gaussian,
  parametric elastic modulus $\ttE(\by)$ as in \eqref{eq:YoungGRF}.
  To this end, with $\ttA_1$ denoting the matrix $\ttA$ in
  \eqref{eq:HookeComp} with $\ttE=1$, we introduce in
  \eqref{eq:LinElVar} the parametric bilinear forms
  \begin{equation*}\label{eq:ParmElas}
    b(\bu,\bv;\by) 
    := 
    \ttE(\by) \int_{\domain}  \bepsilon[\bv] \cdot \ttA_1 \bepsilon[\bu] \rd \bx 
    = 
    \frac{\ttE(\by)}{1+\nu} 
    \left( 
      (\bepsilon[\bu],\bepsilon[\bv]) + \frac{\nu}{1-2\nu} (\div \bu, \div \bv)
    \right)
    \;.
  \end{equation*}
  Let us verify continuity and coercivity of the parametric bilinear
  forms
  \begin{equation} \label{parametric forms} \{ b(\cdot,\cdot ;\by):X\times X\to \R: \by\in U\},
  \end{equation}
  where we recall that  $U:= \RRi$.  
  With $\ttA_1$ as defined above, we write
  for arbitrary $\bv\in X = H^1_0(\domain)^d$, $d=2,3$, and for all
  $\by\in U_0\subset U$ where the set $U_0$ is as in \eqref{eq:U0},
$$
\begin{array}{rcl}
  b(\bv,\bv;\by) 
  & = & \displaystyle
        \int_{\domain} \bepsilon[\bv] \cdot (\ttA \bepsilon[\bv]) \rd\bx 
        =
        \int_{\domain} E(\by)  \left(\bepsilon[\bv] \cdot (\ttA_1 \bepsilon[\bv]) \right)\rd \bx 
  \\
  & \geq & \displaystyle
           c(\nu) \int_{\domain} E(\by)  \| \bepsilon[\bv] \|_2^2 \rd \bx 
  \\
  & \geq & \displaystyle
           c(\nu) \exp(-\| b(\by) \|_{L^\infty}) \int_{\domain} \| \bepsilon[\bv] \|_2^2 \rd \bx
  \\
  & \geq & \displaystyle
           \frac{c(\nu)}{2} a_{\min}(\by) | \bv |_{H^1}^2 
  \\
  & \geq & \displaystyle
           C_P \frac{c(\nu)}{2} a_{\min}(\by) \| \bv \|_{H^1}^2
           \;.
\end{array}
$$
Here, in the last two steps we employed the first Korn's inequality,
and the Poincar\'{e} inequality, respectively.  The lower bound
$E(\by)\geq \exp(-\| b(\by) \|_{L^\infty })$ is identical to
\eqref{eq:Norms} in the scalar diffusion problem.

In a similar fashion, continuity of the bilinear forms
\eqref{parametric forms} may be established: there
exists a constant $c'(\nu)>0$ such that 
$$
\forall \bu,\bv\in X,\; \forall \by \in U_0 : \quad | b(\bu,\bv;\by) |
\leq c'(\nu)\exp(\| b(\by) \|_{L^\infty}) \| \bu \|_{H^1} \| \bv
\|_{H^1}.
$$
With continuity and coercivity of the parametric forms
\eqref{parametric forms} verified for $\by\in U_0$, the Lax-Milgram
lemma ensures for given $\boldf\in L^2(\domain)^d$ the existence of
the parametric solution family
\begin{equation}\label{eq:ParSolElas}
  \{ \bu(\by)  \in X : b(\bu,\bv;\by) = (\boldf,\bv) \ \forall \bv\in X , \by \in U_0 \}\;.
\end{equation}
Similar to the scalar case discussed in Proposition~\ref{prop:Meas1}, 
    the following result on almost everywhere existence and measurability holds.
\begin{proposition}\label{prop:ParSolElas}
  Under Assumption \ref{ass:Ass1}, $\gamma(U_0) = 1$.
  For all $k\in \NN$ there holds, with $\EE(\cdot)$ denoting
  expectation with respect to $\gamma$,
$$
\EE\left( \exp(k\| b(\cdot) \|_{L^\infty}) \right) < \infty \;.
$$
The parametric solution family \eqref{eq:ParSolElas} of the parametric
elliptic boundary value problem \eqref{eq:LinElVar} with log-Gaussian
modulus $E(\by)$ as in \eqref{eq:YoungGRF} is in $L^k(U,V;\gamma)$ for
every finite $k\in \NN$.
\end{proposition}
For the parametric solution family \eqref{eq:ParSolElas}, 
analytic continuations into complex parameter domains, and parametric
regularity results may be developed in analogy to the development in
Sections \ref{sec:HsReg} and \ref{sec:KondrReg}.  
The key result for
bootstrapping to higher order regularity is, in the case of smooth
boundaries $\partial\domain$, classical elliptic regularity for
linear, Agmon-Douglis-Nirenberg elliptic systems which comprise
\eqref{eq:LinElDispl}.  In the polygonal (for $d=2$) or polyhedral
($d=3$) case, weighted regularity shifts in Kondrat'ev type spaces are
available in \cite[Theorem 5.2]{GuoBabElast} (for $d=2$) and in
\cite{CDN12} (for both, $d=2,3$).
%
\subsubsection{Maxwell equations with log-Gaussian permittivity}
\index{Maxwell equation}
\label{sec:Maxwell}
Similar models are available for time-harmonic, electromagnetic waves
in dielectric media with uncertain conductivity. We refer to
\cite{logNMax2018}, where log-Gaussian models are employed.  There,
also the parametric regularity analysis of the parametric electric and
magnetic fields is discussed, albeit by real-variable methods.  The
setting in \cite{logNMax2018} is, however, so that the presently
developed, complex variable methods can be brought to bear on it.  We
refrain from developing the details.
%

\subsubsection{Linear parametric elliptic systems and transmission problems}
\index{Elliptic systems}
\index{Transmission problem}
\label{sec:EllipticSystem}
%
In Section~\ref{sec:KondrAn}, 
Theorem~\ref{thm:bacuta} we obtained parameter-explicit elliptic regularity shifts
for a scalar, linear second order parametric elliptic divergence-form PDE in
polygonal domain $\domain \subset \R^2$. 
A key feature of these estimates in the subsequent
analysis of sparsity of gpc expansions was the \emph{polynomial dependence on the 
parameter in the bounds on parametric solutions in corner-weighted 
Sobolev spaces of Kondrat'ev type}. 
Such a-priori bounds are not limited to the particular setting considered in 
Section~\ref{sec:KondrAn}, but hold for rather general, linear elliptic PDEs 
in smooth domains $\domain \subset \R^d$ of space dimension $d\geq 2$, 
with parametric differential and boundary operators of general integer order.
In particular, for example, for linear, anisotropic elastostatics in $\R^3$,
for parametric fourth order PDEs in $\R^2$ which arise in dimensionally reduced
models of elastic continua (plates, shells, etc.). We refer to \cite{KLMN22}
for statements of results and proofs.

In the results in Section~\ref{sec:KondrAn}, 
we admitted inhomogeneous coefficients which are regular in all of $\domain$. 
In many applications, 
\emph{transmission problems} with parametric, inhomogeneous
coefficients with are piecewise regular on a given, fixed (i.e. non-parametric) 
partition of $\domain$ is of interest. Also in these cases, corresponding
a-priori estimates of parametric solution families with norm bounds
which are polynomial with respect to the parameters hold. 
We refer to \cite{NistLabrPolyBds23} for such results, 
in smooth domains $\domain$, with smooth interfaces.

\newpage
\section{Parametric posterior analyticity and sparsity in BIPs}
\label{sec:BIP}
We have investigated the parametric analyticity of the \emph{forward
  solution maps} of linear PDEs with uncertain parametric inputs which
typically arise from GRF models for these inputs.  We have also provided an
analysis of sparsity in the Wiener-Hermite PC
expansion of the corresponding parametric solution families.

We now explore the notion of parametric holomorphy in the context of
BIPs for linear PDEs.  
For these PDEs we adopt the Bayesian setting as
outlined, e.g.,  in \cite{DashtiStuart17} and the references there.
This Bayesian setting is briefly recapitulated in Section
\ref{sec:BIPFrmWelPsd}.  
With a suitable version of Bayes' theorem,
the main result is a (short) proof of parametric
$(\bb,\xi,\delta,\C)$-holomorphy of the Bayesian posterior density for
unbounded parameter ranges.  
This implies
sparsity of the coefficients in
Wiener-Hermite PC expansions of the Bayesian posterior density, which
can be leveraged to obtain higher-order approximation rates that are
free from the curse of dimensionality for various deterministic
approximation methods of the Bayesian expectations, for several
classes of function space priors modelled by product measures on the
parameter sequences $\by$.  In particular, the construction of
Gaussian priors described in Section \ref{S:GMSepHS} is applicable. Concerning related previous works, we remark the following.
In \cite{CSAMS2011} holomorphy for a bounded parameter
domain (in connection with uniform prior measure) has been addressed
by complex variable arguments in the same fashion. 
In \cite{RSAMSAT_Bip2017}, MC and QMC integration has been analyzed by
real-variable arguments for such Gaussian priors.  
In \cite{HKS19_845}, corresponding results have been obtained 
also for so-called \emph{Besov priors}, 
again by real-variable arguments for the parametric posterior.  
Since the presently developed, quantified
parametric holomorphy results are independent of the particular
measure placed upon the unbounded parameter domain $\RR^\infty$. 
The sparsity
and approximation rate bounds for the parametric deterministic
posterior densities will imply approximate rate bounds also for prior
constructions beyond the Gaussian ones.
\subsection{Formulation and well-posedness}
\label{sec:BIPFrmWelPsd}
With ${\data}$ and $X$ denoting separable Banach and Hilbert spaces
over $\C$, respectively, we consider a forward solution map
$\Uu: {\data}\to X$ and an observation map $\bcalO:X\to \RR^m$. In the
context of the previous sections, $\Uu$ could denote again the map
which associates with a diffusion coefficient
$a\in {\data}:=L^\infty(\D;\C)$ the solution
$\Uu(a) \in X:=H_0^1(\D;\C)$ of  the equation \eqref{eq:elliptic} below. 
We assume the map $\Uu$ to be Borel measurable.

The inverse problem consists in determining the (expected value of an)
uncertain input datum $a\in {\data}$ from noisy observation data
$\bobs\in \RR^m$.  Here, the observation noise $\boldeta\in \RR^m$ is
assumed additive centered Gaussian, i.e., the observation data $\bobs$
for input $a$ is \index{Bayesian inverse problem}
\begin{equation*}\label{eq:ObsNois}
  \bobs = \bcalO\circ\Uu(a) +  \boldeta\;,
\end{equation*}
where $\boldeta \sim \calN(0,\bGamma)$.  We assume the observation
noise covariance $ \bGamma\in\R^{m\times m}$ is symmetric positive
definite.

In the so-called Bayesian setting of the inverse problem, one assumes
that the uncertain input $a$ is modelled as RV which is
distributed according to a prior measure $\pi_0$ on ${\data}$. Then,
under suitable conditions, which are made precise in Theorem
\ref{thm:bdHol} below, the posterior distribution $\pi(\cdot|\bobs)$
on the conditioned {RV} $\Uu|\bobs$ is absolutely
continuous w.r.t.\ the prior measure $\pi_0$ on ${\data}$ and there
holds Bayes' theorem in the form \index{Bayes' theorem}
\begin{equation}\label{eq:Bayes}
  \frac{\rd\pi(\cdot|\bobs)}{\rd\pi_0}(a) = \frac{1}{Z} \Theta(a).
\end{equation}
In \eqref{eq:Bayes}, the posterior density $\Theta$ and the
normalization constant $Z$ are given by
\begin{equation}\label{eq:BayesDens}
  \Theta(a) = \exp(-\Phi(\bobs;a)),\qquad
  \Phi(\bobs;a) = \frac{1}{2} \norm[2]{\bGamma^{-1/2}(\bobs - \bcalO(\Uu(a)))}^2,\qquad
  Z = \EE_{\pi_0}[\Theta(\cdot)]
  \;.
\end{equation}
Additional conditions ensure that the posterior measure
$\pi(\cdot|\bobs)$ is well-defined and that \eqref{eq:Bayes} holds
according to the following result from \cite{DashtiStuart17}.
\begin{proposition}\label{prop:Bayes}
  Assume that $\bcalO\circ\Uu: {\data}\to \RR^m$ is continuous and
  that $\pi_0({\data}) = 1$. Then the posterior $\pi(\cdot|\bobs)$ is
  absolutely continuous with respect to $\pi_0$, and
  \eqref{eq:BayesDens} holds.
\end{proposition}
The condition $\pi_0({\data}) = 1$ can in fact be weakened to
$\pi_0({\data}) > 0$ (e.g.\ \cite[Theorem 3.4]{DashtiStuart17}).

The solution of the BIP amounts to the evaluation of the posterior
expectation $\EE_{\mu^\bobs}[\cdot]$ of a continuous linear map
$\phi:X\to Q$ of the map $\Uu(a)$, where $Q$ 
is a suitable Hilbert space over $\C$. 
Solving the Bayesian inverse problem is thus closely
related to the numerical approximation of the posterior expectation
\begin{equation*}\label{eq:BIPPost}
  \EE_{\pi(\cdot|\bobs)} [\phi(\Uu(\cdot))] \in Q.
\end{equation*}
For computational purposes, and to facilitate Wiener-Hermite PC
approximation of the density $\Theta$ in \eqref{eq:Bayes}, one
parametrizes the input data $a=a(\by)\in \data$ by a Gaussian series
as discussed in Section \ref{S:GSer}.  Inserting into $\Theta(a)$ in
\eqref{eq:Bayes}, \eqref{eq:BayesDens} this results in a
countably-parametric density $U\ni\by\mapsto \Theta(a(\by))$, for
$\by\in U$, 
and the Gaussian reference measure $\pi_0$ on
$\data$ in \eqref{eq:Bayes} is pushed forward into 
a countable product $\gamma$ of the sequence of Gaussian measures 
$\{\gamma_{1,n}\}_{n \in \NN}$ on $\RR$: 
using \eqref{eq:Bayes} and choosing a
Gaussian prior (e.g.\ \cite[Section 2.4]{DashtiStuart17} or
\cite{HKS19_845,LSS2009})
$$
\pi_0 = \gamma = \bigotimes_{j\in \NN}  \gamma_{1,n}
$$ on
$U$ (see Example \ref{ex:prod-measR^infty}), the Bayesian estimate, i.e., the posterior expectation, can then
be written as a (countably) iterated integral
\cite{CSAMS2011,DashtiStuart17,RSAMSAT_Bip2017} 
with respect to the product GM $\gamma$, i.e.
\begin{equation}\label{eq:BIPSparse}
  \EE_{\pi(\cdot|\bobs)}[\phi(\Uu(a(\cdot)))] = \frac{1}{Z}
  \int_{U} \phi(\Uu(a(\by))) \Theta(a(\by)) \dd\gamma(\by) \in Q,
  \quad 
  Z = \int_{U} \Theta(a(\by)) \dd\gamma(\by) \in \RR.
\end{equation}
The parametric density $U\to \RR$ in \eqref{eq:BIPSparse} which arises
in Bayesian PDE inversion under Gaussian prior and also under more
general, so-called Besov prior measures on $U$, see, 
e.g.\ \cite[Section 2.3]{DashtiStuart17}, \cite{HKS19_845,LSS2009}.  
The parametric density
$$
\by \mapsto \phi(\Uu(a(\by))) \Theta(a(\by)) \;,
$$
inherits sparsity from the forward map $\by \mapsto \Uu(a(\by))$,
whose sparsity is expressed as before in terms of  
$\ell^p$-summability and weighted $\ell^2$-summability 
of Wiener-Hermite PC expansion coefficients. 
We employ the
parametric holomorphy of the forward map $a\mapsto \Uu(a)$ to quantify
the sparsity of the parametric posterior densities
$\by\mapsto \Theta(a(\by))$ and $\by\mapsto \phi(\Uu(a(\by))) \Theta(a(\by))$ 
in \eqref{eq:BIPSparse}.
\subsection{Posterior parametric holomorphy}
\label{sec:BIPParmHol}
With a Gaussian series in the data space $E$, 
for the resulting parametric data-to-solution map
$$u:U\to X: \by\mapsto \Uu(a(\by)),$$
we now prove that under certain conditions both, the corresponding
parametric posterior density
\begin{equation}\label{eq:TheParm}
  \by\mapsto \exp\left( -(\bobs-\bcalO(u(\by)))^\top \bGamma^{-1} (\bobs-\bcalO(u(\by))) \right)
\end{equation}
in \eqref{eq:BayesDens}, and the integrand
\begin{equation}\label{eq:integrand}
  \by\mapsto  \phi(u(\by))
  \exp\left( -(\bobs-\bcalO(u(\by)))^\top \bGamma^{-1} (\bobs-\bcalO(u(\by))) \right)
\end{equation}
in \eqref{eq:BIPSparse} are $(\bb,\xi,\delta,\C)$-holomorphic and
$(\bb,\xi,\delta,Q)$-holomorphic, respectively.
\begin{theorem}\label{thm:bdHol}
  Let $r>0$.  Assume that the map $u:U\to X$ is
  $(\bb,\xi,\delta,X)$-holomorphic with constant functions
  $\varphi_N\equiv r$, $N\in\N$, in Definition~\ref{def:bdXHol}.
  Let the observation noise covariance matrix
  $\bGamma\in\R^{m\times m}$ be symmetric positive definite.

  Then, for any bounded linear quantity of interest $\phi\in L(X,Q)$,
  and for any observable $\bcalO\in (X')^m$ with arbitrary, finite
  $m$, the function in \eqref{eq:TheParm} is
  $(\bb,\xi,\delta,\C)$-holomorphic and the function in
  \eqref{eq:integrand} is $(\bb,\xi,\delta,Q)$-holomorphic.
\end{theorem}
\begin{proof}
  We only show the statement for the parametric integrand in
  \eqref{eq:integrand}, as the argument for the posterior density in
  \eqref{eq:TheParm} is completely analogous.

  Consider the map
  \begin{equation*}
    \Xi: \set{v\in X}{\norm[X]{v}\le r} \to Q : 
    v \mapsto \phi(v)\exp(-(\bobs-\bcalO(v))\bGamma^{-1}(\bobs-\bcalO(v))).
  \end{equation*}
  This function is well-defined.  We have
  $|\bcalO(v)|\le \norm[X']{\bcalO} r$ and
  $|\phi(v)|\le \norm[L(X;Q)]{\phi} r$ for all $v\in X$ with
  $\norm[X]{v}\le r$. Since $\exp:\C\to\C$ is Lipschitz continuous on
  compact subsets of $\C$ and since $\phi\in L(X;Q)$ is bounded linear
  map (and thus Lipschitz continuous), we find that
  \begin{equation*}
    \sup_{\norm[X]{v}\le r}\norm[Q]{\Xi(v)}=:\tilde r<\infty
  \end{equation*}
  and that $$\Xi:\set{v\in X}{\norm[X]{v}\le r}\to\C$$ is Lipschitz
  continuous with some Lipschitz constant $L>0$.

  Let us recall that the $(\bb,\xi,\delta,X)$-holomorphy of
  $u:U\to X$, implies the existence of (continuous) functions
  $u_N \in L^2(\R^N,X;\gamma_N)$ such that with
  $\tilde u_N(\by)=u_N(y_1,\dots,y_N)$ it holds
  $\lim_{N\to\infty}\tilde u_N= u$ in the sense of
  $L^2(U,X;\gamma)$. Furthermore, if
  $$\sum_{j=1}^N b_j\varrho_j\le\delta$$
  (i.e.~$\bvarrho=(\varrho_j)_{j=1}^N$ is $(\bb,\xi)$-admissible in
  the sense of Definition~\ref{def:bdXHol}), then $u_N$ allows a
  holomorphic extension $$u_N:\Ss_\bvarrho\to X$$ such that for all
  $\by\in \R^N$
  \begin{equation}\label{eq:unbound}
    \sup_{\bz\in\Bb_\bvarrho}\norm[X]{u_N(\by+\bz)}\le \varphi_N(\by)=r\qquad\forall \by\in\R^N,
  \end{equation}
  see \eqref{eq:Sjrho} for the definition of $\Ss_\bvarrho$ and
  $\Bb_\bvarrho$.

  We want to show that $f(\by):=\Xi(u(\by))$ is well-defined in
  $L^2(U,Q;\gamma)$, and given as the limit of the functions
  $$\tilde f_N(\by)=f_N((y_j)_{j=1}^N)$$ for all $\by\in U$ and
  $N\in\N$, where $$f_N((y_j)_{j=1}^N)=\Xi(u_N((y_j)_{j=1}^N).$$ Note
  at first that $f_N:\R^N\to Q$ is well-defined. In the case
  $$\sum_{j=1}^Nb_j\varrho_j\le\delta,$$ $f_N$ allows a holomorphic
  extension $f_N:\Ss_\bvarrho\to X$ given through $\Xi\circ u_N$.
  Using \eqref{eq:unbound}, this extension satisfies for any $N\in\N$
  and any $(\bb,\xi)$-admissible $\varrho\in (0,\infty)^N$
  \begin{equation*}
    \sup_{\bz\in\Bb_\bvarrho}|f_N(\by+\bz)|
    \le \sup_{\norm[X]{v}\le r}|\Xi(v)|=\tilde r\qquad\forall \by\in \R^N.
  \end{equation*}
  This shows assumptions \ref{item:hol}-\ref{item:varphi} of
  Definition~\ref{def:bdXHol} for $f_N:\R^N\to Q$.

  Finally we show assumption \ref{item:vN} of
  Definition~\ref{def:bdXHol}.  By assumption it holds
  $\lim_{N\to\infty}\tilde u_N= u$ in the sense of
  $L^2(U,X;\gamma)$. Thus for $f=\Xi\circ u$ and with
  $f_N=\Xi\circ u_N$
  \begin{align*}
    \int_U\norm[Q]{f(\by)-f_N(\by)}^2 \dd\gamma(\by)
    & = \int_U\norm[Q]{\Xi(u(\by))-\Xi(u_N(\by))}^2\dd\gamma(\by)
    \\
    &
      \le L^2 \int_U\norm[X]{u(\by)-u_N(\by)}^2\dd\gamma(\by),
  \end{align*}
  which tends to $0$ as $N\to\infty$. Here we used that $L$ is a
  Lipschitz constant of $\Xi$.
\end{proof}

Let us now discuss which functions satisfy the requirements of
Theorem~\ref{thm:bdHol}.  Additional to
$(\bb,\xi,\delta,X)$-holomorphy, we had to assume boundedness of the
holomorphic extensions in Definition~\ref{def:bdXHol}. For functions
of the type as in Theorem \ref{thm:bdX}
$
 u(\by)=\lim_{N\to\infty}\Uu\left(\exp\left(\sum_{j=1}^N
    y_j\psi_j\right)\right), 
 $ 
 the following result gives sufficient
conditions such that the assumptions of Theorem~\ref{thm:bdHol} are
satisfied for the forward map.

\begin{corollary}\label{cor:norma2}
  Assume that $\Uu:O\to X$ and $(\psi_j)_{j\in\N}\subset {\data}$
  satisfy Assumptions \ref{item:uhol}, \ref{item:loclip} and
  \ref{item:psi} of Theorem \ref{thm:bdX} and additionally for some
  $r>0$
  \begin{enumerate}
    \setcounter{enumi}{1}
  \item\label{item:norma2} $\norm[X]{\Uu(a)}\le r$ for all $a\in O$.
  \end{enumerate}
    
  Then
    $$
    u(\by)=\lim_{N\to\infty}\Uu\left(\exp\left(\sum_{j=1}^N
        y_j\psi_j\right)\right) \in L^2(U,X;\gamma)
    $$
    is $(\bb,\xi,\delta,X)$-holomorphic with constant functions
    $\varphi_N\equiv r$, $N\in\N$, in Definition~\ref{def:bdXHol}.
  \end{corollary}
  \begin{proof}
    By Theorem \ref{thm:bdX}, $u$ is $(\bb,\xi,\delta,X)$-holomorphic.
    Recalling the construction of $\varphi_N:\R^N\to\R$ in Step 3 of
    the proof of Theorem \ref{thm:bdX}, we observe that $\varphi_N$
    can be chosen as $\varphi_N\equiv r$.
  \end{proof}
  \subsection{Example: parametric diffusion coefficient}
  \label{sec:ExplBip}
  We revisit the example of the diffusion equation with parametric
  log-Gaussian coefficient as introduced in  Section~\ref{S:ParCoef} and
  used in Section~\ref{sec:pdc}.  
  With the Lipschitz
  continuity of the data-to-solution map established in Section
  \ref{sec:pdc}, we verify the well-posedness of the corresponding
  BIP.

  We fix the dimension $d\in\N$ of the physical domain
  $\D\subseteq\R^d$, being a bounded Lipschitz domain, and choose
  $\data=L^\infty(\D;\C)$ and $X=H_0^1(\D;\C)$.  We assume that
  $f\in X'$ and $a_0\in \data$ with
$$
\rho(a_0)>0.
$$
For
$$a\in O:=\set{a\in \data}{\rho(a) >0},$$ 
let $\Uu(a)$ be the solution to the equation
\begin{equation}\label{eq:elliptic}
  \begin{aligned}
    - \div((a_0+a) \nabla \Uu(a)) = f \text{ in }\D, \;\; \Uu(a) =0
    \text{ on }\partial\D,
  \end{aligned}
\end{equation}
for some fixed $f\in X'$.

Due to
$$\rho(a_0+a)\ge \rho(a_0)>0,$$
for every $a\in O$, as in \eqref{eq:apriori} we find that $\Uu(a)$ is
well-defined and it holds
\begin{equation*}
  \norm[X]{\Uu(a)}\le \frac{\norm[X']{f}}{{\rho(a_0)}}=:r\qquad\forall a\in O.
\end{equation*}
This shows assumption \ref{item:norma2} in Corollary~\ref{cor:norma2}.
Slightly adjusting the arguments in Section~\ref{sec:pdc} one observes
that $\Uu:O\to X$ satisfies assumptions \ref{item:uhol} and
\ref{item:loclip} in Theorem \ref{thm:bdX}.  Fix a representation
system $(\psi_j)_{j\in\N}\subseteq V$ such that with
$b_j:=\norm[{\data}]{\psi_j}$ it holds $(b_j)_{j\in\N}\in\ell^1(\N)$.
Then Corollary~\ref{cor:norma2} implies that the forward map
$$u(\by)=\lim_{N\to\infty}\Uu\bigg(\exp\Big(\sum_{j=1}^Ny_j\psi_j\Big)\bigg)$$ satisfies
the assumptions of Theorem~\ref{thm:bdHol}. Theorem~\ref{thm:bdHol} in
turn implies that the posterior density for this model is
$(\bb,\xi,\delta,X)$-holomorphic. We shall prove in
Section~\ref{sec:StochColl} that sparse-grid quadratures can be
constructed which achieve higher order convergence for the integrands
in \eqref{eq:TheParm} and \eqref{eq:integrand}, with the convergence
rate being a decreasing function of $p\in (0,4/5)$ such that
$\bb\in\ell^p(\N)$, see Theorem~\ref{thm:quad}. Furthermore,
Theorem~\ref{thm:bdXSum} implies a certain sparsity for
the family of Wiener-Hermite PC expansion coefficients 
of the parametric maps in \eqref{eq:TheParm} and \eqref{eq:integrand}.
\newpage
\
\newpage

\section{Smolyak sparse-grid interpolation and quadrature}
\label{sec:StochColl}
Theorem~\ref{thm:bdXSum} shows that if $v$ is $(\bb,\xi,\delta,X)$-holomorphic for some $\bb\in \ell^p(\N)$ and some $p\in (0,1)$, then  $(\norm[X]{v_\bnu})_{\bnu\in\CF}\in \ell^{2p/(2-p)}(\CF).$ In Remark \ref{rmk:bestN}, based on this summability of the 
Wiener-Hermite PC expansion coefficients, 
we derived the convergence rate of best $n$-term approximation 
as in \eqref{eq:bestNgeneral}. 
This approximation is not linear 
since the approximant is taken accordingly to the $N$ largest terms $\norm[X]{v_\bnu}$. 
To construct a linear approximation which gives the same convergence rate 
it is suitable to use the stronger weighted $\ell^2$-summability result \eqref{eq:general} 
in Theorem~\ref{thm:bdXSum}. 

In Theorem \ref{thm:bdXSum} of Section~\ref{sec:SumHolSol}, 
we have obtained the weighted $\ell^2$-summability 
\begin{equation}\label{eq:weightedsum1}
\sum_{\bnu\in\CF}\beta_\bnu(r,\bvarrho) \norm[X]{u_\bnu}^2<\infty
\ \ \ \text{with} \ \ \ 
\big(\beta_\bnu(r,\bvarrho)^{-1/2}\big)_{\bnu\in\FF} \in \ell^{p/(1-p)}(\FF),
\end{equation}
for the norms of the Wiener-Hermite PC expansion coefficients of
$(\bb,\xi,\delta,X)$-holomorphic functions $u$ if $\bb \in \ell^p(\NN)$ 
for some $0 < p < 1$. 
In Section~\ref{sec:bdX}
and Section~\ref{sec:BIP} we saw that solutions to certain parametric PDEs as
well as posterior densities satisfy $(\bb,\xi,\delta,X)$-holomorphy.

The goal of this section is in a constructive way to sharpen and improve 
these results in a form more suitable for numerical implementation by using some ideas
from \cite{dD21,dD-Erratum22,ZS17}.
We shall construct a new weight
family $(c_{\bnu})_{\bnu\in\CF}$ based on
$(\beta_\bnu(r,\bvarrho))_{\bnu\in\CF}$, such that 
\eqref{eq:weightedsum1} with $\beta_\bnu(r,\bvarrho)$ replaced by $c_\bnu$, 
and its generalization of the form \eqref{weighted-summ} 
for $\sigma_\bnu =c_\bnu^{1/2}$ hold.  
Once a suitable family $(c_\bnu)_{\bnu\in\CF}$ has been
identified, we obtain a multiindex set $\Lambda_\eps\subseteq\CF$ 
for $\eps>0$ via
\begin{equation}\label{eq:Leps}
	\Lambda_\eps:=\set{\bnu\in\CF}{c_\bnu^{-1}\ge\eps},
\end{equation}
The set $\Lambda_\eps$ will then serve as an index set to define
interpolation  operators  $\VI_{\Lambda_\eps}$ 
and 
quadrature operators $\VQ_{\Lambda_\eps}$.
As the sequence $(c_{\bnu})_{\bnu\in\CF}$ is used to construct sets of
multiindices, it should possess certain features, including each
$c_\bnu$ to be easily computable for $\bnu\in\CF$, 
and for the resulting numerical algorithm to be efficient.
\subsection{Smolyak sparse-grid interpolation and quadrature}
\label{Smolyak interpolation and quadrature}
\subsubsection{Smolyak sparse-grid interpolation}\label{sec:int} \index{interpolation!Smolyak $\sim$}
Recall that for every $n\in\N_0$ denote by $(\chi_{n,j})_{j=0}^n\subseteq\R$ the
Gauss-Hermite points in one dimension (in particular, $\chi_{0,0}=0$),
that is, the roots of Hermite polynomial $H_{n+1}$. 
Let
$$I_n:C^0(\R)\to C^0(\R)$$ be the univariate polynomial
Lagrange interpolation \index{Lagrange interpolation} operator defined by
\begin{equation*}
  (I_n u)(y):= \sum_{j=0}^n u(\chi_{n,j}) \prod_{\substack{i=0\\
      i\neq j}}^n \frac{y-\chi_{n,i}}{\chi_{n,j}-\chi_{n,i}},\qquad y\in\R,
\end{equation*}
with convention that $I_{-1}:C^0(\R)\to C^0(\R)$ is defined as the
constant $0$ operator.

For any multi-index $\bnu\in\CF$, introduce the tensorized operators
$\VI_\bnu$ by
$$\VI_\bnul u := u((\chi_{0,0})_{j\in\N}),$$ and for $\bnu\neq\bnul$ via
\begin{equation}\label{eq:VInu}
  \VI_\bnu := \bigotimes_{j\in\N} I_{\nu_j},
\end{equation}
i.e.,
\begin{equation}\nonumber
  \VI_\bnu u(\by) = \sum_{\set{\bmu\in\CF}{\bmu\le\bnu}}
  u((\chi_{\nu_j,\mu_j})_{j\in\N}) \prod_{j\in\N}\prod_{\substack{i=0\\ i\neq
      \mu_j}}^{\nu_j}
  \frac{y_j-\chi_{\nu_j,i}}{\chi_{\nu_j,\mu_j}-\chi_{\nu_j,i}},\quad \by\in U.
\end{equation}
The operator $\VI_\bnu$ can thus be applied to functions $u$ which are
pointwise defined at each $(\chi_{\nu_j,\mu_j})_{j\in\N}\in U$. Via
Remark~\ref{rmk:defu}, we can apply it in particular to
$(\bb,\xi,\delta,X)$-holomorphic functions. Observe that the product
over $j\in\N$ in \eqref{eq:VInu} is a finite product, since for every
$j$ with $\nu_j=0$, the inner product over
$i\in\{0,\dots,\mu_j-1,\mu_j+1,\dots,\nu_j\}$ is over an empty set,
and therefore equal to one by convention.  
Then for a finite set $\Lambda\subseteq\CF$
\begin{equation}\label{eq:VILambda0}
  \VI_{\Lambda} :=\sum_{\bnu\in\Lambda}\bigotimes_{j\in\N}(I_{\nu_j}-I_{\nu_j-1}).
\end{equation}

Expanding all tensor product operators, we get
\begin{equation}\label{eq:VILambda}
  \VI_{\Lambda} =\sum_{\bnu\in\Lambda}\sigma_{\Lambda;\bnu}
  \VI_\bnu\qquad\text{where}\qquad
  \sigma_{\Lambda;\bnu}:=\sum_{\set{\bee\in\{0,1\}^\infty}{\bnu+\bee\in\Lambda}} (-1)^{|\bee|}.
\end{equation}

\begin{definition}\label{def:DownClsd}
  An index set $\Lambda\subseteq\CF$ is called \emph{downward closed},
  if it is finite and if for every $\bnu\in\Lambda$ it holds
  $\bmu\in\Lambda$ whenever $\bmu\le\bnu$.  Here, the ordering
  ``$\le$'' between two indices $\bmu = (\mu_j)_{j\in \NN}$ and
  $\bnu = (\nu_j)_{j\in \NN}$ in $\CF$ expresses that for all
  $j\in \NN$ holds $\mu_j \le \nu_j$ with strict inequality for at
  least one index $j$.

\end{definition}

As is well-known, $\VI_{\Lambda}$ possesses the following crucial
property, see for example \cite[Lemma 1.3.3]{JZdiss}.
\begin{lemma}\label{lemma:VIprop}
  Let $\Lambda\subseteq\CF$ be downward closed. Then
  $\VI_{\Lambda}f =f$ for all
  $f\in {\rm span}\set{\by^\bnu}{\bnu\in\Lambda}$.
\end{lemma}

The reason to choose the collocation points $(\chi_{n,j})_{j=0}^n$ as
the Gauss-Hermite points, is that it was recently shown that the
interpolation operators $I_n$ then satisfy the following stability
estimate, see \cite[Lemma 3.13]{ErnstSprgkTam18}.
\begin{lemma}
  For every $n\in\N_0$ and every $m\in\N$ it holds
  \begin{equation*}
    \norm[L^2(\R;\gamma_1)]{I_n(H_m)}\le 4 \sqrt{2m-1}.
  \end{equation*}
\end{lemma}
With the presently adopted normalization of the GM $\gamma_1$, it
holds $H_0\equiv 1$ and therefore $I_n(H_0)=H_0$ for all $n\in\N_0$
(since the interpolation operator $I_n$ exactly reproduces all
polynomials of degree $n\in\N_0$).  Hence
$$\norm[L^2(\R;\gamma_1)]{I_n(H_0)}=\norm[L^2(\R;\gamma_1)]{H_0}=1$$
for all $n\in\N_0$.  
Noting that $4\sqrt{2m-1}\le (1+m)^2$ for all $m\in\N$, 
we get
\begin{equation*}
  \norm[L^2(\R;\gamma_1)]{I_n(H_m)}\le (1+m)^2\qquad\forall n,~m\in\N_0.
\end{equation*}
Consequently 
\begin{equation}\label{eq:L2bound}
  \norm[L^2(U;\gamma)]{\VI_\bnu(H_\bmu)} = 
  \prod_{j\in\N}\norm[L^2(\R;\gamma_1)]{I_{\nu_j}(H_{\mu_j})}
  \le \prod_{j\in\N}(1+\mu_j)^2\qquad\forall \bnu,~\bmu\in\CF.
\end{equation}
Recall that for 
$\bnu \in \FF$ and $\tau\geq 0$,
we denote
\begin{equation*}
  p_\bnu( \tau) :=\prod_{j\in \NN} (1+\nu_j)^\tau.
\end{equation*}
If $\nu_j>\mu_j$ then $(I_{\nu_j}-I_{\nu_j-1})H_{\mu_j}=0.$
Thus, 
$$\bigotimes_{j\in\N}(I_{\nu_j}-I_{\nu_j-1})H_\bmu=0,$$ 
whenever
there exists $j\in\N$ such that $\nu_j>\mu_j$.  
Hence, for any
downward closed set $\Lambda$, it holds
\begin{equation}\label{eq:L2boundLambda}
  \norm[L^2(U;\gamma)]{\VI_\Lambda(H_\bmu)} \le p_{\bmu}(3).
\end{equation}
Indeed,
\begin{equation}\nonumber
  \norm[L^2(U;\gamma)]{\VI_\Lambda(H_\bmu)} 
  \le
  \sum_{\set{\bnu\in\Lambda}{\bnu\le\bmu}}p_\bmu(2)
  \le 
  |{\set{\bnu\in\Lambda}{\bnu\le\bmu}}|p_\bmu(2)
  = 
  \prod_{j\in\N}(1+\mu_j)p_\bmu(2)=p_{\bmu}(3).
\end{equation}
\subsubsection{Smolyak sparse-grid quadrature}
\label{sec:Quadrat}
Recall that  analogously 
to $I_n$ we introduce univariate
polynomial quadrature operators via
\begin{equation*}
  Q_n u := \sum_{j=0}^n u(\chi_{n,j}) \omega_{n,j},\qquad
  \omega_{n,j}:=\int_{\R} \prod_{i\neq j}
  \frac{y-\chi_{n,i}}{\chi_{n,j}-\chi_{n,i}} \dd\gamma_1(y).
\end{equation*}
Furthermore, we define $$\VQ_\bnul u := u((\chi_{0,0})_{j\in\N}),$$
and for $\bnu\neq\bnul$,
\begin{equation*}
  \VQ_\bnu := \bigotimes_{j\in\N} Q_{\nu_j},
\end{equation*}
i.e.,
\begin{equation*}
  \VQ_\bnu u = \sum_{\set{\bmu\in\CF}{\bmu\le\bnu}}
  u((\chi_{\nu_j,\mu_j})_{j\in\N}) \prod_{j\in\N}\omega_{\nu_j,\mu_j},
\end{equation*}
and finally for a finite downward closed $\Lambda\subseteq\CF$ with
$\sigma_{\Lambda;\bnu}$ as in \eqref{eq:VILambda},
\begin{equation*}
  \VQ_{\Lambda}:=\sum_{\bnu\in\Lambda}\sigma_{\Lambda;\bnu} Q_\bnu.
\end{equation*}
Again we emphasize that the above formulas are meaningful as long as
point evaluations of $u$ at each $(\chi_{\nu_j,\mu_j})_{j\in\N}$ are
well defined, $\bnu\in\CF$, $\bmu\le\bnu$. Also note that
\begin{equation}\label{eq-interpolation-quadrature}
\VQ_{\Lambda}f=\int_U \VI_{\Lambda}f(\by)\dd\gamma(\by).
\end{equation}
Recall that the set $\CF_2$ is defined by
\begin{equation} \label{cF_2} 
    \CF_2:=\set{\bnu\in\CF}{\nu_j\neq 1~\forall j}.
\end{equation}
We thus have $\CF_2\subsetneq\CF$.  Similar to Lemma
\ref{lemma:VIprop} we have the following lemma, which can be proven
completely analogous to \cite[Lemma 1.3.16]{JZdiss} (also see
\cite[Remark 4.2]{ZS17}).

\begin{lemma}\label{lemma:VQprop}
  Let $\Lambda\subseteq\CF$ be downward closed. Then
$$\VQ_{\Lambda}v =\int_U v(\by)\dd\gamma(\by)$$ for all
$v \in {\rm
  span}\set{\by^\bnu}{\bnu\in\Lambda\cup(\CF\backslash\CF_2)}$.
\end{lemma}

With \eqref{eq:L2bound} it holds
\begin{equation*}
  |\VQ_\bnu(H_\bmu)|
  =\left|\int_U \VI_\bnu(H_\bmu)(\by)\dd\gamma(\by)\right|
  \le \norm[L^2(U;\gamma)]{\VI_\bnu(H_\bmu)}
  \le \prod_{j\in\N}(1+\mu_j)^2\qquad\forall \bnu,~\bmu\in\CF,
\end{equation*}
and similarly, using \eqref{eq:L2boundLambda}, 
we have the bound
\begin{equation}\label{eq:L1boundLambda}
  |\VQ_\Lambda(H_\bmu)| \le p_{\bmu}(3).
\end{equation}
%
\subsection{Multiindex sets}\label{sec:mi}
In this section, we first recall some arguments from
\cite{dD21,dD-Erratum22,ZS17} which allow to bound the number of required function
evaluations in the interpolation an quadrature algorithm. 
Subsequently, a construction of a suitable family
$(c_{k,\bnu})_{\bnu\in\CF}$ is provided for $k\in\{1,2\}$. 
The index
$k$ determines whether the family will be used for a sparse-grid
interpolation ($k=1$) or a Smolyak-type sparse-grid quadrature ($k=2$)
algorithm.  Finally, it is shown that the multiindex sets
$\Lambda_{k,\eps}$ as in \eqref{eq:Leps} based on
$(c_{k,\bnu})_{\bnu\in\CF}$, guarantee algebraic convergence rates for
certain truncated Wiener-Hermite PC expansions. 
This will be
exploited 
to verify convergence rates for interpolation in
Section~\ref{sec:intrate} and for quadrature in
Section~\ref{sec:quadrate}.
\subsubsection{Number of function evaluations}
\label{sec:NoFnEval}
In order to obtain a convergence rate in terms of the number of
evaluations of $u$, we need to determine the number of interpolation
points used by the operator $\VI_{\Lambda}$ or $\VQ_{\Lambda}$.  Since
the discussion of $\VQ_{\Lambda}$ is very similar, we concentrate here
on $\VI_{\Lambda}$.

Computing the interpolant $\VI_{\bnu}u$ in \eqref{eq:VInu} requires
knowledge of the function values of $u$ at each point in
\begin{equation*}
  \set{(\chi_{\nu_j,\mu_j})_{j\in\N}}{\bmu\le\bnu}.
\end{equation*}
The cardinality of this set is bounded by
$\prod_{j\in\N}(1+\nu_j)=p_\bnu(1)$. Denote by
\begin{equation}\label{eq:pts}
  \pts (\Lambda)
  := 
  \set{(\chi_{\nu_j,\mu_j})_{j\in\N}}{\bmu\le\bnu,~\bnu\in\Lambda}
\end{equation}
the set of interpolation points defining the interpolation operator
$\VI_{\Lambda}$ (i.e., ~$|\pts (\Lambda)|$ is the number of function
evaluations of $u$ required to compute $\VI_{\Lambda} u$).  By
\eqref{eq:VILambda} we obtain the bound
\begin{equation}\label{eq:ptsLbound}
  |\pts (\Lambda)|
  \le 
  \sum_{\set{\bnu\in\Lambda}{\sigma_{\Lambda,\bnu}\neq 0}} \prod_{j\in\N} (1+\nu_j)
  = \sum_{\set{\bnu\in\Lambda}{\sigma_{\Lambda,\bnu}\neq 0}} p_\bnu(1).
\end{equation}

\subsubsection{Construction of $(c_{k,\bnu})_{\bnu\in\CF}$}
We are now in position to construct $(c_{k,\bnu})_{\bnu\in\CF}$.  As
mentioned above, we distinguish between the cases $k=1$ and $k=2$,
which correspond to polynomial interpolation or quadrature.  Note that
in the next lemma we define $c_{k,\bnu}$ for all $\bnu\in\CF$, but the
estimate provided in the lemma merely holds for $\bnu\in\CF_k$,
$k\in\{1,2\}$, where $\CF_1:=\CF$ and $\CF_2$ is defined in
\eqref{cF_2}.  
Throughout what follows, 
empty products shall equal $1$ by convention.
\begin{lemma}\label{lemma:cnu}
  Assume that $\tau>0$, $k\in\{1,2\}$ and $r>\max\{\tau,k\}$.  
  Let
  $\bvarrho\in (0,\infty)^\infty$ be such that $\varrho_j\to\infty$ as
  $j\to\infty$.  

  Then there exist $K>0$ and $C_0>0$ such that
  \begin{equation}\label{eq:cnu}
    c_{k,\bnu}:= \prod_{j\in\supp(\bnu)}\max\left\{1,K
      \varrho_j\right\}^{2k} \nu_j^{r-\tau},\quad\bnu\in\CF,
  \end{equation}
  satisfies
  \begin{equation}\label{eq:cknubound}
    C_0 c_{k,\bnu} p_\bnu(\tau) \le \beta_\bnu(r,\bvarrho) \quad\forall \bnu\in\CF_k
  \end{equation}
 with $\beta_\bnu(r,\bvarrho)$ as in \eqref{beta}.
\end{lemma}
\begin{proof}
  \textbf{Step 1.} Fix $\bnu\in\CF_k$, then $j\in\supp(\bnu)$ implies
  $\nu_j\ge k$ and thus $\min\{r,\nu_j\}\ge k$ since $r> k$ by
  assumption.  With $s:=\min\{r,\nu_j\}\le \nu_j$, for all $j\in\N$
  holds
  \begin{equation*}
    \binom{\nu_j}{s} 
    = 
    \frac{\nu_j!}{(\nu_j-s)! s!} 
    \ge 
    \frac{1}{s!} (\nu_j-s+1)^s\ge \nu_j^s \frac{1}{s! s^s}
    \ge
    \nu_j^s \frac{1}{r! r^r}= \nu_j^{\min\{\nu_j,r\}} \frac{1}{r!r^r}
    \ge 
    \nu_j^{r} \frac{1}{r!r^{2r}}.
  \end{equation*}
  Furthermore, 
  if $j\in\supp(\bnu)$, then due to
  $s=\min\{\nu_j,r\}\ge k$, with
  $\varrho_0:=\min \{1,\min_{j\in\N}\varrho_j\}$ we have
$$
\varrho_0^{2r} \leq \min\{ 1,\varrho_j\}^{2r} \leq \varrho_j^{2(s-k)}.
$$
Thus
$$\varrho_j^{\min\{\nu_j,r\}}\ge \varrho_0^{2r} \varrho_j^{2k}$$ for all $j\in\N$. 
In all, we conclude
\begin{equation}\label{eq:estb}
  \beta_\bnu(r,\bvarrho) = \prod_{j\in\N} \left(\sum_{l=0}^r
    \binom{\nu_j}{l}\varrho_j^{2l} \right) \ge
  \prod_{j\in\supp(\bnu)}\binom{\nu_j}{\min\{\nu_j,r\}}
  \varrho_j^{2\min\{\nu_j,r\}} \ge \prod_{j\in\supp(\bnu)}
  \frac{\varrho_0^{2r}}{r!r^{2r}} \varrho_j^{2k}\nu_j^{r}.
\end{equation}
Since $\bnu\in\CF_k$ was arbitrary, 
this estimate holds for all $\bnu\in\CF_k$.

\textbf{Step 2.}
Denote $\hat\varrho_j:=\max\{1,K\varrho_j\}$, 
where $K>0$ is still at our disposal.  
We have
\begin{equation*}
  p_\bnu(\tau)\le\prod_{j\in\supp(\bnu)} 2^{\tau} \nu_j^{\tau}
\end{equation*}
and thus
\begin{equation}\label{eq:estc}
  c_{k,\bnu} p_\bnu(\tau) \le
  \prod_{j\in\supp(\bnu)}
  2^\tau
  \hat\varrho_j^{2k} \nu_j^{r}.
\end{equation}
Again, this estimate holds for any $\bnu\in\CF_k$.

With $\varrho_0:=\min\{1,\min_{j\in\N}\varrho_j\}$ denote
\begin{equation*}
  C_b:=\left(\frac{\varrho_0^{2r}}{r!r^{2r}}\right)^{1/(2k)}
  \qquad\text{and}\qquad C_c:= (2^\tau)^{1/(2k)}.
\end{equation*}
Set
\begin{equation*}
  K:=\frac{C_b}{C_c},\qquad \tilde\varrho_j= K \varrho_j
\end{equation*}
for all $j\in\N$. 
Then
\begin{equation*}
  C_b\varrho_j=C_c\tilde\varrho_j = C_c\hat\varrho_j\begin{cases}
    1 &\text{if }K\varrho_j\ge 1,\\
    K\varrho_j  &\text{if }K\varrho_j<1.
  \end{cases}
\end{equation*}
Let
\begin{equation*}
  C_0:= \prod_{\set{j\in\N}{K\varrho_j<1}} (K\varrho_j)^{2k}
\end{equation*}
and note that this product is over a finite number of indices, since
$\varrho_j\to\infty$ as $j\to\infty$. Then for any $\bnu\in\CF_k$
\begin{equation*}
  \prod_{j\in\supp(\bnu)} C_c\tilde\varrho_j \ge C_0^{\frac{1}{2k}}
  \prod_{j\in\supp(\bnu)} C_c\hat\varrho_j.
\end{equation*}
With \eqref{eq:estb} and \eqref{eq:estc} we thus obtain for every
$\bnu\in\CF_k$,
\begin{align}
  \beta_\bnu(r,\bvarrho) &\ge \prod_{j\in\supp(\bnu)} (C_b\varrho_j)^{2k}\nu_j^r =
                           \prod_{j\in\supp(\bnu)} \left(C_c \tilde
                           \varrho_j\right)^{2k}\nu_j^r 
                           \notag
  \\
                         &\ge C_0 
                           \prod_{j\in\supp(\bnu)} (C_c\hat\varrho_j)^{2k}\nu_j^r \ge C_0 c_{k,\bnu}p_\bnu(\tau).
                           \notag
\end{align}

\end{proof}
\subsubsection{Summability properties of the collection $(c_{k,\bnu})_{\bnu\in\CF}$}
\label{S:SumPrpcknu}
First we discuss the summability of the collection
$(c_{k,\bnu})_{\bnu\in\cF}$.  
We will require the following lemma
which is a modification of \cite[Lemma 6.2]{dD21}.
\begin{lemma}\label{lemma:summabcnu}
  Let $\theta \geq 0$.  Let further $k\in\{1,2\}$, $\tau>0$,
  $r>\max\{k,\tau\}$ and $q>0$ be such that
  $(r-\tau)q/(2k)-\theta>1$. Assume that
  $(\varrho_j)_{j\in\N}\in (0,\infty)^\infty$ satisfies
  $(\varrho_j^{-1})_{j\in\N}\in\ell^q(\N)$. Then with
  $(c_{k,\bnu})_{\bnu\in\CF}$ as in Lemma \ref{lemma:cnu} it holds
  \begin{equation*}
    \sum_{\bnu\in\cF}p_\bnu( \theta) c_{k,\bnu}^{-\frac{q}{2k}}<\infty.
  \end{equation*}
\end{lemma}
\begin{proof}
  This lemma can be proven in the same way as the proof of \cite[Lemma
  6.2]{dD21}.  We provide a proof for completeness.  With
  $\hat\varrho_j:=\max\{1,K\varrho_j\}$ it holds
  $(\hat\varrho_j^{-1})_{j\in\N}\in\ell^q(\N)$. By definition of
  $c_{k,\bnu}$, factorizing, we get
  \begin{equation*}
    \begin{split}
      \sum_{\bnu\in\CF} p_\bnu( \theta) c_{k,\bnu}^{-\frac{q}{2k}} & =
      \sum_{\bnu\in\CF} \prod_{j\in\supp(\bnu)}(1+\nu_j)^\theta
      \left(\hat\varrho_j^{2k} \nu_j^{r-\tau}\right)^{-\frac{q}{2k}}
      \leq \prod_{j\in\N} \left(2^\theta \hat\varrho_j^{-q}
        \sum_{n\in\N} n^{\frac{-q(r-\tau)}{2k}} n^\theta\right).
    \end{split}
  \end{equation*}
  The sum over $n$ equals some finite constant $C$ since by assumption
  $q(r-\tau)/2k-\theta>1$.
  Using the inequality $\log(1+x)\le x$ for all $x>0$, we get
  \begin{equation*}
    \sum_{\bnu\in\CF} c_{k,\bnu}^{-\frac{q}{2k}}\le \prod_{j\in\N}
    \left( 1 + C \hat\varrho_j^{-q}\right)  
    = \exp\left(\sum_{j\in\N}\log (1+C \hat\varrho_j^{-q}) \right)
    \le \exp\left(\sum_{j\in\N} C \hat\varrho_j^{-q}\right),
  \end{equation*}
  which is finite since $(\hat\varrho_j^{-1})\in\ell^q(\N)$.
\end{proof}

Based on \eqref{eq:Leps}, 
for $\eps>0$ and $k\in\{1,2\}$ let
\begin{equation}\label{eq:Lkeps}
  \Lambda_{k,\eps}:=\{\bnu\in \CF: c_{k,\bnu}^{-1}\ge \eps \} \subseteq \CF.
\end{equation}
The summability shown in Lemma \ref{lemma:summabcnu} implies algebraic
convergence rates of the tail sum as provided by the following
proposition.  
This is well-known and follows by Stechkin's lemma
\cite{stechkin} which itself is a simple consequence of 
H\"older's inequality.
\begin{proposition}\label{prop:bestN}
  Let $k\in\{1,2\}$, $\tau>0$, and $q>0$. 
  Let
  $(\varrho_j^{-1})_{j\in\N}\in\ell^q(\N)$ and $r>\max\{k,\tau\}$,
  $(r-\tau)q/(2k)>2$. 
  Assume that
  $(a_\bnu)_{\bnu\in\CF}\in [0,\infty)^\infty$ is such that
  \begin{equation}\label{eq:betabnulbnulinf}
    \sum_{\bnu\in\CF}\beta_\bnu(r,\bvarrho) a_\bnu^2<\infty.
  \end{equation}
  Then there exists a constant $C$ solely depending on
  $(c_{k,\bnu})_{\bnu\in\CF}$ in \eqref{eq:cnu} such that for all
  $\eps>0$ it holds that
  \begin{equation*}
    \sum_{\bnu\in\CF_k\backslash\Lambda_{k,\eps}} p_\bnu(\tau)
    a_\bnu
    \le C\left(\sum_{\bnu\in\CF}\beta_\bnu(r,\bvarrho)a_\bnu^2\right)^{\frac 1 2}  \eps^{\frac{1}{2}- \frac{q}{4k}},
  \end{equation*}
  and
  \begin{equation} \label{ptsLambda0} |\pts (\Lambda_{k,\epsilon})|
    \leq C \varepsilon^{-\frac{q}{2k}}.
  \end{equation}
\end{proposition}
\begin{proof}
  We estimate
  \begin{equation*} \label{sum-estimate1}
    \begin{split} 
      \sum_{\bnu\in\CF_k\backslash\Lambda_{k,\eps}} p_\bnu(\tau)
      a_\bnu & \le \Bigg(
      \sum_{\bnu\in\CF_k\backslash\Lambda_{k,\eps}}
      p_\bnu(\tau)^2a_\bnu^2 c_{k,\bnu}\Bigg)^{1/2} \Bigg(
      \sum_{\bnu\in\CF_k\backslash\Lambda_{k,\eps}}
      c_{k,\bnu}^{-1}\Bigg)^{1/2}.
    \end{split}
  \end{equation*}
  The first sum is finite by \eqref{eq:betabnulbnulinf} and because
  $C_0 p_\bnu(\tau)^2 c_{k,\bnu}\le \beta_\bnu(r,\bvarrho)$ according
  to \eqref{eq:cknubound}. By Lemma \ref{lemma:summabcnu} and
  \eqref{eq:Lkeps} we obtain
  \begin{equation*} \label{sum-estimate2}
    \sum_{\bnu\in\CF_k\backslash\Lambda_{k,\eps}} c_{k,\bnu}^{-1} =
    \sum_{c_{k,\bnu}^{-1}<\varepsilon } c_{k,\bnu}^{-\frac{q}{2k}}
    c_{k,\bnu}^{-1+\frac{q}{2k}} \leq C\varepsilon^{1-\frac{q}{2k}}
  \end{equation*} 
  which proves the first statement. Moreover, for each $\bnu\in \FF$,
  the number of interpolation (quadrature) points is $p_\bnu(1)$.
  Hence
  \begin{equation*} \label{ptsLambda} |\pts (\Lambda_{k,\epsilon})| =
    \sum_{\bnu \in \Lambda_{k,\epsilon}}p_\bnu(1) =
    \sum_{c_{k,\bnu}^{-1}\geq \varepsilon}
    p_\bnu(1)c_{k,\bnu}^{-\frac{q}{2k}} c_{k,\bnu}^{\frac{q}{2k}} \leq
    \varepsilon^{-\frac{q}{2k}} \sum_{\bnu \in \cF_{k}}
    p_\bnu(1)c_{k,\bnu}^{-\frac{q}{2k}} \leq C
    \varepsilon^{-\frac{q}{2k}}
  \end{equation*}
  again by Lemma \ref{lemma:summabcnu} and \eqref{eq:Lkeps}.
\end{proof}

\subsubsection{Computing $\Lambda_\eps$}\label{sec:mleps}
Having identified appropriate sequences $(c_\bnu)_{\bnu\in\cF}$,
  in order to be able to implement the Smolyak sparse-grid interpolation operator
  $\VI_{\Lambda_\eps}$ and the Smolyak sparse-grid quadrature operator
  $\VQ_{\Lambda_\eps}$, in practice it remains to compute the sets
  $\Lambda_\eps=$ in \eqref{eq:Leps}.  We now recall Algorithm 2 in
  \cite[Sec.~3.1.3]{JZdiss} which achieves this in $O(|\Lambda_\eps|)$
  work and memory. For the convenience of the reader we recall the
  main statement regarding the algorithm's complexity below in Lemma
  \ref{lemma:alg}. Additionally, we point to
  \cite[Alg.~4.13]{MR2566594} which presents an alternative
  approach---a recursive algorithm that also achieves linear computational
  complexity.

In the following denote $\be_j:=(\delta_{ij})_{j\in\N}\in\N_0^\infty$.

\begin{algorithm} 
    \caption{Lambda($\eps,(c_\bnu)_{\bnu\in\CF})$)}
    \label{alg:Lambda}
    \begin{algorithmic}[1]
  \State $\bnu\leftarrow\bnul$
  \If{$c_\bnu<\eps$}
  \State $\Lambda\leftarrow \emptyset$
  \State\Return $\Lambda$
  \Else
  \State $\Lambda\leftarrow \{\bnu\}$
  \EndIf
  \While{True}
  \State $d\leftarrow 1$
  \While{$a_{\bnu+\be_d}<\eps$} 
  \If{$\nu_d\neq 0$}\Comment{Reject $\bnu+\be_d$ where $\nu_d\neq
    0$}
  \State $\nu_d\leftarrow 0$
  \State $d\leftarrow d+1$ 
  \ElsIf{$\bnu\neq\bnul$} \Comment{Reject $\bnu+\be_d$ where $\nu_d= 0$}
  \State $d=\min\set{j\in\N}{\nu_j\neq 0}$  
  \Else \Comment{Reject $\be_d$ $\Rightarrow$ stop algorithm}
  \State\Return $\Lambda$  
  \EndIf
  \EndWhile
  \State $\bnu\leftarrow \bnu+\be_d$
  \State $\Lambda\leftarrow \Lambda\cup\{\bnu\}$ 
  \EndWhile
\end{algorithmic}
\end{algorithm}

The algorithm is of linear complexity in the following sense
\cite[3.1.12]{JZdiss}:
\begin{lemma}\label{lemma:alg}
  Let $(c_\bnu)_{\bnu\in\CF}\subseteq [0,\infty)$ be a null-sequence
  such that (i) $\bmu\le\bnu$ implies $c_\bmu\ge c_{\bnu}$ and
  (ii) if $\bnu\in\CF$ and for some $i<j$ it holds $\nu_i=\nu_j=0$,
  then $c_{\bnu+\be_i}\ge c_{\bnu+\be_j}$.

  Then for any $\eps>0$, Algorithm \ref{alg:Lambda} terminates and
  returns $\Lambda_\eps$ in \eqref{eq:Leps}. Moreover each line of
  Algorithm \ref{alg:Lambda} is executed at most $4|\Lambda_\eps|+1$
  times.
\end{lemma}

\subsection{Interpolation convergence rate}
\label{sec:intrate}
If $X$ is a Hilbert space, then the Wiener-Hermite PC expansion of
$u:U\to X$ converges in general only in $L^2(U,X;\gamma)$.
As mentioned before this creates some subtleties when working with
interpolation and quadrature operators based on pointwise evaluations
of the target function.
To demonstrate this, we recall the following example from
\cite{CSZ16}, which \emph{does not satisfy
  $(\bb,\xi,\delta,\C)$-holomorphy}, since Definition~\ref{def:bdXHol}
\ref{item:vN} does not hold.
\begin{example}\label{ex:u}
  Define $u:U\to\C$ pointwise by
  \begin{equation*}
    u(\by):=\begin{cases}
      1 &\text{if }|\set{j\in\N}{y_j\neq 0}|<\infty\\
      0 &\text{otherwise.}
    \end{cases}
  \end{equation*}
  Then $u$ vanishes on the complement of the $\gamma$-null set
  \begin{equation*}
    \bigcup_{n\in\N} \R^n\times \{0\}^\infty.
  \end{equation*}
  Consequently $u$ is equal to the constant zero function in the sense
  of $L^2(U;\gamma)$. Hence there holds the expansion
  $u=\sum_{\bnu\in\CF}0\cdot H_\bnu$ with convergence in
  $L^2(U;\gamma)$. Now let $\Lambda\subseteq\CF$ be nonempty, finite
  and downward closed. As explained in Section~\ref{sec:int}, the
  interpolation operator $\VI_{\Lambda}$ reproduces all polynomials in
  ${\rm span}\set{\by^\bnu}{\bnu\in\Lambda}$. Since any point
  $(\chi_{\nu_j,\mu_j})_{j\in\N}$ with $\mu_j\le\nu_j$ is zero in all
  but finitely many coordinates (due to $\chi_{0,0}=0$), we observe
  that
  \begin{equation*}
    \VI_{\Lambda} u \equiv 1 \neq 0\equiv \sum_{\bnu\in\CF} 0\cdot
    \VI_{\Lambda} H_\bnu.
  \end{equation*}
  This is due to the fact that $u = \sum_{\bnu\in\CF} 0\cdot H_\bnu$
  only holds in the $L^2(U;\gamma)$ sense, and interpolation or
  quadrature (which require pointwise evaluation of the function) are
  not meaningful for $L^2(U;\gamma)$ functions.
\end{example}

The above example shows that if
$$u=\sum_{\bnu\in\CF} u_\bnu H_\bnu\in L^2(U;\gamma)$$ with Wiener-Hermite PC
expansion coefficients $(u_\bnu)_{\nu\in\CF}\subset\R$, then the
formal equalities
$$\VI_{\Lambda}u=\sum_{\bnu\in\CF} u_\bnu \VI_{\Lambda}H_\bnu, $$ and
$$\VQ_{\Lambda}u=\sum_{\bnu\in\CF} u_\bnu \VQ_{\Lambda}H_\bnu $$ 
do in general not hold in $L^2(U;\gamma)$. 
Our definition of
$(\bb,\xi,\delta,X)$-holomorphy allows to circumvent this by
interpolating not $u$ itself but the approximations $u_N$ to $u$ which
are pointwise defined and only depend on finitely many variables,
cp.~Definition~\ref{def:bdXHol}.

Our analysis starts with the following result about pointwise
convergence. 
For $k\in\{1,2\}$ and $N\in\N$ we introduce the notation
\begin{equation*}\label{eq:CFkN}
  \CF_{k}^N := \set{\bnu\in \CF_k}{\supp(\bnu)\subseteq\{1,\dots,N\}}.
\end{equation*}
These sets thus contain multiindices $\bnu$ for which $\nu_j=0$ for
all $j>N$. 

\begin{lemma}\label{lemma:uN}
  Let $u$ be $(\bb,\xi,\delta,X)$-holomorphic for some
  $\bb\in (0,\infty)^\infty$. Let $N\in\N$, and let
  $\tilde u_N:U\to X$ be as in Definition~\ref{def:bdXHol}.  For
  $\bnu\in\CF$ define
  $$\tilde u_{N,\bnu}:=\int_{U} \tilde u_N(\by)H_\bnu(\by)\dd\gamma(\by).$$ 
  Then,
  \begin{equation} \label{pointwise-abs-conv}
  \tilde u_N(\by)=\sum_{\bnu\in\CF_1^N} \tilde u_{N,\bnu}
  H_\bnu(\by)
  \end{equation}
  with the equality and
  pointwise absolute convergence in $X$ for all $\by\in U$.
\end{lemma}
\begin{proof}
  From the Cram\'er bound
  $$
  |\tilde{H}_n(x)| <
  2^{n/2}\sqrt{n!}\exp(x^2/2), 
  $$
  see \cite{Indritz}, and where
  $\tilde{H}_n(x/\sqrt{2}) := 2^{n/2}\sqrt{n!}H_n(x)$, see \cite[Page
  787]{AS}, we have for all $n\in\N_0$
  \begin{equation}\label{eq:supxHn}
    \sup_{x\in\R} \exp(-x^2/4) |H_n(x)|\le 1.
  \end{equation}

  By Theorem.~\ref{thm:bdHolSum} 
  $(\tilde u_{N,\bnu})_{\nu\in\CF}\in\ell^1(\CF)$.  Note that for
  $\bnu\in\CF_1^N$
  \begin{equation*}
    \tilde u_{N,\bnu} = \int_U \tilde u_N(\by)H_\bnu(\by)\dd\gamma(\by)
    = \int_{\R^N} u_N (y_1,\dots,y_N)\prod_{j=1}^NH_{\nu_j}(y_j) \dd\gamma_N((y_j)_{j=1}^N)
  \end{equation*}
  and thus $\tilde u_{N,\bnu}$ coincides with the Wiener-Hermite PC
  expansion coefficient of $u_N$ w.r.t.~the multiindex
  $(\nu_j)_{j=1}^N\in\N_0^N$.
  The summability of the collection 
$$
\left(\norm[X]{u_{N,\bnu}}\norm[L^2(\R^N;\gamma_N)]{\prod_{j=1}^N H_{\nu_j}(y_j)}\right)_{\bnu\in\CF_1^N}
$$
now implies in particular,
$$u_N((y_j)_{j=1}^N) =\sum_{\bnu\in\CF_1^N} u_{N,\bnu} \prod_{j=1}^N
H_{\nu_j}(y_j)$$ in the sense of $L^2(\R^N;\gamma_N)$.

Due to \eqref{eq:supxHn} and
$(\norm[X]{u_{N,\bnu}})_{\bnu\in\CF_1^N}\in\ell^1(\CF_1^N)$ 
we can define a continuous function
\begin{equation}\label{eq:hatuN}
\hat u_N: (y_j)_{j=1}^N \mapsto 
\sum_{\bnu\in\N_0^N} u_{N,\bnu} \prod_{j=1}^N H_{\nu_j}(y_j)
\end{equation}
on $\R^N$.  By \eqref{eq:supxHn}, for every fixed
  $(y_j)_{j=1}^N \in\R^N$ we have the uniform bound
  $|\prod_{j=1}^N H_{\nu_j}(y_j)|\le \prod_{j=1}^N\exp(\frac{y_j^2}{4})$
  independent of $\bnu\in\cF_1^N$. The summability of
  $(\norm[X]{u_{N,\bnu}})_{\bnu\in\cF_1^N}$ implies
 the  absolute convergence of the series in \eqref{eq:hatuN} for every
  fixed $(y_j)_{j=1}^N \in\R^N$.
  
Since they have the same Wiener-Hermite PC expansion, it holds
$\hat u_N=u_N$ in the sense of $L^2(\R^N;\gamma_N)$.

  By Definition~\ref{def:bdXHol} the function $u:\R^N\to X$ is in
  particular continuous (it even allows a holomorphic extension to
  some subset of $\CC^N$ containing $\R^N$). Now $\hat u_N$,
  $u_N:\R^N\to X$ are two continuous functions which are equal in the
  sense of $L^2(\R^N;\gamma_N)$. Thus they coincide pointwise and it
  holds in $X$ for every $\by\in U$,
$$
\tilde u_N(\by)=u_N((y_j)_{j=1}^N)=\sum_{\bnu\in\CF_1^N} \tilde u_{N,\bnu} H_\bnu(\by).
$$
\end{proof}

The result on the pointwise absolute convergence in Lemma \ref{lemma:uN} 
is not sufficient for establishing the convergence rate 
of the interpolation approximation in the space $L^2(U, X;\gamma)$. 
To this end, we need the result on convergence in the space $L^2(U, X;\gamma)$ 
in the following lemma.

\begin{lemma} \label{lemma:L^2-convergence}
  Let $u$ be $(\bb,\xi,\delta,X)$-holomorphic for some
$\bb\in (0,\infty)^\infty$. Let $N\in\N$, and let
$\tilde u_N:U\to X$ be as in Definition~\ref{def:bdXHol} and $\tilde u_{N,\bnu}$ as in Lemma \ref{lemma:uN}. 
Let
$\Lambda \subset \CF_1$ be a finite, downward closed set. 

Then we have
	\begin{align}
	\VI_\Lambda \tilde u_N=	\sum_{\bnu\in\CF_1^N} \tilde u_{N;\bnu} \VI_\Lambda H_\bnu
\end{align}
with the equality and unconditional convergence in the space $L^2(U, X;\gamma)$.
\end{lemma}

\begin{proof} For a function $v: U\to X$ we have
\begin{equation} \label{VI_{Lambda} v(by)}
  \VI_{\Lambda} v(\by) =\sum_{\bnu\in\Lambda}\sigma_{\Lambda;\bnu}
\sum_{\mu\in \FF, \bmu\leq \bnu}v(\chi_{\bnu,\bmu})L_{\bnu,\bmu}(\by), 
\end{equation}
where $\sigma_{\Lambda;\bnu}$ is defined in \eqref{eq:VILambda} and 
recall, $\chi_{\bnu,\bmu}=(\chi_{\nu_j,\mu_j})_{j\in \NN}$ and
\begin{equation} \label{L_{bnu,bmu}(by)}
L_{\bnu,\bmu}(\by):=\prod_{j\in\N}\prod_{\substack{i=0\\ i\neq
		\mu_j}}^{\nu_j}
\frac{y_j-\chi_{\nu_j,i}}{\chi_{\nu_j,\mu_j}-\chi_{\nu_j,i}},\quad \by\in U.
\end{equation}
Since in a Banach space the absolute convergence implies the unconditional convergence, 
from Lemma \ref{lemma:uN}  it follows that for any $\by\in U$,
	\begin{equation} \label{tilde u_N(by)=}
\tilde u_N(\by)=\sum_{\bnu\in\CF_1^N} \tilde u_{N,\bnu}
	H_\bnu(\by)
	\end{equation}
	with the equality and unconditional convergence in $X$.
Let $\{ F_n\}_{n\in \NN}\subset \FF_1^N$ be any sequence of finite sets in $\FF_1^N$ exhausting $\FF_1^N$. 
Then
\begin{equation} \label{tilde{u}_N^{(n)}(by)}
\forall \by\in U: \ \ \tilde{u}_N^{(n)}(\by):=\sum_{\bnu\in F_n}\tilde{u}_{N,\bnu} H_\bnu(\by) \ \to \ \tilde{u}_N(\by), \ \ n \to \infty,
\end{equation}
 with the sequence convergence in the space $X$. 
Notice that the functions 
	$\VI_{\Lambda} \tilde{u}_N $ and $\sum_{\bnu\in F_n}\tilde{u}_{N,\bnu} \VI_{\Lambda} H_\bnu$ belong to the space $L^2(U,X;\gamma)$. Hence we have that
\begin{equation}
	\begin{split} 
&\bigg\|\VI_{\Lambda} \tilde{u}_N  -  \sum_{\bnu\in F_n}\tilde{u}_{N,\bnu} \VI_{\Lambda} H_\bnu \bigg\|_{L^2(U,X;\gamma)}
=
\big\|\VI_{\Lambda} \tilde{u}_N  -  \VI_{\Lambda} \tilde{u}_N^{(n)}\|_{L^2(U,X;\gamma)}	 
=\big\|\VI_{\Lambda} \big(\tilde{u}_N  -  \tilde{u}_N^{(n)}\big)\|_{L^2(U,X;\gamma)}
		\\
		&
		\leq   
		\sum_{\bnu\in\Lambda}|\sigma_{\Lambda;\bnu}|
		\sum_{\bmu\in \FF, \bmu\leq \bnu}\big\|\tilde{u}_N (\chi_{\bnu,\bmu})-\tilde{u}_N^{(n)}(\chi_{\bnu,\bmu}) \big\|_X\int_{U}|L_{\bnu,\bmu}(\by) |\rd\gamma(\by).
	\end{split}
\end{equation}
Observe that $L_{\bnu,\bmu}$ is a polynomial of order $|\bnu|$. Since $\{\bmu\in \FF: \bmu\leq \bnu\}$ and $\Lambda$ are finite sets, we can choose $C:= C(\Lambda)>0$ so that 
$$
\int_{U}|L_{\bnu,\bmu}(\by) |\rd\gamma(\by)\leq C
$$ 
for all $\bmu\leq \bnu $ and $\bnu\in \Lambda$, and,  moreover,
by using \eqref{tilde{u}_N^{(n)}(by)}  we can choose $n_0$ so that
\begin{equation*}\label{eq:varepsilon}
\|\tilde{u}_N (\chi_{\bnu,\bmu})-\tilde{u}_N^{(n)}(\chi_{\bnu,\bmu})\|_X\leq \varepsilon
\end{equation*}
for all $n\geq n_0$ and $\bmu\leq \bnu$, $\bnu\in \Lambda$.
Consequently, we have that  for all $n\geq n_0$,
\begin{equation}
	\begin{split} 
		\bigg\|\VI_{\Lambda} \tilde{u}_N  -  \sum_{\bnu\in F_n}\tilde{u}_{N,\bnu} \VI_{\Lambda} H_\bnu \bigg\|_{L^2(U,X;\gamma)}
		\leq 
		C\sum_{\bnu\in \Lambda}|\sigma_{\Lambda;\bnu}|\sum_{\bmu\leq \bnu}\varepsilon
	 =    C\varepsilon\sum_{\bnu\in\Lambda}|\sigma_{\Lambda;\bnu}|p_\bnu(1) .
	\end{split}
\end{equation}
 Hence we derive the convergence in the space $L^2(U, X;\gamma)$ of the sequence $\sum_{\bnu\in F_n}\tilde{u}_{N,\bnu} \VI_{\Lambda} H_\bnu$ to $\VI_{\Lambda} \tilde{u}_N$  ($n \to \infty$)  for any sequence of finite sets $\{ F_n\}_{n\in \NN}\subset \FF_1^N$ exhausting $\FF_1^N$. This proves the lemma.
\end{proof}

\begin{remark}
{\rm  Under the assumption of Lemma \ref{lemma:L^2-convergence}, in a similar way, we can prove that for every $\by\in U$
	\begin{equation}  
	\VI_\Lambda \tilde u_N(\by) 
	= 
	\sum_{\bnu\in\CF_1^N} \tilde u_{N;\bnu} \VI_\Lambda H_\bnu(\by)
	\end{equation}
	with the equality and unconditional convergence in the space $X$.
}
\end{remark}

We arrive at the following convergence rate result, 
which improves the convergence rate in \cite{ErnstSprgkTam18} (in terms of
the number of function evaluations) by a factor $2$ (for
the case when the elements of the representation system are supported globally in $\domain$). 
Additionally, 
we provide an explicit construction of suitable index sets.  
Recall that pointwise
evaluations of a $(\bb,\xi,\delta,X)$-holomorphic functions are
understood in the sense of Remark.~\ref{rmk:defu}.

\begin{theorem}\label{thm:int}
  Let $u$ be $(\bb,\xi,\delta,X)$-holomorphic for some
  $\bb\in\ell^p(\N)$ and some $p\in (0,2/3)$.
  Let $(c_{1,\bnu})_{\bnu\in\CF}$ be as in Lemma \ref{lemma:cnu} with
  $\bvarrho$ as in Theorem~\ref{thm:bdHolSum}.
		
  Then there exist $C>0$ and, for every $n \in \NN$, $\eps_n>0$ such
  that $|\pts (\Lambda_{1,\eps_n})|\le n$ (with $\Lambda_{1,\eps_n}$
  as in \eqref{eq:Lkeps}) and
  \begin{equation*}\label{eq:L2interrn}
    \norm[L^2(U,X;\gamma)]{u - \VI_{\Lambda_{1,\eps_n}}u} \	\le \ C n^{{-\frac{1}{p}+\frac{3}{2}}}.
  \end{equation*}
\end{theorem}
\begin{proof}
  For $\eps>0$ small enough and satisfying $|\Lambda_{1,\eps}|>0$,
  take $N\in\N$ with
	$$N\ge\max\set{j\in\supp(\bnu)}{\bnu\in\Lambda_{1,\eps}},$$ so large that
	\begin{equation}\label{eq:truncerr}
          \norm[L^2(U,X;\gamma)]{u-\tilde u_N}\le \eps^{\frac{1}{2}-\frac{p}{4(1-p)}},
	\end{equation}
	which is possible due to the $(\bb,\xi,\delta,X)$-holomorphy
        of $u$ (cp.~Definition~\ref{def:bdXHol} \ref{item:vN}).  An
        appropriate value of $\varepsilon$ depending on $n$ will be
        chosen below.  In the following for $\bnu\in\CF_1^N$ we denote
        by $\tilde u_{N,\bnu}\in X$ the PC coefficient of $\tilde u_N$
        and for $\bnu\in\CF$ as earlier $u_\bnu\in X$ is the PC
        coefficient of $u$.
	
	Because $$N\ge\max\set{j\in\supp(\bnu)}{\bnu\in\Lambda_{1,\eps}}$$
        and $\chi_{0,0}=0$, we have
        $$
        \VI_{\Lambda_{1,\eps}}u= \VI_{\Lambda_{1,\eps}} \tilde u_N
        $$
        (cp.~Remark~\ref{rmk:defu}). Hence by \eqref{eq:truncerr}
	\begin{equation} \label{norm-ineq}
          \norm[L^2(U,X;\gamma)]{u-\VI_{\Lambda_{1,\eps}}u}
          =\norm[L^2(U,X;\gamma)]{u-\VI_{\Lambda_{1,\eps}} \tilde u_N}
          \le \eps^{\frac{1}{2}-\frac{p}{4(1-p)}}+
          \norm[L^2(U,X;\gamma)]{ \tilde u_N-\VI_{\Lambda_{1,\eps}} \tilde u_N}.    
	\end{equation}

We now  give a bound of the second term on the right side of   \eqref{norm-ineq}. 	By Lemma \ref{lemma:L^2-convergence} we can write
       	\begin{align*} \label{interpolation series-eps}
    \VI_{\Lambda_{1,\eps}}\tilde u_N=   		\sum_{\bnu\in\CF_1^N} \tilde u_{N;\bnu} \VI_{\Lambda_{1,\eps}} H_\bnu
       	\end{align*}
     with the equality and unconditional in $L^2(U, X;\gamma)$.
       	Hence by Lemma \ref{lemma:VIprop} and \eqref{eq:L2boundLambda} we have that
	\begin{align*}
          \norm[L^2(U,X;\gamma)]{ \tilde u_N-\VI_{\Lambda_{1,\eps}} \tilde u_N}
          &=\normc[L^2(U,X;\gamma)]{\sum_{\bnu\in\CF\backslash\Lambda_{1,\eps}} \tilde u_{N;\bnu} (H_\bnu-\VI_{\Lambda_{1,\eps}}H_\bnu)}\nonumber\\
          &\le \sum_{\bnu\in\CF\backslash\Lambda_{1,\eps}} \norm[X]{\tilde u_{N;\bnu}}
            \big(\norm[L^2(U;\gamma)]{H_\bnu}+\norm[L^2(U;\gamma)]{\VI_{\Lambda_{1,\eps}}H_\bnu}\big)\nonumber\\
          &\leq \sum_{\bnu\in\CF_1^N\backslash \Lambda_{1,\eps}} \norm[X]{\tilde u_{N;\bnu}}
            \left(1+p_{\bnu}(3)\right)
            \nonumber\\
          &\leq 2\sum_{\bnu\in\CF_1^N\backslash \Lambda_{1,\eps}} \norm[X]{\tilde u_{N;\bnu}}
            p_{\bnu}(3).
	\end{align*}

        Choosing $r>4/p-1$ ($q:=p/(1-p)$, $\tau=3$), according to
        Proposition~\ref{prop:bestN}, \eqref{eq:cknubound} and
        Theorem~\ref{thm:bdHolSum} (with
        $(\varrho_j^{-1})_{j\in\N}\in\ell^{p/(1-p)}(\N)$ as in
        Theorem~\ref{thm:bdHolSum}) the last sum is bounded by
	\begin{equation*}
          C \left(\sum_{\bnu\in\CF_1^N}
            \beta_\bnu(r,\bvarrho) \norm[X]{\tilde u_{N,\bnu}}^2  \right)\eps^{\frac{1}{2}-\frac{q}{4}}	\le 
          C(\bb) \delta^2 \eps^{\frac{1}{2}-\frac{q}{4} }=  C(\bb) \delta^2 \eps^{\frac{1}{2}-\frac{p}{4(1-p)}},
	\end{equation*}
	and the constant $C(\bb)$ from Theorem~\ref{thm:bdHolSum} does
        not depend on $N$ and $\delta$. Hence, by \eqref{norm-ineq} we obtain
 	\begin{equation} \label{norm-ineq-2}
 	\norm[L^2(U,X;\gamma)]{u-\VI_{\Lambda_{1,\eps}}u}
 	\le C_1 \eps^{\frac{1}{2}-\frac{p}{4(1-p)}}.    
 \end{equation}        
        From \eqref{ptsLambda0} it follows that
        \begin{equation*}%
          |\pts (\Lambda_{1,\eps})| 
          \leq C_2 \varepsilon^{-\frac{q}{2}}=C_2\varepsilon^{-\frac{p}{2(1-p)}}.
        \end{equation*}
	For every $n \in \NN$, we choose an $\eps_n>0$ satisfying the
        condition
        $$n/2 \le C_2\varepsilon_n^{-\frac{p}{2(1-p)}} \le n.$$ Then due to \eqref{norm-ineq-2},
        the claim holds true for the chosen $\eps_n$.
      \end{proof}

      \begin{remark} {\rm
        Comparing the best $n$-term convergence result in Remark
        \ref{rmk:bestN} with the interpolation result of Theorem
        \ref{thm:int}, we observe that the convergence rate is reduced
        by $1/2$, and moreover, rather than $p\in (0,1)$ as in Remark
        \ref{rmk:bestN}, Theorem \ref{thm:int} requires
        $p\in (0,2/3)$. This discrepancy can be explained as follows:
        Since $(H_\bnu)_{\bnu\in\cF}$ forms an orthonormal basis of
        $L^2(U;\gamma)$, for the best $n$-term result we could resort
        to Parseval's identity, which merely requires
        $\ell^2$-summability of the Hermite PC coefficients, i.e.\
        $(\norm[X]{u_\bnu})_{\bnu\in\cF}\in\ell^2(\cF)$.  Due to
        $(\norm[X]{u_\bnu})_{\bnu\in\cF}\in\ell^{\frac{2p}{2-p}}$ by
        Theorem \ref{thm:bdXSum}, this is ensured as long as
        $p\in (0,1)$. On the other hand, for the interpolation result
        we had to use the triangle inequality, since the family
        $(\VI_{\Lambda_{1,\eps_n}}H_\bnu)_{\bnu\in\cF}$ of
        interpolated multivariate Hermite polynomials does not form an
        orthonormal family of $L^2(U;\gamma)$.  This argument
        requires the stronger condition
        $(\norm[X]{u_\bnu})_{\bnu\in\cF}\in\ell^1(\cF)$, resulting in
        the stronger assumption $p\in (0,2/3)$ of Theorem
        \ref{thm:int}.
      } \end{remark} 
      
\subsection{Quadrature convergence rate}\label{sec:quadrate}
	We first prove a  result on equality and unconditional convergence 
        in the space $X$ for quadrature operators, 
        which is similar to that in Lemma \ref{lemma:L^2-convergence}. 
        It is needed to establish the quadrature convergence rate.

\begin{lemma} \label{lemma:uncond-conv-Q_Lambda}
 Let $u$ be $(\bb,\xi,\delta,X)$-holomorphic for some
	$\bb\in (0,\infty)^\infty$. Let $N\in\N$, and let
	$\tilde u_N:U\to X$ be as in Definition~\ref{def:bdXHol} 
        and $\tilde u_{N,\bnu}$ as in Lemma \ref{lemma:uN}. Let
	$\Lambda \subset \CF_1$ be a finite downward closed set. 

        Then we have
	\begin{equation}  
	\VQ_\Lambda \tilde u_N 
	= 
	\sum_{\bnu\in\CF_1^N} \tilde u_{N;\bnu} \VQ_\Lambda H_\bnu
	\end{equation}
	with the equality and unconditional convergence in the space $X$.
\end{lemma}

\begin{proof} For a function $v: U\to X$ by \eqref{eq-interpolation-quadrature} and \eqref{VI_{Lambda} v(by)} we have
	$$
	\VQ_{\Lambda} v  =\sum_{\bnu\in\Lambda}\sigma_{\Lambda;\bnu}
	\sum_{\mu\in \FF, \bmu\leq \bnu}v(\chi_{\bnu,\bmu}) \int_{U} L_{\bnu,\bmu}(\by) \rd\gamma(\by),
	$$
	where $\chi_{\bnu,\bmu}=(\chi_{\nu_j,\mu_j})_{j\in \NN}$, $\sigma_{\Lambda;\bnu}$,  $L_{\bnu,\bmu}$ are defined in \eqref{eq:VILambda} and \eqref{L_{bnu,bmu}(by)}, respectively. By using this representation, we can prove the lemma in a way similar to the proof of Lemma \ref{lemma:L^2-convergence} with  some appropriate modifications.
	\end{proof}

Analogous to Theorem~\ref{thm:int} we obtain the following result for
the quadrature convergence with an improved convergence rate compared
to interpolation.
\begin{theorem}\label{thm:quad}
  Let $u$ be $(\bb,\xi,\delta,X)$-holomorphic for some
  $\bb\in\ell^p(\N)$ and some $p\in (0,4/5)$.  Let
  $(c_{2,\bnu})_{\bnu\in\CF}$ be as in Lemma \ref{lemma:cnu} with
  $\varrho$ as in Theorem~\ref{thm:bdHolSum}.  Then there exist $C>0$
  and, for every $N\in\N$, $\eps_n>0$
  such that $|\pts (\Lambda_{2,\eps_n})|\le n$ (with
  $\Lambda_{2,\eps_n}$ as in \eqref{eq:Lkeps}) and
  \begin{equation*}
    \normc[X]{\int_U u(\by)\rd \gamma(\by)- \VQ_{\Lambda_{2,\eps_n}}u} \
    \le \ C n^{-\frac{2}{p}+\frac{5}{2}}.
  \end{equation*}
\end{theorem}
\begin{proof}
  For $\eps>0$ small enough and satisfying $|\Lambda_{2,\eps}|>0$,
  take $N\in\N$,
  $N\ge\max\set{j\in\supp(\bnu)}{\bnu\in\Lambda_{2,\eps}}$ so large
  that
  \begin{equation}\label{eq:truncerrQ}
    \normc[X]{\int_U \big[u(\by)-\tilde u_N(\by)\big]\dd\gamma(\by)}
    \le \norm[L^2(U,X;\gamma)]{u-\tilde u_N}\le \eps^{\frac{1}{2}-\frac{p}{8(1-p)}},
  \end{equation}
  which is possible due to the $(\bb,\xi,\delta,X)$-holomorphy of $u$
  (cp.~Definition~\ref{def:bdXHol} \ref{item:vN}). An appropriate
  value of $\varepsilon$ depending on $n$ will be chosen below. In the
  following for $\bnu\in\CF$ we denote by $\tilde u_{N,\bnu}$ the
  Wiener-Hermite PC expansion coefficient of $\tilde u_N$ and as
  earlier $u_\bnu$ is the Wiener-Hermite PC expansion coefficient of
  $u$.
  
  Because $$N\ge\max\set{j\in\supp(\bnu)}{\bnu\in\Lambda_{2,\eps}}$$
  and $\chi_{0,0}=0$, we have
  $\VQ_{\Lambda_{2,\eps}}u= \VQ_{\Lambda_{2,\eps}} \tilde u_N$
  (cp.~Remark.~\ref{rmk:defu}). Hence by \eqref{eq:truncerrQ}
  \begin{align} \label{estimate1}
    \normc[X]{\int_U u(\by)\dd\gamma(\by)-\VQ_{\Lambda_{2,\eps}}u}
    &=\normc[X]{\int_U u(\by)\dd\gamma(\by)-\VQ_{\Lambda_{2,\eps}}\tilde u_N}\nonumber\\
    &\le \eps^{\frac{1}{2}-\frac{p}{8(1-p)}}+
      \normc[X]{\int_U \tilde u_N(\by)\dd\gamma(\by)-\VQ_{\Lambda_{2,\eps}} \tilde u_N}.    
  \end{align}
By Lemma \ref{lemma:uncond-conv-Q_Lambda} we have
  $$
  \VQ_{\Lambda_{2,\eps}} \tilde u_N =\sum_{\bnu\in\CF_1^N} 
  \tilde u_{N;\bnu} \VQ_{\Lambda_{2,\eps}}H_\bnu= \sum_{\bnu\in\CF_2^N} 
  \tilde u_{N;\bnu} \VQ_{\Lambda_{2,\eps}}H_\bnu
  $$
  with the equality and unconditional convergence in the space $X$.
Since $\Lambda_{2,\eps}$ is nonempty and
  downward closed we have $\bnul\in \Lambda_{2,\eps}$.  Then, by Lemma
  \ref{lemma:VQprop}, \eqref{eq:L1boundLambda}, and using
 $$
 \int_U H_\bnu(\by)\dd\gamma(\by)=0
 $$ 
 for all $\bnul\neq\bnu\in \CF\backslash\CF_2$, we have that
 \begin{align*}
   \normc[X]{\int_U \tilde u_N(\by)\dd\gamma(\by)-\VQ_{\Lambda_{2,\eps}} \tilde u_N}
   &=\normc[X]{\sum_{\bnu\in\CF_2\backslash\Lambda_{2,\eps}} \tilde u_{N;\bnu} \left(\int_{U}H_\bnu(\by)\dd\gamma(\by)-\VQ_{\Lambda_{2,\eps}}H_\bnu\right)}\nonumber\\
   &\le \sum_{\bnu\in\CF_2\backslash\Lambda_{2,\eps}} \norm[X]{\tilde u_{N;\bnu}}
     (\norm[L^2(U;\gamma)]{H_\bnu} + |\VQ_{\Lambda_{2,\eps}}H_\bnu|)\nonumber\\
   &\le \sum_{\bnu\in\CF_2\backslash \Lambda_{2,\eps}} \norm[X]{\tilde u_{N;\bnu}} \left(1+p_\bnu(3)\right) 
     \nonumber\\
   &\le 2\sum_{\bnu\in\CF_2\backslash \Lambda_{2,\eps}} \norm[X]{\tilde u_{N;\bnu}} p_\bnu(3)
     .
 \end{align*}
 Choosing $r>8/p-5$ ($q=\frac{p}{1-p}$, $\tau=3$), according to
 Proposition~\ref{prop:bestN}, \eqref{eq:cknubound} and
 Theorem~\ref{thm:bdHolSum} (with
 $(\varrho_j^{-1})_{j\in\N}\in\ell^{p/(1-p)}(\N)$ as in
 Theorem~\ref{thm:bdHolSum}) the last sum is bounded by
 \begin{equation*}
   C \left(\sum_{\bnu\in\CF}\beta_\bnu(r,\bvarrho) \norm[X]{\tilde u_{N,\bnu}}^2\right)
   \eps^{\frac{1}{2}-\frac{q}{8}}
   \le C(\bb)\delta^2\eps^{\frac{1}{2}-\frac{q}{8}}=C(\bb) \eps^{\frac{1}{2}-\frac{p}{8(1-p)}},
 \end{equation*}
 and the constant $C(\bb)$ from Theorem~\ref{thm:bdHolSum} does not
 depend on $N$ and $\delta$.
 Hence, by \eqref{eq:truncerrQ} and \eqref{estimate1} we obtain that
 \begin{align} \label{estimate2}
 	\normc[X]{\int_U u(\by)\dd\gamma(\by)-\VQ_{\Lambda_{2,\eps}}u}
 	\le C_1 \eps^{\frac{1}{2}-\frac{p}{8(1-p)}}.   
 \end{align}
 From \eqref{ptsLambda0} it follows that
 \begin{equation*}
   |\pts (\Lambda_{k,\epsilon})| 
   \leq C_2 \varepsilon^{-\frac{q}{4}}=C_2 \varepsilon^{-\frac{p}{4(1-p)}}.
 \end{equation*}
 For every $n \in \NN$, we choose an $\eps_n>0$ satisfying the
 condition $$n/2 \le C_2\varepsilon_n^{-\frac{p}{4(1-p)}} \le n.$$ Then due to \eqref{estimate2}
 the claim holds true for the chosen $\eps_n$.
\end{proof}

\begin{remark} {\rm
 Interpolation formulas based on index sets like
  $$
  \Lambda(\xi):= \{\bnu \in \cF: \beta_\bnu(r,\bvarrho) \le \xi^{2/q}\},
  $$
   (where $\xi >0$ is a large parameter), 
  have been proposed in \cite{ErnstSprgkTam18, dD21} for the parametric, elliptic divergence-form PDE \eqref{SPDE} 
  with log-Gaussian inputs \eqref{eq:CoeffAffin} satisfying the assumptions of 
  Theorem \ref{thm[ell_2summability]} with $i=1$. 
  There,
  dimension-independent convergence rates of sparse-grid interpolation
  were obtained. 
  Based on the weighted $\ell^2$-summability of the Wiener-Hermite PC expansion coefficients of the form
 \begin{equation}\label{eq:weightedsum2}
 	\sum_{\bnu\in\CF}\beta_\bnu(r,\bvarrho) \norm[X]{u_\bnu}^2<\infty
 	\ \ \ \text{with} \ \ \ \big(p_\bnu(\tau, \lambda)\beta_\bnu(r,\bvarrho)^{-1/2}\big)_{\bnu\in\FF} \in \ell^q(\FF) \ \ (0 < q < 2),
 \end{equation} 
  the rate established in \cite{ErnstSprgkTam18} is $\frac{1}{2}(1/q - 1/2)$ which  lower than
  those obtained in the present analysis.  
  The improved rate $1/q - 1/2$ has been established in \cite{dung2021collocation}.
  This rate coincides with the rate in Theorem~\ref{thm:int} for the choice $q = p/(1-p)$. 

  The existence of Smolyak
  type quadratures with a proof of dimension-independent convergence
  rates was shown first in \cite{ChenlogNQuad2018} and then in \cite{dD21}.  In \cite{ChenlogNQuad2018}, the symmetry
  of the GM and corresponding cancellations were not exploited, and
  these quadrature formulas provide the convergence rate  $\frac{1}{2}(1/q - 1/2)$ which is lower (albeit
  dimension-independent) convergence rates in terms of the number of
  function evaluations as in Theorems \ref{thm:int} and
  \ref{thm:quad}. 
  By using this symmetry, for a given weighted $\ell^2$-summability of the Wiener-Hermite PC expansion
  coefficients \eqref{weighted-summ} with $\sigma_\bnu =\beta_\bnu(r,\bvarrho)^{1/2}$, 
  the rate established in \cite{dD21} (see also \cite{dD-Erratum22}) 
  is $2/q - 1/2$ which coincides with the rate of convergence that was obtained 
  in Theorem \ref{thm:quad} for the choice $q = p/(1-p)$.
}
\end{remark} 
\newpage
\
\newpage
\section{Multilevel Smolyak sparse-grid interpolation and quadrature}
\label{sec:MLApprox}
In this section we introduce a multilevel interpolation and quadrature
algorithm which are suitable for numerical implementation.  
The presentation and arguments follow mostly \cite{ZDS19} and \cite[Section 3.2]{JZdiss}, 
where multilevel algorithms for the uniform measure on
the hypercube $[-1,1]^\infty$ were analyzed (in contrast to the case
of a product GM on $U$, which we consider here). 
In Section~\ref{sec:SetNot}, we introduce the setting for the multilevel algorithms,
in particular a notation of ``work-measure'' related to the discretization of a
Wiener-Hermite PC expansion coefficient $u_\bnu$ for $\bnu$ in the set of multi-indices that are active 
in a given (interpolation or quadrature) approximation.
Section~\ref{sec:MLAlg} describes the general structure of the algorithms,
Section~\ref{app:mlweight} addresses algorithms for the determination 
of sets $\Lambda \subset \FF$ of active
multi-indices and a corresponding allocation of discretization levels in 
linear in $|\Lambda|$ work and memory. 
Section~\ref{sec:MLInterpol} addresses the error analysis of the 
Smolyak sparse-grid interpolation, 
and 
Section~\ref{sec:MLQuad} contains the error analysis of the corresponding
Smolyak sparse-grid quadrature algorithm.
All algorithms are formulated and analyzed in terms of several abstract hypotheses.
Section~\ref{sec:Approx} verifies these abstract conditions for a concrete family 
of parametric, elliptic PDEs.
Finally, Section~\ref{Multilevel approximation and quadrature in Bochner spaces}
addresses convergence rates achieveable with the mentioned Smolyak sparse-grid interpolation
and quadrature algorithms \emph{assuming at hand optimal multi-index sets}.
The major finding being that the corresponding rates differ only by logarithmic terms
from the error bounds furnished by those realized by the algorithms in 
Sections~\ref{sec:MLAlg}-\ref{sec:MLQuad}.
\subsection{Setting and notation}
\label{sec:SetNot}
To approximate the solution $u$ to a parametric PDE as in the examples
of the preceding sections, the interpolation operator $\VI_\Lambda$
introduced in Section \ref{sec:int} requires function values of $u$ at
different interpolation points in the parameter space $U$. For a
parameter $\by\in U$, typically the PDE solution $u(\by)$, which is a
function belonging to a Sobolev space over a \emph{physical domain}
$\D$, is not given in closed form and has to be approximated. The idea
of multilevel approximations is to combine interpolants of
approximations to $u$ at different spatial accuracies, in order to
reduce the overall computational complexity. This will now be
formalized.

First, we assume given a sequence $(\sw{\lev})_{\lev\in\N_0} \subset \N$,
exhibiting the properties of the following assumption.  
Throughout
$\sw{\lev}$ will be interpreted as a measure for the computational
complexity of evaluating an approximation $u^\lev:U\to X$ of
$u:U\to X$ at a parameter $\by\in U$. 
Here we use a superscript $\lev$
rather than a subscript for the approximation level, as the subscript
is reserved for the dimension truncated version $u_N$ of $u$ as in
Definition~\ref{def:bdXHol}.

\begin{assumption}\label{ass:SW}
  The sequence $(\sw{\lev})_{\lev\in\N_0}\subseteq\N_0$ is strictly
  monotonically increasing and $\sw{0}=0$.  There exists a constant
  $\KW\ge 1$ such that for all $\lev\in\N$
  \begin{enumerate}
  \item\label{item:SW1} $\sum_{j=0}^{\lev}\sw{j}\le \KW \sw{\lev}$,
  \item\label{item:SW2} $\lev\le \KW(1+\log(\sw{\lev}))$,
  \item\label{item:SW4} $\sw{\lev}\le \KW (1+\sw{\lev-1})$,
  \item\label{item:SW3} for every $r>0$ there exists $C=C(r)>0$
    independent of $\lev$ such that
    $$\sum_{j=\lev}^{\infty} \sw{j}^{-r}\le C (1+\sw{\lev})^{-r}.$$
  \end{enumerate}
\end{assumption}
Assumption \ref{ass:SW} is satisfied if $(\sw{\lev})_{\lev\in\N}$ is
exponentially increasing, (for instance $\sw{\lev}=2^\lev$,
$\lev\in\N$).  In the following we write
$\SW:=\set{\sw{\lev}}{\lev\in\N_0}$ and
\begin{equation*}
  \lfloor x \rfloor_\SW :=\max\set{\sw{\lev}}{\sw{\lev}\le x}.
\end{equation*}

We work under the following \emph{hypothesis on the discretization
  errors in physical space}: we quantify the convergence of the
discretization scheme with respect to the discretization level
$\lev\in\N$.  Specifically, we assume the approximation $u^\lev$ to
$u$ to behave asymptotically
as 
\begin{equation}\label{eq:FEMrate}
  \norm[X]{u(\by)-u^\lev(\by)}\le C(\by) \sw{\lev}^{-{\alpha}}\qquad\forall \lev\in\N,
\end{equation}
for some fixed convergence rate ${\alpha}>0$ of the ``physical space
discretization'' and with constant $C(\by) > 0$ depending on the
parameter sequence $\by$.  We will make this assumption on $u^\lev$
more precise shortly.  If we think of $u^\lev(\by)\in H^1(\D)$ for the
moment as a FEM approximation to the exact solution
$u(\by)\in H^1(\D)$ of some $\by$-dependent elliptic PDE, then
$\sw{\lev}$ could stand for the number of degrees of freedom of the
finite element space. In this case ${\alpha}$ corresponds to the FEM
convergence rate.  Assumption \eqref{ass:SW} will for instance be
satisfied if for each consecutive level the meshwidth is cut in
half. Examples are provided by the FE spaces discussed in Section
\ref{S:FEIntrp}, Proposition \ref{prop:FECorner}.  As long as the
computational cost of computing the FEM solution is proportional to
the dimension $\sw{\lev}$ of the FEM space, $\sw{\lev}^{-{\alpha}}$ is
the error in terms of the work $\sw{\lev}$. Such an assumption usually
holds in one spatial dimension, where the resulting stiffness matrix
is tridiagonal. For higher spatial dimensions solving the
corresponding linear system is often times not of linear complexity,
in which case the convergence rate ${\alpha}>0$ has to be adjusted
accordingly.

We now state our assumptions on the sequence of functions
$(u^\lev)_{\lev\in\N}$ approximating $u$. Equation \eqref{eq:FEMrate}
will hold in the  $L^2$ sense over all parameters $\by\in U$,
cp.~Assumption \ref{ass:ml} \ref{item:u-ujclose}, and Definition
\ref{def:bdXHol} \ref{item:varphi}.

\begin{assumption}\label{ass:ml}
  Let $X$ be a separable Hilbert space and let
  $(\sw{\lev})_{\lev\in\N_0}$ satisfy Assumption \ref{ass:SW}.
  Furthermore, $0<p_1\le p_2<\infty$, $\bb_1\in\ell^{p_1}(\N)$,
  $\bb_2\in\ell^{p_2}(\N)$, $\xi>0$, $\delta>0$ and there exist
  functions $u\in L^2(U,X;\gamma)$,
  $(u^\lev)_{\lev\in\N}\subseteq L^2(U,X;\gamma)$ such that
  \begin{enumerate}
  \item $u\in L^2(U,X;\gamma)$ is $(\bb_1,\xi,\delta,X)$-holomorphic,
  \item\label{item:u-ujbound} $(u-u^\lev)\in L^2(U,X;\gamma)$ is
    $(\bb_1,\xi,\delta,X)$-holomorphic for every $l\in\N$,
  \item\label{item:u-ujclose} $(u-u^\lev)\in L^2(U,X;\gamma)$ is
    $(\bb_2,\xi,\delta \sw{\lev}^{-{\alpha}},X)$-holomorphic for every
    $l\in\N$.
  \end{enumerate}
\end{assumption}

\begin{remark} {\rm
  Items \ref{item:u-ujbound} and \ref{item:u-ujclose} are two
  assumptions on the domain of holomorphic extension of the
  discretization error $e_\lev:=u-u^\lev:U\to X$. As pointed out in
  Remark~\ref{rmk:bdexpl}, the faster the sequence $\bb$ decays the
  larger the size of holomorphic extension, and the smaller $\delta$
  the smaller the upper bound of this extension.
  
  Hence items \ref{item:u-ujbound} and \ref{item:u-ujclose} can be
  interpreted as follows: Item \ref{item:u-ujbound} implies that
  $e_\lev$ has a large domain of holomorphic extension. 
  Item (iii) is related to the assumption \eqref{eq:FEMrate}. 
  It yields  that by considering the extension of
  $e_\lev$ on a smaller domain, we can get a ($\lev$-dependent)
  smaller upper bound of the extension of $e_\lev$ (in the sense of
  Definition~\ref{def:bdXHol} \ref{item:varphi}). Hence there is a
  tradeoff between choosing the size of the domain of the holomorphic
  extension and the upper bound of this extension. 
} \end{remark} 
\subsection{Multilevel Smolyak sparse-grid algorithms}
\label{sec:MLAlg}
	Let $\blev=(\lev_\bnu)_{\bnu\in\CF}\subseteq\N_0$ be a family of  natural numbers
	associating with each multiindex $\bnu\in\cF$ of a PC expansion a  
        discretization level $\lev_\bnu\in\N_0$.  	
Typically, this is a family of
discretization levels for some \emph{hierarchic, numerical}
approximation of the PDE in the physical domain $\domain$, 
associating with each multiindex $\bnu\in\cF$ of a PC expansion 
of the parametric solution 
in the parameter domain a possibly 
coefficient-dependent discretization level $\lev_\bnu\in\N_0$.  
With the sequence
$\lev_\bnu\in\N_0$, 
we associate sets of multiindices via
\begin{equation}\label{eq:Gamma}
  \Gamma_j=\Gamma_j(\blev):=\set{\bnu\in\CF}{\lev_\bnu\ge j}\qquad\forall j\in\N_0.
\end{equation}
Throughout we will assume that
$$
|\blev|:= \|\blev\|_{\ell^1(\Ff)}: = \sum_{\bnu\in\CF}\lev_\bnu<\infty
$$ 
and that $\blev$ is
monotonically decreasing, meaning that $\bnu\le\bmu$ implies
$\lev_\bnu\ge\lev_\bmu$. In this case each $\Gamma_j\subseteq\CF$,
$j\in\N$, is finite and downward closed. Moreover $\Gamma_0=\CF$, and
the sets $(\Gamma_j )_{j\in \NN_0}$ are nested according to
\begin{equation*}
  \CF = \Gamma_0\supseteq\Gamma_1\supseteq \Gamma_2\dots.
\end{equation*}

With $(u^\lev)_{\lev\in\N}$ as in Assumption \ref{ass:ml}, we now
define the multilevel sparse-grid interpolation algorithm
\begin{equation}\label{eq:VIml}
  \VIml_\blev u :=
  \sum_{j\in\N} (\VI_{\Gamma_j}-\VI_{\Gamma_{j+1}})u^j.
\end{equation}
A few remarks are in order. First, the index $\blev$ indicates that
the sets $\Gamma_j=\Gamma_j(\blev)$ depend on the choice of $\blev$,
although we usually simply write $\Gamma_j$ in order to keep the
notation succinct. Secondly, due to $|\blev|<\infty$ it holds
$$\max_{\bnu\in\CF}\lev_\bnu=:L<\infty$$ and thus $\Gamma_j=\emptyset$
for all $j>L$. Defining $\VI_{\emptyset}$ as the constant $0$
operator, the infinite series \eqref{eq:VIml} can also be written as
the finite sum
\begin{equation*}
  \VIml_\blev u = \sum_{j=1}^L(\VI_{\Gamma_j}-\VI_{\Gamma_{j+1}})u^j
  = \VI_{\Gamma_1}u^1+\VI_{\Gamma_2}(u^2-u^1)+\dots +\VI_{\Gamma_L}(u^L-u^{L-1}),
\end{equation*}
where we used $\VI_{\Gamma_{L+1}}=0$.  If we had
$\Gamma_1=\dots=\Gamma_L$, this sum would reduce to
$\VI_{\Gamma_L}u^L$, which is the interpolant of the approximation
$u^L$ at the (highest) discretization level $L$.  The main observation
of multilevel analyses is that it is beneficial not to choose all
$\Gamma_j$ equal, but instead to balance out the accuracy of the
interpolant $\VI_{\Gamma_j}$ (in the parameter) and the accuracy of
the approximation $u^j$ of $u$.

A multilevel sparse-grid quadrature algorithm is defined  \index{quadrature!multilevel $\sim$}
analogously via
\begin{equation}\label{eq:VQml}
  \VQml_\blev u :=
  \sum_{j\in\N}(\VQ_{\Gamma_j}-\VQ_{\Gamma_{j+1}})u^j,
\end{equation}
with $\Gamma_j=\Gamma_j(\blev)$ as in \eqref{eq:Gamma}.
In the following we will prove algebraic convergence rates of multilevel
interpolation and quadrature algorithms w.r.t.\ the $L^2(U,X;\gamma)$-norm and $X$, respectively. 
The convergence rates will hold in terms of the \emph{work} of computing
$\VIml_\blev$ and $\VQml_\blev$.

As mentioned above, for a level $\lev\in\N$, we interpret
$\sw{\lev}\in\N$ as a measure of the computational complexity of
evaluating $u^\lev$ at an arbitrary parameter $\by\in U$.  As
discussed in Section \ref{sec:NoFnEval}, computing $\VI_{\Gamma_j} u$
or $\VQ_{\Gamma_j}u$ 
requires to evaluate the function $u$ at each
parameter in the set $\pts(\Gamma_j)\subseteq U$ introduced in
\eqref{eq:pts}.  We recall the bound
\begin{equation*}
  |\pts(\Gamma_j)|\le \sum_{\bnu\in\Gamma_j} p_\bnu(1),
\end{equation*}
on the cardinality of this set obtained in \eqref{eq:ptsLbound}.  As
an upper bound of the work corresponding to the evaluation of all
functions required for the multilevel interpolant in \eqref{eq:VIml},
we obtain
\begin{equation} \label{work}
  \sum_{j\in\N}\sw{j}
  \left(\sum_{\bnu\in\Gamma_j(\blev)} p_\bnu(1) 
    + \sum_{\bnu\in\Gamma_{j+1}(\blev)} p_\bnu(1) \right).
\end{equation}
Since $\Gamma_{j+1}\subseteq \Gamma_{j}$, up the factor $2$ the work
of a sequence $\blev$ is defined by
\begin{equation}\label{eq:work}
  \work(\blev) := \sum_{j=1}^L \sw{j} \sum_{\bnu\in\Gamma_j(\blev)} p_\bnu(1)
  = \sum_{\bnu\in\CF(\blev)} p_\bnu(1) \sum_{j=1}^{l_\bnu}\sw{j},
\end{equation}
where we used the definition of $\Gamma_j(\blev)$ in \eqref{eq:Gamma},
$L:= \max_{\bnu\in\CF}\lev_\bnu < \infty$ and the finiteness of the
set
$$\CF(\blev):= \{\bnu \in \CF: \lev_\bnu > 0\}.$$

The efficiency of the multilevel interpolant critically relies on a
suitable choice of levels $\blev=(\lev_\bnu)_{\bnu\in\CF}$. This will
be achieved with the following algorithm, which constructs $\blev$
based on two collections of positive real numbers,
$(\seqi{\bnu})_{\bnu\in\CF}\in\ell^{q_1}(\CF)$ and
$(\seqii{\bnu})_{\bnu\in\CF}\in\ell^{q_2}(\CF)$. 
The algorithm is
justified due to Lemma \ref{LEMMA:MLWEIGHTNEW} 
which was shown in Section \ref{app:mlweight}. 
This technical lemma, which is a variant of \cite[Lemma 3.2.7]{JZdiss}, 
constitutes the central part of the proofs
of the convergence rate results presented in the rest of this section.
\begin{algorithm}[H]
  \caption{
    $(\lev_\bnu)_{\bnu\in\CF}={\rm
      ConstructLevels}((\seqi{\bnu})_{\bnu\in\CF},(\seqii{\bnu})_{\bnu\in\CF},q_1,{\alpha},\eps)$}
  \label{alg:levels}
  \begin{algorithmic}[1]
    \State $(\lev_{\bnu})_{\bnu\in\CF}\leftarrow (0)_{\bnu\in\CF}$
    \State
    $\Lambda_\eps\leftarrow\set{\bnu\in\CF}{\seqi{\bnu}^{-1}\ge \eps}$
    \For{$\bnu\in\Lambda_\eps$} \State
    $\delta \leftarrow \eps^{-\frac{1/2-q_1/4}{{\alpha}}}
    \seqii{\bnu}^{\frac{-1}{1+2{\alpha}}}
    \left(\sum_{\bmu\in\Lambda_\eps}
      \seqii{\bmu}^{\frac{-1}{1+2{\alpha}}}
    \right)^{\frac{1}{2{\alpha}}}$ \State
    $\lev_{\bnu}\leftarrow \max\set{j\in\N_0}{\sw{j}\le \delta}$
    \EndFor \State \Return $(\lev_{\bnu})_{\bnu\in\CF}$
  \end{algorithmic}
\end{algorithm}

Note that the determination of the sets $\Lambda_\eps$
in line 2 of Algorithm \ref{alg:levels} can be done with
Algorithm~\ref{alg:Lambda}.
\subsection{Construction of an allocation of discretization levels}
\label{app:mlweight}
We detail the construction of an allocation of discretization levels
along the coefficients of Wiener-Hermite PC expansion. 
It is
valid for collections $( u_\bnu )_{\bnu \in \cF}$ of Wiener-Hermite PC
expansion coefficients taking values in a separable Hilbert space, say
$X$, with additional regularity, being $X^s \subset X$, allowing for
weaker (weighted) summability of the $V^s$-norms
$( \|u_\bnu\|_{X^s})_{\bnu \in \cF}$.  In the setting of elliptic BVPs
with log-Gaussian diffusion coefficient, $X = V =H^1_0(\domain)$, and
$X^s$ is, for example, a weighted Kondrat'ev space in $\domain$ as
introduced in Section \ref{sec:KondrAn}.  We phrase the result and the
construction in abstract terms so that the allocation is applicable to
more general settings, such as the parabolic IBVP in Section
\ref{sec:LinParPDE}.

For a given, dense sequence $({X}_l)_{l\in \NN_0} \subset X$ of
nested, finite-dimensional subspaces and target accuracy
$0 < \eps \leq 1$, in the numerical approximation of Wiener-Hermite PC
expansions of random fields $u$ taking values in $X$, we consider
approximating the Wiener-Hermite PC expansion coefficients $u_\bnu$ in
$X$ from $X_l$.  The assumed density of the sequence
$(X_l)_{l\in \NN_0} \subset X$ in $X$ ensures that for
$u\in L^2(U,X;\gamma)$ the coefficients
$( u_\bnu )_{\bnu \in \cF} \subset X$ are square summable, in the sense
that $( \|u_\bnu\|_{X})_{\bnu \in \cF} \in \ell_2(\calF)$

The following lemma is a variation of \cite[Lemma 3.2.7]{JZdiss}. Its
proof, is, with several minor modifications, taken from \cite[Lemma
3.2.7]{JZdiss}.  We remark that the construction of the map
$\blev(\eps,\bnu)$, as described in the lemma, mimicks Algorithm
\ref{alg:levels}. Again, a convergence rate is obtained that is not
prone to the so-called ``curse of dimensionality'', being limited only
by the available sparsity in the coefficients of
Wiener-Hermite PC expansion for the parametric solution manifold.
\begin{lemma}\label{LEMMA:MLWEIGHTNEW}
  Let $\SW=\set{\sw{\lev}}{\lev\in\N_0}$ satisfy Assumption
  \ref{ass:SW}.  Let $q_1\in [0,2)$, ${q_2}\in [q_1,\infty)$ and
  ${\alpha}>0$.  Let
  \begin{enumerate}
  \item $(a_{j,\bnu})_{\bnu\in\CF}\subseteq [0,\infty)$ for every
    $j\in\N_0$,
  \item $(\seqi{\bnu})_{\bnu\in\CF}\subseteq (0,\infty)$ and
    $(\seqii{\bnu})_{\bnu\in\CF}\subseteq (0,\infty)$ be such that
		$$(\seqi{\bnu}^{-1/2})_{\bnu\in\CF}\in\ell^{q_1}(\CF) \;\mbox{and}\;
		\big(\seqii{\bnu}^{-1/2}p_\bnu(1/2+\alpha)\big)_{\bnu\in\CF}\in\ell^{{q_2}}(\CF),
		$$
              \item
		\begin{equation}\label{eq:sumsleinfty}
                  \sup_{j\in\N_0} \left(\sum_{\bnu\in\CF}a_{j,\bnu}^2\seqi{\bnu}\right)^{1/2}=:C_1<\infty,\qquad
                  \sup_{j\in\N_0} \left(\sum_{\bnu\in\CF}(\sw{j}^{{\alpha}}a_{j,\bnu})^2\seqii{\bnu}\right)^{1/2} 
                  =:C_2<\infty.
		\end{equation}
              \end{enumerate}
              For every $\eps>0$ define
              $\Lambda_\eps =
              \set{\bnu\in\CF}{\seqi{\bnu}^{-1}\ge\eps}$,
              $\wk{\eps,\bnu}:= 0$ for all
              $\bnu\in\CF\backslash\Lambda_\eps$, and define
              \begin{equation*}\label{eq:wepsnuweight}
                \wk{\eps,\bnu}:= \left\lfloor
                  \eps^{-\frac{1/2 - q_1/4}{{\alpha}}}
                  \seqii{\bnu}^{\frac{-1}{1+2{\alpha}}}
                  \Bigg(\sum_{\bmu\in\Lambda_\eps}
                  \seqii{\bmu}^{\frac{-1}{1+2{\alpha}}}
                  \Bigg)^{\frac{1}{2{\alpha}}}\right\rfloor_{\SW}\in\SW\qquad\forall\bnu\in\Lambda_\eps.
              \end{equation*}
              Furthermore, for every $\eps>0$ and $\bnu\in\CF$ let
              $\lev_{\eps,\bnu}\in\N_0$ be the corresponding
              discretization level, i.e., 
              $\wk{\eps,\bnu}=\sw{\lev_{\eps,\bnu}}$, and define the
              maximal discretization level
              $$L(\eps):=\max \set{{\lev_{\eps,\bnu}}}{\bnu\in\CF}.$$
              Denote $\blev_\eps=({\lev_{\eps,\bnu}})_{\bnu\in\CF}$.
		 
              Then there exists a constant $C>0$ and tolerances
              $\eps_n\in (0,1]$ such that for every $n \in \NN$ holds
              $\work(\blev_{\eps_n}) \le n$ and
              \begin{equation*}\label{eq:mlestweight}
                \sum_{\bnu\in\CF}\sum_{j={\lev_{\eps_n,\bnu}}}^{L(\eps_n)} a_{j,\bnu}    
                \le C (1+\log n )n^{-R},
              \end{equation*}
              where the rate $R$ is given by
              \begin{equation*}
                R = \min\left\{{\alpha},\frac{{\alpha}(q_1^{-1}-1/2)}{{\alpha}+q_1^{-1}-{q_2}^{-1}}\right\}.
              \end{equation*}
            \end{lemma}
            \begin{proof}
              Throughout this proof denote $\delta:= 1/2 - q_1/4 >0$.
              In the following
              \begin{equation*}
                \twk{\eps,\bnu}:= 
                \eps^{-\frac{\delta}{{\alpha}}}
                \seqii{\bnu}^{\frac{-1}{1+2{\alpha}}}
                \Bigg(\sum_{\bmu\in\Lambda_\eps}
                \seqii{\bmu}^{\frac{-1}{1+2{\alpha}}}
                \Bigg)^{\frac{1}{2{\alpha}}} \qquad\forall\bnu\in\Lambda_\eps,
              \end{equation*}
              i.e.\
              $\wk{\eps,\bnu}=\lfloor \twk{\eps,\bnu}\rfloor_\SW$.
              Note that $0<\twk{\eps,\bnu}$ is well-defined for all
              $\bnu\in\Lambda_\eps$ since $\seqii{\bnu}>0$ for all
              $\bnu\in\CF$ by assumption.  Due to Assumption
              \ref{ass:SW} \ref{item:SW4} it holds
              \begin{equation}\label{eq:sw1Kw}
                \frac{\twk{\eps,\bnu}}{\KW}\le 1+\wk{\eps,\bnu}\le 
                1+\twk{\eps,\bnu} 
                \qquad \forall \bnu\in\Lambda_\eps.
              \end{equation}
              Since
              $(\seqi{\bnu}^{-1/2})_{\bnu\in\CF}\in\ell^{q_1}(\CF)$
              and \eqref{eq:sumsleinfty}, we get
		$$
		\sum_{\bnu\in\CF\backslash\Lambda_{\eps}} a_{j,\bnu}
                \le \Bigg( \sum_{\bnu\in\CF\backslash\Lambda_{\eps}}
                a_{j,\bnu}^2 c_{\bnu}\Bigg)^{1/2} \Bigg(
                \sum_{\bnu\in\CF\backslash\Lambda_{\eps}}
                c_{\bnu}^{-1}\Bigg)^{1/2} \leq C_1 \Bigg(
                \sum_{\seqi{\bnu}^{-1}\leq\eps}
                c_{\bnu}^{-\frac{q_1}{2}}c_{\bnu}^{\frac{q_1}{2}-1}\Bigg)^{1/2}
                \le C \eps^{\delta}
		$$
		with the constant $C$ independent of $j$ and
                $\eps$. Thus,
		\begin{equation} \label{sum_notLambda}
                  \sum_{\bnu\in\CF\backslash\Lambda_\eps}
                  \sum_{j=0}^{L(\eps)} a_{j,\bnu} =
                  \sum_{j=0}^{L(\eps)}\sum_{\bnu\in\CF\backslash\Lambda_\eps}a_{j,\bnu}
                  \le C_1 (1+L(\eps))\eps^{\delta}.
		\end{equation}
		Next with $C_2$ as in \eqref{eq:sumsleinfty},
		\begin{align}\label{eq:sumsuma2jbnu}
                  \sum_{\bnu\in\Lambda_\eps}\sum_{j= {\lev_{\eps,\bnu}}}^{L(\eps)} a_{j,\bnu}
                  &= \sum_{\bnu\in\Lambda_\eps}
                    \sum_{j={\lev_{\eps,\bnu}}}^{L(\eps)} a_{j,\bnu}
                    \sw{j}^{\alpha} \sw{j}^{-{\alpha}}
                    \seqii{\bnu}^{1/2} \seqii{\bnu}^{-1/2}\nonumber\\
                  &\le\Bigg(\sum_{\bnu\in\Lambda_\eps}\sum_{j=0}^{L(\eps)}\big(a_{j,\bnu}\sw{j}^{\alpha} \seqii{\bnu}^{1/2}\big)^2 \Bigg)^{\frac{1}{2}} 
                    \Bigg( \sum_{\bnu\in\Lambda_\eps}\sum_{j\ge {\lev_{\eps,\bnu}}}
                    \big(\seqii{\bnu}^{-1/2}
                    \sw{j}^{-{\alpha}}\big)^{2}\Bigg)^{\frac{1}{2}}\nonumber\\
                  &\le  C_2 (1+L(\eps))
                    \Bigg( \sum_{\bnu\in\Lambda_\eps}\sum_{j\ge {\lev_{\eps,\bnu}}}
                    \big(\seqii{\bnu}^{-1/2}
                    \sw{j}^{-{\alpha}}\big)^{2}\Bigg)^{\frac{1}{2}}.
		\end{align}
		Assumption \ref{ass:SW} \ref{item:SW3} implies for
                some $C_3$
		\begin{equation*}
                  \sum_{j\ge {\lev_{\eps,\bnu}}} \sw{j}^{-2{\alpha}}\le {C_3^2}
                  (1+\sw{{\lev_{\eps,\bnu}}})^{-2{\alpha}}= {C_3^2}
                  (1+\wk{\eps,\bnu})^{-2{\alpha}},
		\end{equation*}
		so that by \eqref{eq:sw1Kw} and
                \eqref{eq:sumsuma2jbnu}
		\begin{align}\label{eq:firstsummlweight2}
                  \sum_{\bnu\in\Lambda_\eps}\sum_{j= {\lev_{\eps,\bnu}}}^{L(\eps)} a_{j,\bnu}&\le
                                                                                               C_3C_2(1+L(\eps))
                                                                                               \left( \sum_{\bnu\in\Lambda_\eps}
                                                                                               \big(\seqii{\bnu}^{-1/2}(1+\wk{\eps,\bnu})^{-{\alpha}}\big)^{2}\right)^{\frac{1}{2}}\nonumber\\
                                                                                             &\le  C_3C_2{\KW^{\alpha}}(1+L(\eps))
                                                                                               \left( \sum_{\bnu\in\Lambda_\eps}
                                                                                               \big(\seqii{\bnu}^{-1/2}\twk{\eps,\bnu}^{-{\alpha}}\big)^{2}\right)^{\frac{1}{2}}.
		\end{align}
		Inserting the definition of $\twk{\eps,\bnu}$, 
                we have
		\begin{align}\label{eq:firstsummlweight3}
                  \left( \sum_{\bnu\in\Lambda_\eps}
                  \big(\seqii{\bnu}^{-1/2}\twk{\eps,\bnu}^{-{\alpha}}\big)^{2}\right)^{\frac{1}{2}}
                  &=
                    \eps^{\delta}
                    \Bigg(\sum_{\bmu\in\Lambda_\eps}
                    \seqii{\bmu}^{\frac{-1}{1+2{\alpha}}}
                    \Bigg)^{-{\alpha} \frac{1}{2{\alpha}}} 
                    \left(\sum_{\bnu\in\Lambda_\eps}
                    \seqii{\bnu}^{-1}\seqii{\bnu}^{\frac{2{\alpha}}{1+2{\alpha}} }  \right)^{\frac{1}{2}}
                    =
                    \eps^{\delta},
		\end{align}
		where we used
		\begin{equation*}
                  -1 + \frac{2{\alpha}}{1+2{\alpha}} =
                  \frac{-(1+2{\alpha})+2{\alpha}}{1+2{\alpha}}
                  = \frac{-1}{1+2{\alpha}}.
		\end{equation*}
		Using Assumption \ref{ass:SW} \ref{item:SW2} and the
                definition of $\work(\blev_\eps)$ in \eqref{eq:work}
                we get
		\begin{equation}\label{eq:Jeps}
                  L(\eps)
                  \le \log(1+\max_{\bnu\in\CF}\wk{\eps,\bnu})
                  \le\log(1+\work(\blev_\eps)).
		\end{equation}
		Hence, \eqref{sum_notLambda},
                \eqref{eq:firstsummlweight2},
                \eqref{eq:firstsummlweight3} {and \eqref{eq:Jeps}}
                yield
		\begin{equation}\label{eq:firstcasesn1est}
                  \sum_{\bnu\in\CF}\sum_{j={\lev_{\eps,\bnu}}}^{L(\eps)}a_{j,\bnu}=
                  \sum_{\bnu\in\Lambda_\eps}\sum_{j={\lev_{\eps,\bnu}}}^{L(\eps)} a_{j,\bnu}+
                  \sum_{\bnu\in\CF\backslash\Lambda_\eps} \sum_{j=0}^{L(\eps)}a_{j,\bnu}\le
                  C \big(1+\log(\work(\blev_\eps))\big) \eps^{\delta}.
		\end{equation}
		Next, we compute an upper bound for
                $\work(\blev_\eps)$. By definition of
                $\work(\blev_\eps)$ in \eqref{eq:work}, and using
                Assumption \ref{ass:SW} \ref{item:SW1} as well as
                $\wk{\eps,\bnu}=\sw{{\lev_{\eps,\bnu}}}$,
		\begin{align}\label{eq:sn1vwk}
                  \work(\blev_\eps)
                  &=                          \sum_{\bnu\in\Lambda_\eps}p_\bnu(1)\sum_{\set{
                    j\in\N}{j\le{\lev_{\eps,\bnu}}}}\sw{j}
                    \le \sum_{\bnu\in\Lambda_\eps}p_\bnu(1)\KW\wk{\eps,\bnu}\nonumber\\
                  &\le \KW \sum_{\bnu\in\Lambda_\eps}p_\bnu(1) \twk{\eps,\bnu}
                    \le \KW
                    \eps^{-\frac{\delta}{{\alpha}}} \left(\sum_{\bnu\in\Lambda_\eps}
                    p_\bnu(1)\seqii{\bnu}^{\frac{-1}{1+2{\alpha}}}\right)^{\frac{1}{2{\alpha}}+1}\nonumber\\
                  &=\KW                  \eps^{-\frac{\delta}{{\alpha}}} \left(\sum_{\bnu\in\Lambda_\eps} \left(p_\bnu(1/2+{\alpha})\seqii{\bnu}^{-1/2}\right)^{\frac{2}{1+2{\alpha}}}\right)^{\frac{1}{2{\alpha}}+1},
		\end{align}
		where we used
                $p_\bnu(1)=p_\bnu(1/2+{\alpha})^{2/(1+2{\alpha})}$ and
                the fact that $p_\bnu(1)\ge 1$ for all $\bnu$.

                We distinguish between the two cases
		\begin{equation*}\label{eq:2casessn1}
                  \frac{2}{1+2{\alpha}}\ge {q_2}\qquad \text{and} \qquad
                  \frac{2}{1+2{\alpha}}< {q_2}.
		\end{equation*}
		In the first case, since
                $(p_\bmu(1/2+{\alpha})\seqii{\bmu}^{-1/2})_{\bmu\in\CF}\in\ell^{{q_2}}(\CF)$,
                {\eqref{eq:sn1vwk} implies}
		\begin{equation}\label{eq:firstcasesn1w}
                  \work(\blev_\eps)\le C \eps^{-\frac{\delta}{{\alpha}}}
		\end{equation}
		and hence,
                $\log(\work(\blev_\eps)) \le \log
                \big(C\eps^{-\frac{\delta}{{\alpha}}}\big)$.  Then
                \eqref{eq:firstcasesn1est} together with
                \eqref{eq:firstcasesn1w} implies
		\begin{equation*}
                  \sum_{\bnu\in\CF}\sum_{j={\lev_{\eps,\bnu}}}^{L(\eps)} a_{j,\bnu}
                  \le C(1+{|\log(\eps^{-1})|})\eps^\delta .
		\end{equation*}
		For every $n \in \NN$, we can find $\eps_n>0$ such
                that
                $\frac{n}{2} \le C\eps_n^{-\frac{\delta}{{\alpha}}}
                \le n$.  Then the claim of the corollary in the case
                $\frac{2}{1+2{\alpha}} \ge q_2$ holds true for the
                chosen $\eps_n$.

		Finally, let us address the case
                $\frac{2}{1+2{\alpha}}<{q_2}$. Then, by
                \eqref{eq:sn1vwk} and using H\"older's inequality with
                $ {q_2}\frac{1+2{\alpha}}{2}>1$ we get
		\begin{equation*}
                  \work(\blev_\eps)\le
                  \KW \eps^{-\frac{\delta}{{\alpha}}}\norm[\ell^{{q_2}}(\CF)]{(p_\bnu(1/2+{\alpha})\seqii{\bnu}^{-1/2})_{\bnu\in\CF}}^{\frac{1}{{\alpha}}}
                  |\Lambda_\eps|^{\big(1-\frac{2}{q_2(1+2\alpha)}\big)\frac{1+2{\alpha}}{2{\alpha}}}. 
		\end{equation*}
		Since
		$$
		|\Lambda_\eps|=\sum_{\bnu \in \Lambda_\eps}1 =
                \sum_{\seqi{\bnu}^{-1}\geq\eps}
                c_{\bnu}^{-\frac{q_1}{2}}c_{\bnu}^{\frac{q_1}{2}} \leq
                C \eps^{-\frac{q_1}{2}},
		$$
		we obtain
		\begin{align*}
                  \work(\blev_\eps)\le
                  \KW \eps^{-\frac{\delta}{{\alpha}}-\frac{q_1}{2}(1-\frac{2}{q_2(1+2\alpha)})\frac{1+2{\alpha}}{2{\alpha}}} \le C \eps^{-\frac{q_1}{2\alpha}\big(\alpha-\frac{1}{q_2}+\frac{1}{q_1}\big)} .
		\end{align*}
		For every $n \in \NN$, we can find $\eps_n>0$ such
                that
                $$ \frac{n}{2} \le C
                \eps_n^{-\frac{q_1}{2\alpha}\big(\alpha-\frac{1}{q_2}+\frac{1}{q_1}\big)}
                \le n.$$ Thus the claim also holds true in the case
                $\frac{2}{1+2{\alpha}}<q_2$.
              \end{proof}
              
              \subsection{Multilevel Smolyak sparse-grid interpolation algorithm} \index{interpolation!multilevel $\sim$}
              \label{sec:MLInterpol}
              We are now in position to formulate a multilevel Smolyak sparse-grid
              interpolation convergence theorem.  To this end, we
              observe that our proofs of approximation rates have been
              constructive: rather than being based on a best $N$-term
              selection from the infinite set of Wiener-Hermite PC
              expansion coefficients, a constructive selection process
              of ``significant'' Wiener-Hermite PC expansion
              coefficients, subject to a given prescribed
              approximation tolerance, has been provided.  In the
              present section, we turn this into a concrete, numerical
              selection process with complexity bounds. In particular,
              we provide an \emph{a-priori allocation} of
              discretization levels to Wiener-Hermite PC expansion
              coefficients.  This results on the one hand in an
              explicit, algorithmic definition of a family of
              multilevel interpolants which is parametrized by an
              approximation threshold $\eps>0$.  On the other hand, it
              will result in \emph{mathematical convergence rate
                bounds in terms of computational work} rather than in
              terms of, for example, number of active Wiener-Hermite
              PC expansion coefficients, which rate bounds are free
              from the curse of dimensionality.

              The idea is as follows: let
              $\bb_1=(b_{1,j})_{j\in\N}\in\ell^{p_1}(\NN)$,
              $\bb_2=(b_{2,j})_{j\in\N}\in\ell^{p_2}(\NN)$, and $\xi$
              be the two sequences and constant from Assumption
              \ref{ass:ml}.  For two constants $K>0$ and $r>3$ (which
              are still at our disposal and which will be specified
              below), set for all $j\in\N$
              \begin{equation}\label{eq:bvarrhoi}
                \varrho_{1,j}
                :=
                b_{1,j}^{p_1-1} \frac{\xi}{4\norm[\ell^{p_1}]{\bb_1}},
                \qquad 
                \varrho_{2,j}:=b_{2,j}^{p_2-1} \frac{\xi}{4\norm[\ell^{p_2}]{\bb_2}}.
              \end{equation}
              We let for all $\bnu\in\CF$ (as in Lemma \ref{lemma:cnu}
              for $k=1$ and with $\tau=3$)
              \begin{equation}\label{eq:seqiseqii}
                \seqi{\bnu}:=\prod_{j\in\N}\max\{1,K\varrho_{1,j}\}^2\nu_j^{r-3},\qquad
                \seqii{\bnu}:=\prod_{j\in\N}\max\{1,K\varrho_{2,j}\}^2\nu_j^{r-3}.
              \end{equation}
              Based on those two multi-index collections,
              Algorithm \ref{alg:levels} provides a collection of
              discretization levels which sequence depends on $\eps>0$
              and is indexed over $\CF$.  We denote it by
              $\blev_\eps=(\lev_{\eps,\bnu})_{\bnu\in\CF}$.  We now
              state an upper bound for the error of the corresponding
              multilevel interpolants in terms of the work measure in
              \eqref{eq:work} as $\eps\to 0$.
              \begin{theorem}\label{thm:mlint}
                Let $u\in L^2(U,X;\gamma)$ and
                $u^\lev\in L^2(U,X;\gamma)$, $\lev\in\N$, satisfy
                Assumption \ref{ass:ml} with some constants
                ${\alpha}>0$ and $0<p_1<2/3$ and $p_1\le p_2<1$.  Set
                $q_1:=p_1/(1-p_1)$.  Assume that
                $r>2(1+({\alpha}+1)q_1)/q_1+3$ (for $r$ as defined in
                \eqref{eq:seqiseqii}).  There exist constants $K>0$
                (in \eqref{eq:seqiseqii}) and $C>0$ such that the
                following holds.
  
                For every $n \in \NN$, there are positive constants
                $\eps_n\in (0,1]$ such that
                $\work(\blev_{\eps_n}) \le n$ and with
                $\blev_{\eps_n}=({\lev_{\eps_n,\bnu}})_{\bnu\in\CF}$
                as defined in Lemma \ref{LEMMA:MLWEIGHTNEW} (where
                $c_\bnu$, $d_\bnu$ as in \eqref{eq:seqiseqii}) it
                holds
                \begin{equation*}
                  \norm[L^2(U,X;\gamma)]{u-\VI_{\blev_{\eps_n}}^{\rm ML}u} \	\le \ C (1+\log n )n^{-R}
                \end{equation*}
                with the convergence rate
                \begin{equation} \label{R-I}
                  R := \min\left\{{\alpha},\frac{{\alpha}(p_1^{-1}-3/2)}{{\alpha}+p_1^{-1}-p_2^{-1}}\right\}.
                \end{equation}
              \end{theorem}
              \begin{proof}
                Throughout this proof we write
                $\bb_1=(b_{1,j})_{j\in\N}$ and
                $\bb_2=(b_{2,j})_{j\in\N}$ for the two sequences in
                Assumption \ref{ass:ml}.  
                We observe that $\Gamma_j$
                defined in \eqref{eq:Gamma} is downward closed for all
                $j\in\N_0$.  This can be easily deduced from the fact
                that the multi-index collections
                $(\seqi{\bnu})_{\bnu\in\CF}$ 
                and
                $(\seqii{\bnu})_{\bnu\in\CF}$ are monotonically
                increasing (i.e., e.g., $\bnu\le\bmu$ implies
                $\seqi{\bnu}\le\seqi{\bmu}$) and the definition of
                $\Lambda_\eps$ and $\lev_{\eps,\bnu}$ in Algorithm
                \ref{alg:levels}. We will use this fact throughout the
                proof, without mentioning it at every instance.

                {\bf Step 1.}  Given $n\in \N$, we choose
                $\eps:=\eps_n$ as in Lemma \ref{LEMMA:MLWEIGHTNEW}.
                Fix $N\in\N$ such that
                $$
                N>\max \{j:\, j \in\supp(\bnu),\, \lev_{\eps,\bnu}>  0\}
                $$ and so large that
                \begin{equation}\label{eq:truncerr_mlint}
                  \norm[L^2(U,X;\gamma)]{u-\tilde u_N}\le  n^{-R},
                \end{equation}
                where $\tilde u_N:U\to X$ is as in
                Definition~\ref{def:bdXHol}. This is possible due to
  $$\lim_{N\to\infty}\norm[L^2(U,X;\gamma)]{u-\tilde u_N}=0,$$ which
  holds by the $(\bb_1,\xi,\delta,X)$-holomorphy of $u$.  By
  Assumption \ref{ass:ml}, for every $j \in\N$ the function
  $e^j:=u-u^j\in L^2(U,X;\gamma)$ is
  $(\bb_1,\xi,\delta,X)$-holomorphic and
  $(\bb_2,\xi,\delta\sw{j}^\gamma,X)$-holomorphic. 
  For notational
  convenience we set $e^0:=u-0=u\in L^2(U,X;\gamma)$, 
  so that $e^0$ is
  $(\bb_1,\xi,\delta,X)$-holomorphic and
  $(\bb_2,\xi,\delta,X)$-holomorphic. 
  Hence, for every $j\in\N_0$ there
  exists a function $\tilde e_N^j=\tilde u_N-\tilde u_N^j$ 
  as in Definition~\ref{def:bdXHol} \ref{item:vN}.

  In the rest of the proof we use the following facts:
  \begin{enumerate}
  \item\label{item:u-uswunifabs} By Lemma \ref{lemma:uN}, for every
    $j\in\N_0$, with the Wiener-Hermite PC expansion coefficients
    $$\tilde e_{N,\bnu}^j:= \int_{U} H_\bnu(\by)\tilde e_{N}^j(\by)
    \dd\gamma(\by),$$ it holds
    \begin{equation*}\label{eq:u-uswunifabs}
      \tilde e_{N}^j(\by) 
    = \sum_{\bnu\in\CF} \tilde e_{N,\bnu}^j H_\bnu(\by)\qquad\forall \by\in U,
    \end{equation*}
    with pointwise absolute convergence.
  \item\label{item:aux} By Lemma \ref{lemma:cnu}, upon choosing $K>0$
    in \eqref{eq:seqiseqii} large enough, and because $r>3$,
    \begin{equation*}
      C_0 \seqi{\bnu} p_\bnu(3)\le \beta_\bnu(r,\bvarrho_1),\qquad
      C_0 \seqii{\bnu} p_\bnu(3)\le \beta_\bnu(r,\bvarrho_2)
      \qquad \forall \bnu\in \CF_1.
    \end{equation*}
    We observe that by definition of $\bvarrho_i$, $i\in\{1,2\}$, 
    in \eqref{eq:bvarrhoi}, it holds
    $\varrho_{i,j}\sim b_{i,j}^{-(1-p_i)}$ and therefore
    $(\varrho_{i,j}^{-1})_{j\in\N}\in\ell^{q_i}(\N)$ with
    $q_i:=p_i/(1-p_i)$, $i\in\{1,2\}$.
  \item\label{item:seqiseqiisum} Due to
    $r>2(1+({\alpha}+1)q_1)/q_1+3$,
    the condition of Lemma \ref{lemma:summabcnu} is satisfied (with
    $k=1$, $\tau=3$ and $\theta=({\alpha}+1)q_1$). Hence the lemma
    gives
    \begin{equation*}
      \sum_{\bnu\in\CF} p_\bnu(({\alpha}+1)q_1)\seqi{\bnu}^{-q_1/2}<\infty
      \qquad
      \Rightarrow\qquad
      (p_\bnu({\alpha}+1)\seqi{\bnu}^{-1/2})_{\bnu\in\CF}\in\ell^{q_1}(\CF)
    \end{equation*}
    and similarly
    \begin{equation*}
      \sum_{\bnu\in\CF} p_\bnu(({\alpha}+1)q_2)\seqii{\bnu}^{-q_2/2}<\infty
      \qquad
      \Rightarrow\qquad
      (p_\bnu({\alpha}+1)\seqii{\bnu}^{-1/2})_{\bnu\in\CF}\in\ell^{q_2}(\CF).
    \end{equation*}
  \item\label{item:tedecay} By Theorem~\ref{thm:bdHolSum} and item
    \ref{item:aux}, for all $j\in\N_0$
    \begin{equation*}
      C_0 \sum_{\bnu\in\CF}\seqi{\bnu}\norm[X]{\tilde e_{N,\bnu}^j}^2p_\bnu(3)\le    
      \sum_{\bnu\in\CF}\beta_\bnu(r,\bvarrho_1)\norm[X]{\tilde e_{N,\bnu}^j}^2 \le
      C \delta^2
    \end{equation*}
    and
    \begin{equation*}
      C_0 \sum_{\bnu\in\CF}\seqii{\bnu}\norm[X]{\tilde e_{N,\bnu}^j}^2p_\bnu(3) \le
      \sum_{\bnu\in\CF}\beta_\bnu(r,\bvarrho_2)\norm[X]{\tilde e_{N,\bnu}^j}^2 \le
      C \frac{\delta^2}{\sw{j}^{2{\alpha}}},
    \end{equation*}
    with the constant $C$ independent of $j$, $\sw{j}$ and $N$.
  \item\label{item:IGammae} Because
    $N\ge\max\set{j\in\supp(\bnu)}{\lev_{\eps,\bnu}\ge 0}$ and
    $\chi_{0,0}=0$ we have
    $$\VI_{\Gamma_j}(u-u^j)=\VI_{\Gamma_j}e^j= \VI_{\Gamma_j}\tilde
    e_N^j$$ for all $j\in\N$ (cp.~Remark~\ref{rmk:defu}). Similarly
    $\VI_{\Gamma_j}u=\VI_{\Gamma_j}\tilde u_N$ for all $j\in\N$.
  \end{enumerate}

  {\bf Step 2.} Observe that $\Gamma_j=\emptyset$ for all
  $j>L(\eps):= \max_{\bnu\in\CF}{\lev_{\eps,\bnu}}$
  (cp.~\eqref{eq:Gamma}), which is finite due to
  $|\blev_\eps|<\infty$.  With the conventions
  $\VI_{\Gamma_0}=\VI_{\CF}={\rm Id}$ (i.e.~$\VI_{\Gamma_0}$ is the
  identity) and $\VI_{\emptyset} \equiv 0$ this implies
  \begin{equation*}
    u = \VI_{\Gamma_0}u =
    \sum_{j=0}^{L(\eps)} (\VI_{\Gamma_j}-\VI_{\Gamma_{j+1}}) u
    = (\VI_{\Gamma_0}-\VI_{\Gamma_1})u+\dots+
    (\VI_{\Gamma_{L(\eps)-1}}-\VI_{\Gamma_{L(\eps)}})u+\VI_{\Gamma_{L(\eps)}}u.
  \end{equation*}
  By definition of the multilevel interpolant in \eqref{eq:VIml}
  \begin{equation*}
    \VIml_{\blev_\eps} u =
    \sum_{j=1}^{L(\eps)}(\VI_{\Gamma_j}-\VI_{\Gamma_{j+1}})u^j
    =(\VI_{\Gamma_1}-\VI_{\Gamma_2}) u^1+\dots+
    (\VI_{\Gamma_{L(\eps)-1}}-\VI_{\Gamma_{L(\eps)}})   
     u^{L(\eps)-1} + \VI_{\Gamma_{L(\eps)}}u^{L(\eps)}.
  \end{equation*}
  By item \ref{item:IGammae} of Step 1, we can write
  \begin{equation*}
    (\VI_{\Gamma_0}-\VI_{\Gamma_1})u =
    u-\VI_{\Gamma_1}u = u-\VI_{\Gamma_1}\tilde u_N
    = (u-\tilde u_N)+(\VI_{\Gamma_0}-\VI_{\Gamma_1})\tilde u_N
    = (u-\tilde u_N)+(\VI_{\Gamma_0}-\VI_{\Gamma_1})\tilde e_N^0,
  \end{equation*}
  where in the last equality we used $e_N^0=u_N$, by definition of
  $e^0=u$ (and $\tilde e_N^0=\tilde u_N\in L^2(U,X;\gamma)$ as in
  Definition~\ref{def:bdXHol}).  Hence, again by item \ref{item:IGammae},
  \begin{align} \label{eq:u-Iwu}
    u-\VIml_{\blev_\eps}u &=(\VI_{\Gamma_0}-\VI_{\Gamma_1}) u+
                            \sum_{j=1}^{L(\eps)}(\VI_{\Gamma_j}-\VI_{\Gamma_{j+1}}) (u-u^j)
                           \nonumber\\
                          &=(u-\tilde u_N)+(\VI_{\Gamma_0}-\VI_{\Gamma_1})\tilde u_N+\sum_{j=1}^{L(\eps)}(\VI_{\Gamma_j}-\VI_{\Gamma_{j+1}})\tilde e_N^j
                           \nonumber\\
                          &=(u-\tilde u_N)+\sum_{j=0}^{L(\eps)}(\VI_{\Gamma_j}-\VI_{\Gamma_{j+1}})\tilde e_N^j.
\end{align}                          
We will use this representation to bound the norm $\norm[L^2(U,X;\gamma)]{u-\VIml_{\blev_\eps}u}$. From item \ref{item:u-uswunifabs} of Step 1 it follows that
  for every $j\in\N_0$ 
 $$\tilde e_{N}^j(\by)=\sum_{\bnu\in\CF} \tilde e_{N,\bnu}^j
 H_\bnu(\by),$$ with the equality and unconditional convergence in the space $X$ for all
 $\by\in U$.   
Therefore, by the same argument as in the proof of Lemma \ref{lemma:L^2-convergence}, we can prove  that 
\begin{equation}\label{eq-kien-02}
	(\VI_{\Gamma_j}-\VI_{\Gamma_{j+1}})\tilde e_N^j = \sum_{\bnu\in\cF}\tilde e_{N,\bnu}^j
	(\VI_{\Gamma_j}-\VI_{\Gamma_{j+1}}) H_\bnu
\end{equation}
with equality and unconditional convergence in the space $L^2(\R^N,X;\gamma_N)$. 

Using \eqref{eq:u-Iwu} and
  $$
     (\VI_{\Gamma_j}-\VI_{\Gamma_{j+1}})H_\bnu = 0
  $$ 
  for all $\bnu\in\Gamma_{j+1}\subseteq\Gamma_j$ by Lemma \ref{lemma:VIprop}, 
  we get
  \begin{equation}\label{eq:u-vimlu}
    \norm[L^2(U,X;\gamma)]{u-\VIml_{\blev_\eps}u}\le
    \norm[L^2(U,X;\gamma)]{u-\tilde u_N}+
    \sum_{\bnu\in\CF}\sum_{j={\lev_{\eps,\bnu}}}^{L(\eps)} \norm[X]{\tilde e_{N,\bnu}^j}\norm[L^2(U;\gamma)]{(\VI_{\Gamma_j}-\VI_{\Gamma_{j+1}})H_\bnu}.
  \end{equation}
  
  {\bf Step 3.}   
  We wish to apply Lemma \ref{LEMMA:MLWEIGHTNEW} to the
  bound \eqref{eq:u-vimlu}. By \eqref{eq:L2boundLambda}, we have for
  all $\bnu\in\CF$
  \begin{equation*}
    \norm[L^2(U;\gamma)]{(\VI_{\Gamma_j}-\VI_{\Gamma_{j+1}})H_\bnu}
    \le \norm[L^2(U;\gamma)]{\VI_{\Gamma_j}H_\bnu}+
    \norm[L^2(U;\gamma)]{\VI_{\Gamma_{j+1}}H_\bnu}
    \le 2 p_\bnu(3).
  \end{equation*}
Note  these inequalities also hold when $j=0$, 
that is when  $\VI_{\Gamma_0}={\rm Id}$. 
  By items \ref{item:seqiseqiisum} and \ref{item:tedecay} of Step 1,
  the collections $(a_{j,\bnu})_{\bnu\in\CF}$, $j\in\N_0$, 
  and
  $(\seqi{\bnu})_{\bnu\in\CF}$, $(\seqii{\bnu})_{\bnu\in\CF}$, 
  satisfy the assumptions of Lemma \ref{LEMMA:MLWEIGHTNEW}.  
  Therefore,
  \eqref{eq:u-vimlu}, \eqref{eq:truncerr_mlint} and Lemma
  \ref{LEMMA:MLWEIGHTNEW} give
  \begin{equation*}
    \norm[L^2(U,X;\gamma)]{u-\VIml_{\blev_{\eps_n}}u} 
    \le
    n^{-R} + \sum_{\bnu\in\CF}\sum_{j={\lev_{\eps_n,\bnu}}}^{L(\eps_n)} a_{j,\bnu}
    \le 
    C (1+\log n )n^{-R},
  \end{equation*}
  with the convergence rate
  \begin{equation*}
    R = \min\left\{{\alpha},\frac{{\alpha} (q_1^{-1}-1/2)}{{\alpha} + q_1^{-1}-q_2^{-1}}
    \right\} =
    \min\left\{{\alpha},\frac{{\alpha} (p_1^{-1}-3/2)}{{\alpha} + p_1^{-1}-p_2^{-1}}
    \right\},
  \end{equation*}
  where we used $q_1=p_1/(1-p_1)$ and $q_2=p_2/(1-p_2)$ as stated in
  item \ref{item:aux} of Step 1.
\end{proof}
\subsection{Multilevel Smolyak sparse-grid quadrature algorithm}
\label{sec:MLQuad}
We next formulate an analog of Theorem~\ref{thm:mlint} 
for a multilevel Smolyak sparse-grid quadrature algorithm.
First, the definition of the multi-index sets in
\eqref{eq:seqiseqii} (which are used to construct the quadrature via
Algorithm \ref{alg:levels}) has to be slightly adjusted.  
Then, we state and prove a convergence 
rate result for the corresponding algorithm.  
Its proof is along the lines of the proof of Theorem~\ref{thm:mlint}.

Let 
$\bb_1=(b_{1,j})_{j\in\N}\in\ell^{p_1}(\NN)$,
$\bb_2=(b_{2,j})_{j\in\N}\in\ell^{p_2}(\NN)$, 
and $\xi$ 
be the two sequences and the constant from Assumption \ref{ass:ml}. 
For two
constants $K>0$ and $r>3$, which are still at our disposal and which
will be defined below, we set for all $j\in\N$
\begin{equation}\label{eq:bvarrhoiquad}
  \varrho_{1,j}:=b_{1,j}^{p_1-1} \frac{\xi}{4\norm[\ell^{p_1}]{\bb_1}},\qquad \varrho_{2,j}:=b_{2,j}^{p_2
    -1} \frac{\xi}{4\norm[\ell^{p_2}]{\bb_2}}.
\end{equation}
Furthermore, we let for all $\bnu\in\CF$ (as in Lemma \ref{lemma:cnu}
for $k=2$ and with $\tau=3$)
\begin{equation}\label{eq:seqiseqiiquad}
  \seqi{\bnu}:=\prod_{j\in\N}\max\{1,K\varrho_{1,j}\}^4\nu_j^{r-3},\qquad
  \seqii{\bnu}:=\prod_{j\in\N}\max\{1,K\varrho_{2,j}\}^4\nu_j^{r-3}.
\end{equation}

\begin{theorem}\label{thm:mlquad}
  Let $u\in L^2(U,X;\gamma)$ and $u^\lev\in L^2(U,X;\gamma)$,
  $\lev\in\N$, satisfy Assumption \ref{ass:ml} with some constants
  ${\alpha}>0$ and $0<p_1<4/5$ and $p_1\le p_2<1$. 
  Set $q_1:=p_1/(1-p_1)$.  
  Assume that $r>2(1+({\alpha}+1)q_1/2)/q_1+3$
  (for $r$ in \eqref{eq:seqiseqiiquad}).  
  There exist constants $K>0$ (in \eqref{eq:seqiseqiiquad}) and $C>0$ 
  such that the following holds.
  
  There exist $C>0$ and, for every $n \in \NN$ there exists
  $\eps_n\in (0,1]$ such that such that $\work(\blev_{\eps_n}) \le n$
  and with $\blev_{\eps_n}=({\lev_{\eps_n,\bnu}})_{\bnu\in\CF}$ as in
  Corollary \ref{LEMMA:MLWEIGHTNEW} (with $c_\bnu$, $d_\bnu$ as in
  \eqref{eq:seqiseqiiquad}) it holds
  \begin{equation*}
    \normc[X]{\int_U u(\by)\dd\gamma(\by)-\VQ_{\blev_{\eps_n}}^{\rm ML}u} 
        	\le \ C (1+\log n )n^{-R},
  \end{equation*}
  with the convergence rate
  \begin{equation} \label{R-Q}
    R:=\min\left\{{\alpha},\frac{{\alpha}(2p_1^{-1}-5/2)}{{\alpha}+2p_1^{-1}-2p_2^{-1}}\right\}.
  \end{equation}
\end{theorem}
\begin{proof}
  Throughout this proof we write $\bb_1=(b_{1,j})_{j\in\N}$ and
  $\bb_2=(b_{2,j})_{j\in\N}$ for the two sequences in Assumption \ref{ass:ml}. 
  As in the proof of Theorem~\ref{thm:mlint} we 
  highlight that the multi-index set $\Gamma_j$ which was 
  defined in \eqref{eq:Gamma} is downward closed for all $j\in\N_0$.

  {\bf Step 1.}  Given $n\in \N$, we choose $\eps:=\eps_n$ as in 
  Lemma \ref{LEMMA:MLWEIGHTNEW}.  
  Fix $N\in\N$ such that $N>\max \{j:\, j \in\supp(\bnu),\, \lev_{\eps,\bnu}> 0\}$ 
  and so large that
  \begin{equation}\label{eq:truncerr_mlquad}
    \normc[X]{\int_U(u(\by)-\tilde u_N(\by)) \dd\gamma(\by)}\le n^{-R},
  \end{equation}
  where $\tilde u_N:U\to X$ is as in Definition~\ref{def:bdXHol} 
  (this is possible due
  $\lim_{N\to\infty}\norm[L^2(U,X;\gamma)]{u-\tilde u_N}=0$ which
  holds by the $(\bb_1,\xi,\delta,X)$-holomorphy of $u$).

  By Assumption \ref{ass:ml}, for every $j \in\N$ the function
  $e^j:=u-u^j\in L^2(U,X;\gamma)$ is
  $(\bb_1,\xi,\delta,X)$-holomorphic and
  $(\bb_2,\xi,\delta\sw{j}^{\alpha},X)$-holomorphic. For notational
  convenience we set $e^0:=u-0=u\in L^2(U,X;\gamma)$, so that $e^0$ is
  $(\bb_1,\xi,\delta,X)$-holomorphic and
  $(\bb_2,\xi,\delta,X)$-holomorphic. Hence for every $j\in\N_0$ there
  exists a function $\tilde e_N^j=\tilde u_N-\tilde u_N^j$ as in
  Definition~\ref{def:bdXHol} \ref{item:vN}.

  The following assertions are identical to the ones in the proof of
  Theorem~\ref{thm:mlint}, except that we now admit different
  summability exponents $q_1$ and $q_2$.
  \begin{enumerate}
  \item\label{item:u-uswunifabs_quad} By Lemma \ref{lemma:uN}, for
    every $j\in\N_0$, with the Wiener-Hermite PC expansion
    coefficients
    $$\tilde e_{N,\bnu}^j:= \int_{U} H_\bnu(\by)\tilde e_{N}^j(\by)
    \dd\gamma(\by),$$ it holds
    \begin{equation*}\label{eq:u-uswunifabs2}
      \tilde e_{N}^j(\by) = \sum_{\bnu\in\CF} \tilde e_{N,\bnu}^j H_\bnu(\by)\qquad\forall \by\in U,
    \end{equation*}
    with pointwise absolute convergence.
  \item\label{item:aux_quad} By Lemma \ref{lemma:cnu}, upon choosing
    $K>0$ in \eqref{eq:seqiseqii} large enough, and because $r>3$,
    \begin{equation*}
      C_0 \seqi{\bnu} p_\bnu(3)\le \beta_\bnu(r,\bvarrho_1),\qquad
      C_0 \seqii{\bnu} p_\bnu(3)\le \beta_\bnu(r,\bvarrho_2)\qquad\forall \bnu\in\CF_2.
    \end{equation*}
    Remark that by definition of $\bvarrho_i$, $i\in\{1,2\}$, in
    \eqref{eq:bvarrhoiquad}, it holds
    $\varrho_{i,j}\sim b_{i,j}^{-(1-p_i)}$ and therefore
    $(\varrho_{i,j}^{-1})_{j\in\N}\in\ell^{q_i}(\N)$ with
    $q_i:=p_i/(1-p_i)$, $i\in\{1,2\}$.
  \item\label{item:seqiseqiisum_quad} Due to
    $r>2(1+2({\alpha}+1)q_1)/q_1+3$,
    the condition of Lemma \ref{lemma:summabcnu} is satisfied (with
    $k=2$, $\tau=3$ and $\theta=({\alpha}+1)q_1/2$). Hence the lemma
    gives
    \begin{equation*}
      \sum_{\bnu\in\CF} p_\bnu(({\alpha}+1)q_1/2)\seqi{\bnu}^{-q_1/4}<\infty
      \qquad
      \Rightarrow\qquad
      (p_\bnu({\alpha}+1)\seqi{\bnu}^{-1/2})_{\bnu\in\CF}\in\ell^{q_1/2}(\CF)
    \end{equation*}
    and similarly
    \begin{equation*}
      \sum_{\bnu\in\CF} p_\bnu(({\alpha}+1)q_2/2)\seqii{\bnu}^{-q_2/4}<\infty
      \qquad
      \Rightarrow\qquad
      (p_\bnu({\alpha}+1)\seqii{\bnu}^{-1/2})_{\bnu\in\CF}\in\ell^{q_2/2}(\CF).
    \end{equation*}
  \item\label{item:tedecay_quad} By Theorem~\ref{thm:bdHolSum} and
    item \ref{item:aux_quad}, for all $j\in\N_0$
    \begin{equation*}
      C_0 \sum_{\bnu\in\CF_2}\seqi{\bnu}\norm[X]{\tilde e_{N,\bnu}^j}^2p_\bnu(3)\le    
      \sum_{\bnu\in\CF_2}\beta_\bnu(r,\bvarrho_1)\norm[X]{\tilde e_{N,\bnu}^j}^2 \le
      C \delta^2
    \end{equation*}
    and
    \begin{equation*}
      C_0 \sum_{\bnu\in\CF_2}\seqii{\bnu}\norm[X]{\tilde e_{N,\bnu}^j}^2p_\bnu(3) \le
      \sum_{\bnu\in\CF_2}\beta_\bnu(r,\bvarrho_2)\norm[X]{\tilde e_{N,\bnu}^j}^2 \le
      C \frac{\delta^2}{\sw{j}^{2{\alpha}}},
    \end{equation*}
    with the constant $C$ independent of $j$, $\sw{j}$ and $N$.
  \item\label{item:IGammae_quad} Because
    $N\ge\max\set{j\in\supp(\bnu)}{\lev_{\eps,\bnu}\ge 0}$ and
    $\chi_{0,0}=0$ we have
    $\VQ_{\Gamma_j}(u-u^j)=\VQ_{\Gamma_j}e^j= \VQ_{\Gamma_j}\tilde
    e_N^j$ for all $j\in\N$ (cp.~Remark~\ref{rmk:defu}). Similarly
    $\VQ_{\Gamma_j}u=\VQ_{\Gamma_j}\tilde u_N$ for all $j\in\N$.
  \end{enumerate}

  {\bf Step 2.} Observe that $\Gamma_j=\emptyset$ for all
  $$j>L(\eps):= \max_{\bnu\in\CF}{\lev_{\eps,\bnu}}$$
  (cp.~\eqref{eq:Gamma}), which is finite due to
  $|\blev_\eps|<\infty$. With the conventions
  $$\VQ_{\Gamma_0}=\VQ_{\CF}=\int_U \cdot\dd\gamma(\by)$$
  (i.e.~$\VQ_{\Gamma_0}$ is the exact integral operator) and
  $\VQ_{\emptyset} \equiv 0$ this implies
  \begin{align*}
    \int_Uu(\by)\dd\gamma(\by) 
    &  = \VQ_{\Gamma_0}u 
      = \sum_{j=0}^{L(\eps)} (\VQ_{\Gamma_j}-\VQ_{\Gamma_{j+1}}) u
    \\
    &   = (\VQ_{\Gamma_0}-\VQ_{\Gamma_1})u+\ldots+
      (\VQ_{\Gamma_{L(\eps)-1}}-\VQ_{\Gamma_{L(\eps)}})u+\VQ_{\Gamma_{L(\eps)}}u.
  \end{align*}
  By definition of the multilevel quadrature in \eqref{eq:VQml}
  \begin{align*}
    \VQml_{\blev_\eps} u & =
                           \sum_{j=1}^{L(\eps)}(\VQ_{\Gamma_j}-\VQ_{\Gamma_{j+1}})u^j
    \\
                         &  =(\VQ_{\Gamma_1}-\VQ_{\Gamma_2}) u^1+\ldots+
                           (\VQ_{\Gamma_{L(\eps)-1}}-\VQ_{\Gamma_{L(\eps)}}) u^{L(\eps)}+
                           \VQ_{\Gamma_{L(\eps)}}u^{L(\eps)}.
  \end{align*}
  By item \ref{item:IGammae_quad} of Step 1, we can write
  \begin{align*}
    (\VQ_{\Gamma_0}-\VQ_{\Gamma_1})u &=
                                       \int_Uu(\by)\dd\gamma(\by)-\VQ_{\Gamma_1}u \nonumber\\
                                     &= \int_Uu(\by)\dd\gamma(\by)-\VQ_{\Gamma_1}\tilde u_N
                                       \nonumber\\
                                     &= \int_U(u(\by)-\tilde u_N(\by))\dd\gamma(\by)+(\VQ_{\Gamma_0}-\VQ_{\Gamma_1})\tilde u_N\nonumber\\
                                     &= \int_U(u(\by)-\tilde u_N(\by))\dd\gamma(\by)
                                       +(\VQ_{\Gamma_0}-\VQ_{\Gamma_1})\tilde e_N^0,
  \end{align*}
  where in the last equality we used $e_N^0=u_N$, by definition of
  $e^0=u$ (and $\tilde e_N^0=\tilde u_N\in L^2(U,X;\gamma)$ as in
  Definition~\ref{def:bdXHol}).  Hence, again by item
  \ref{item:IGammae_quad},
  \begin{align*}\label{eq:u-Iwu_quad}
    \int_Uu(\by)\dd\gamma(\by)-\VQml_{\blev_\eps}u &=(\VQ_{\Gamma_0}-\VQ_{\Gamma_1}) u+
                                                     \sum_{j=1}^{L(\eps)}(\VQ_{\Gamma_j}-\VQ_{\Gamma_{j+1}}) (u-u^j)\nonumber\\
                                                   &=\int_U(u(\by)-\tilde u_N(\by))\dd\gamma(\by)+(\VQ_{\Gamma_0}-\VQ_{\Gamma_1})\tilde u_N+\sum_{j=1}^{L(\eps)}(\VQ_{\Gamma_j}-\VQ_{\Gamma_{j+1}})\tilde e_N^j\nonumber\\
                                                   &=\int_U(u(\by)-\tilde u_N(\by))\dd\gamma(\by)+\sum_{j=0}^{L(\eps)}(\VQ_{\Gamma_j}-\VQ_{\Gamma_{j+1}})\tilde e_N^j.
  \end{align*}
  Let us bound the norm. From item \ref{item:u-uswunifabs_quad} of Step
  1 it follows that for every $j\in\N_0$,
  $$\tilde e_{N}^j(\by)=\sum_{{\bnu\in\CF_1^N}} \tilde e_{N,\bnu}^j
  H_\bnu(\by),$$ with  the equality and unconditional convergence in $X$ for all
  $\by\in \CF_1^N$. 
  Hence similar to Lemma~\ref{lemma:uncond-conv-Q_Lambda} we have
  	$$
  	(\VQ_{\Gamma_j}-\VQ_{\Gamma_{j+1}})e_{N}^j =\sum_{{\bnu\in\CF_1^N}} \tilde e_{N,\bnu}^j (\VQ_{\Gamma_j}-\VQ_{\Gamma_{j+1}})H_\bnu
  	$$
with the equality and unconditional convergence   
in $X$. Since
  $(\VQ_{\Gamma_j}-\VQ_{\Gamma_{j+1}})H_\bnu= 0\in X$ for all
  $\bnu\in\Gamma_{j+1}\subseteq\Gamma_j$ and all
    $\bnu\in\CF\backslash\CF_2$  by Lemma \ref{lemma:VIprop}, we get
  \begin{equation}\label{eq:u-vimlu_quad}
    \normc[X]{\int_U u(\by)\dd\gamma(\by)-\VQml_{\blev_\eps}u}\le
    \normc[X]{\int_U(u(\by)-\tilde u_N(\by))\dd\gamma(\by)}+
    \sum_{\bnu\in\CF_2}\sum_{j={\lev_{\eps,\bnu}}}^{L(\eps)} 
    \norm[X]{\tilde e_{N,\bnu}^j}|(\VQ_{\Gamma_j}-\VQ_{\Gamma_{j+1}})H_\bnu|.
  \end{equation}
  
  {\bf Step 3.} We wish to apply Lemma \ref{LEMMA:MLWEIGHTNEW} to the
  bound \eqref{eq:u-vimlu_quad}. By \eqref{eq:L1boundLambda}, for all
  $\bnu\in\CF$
  \begin{equation*}
    |(\VQ_{\Gamma_j}-\VQ_{\Gamma_{j+1}})H_\bnu|
    \le |\VQ_{\Gamma_j}H_\bnu|+
    |\VQ_{\Gamma_{j+1}}H_\bnu|
    \le 2 p_\bnu(3).
  \end{equation*}
  Define
  \begin{equation*}
    a_{j,\bnu}:=\norm[X]{\tilde e_{N,\bnu}^j}p_\bnu(3)\qquad\forall\bnu\in\CF_2,
  \end{equation*}
  and $a_{j,\bnu}:=0$ for $\bnu\in\CF\backslash\CF_2$. By items
  \ref{item:seqiseqiisum_quad} and \ref{item:tedecay_quad} of Step 1,
  the collections $(a_{j,\bnu})_{\bnu\in\CF}$, $j\in\N_0$, and
  $(\seqi{\bnu})_{\bnu\in\CF}$, $(\seqii{\bnu})_{\bnu\in\CF}$, 
  satisfy
  the assumptions of Lemma \ref{LEMMA:MLWEIGHTNEW} (with
  $\tilde q_1:=q_1/2$ and $\tilde q_2:=q_2/2$). 
  Therefore
  \eqref{eq:u-vimlu_quad}, \eqref{eq:truncerr_mlquad} and 
  Lemma \ref{LEMMA:MLWEIGHTNEW} 
  give
  \begin{equation*}
    \normc[X]{\int_U u(\by)\dd\gamma(\by)-\VQml_{\blev_\eps}u}
    \le
    n^{-R}+
    \sum_{\bnu\in\CF}\sum_{j={\lev_{\eps,\bnu}}}^{L(\eps)} a_{j,\bnu}
    \le C (1+\log n)n^{-R},
  \end{equation*}
  with
  \begin{equation*}
    R = \min\left\{{\alpha},\frac{{\alpha} (\tilde q_1^{-1}-1/2)}{{\alpha} + \tilde q_1^{-1}-\tilde q_2^{-1}}
    \right\} =
    \min\left\{{\alpha},\frac{{\alpha} (2 p_1^{-1}-5/2)}{{\alpha} + 2 p_1^{-1}-2 p_2^{-1}}
    \right\},
  \end{equation*}
  where we used $\tilde q_1=q_1/2=p_1/(2-2p_1)$ and
  $\tilde q_2=q_2/2=p_2/(2-2p_2)$ as stated in item
  \ref{item:aux_quad} of Step 1.
\end{proof}
\subsection{Examples for multilevel interpolation and quadrature}
\label{sec:Approx}
We revisit the examples in Sections \ref{sec:SumHolSol} and
\ref{sec:BIP}, and demonstrate how to verify the assumptions required
for the multilevel convergence rate results
in Theorems~\ref{thm:mlint} and \ref{thm:mlquad}.
\subsubsection{Parametric diffusion coefficient in polygonal domain}
\label{S:DiffPolyg}
Let $\D\subseteq\R^2$ be a bounded polygonal domain, 
and consider once more the elliptic equation \index{PDE!linear elliptic $\sim$}
\begin{equation}\label{eq:elliptic2}
  - \div(a \nabla \Uu(a))=f\quad\text{in }\D,\qquad
  \Uu(a) =0\quad\text{on }\partial\D,
\end{equation}
as in Section~\ref{sec:pdc}.

For $s\in\N_0$ and $\varkappa\in\R$, recall the Kondrat'ev spaces
$\Ww^{s}_\infty(\D)$ and $\Kk^{s}_{\varkappa}(\D)$ with norms
\begin{equation*}
  \|u\|_{ \Kk^s_\varkappa}
  :=
  \sum_{|\balpha|\leq s}\|r_\domain^{|\balpha|-\varkappa}D^\balpha u\|_{L^2}
  \qquad\text{and}\qquad
  \|u\|_{ \Ww^s_\infty}:=\sum_{|\balpha|\leq s}\|r_\domain^{|\balpha|}D^\balpha u\|_{L^\infty}
\end{equation*}
introduced in Section~\ref{sec:KondrAn}.  Here, as earlier,
$r_{\D}:\D\to [0,1]$ denotes a fixed smooth function that coincides
with the distance to the nearest corner, in a neighbourhood of each
corner.
According to Theorem~\ref{thm:bacuta}, assuming $s\ge 2$,
$f\in \Kk^{s-2}_{\varkappa-1}(\D)$ and $a\in \Ww^{s-1}_\infty(\D)$ the
solution $\Uu(a)$ of \eqref{eq:elliptic2} belongs to
$\Kk_{\varkappa+1}^{s}(\D)$ provided that with
\begin{equation*}
  \rho(a) := \underset{\bx\in\D}{\essinf}\,\Re(a(\bx))>0,
\end{equation*}
\begin{equation}\label{eq:kappaa}
  |\varkappa|<\frac{\rho(a)}{\tau \norm[L^\infty]{a}},
\end{equation}
where $\tau$ is a constant depending on $\D$ and $s$.  
Our goal is to
treat, in a unified manner, 
a family of diffusion coefficients $a(\by)$, $\by\in U$, 
where for certain $\by\in U$ the diffusion coefficient
$a(\by)$ is such that the right-hand side of \eqref{eq:kappaa} might
be arbitrarily small.  This only leaves us with the choice
$\varkappa=0$, see Remark~\ref{rmk:kappa}.  On the other hand, the
motivation of using Kondrat'ev spaces in the analysis of
approximations to PDE solutions $\Uu(a(\by))$, is that functions in
$\Kk_{\varkappa+1}^s(\D)$ on polygonal domains in $\R^2$ can be
approximated with the optimal convergence rate $\frac{s-1}{2}$ w.r.t.\
the $H^1$-norm by suitable finite element spaces (on graded meshes;
i.e.\ this analysis accounts for corner singularities which prevent
optimal convergence rates on uniform meshes). Such results are
well-known, see for example \cite{BNZPolygon}, however they require
$\varkappa>0$. For this reason we need a stronger regularity result,
giving uniform $\Kk_{\varkappa+1}^{s}$-regularity with $\varkappa>0$
independent of the parameter.  This is the purpose of the next theorem.
For its proof we shall need the following lemma, which is shown in a
similar way as in \cite[Lemma C.2]{2006.06994}.  We recall that
$$\norm[W^{s}_\infty]{f}:=\sum_{|\bnu|\le s}\norm[L^\infty]{D^\bnu f}.$$
\begin{lemma}\label{lemma:fgWm}
  Let $s\in\N_0$ and let $\D\subseteq\R^2$ be a bounded polygonal domain,
  $d\in\N$.  

  Then there exist $C_s$ and $\tilde C_s$ such that for any two
  functions $f$, $g\in \Ww^{s}_{\infty}(\D)$
  \begin{enumerate}
  \item\label{item:prodinfty}
    $\norm[\Ww^{s}_{\infty}]{fg}\le C_s
    \norm[\Ww^{s}_{\infty}]{f}\norm[\Ww^{s}_{\infty}]{g}$,
  \item\label{item:fracinfty}
    $\norm[\Ww^{s}_{\infty}]{\frac 1 f}\le \tilde C_s
    \frac{\norm[\Ww^{s}_{\infty}]{f}^{s}}{\essinf_{\bx\in
        \D}|f(\bx)|^{s+1}}$ if $\essinf_{\bx\in\D}|f(\bx)|>0$.
  \end{enumerate}
  These statements remain true if $\Ww^{s}_{\infty}(\D)$ is replaced
  by $W^{s}_\infty(\D)$.  Furthermore, if $\varkappa\in\R$, then for
  $f\in\Kk_\varkappa^s(\D)$ and $a\in \Ww^{s}_{\infty}(\D)$
  \begin{enumerate}
    \setcounter{enumi}{2}
  \item\label{item:prodkond}
    $\norm[\Kk_\varkappa^s]{fa} 
     \le C_s
    \norm[\Kk_\varkappa^s]{f}\norm[\Ww^{s}_{\infty}]{a}$,
  \item\label{item:proddiffkond}
    $\norm[\Kk_{\varkappa-1}^{s-1}]{\nabla f \cdot\nabla a} \le
    C_{s-1}
    \norm[\Kk_{\varkappa+1}^{s}]{f}\norm[\Ww^{s}_{\infty}]{a}$ 
    if $s\ge 1$.
  \end{enumerate}
\end{lemma}
\begin{proof}
  We will only prove \ref{item:prodinfty} and \ref{item:fracinfty} for
  functions in $\Ww^{s}_\infty(\D)$.  The case of $W^{s}_\infty(\D)$
  is shown similarly (by omitting all occurring functions $r_\D$ in
  the following).

  {\bf Step 1.}  We start with \ref{item:prodinfty}, and show a
  slightly more general bound: for $\tau\in\R$ introduce
  \begin{equation*}
    \norm[\Ww^{s}_{\tau,\infty}]{f}:=\sum_{|\bnu|\le s}\norm[L^\infty]{r_\D^{\tau+|\bnu|}D^\bnu f},
  \end{equation*}
  i.e.\ $\Ww^s_{0,\infty}(\D)=\Ww^s_\infty(\D)$. We will show that for
  $\tau_1+\tau_2=\tau$
  \begin{equation}\label{eq:wwstau}
    \norm[\Ww^s_{\tau,\infty}]{fg}\le
    C_{s} \norm[\Ww^s_{\tau_1,\infty}]{f}\norm[\Ww^s_{\tau_2,\infty}]{g}.
  \end{equation}
  Item \ref{item:prodinfty} then follows with $\tau=\tau_1=\tau_2=0$.
  
  Using the multivariate Leibniz rule for Lipschitz functions, for any
  multiindex $\bnu\in\N_0^d$ with $d\in\N$ fixed,
  \begin{equation}\label{eq:leibniz}
    D^\bnu(fg)=\sum_{\bmu\le\bnu}\binom{\bnu}{\bmu}
    D^{\bnu-\bmu}f D^{\bmu}g.
  \end{equation}
  Thus if $|\bnu|\le s$
  \begin{equation*}
    \norm[L^\infty]{r_\D^{\tau+|\bnu|}D^\bnu(fg)}
    \le \sum_{\bmu\le\bnu}\binom{\bnu}{\bmu}
    \norm[L^\infty]{r_\D^{\tau_1+|\bnu-\bmu|}D^{\bnu-\bmu}f}
    \norm[L^\infty]{r_\D^{\tau_2+|\bmu|}D^{\bmu}g}
    \le 2^{|\bnu|}\norm[\Ww^{s}_{\tau_1,\infty}]{f}\norm[\Ww^{s}_{\tau_2,\infty}]{g},
  \end{equation*}
  where we used $\binom{\bnu}{\bmu}=\prod_{j=1}^d\binom{\nu_j}{\mu_j}$
  and $\sum_{i=0}^{\nu_j}\binom{\nu_j}{i}=2^{\nu_j}$.
  We conclude
  $$\norm[\Ww^{s}_{\tau,\infty}]{fg}\le
  C_s\norm[W^{s}_{\tau,\infty}]{f}\norm[W^{s}_{\infty}]{g}$$ with
  $C_s=\sum_{|\bnu|\le s}2^{|\bnu|}$. Hence \ref{item:prodinfty}
  holds.

  {\bf Step 2.} We show \ref{item:fracinfty}, and claim that for all
  $|\bnu|\le s$ it holds
  \begin{equation}\label{eq:fracclaimtotal}
    D^\bnu \left(\frac{1}{f}\right)=\frac{p_\bnu}{f^{|\bnu|+1}}\,,
  \end{equation}
  where $p_\bnu$ satisfies
  \begin{equation}\label{eq:fracclaim}
    \norm[\Ww^{s-|\bnu|}_{|\bnu|,\infty}]{p_\bnu}\le \hat C_{|\bnu|}\norm[\Ww^{s}_\infty]{f}^{|\bnu|}
  \end{equation}

  for some $\hat C_{|\bnu|}$ solely depending on $|\bnu|$. We proceed
  by induction over $|\bnu|$ and start with $|\bnu|=1$, i.e.,\
  $\bnu=\bee_j=(\delta_{ij})_{i=1}^d$ for some
  $j\in\{1,\dots,d\}$. Then
  $D^{\bee_j}\frac{1}{f}=\frac{-\partial_j f}{f^2}$ and
  $p_{\bee_j}=-\partial_j f$ satisfies
  \begin{equation*}
    \norm[\Ww^{s-1}_{1,\infty}]{p_{\bee_j}}=
    \sum_{|\bmu|\le s-1}\norm[L^\infty]{r_\D^{1+|\bmu|}D^{\bmu}p_{\bee_j}}
    =\sum_{|\bmu|\le s-1}\norm[L^\infty]{r_\D^{|\bmu+\bee_j|}D^{\bmu+\bee_j} f}
    \le \norm[\Ww^{s}_{\infty}]{f},
  \end{equation*}
  i.e.\ $\hat C_1=1$. For the induction step fix $\bnu$ with
  $1<|\bnu|<s$ and $j\in\{1,\dots,d\}$. Then by the induction
  hypothesis $D^\bnu \frac{1}{f}=\frac{p_\bnu}{f^{|\bnu|+1}}$ and
  \begin{equation*}
    D^{\bnu+\bee_j} \frac{1}{f}=
    \partial_j\left(\frac{p_{\bnu}}{f^{|\bnu|+1}}\right)
    =\frac{f^{|\bnu|+1}\partial_j p_{\bnu}-(|\bnu|+1)f^{|\bnu|}p_{\bnu}\partial_j f}{f^{2|\bnu|+2}}
    =\frac{f \partial_j p_{\bnu}-(|\bnu|+1)p_{\bnu}\partial_j f}{f^{|\bnu|+2}},
  \end{equation*}
  and thus
  $$p_{\bnu+\bee_j}:=f \partial_j p_{\bnu}-(|\bnu|+1)p_{\bnu}\partial_j f.$$ 
  Observe that
  \begin{equation}\label{eq:partialjg}
    \norm[\Ww^s_{\tau,\infty}]{\partial_j g}
    =\sum_{|\bmu|\le s}\norm[L^\infty]{r_\D^{\tau+|\bmu|} D^{\bmu+\bee_j} g}
    \le \sum_{|\bmu|\le s+1}\norm[L^\infty]{r_\D^{\tau+|\bmu|-1} D^\bmu g}
    =\norm[\Ww^{s+1}_{\tau-1,\infty}]{g}.
  \end{equation}
  Using \eqref{eq:wwstau} and \eqref{eq:partialjg}, we get with
  $\tau:=|\bnu|+1$
  \begin{align*}
    \norm[\Ww^{s-\tau}_{\tau,\infty}]{p_{\bnu+\bee_j}}
    &\le \norm[\Ww^{s-\tau}_{\tau,\infty}]{f\partial_jp_{\bnu}}
      +(|\bnu|+1)\norm[\Ww^{s-\tau}_{\tau,\infty}]{p_{\bnu}\partial_j f}\nonumber\\
    &\le C_{s-\tau} \norm[\Ww^{s-\tau}_{0,\infty}]{f}
      \norm[\Ww^{s-\tau}_{\tau,\infty}]{\partial_jp_{\bnu}}+(|\bnu|+1)
      C_{s-\tau}
      \norm[\Ww^{s-\tau}_{\tau-1,\infty}]{p_{\bnu}}
      \norm[\Ww^{s-\tau}_{1,\infty}]{\partial_j f}\nonumber\\
    &\le C_{s-\tau} \norm[\Ww^{s-\tau}_{0,\infty}]{f}
      \norm[\Ww^{s-\tau+1}_{\tau-1,\infty}]{p_{\bnu}}+(|\bnu|+1)
      C_{s-\tau}
      \norm[\Ww^{s-\tau+1}_{\tau-1,\infty}]{p_{\bnu}}
      \norm[\Ww^{s-\tau+1}_{0,\infty}]{f}.
  \end{align*}
  Due to $\tau-1=|\bnu|$ and the induction hypothesis
  \eqref{eq:fracclaim} for $p_\bnu$,
  \begin{align*}
    \norm[\Ww^{s-(|\bnu|+1)}_{|\bnu|+1,\infty}]{p_{\bnu+\bee_j}}
    &\le C_{s-(|\bnu|+1)}\left(\hat C_{|\bnu|}\norm[\Ww^{s-(|\bnu|+1)}_{\infty}]{f}\norm[\Ww^s_{\infty}]{f}^{|\bnu|}+
      (|\bnu|+1)\hat C_{|\bnu|}\norm[\Ww^s_{\infty}]{f}^{|\bnu|}\norm[\Ww^{s-|\bnu|}_\infty]{f} \right)\nonumber\\
    &\le C_{s-(|\bnu|+1)}\hat C_{|\bnu|} (|\bnu|+2)\norm[\Ww_s^\infty]{f}^{|\bnu|+1}.
  \end{align*}
  In all this shows the claim with $\hat C_1:=1$ and inductively for
  $1<k\le s$,
  $$\hat C_{k}:= C_{s-k}\hat C_{k-1}(k+1).$$
  By \eqref{eq:fracclaimtotal} and \eqref{eq:fracclaim}, for every
  $|\bnu|\le s$
  \begin{equation*}
    \normc[L^\infty]{r_\D^{|\bnu|}D^\bnu\left(\frac{1}{f}\right)}\le \hat C_{|\bnu|}\frac{\norm[\Ww^{s}_{\infty}]{f}^{|\bnu|}}{\essinf_{\bx\in\D}|f(\bx)|^{|\bnu|+1}}\,.
  \end{equation*}
  Due to
  $$\norm[\Ww^s_\infty]{f}\ge \norm[L^\infty]{f}\ge
  \essinf_{\bx\in\D}|f(\bx)|,$$ this implies
  \begin{equation*}
    \normc[\Ww^{s}_{\infty}]{\frac{1}{f}}
    = \sum_{|\bnu|\le s}\normc[L^\infty]{r_\D^{|\bnu|}D^\bnu \left(\frac{1}{f}\right)}
    \le \tilde C_s \frac{\norm[\Ww^{s}_{\infty}]{f}^{s}}{\essinf_{\bx\in\D}|f(\bx)|^{s+1}}
  \end{equation*}
  with $\tilde C_s:=\sum_{|\bnu|\le s}\hat C_{|\bnu|}$.

  {\bf Step 3.} We show \ref{item:prodkond} and
  \ref{item:proddiffkond}. If $f\in\Kk_\varkappa^s(\D)$ and
  $a\in \Ww^{s}_{\infty}(\D)$, then by \eqref{eq:leibniz} for Sobolev
  functions,
  \begin{equation*}
    r_\D^{\bnu-\varkappa}D^\bnu (fa)=\sum_{\bmu\le\bnu}\binom{\bnu}{\bmu} (r_\D^{|\bnu-\bmu|-\varkappa}D^{\bnu-\bmu}f)(r_\D^{|\bmu|}D^\bmu a)
  \end{equation*}
  and hence
  \begin{align*}
    \norm[\Kk_\varkappa^s]{fa}&=\sum_{|\bnu|\le s}\norm[L^2]{r_\D^{|\bnu|-\varkappa} D^\bnu (fa)}\nonumber\\
                              &\le
                                \sum_{|\bnu|\le s}\sum_{\bmu\le\bnu} \binom{\bnu}{\bmu}\norm[L^2]{r_\D^{|\bnu-\bmu|-\varkappa} D^{\bnu-\bmu}f}\norm[L^\infty]{r_\D^{|\bmu|}D^\bmu a}\nonumber\\
                              &\le C_s\sum_{|\bnu|\le s}
                                \norm[L^2]{r_\D^{|\bnu|-\varkappa} D^{\bnu}f}
                                \sum_{|\bmu|\le s}\norm[L^\infty]{r_\D^{|\bmu|}D^\bmu a}\nonumber\\
                              &=C_s \norm[\Kk_\varkappa^s]{f}\norm[\Ww^s_\infty]{a}.
  \end{align*}

  Finally if $s\ge 1$,
  \begin{align*}
    \norm[\Kk_{\varkappa-1}^{s-1}]{\nabla f\cdot\nabla a}&=\sum_{|\bnu|\le s-1}\normc[L^2]{r_\D^{|\bnu|-\varkappa+1} D^\bnu \left(\sum_{j=1}^d\partial_jf\partial_j a\right)}\nonumber\\
                                                         &\le
                                                           \sum_{|\bnu|\le s-1}\sum_{\bmu\le\bnu}\binom{\bnu}{\bmu}\sum_{j=1}^d\norm[L^2]{r_\D^{|\bnu-\bmu|-\varkappa} D^{\bnu-\bmu+\bee_j}f}\norm[L^\infty]{r_\D^{|\bmu|+1}D^{\bmu+\bee_j} a}\nonumber\\
                                                         &\le C_{s-1} d \sum_{|\bnu|\le s}
                                                           \norm[L^2]{r_\D^{|\bnu|-\varkappa-1} D^{\bnu}f}
                                                           \sum_{|\bmu|\le s}\norm[L^\infty]{r_\D^{|\bmu|}D^\bmu a}\nonumber\\
                                                         &= C_{s-1}d \norm[\Kk_{\varkappa+1}^s]{f}\norm[\Ww^{s}_\infty]{a}.\qedhere    
  \end{align*}
\end{proof}

The proof of the next theorem is based on Theorem~\ref{thm:bacuta}.
In order to get regularity in $\Kk_{\varkappa+1}^s(\D)$ with
$\varkappa>0$ independent of the diffusion coefficient $a$, we now
assume $a\in W^{1}_\infty(\D)\cap\Ww_{\infty}^{s-1}(\D)$ in lieu of
the weaker assumption $a\in \Ww^{s-1}_\infty$
that was required in Theorem~\ref{thm:bacuta}.
\begin{theorem}\label{thm:bacuta2}
  Let $\D\subseteq\R^2$ be a bounded polygonal domain and
  $s\in\N$, $s\ge 2$.  Then there exist $\varkappa>0$ and $C_s>0$
  depending on $\D$ and $s$ (but independent of $a$) such that for all
  $a\in W^{1}_\infty(\D)\cap\Ww_{\infty}^{s-1}(\D)$ and all
  $f\in \Kk_{\varkappa-1}^{s-2}(\D)$ the weak solution
  $\Uu\in H_0^1(\D)$ of \eqref{eq:elliptic2} satisfies with
  $N_s:=\frac{s(s-1)}{2}$
  \begin{equation}\label{eq:Uuapriori}
    \norm[\Kk_{\varkappa+1}^{s}]{\Uu}
    \le C_s \frac{1}{\rho(a)}\left(\frac{\norm[\Ww^{s-1}_{\infty}]{a}+\norm[W^{1}_\infty]{a}}{\rho(a)}\right)^{N_s}\norm[\Kk_{\varkappa-1}^{s-2}]{f}.
  \end{equation}
\end{theorem}
\begin{proof}
  Throughout this proof let $\varkappa\in (0,1)$ be a constant such
  that
  \begin{equation}\label{eq:DeltaIso}
    -\Delta:\Kk_{\varkappa+1}^{j}(\D)\cap H_0^1(\D)\to \Kk_{\varkappa-1}^{j-2}(\D)
  \end{equation}
  is a boundedly invertible operator for all $j\in\{2,\dots,s\}$; such
  $\varkappa$ exists by Theorem~\ref{thm:bacuta}, and $\varkappa$
  merely depends on $\D$ and $s$.

  {\bf Step 1.}  We prove the theorem for $s=2$, in which case
  $a\in W^{1}_\infty\cap\Ww^{1}_\infty=W^{1}_\infty$.
  
  Applying Theorem~\ref{thm:bacuta} directly to \eqref{eq:elliptic2}
  yields the existence of \emph{some}
  $\tilde\varkappa\in (0,\varkappa)$ (depending on $a$) such that
  $\Uu\in\Kk_{\tilde\varkappa+1}^2$. Here we use
  $$f\in\Kk_{\varkappa-1}^0(\D)\hookrightarrow \Kk_{\tilde\varkappa-1}^0(\D)$$
  due to $\tilde\varkappa\in(0,\varkappa)$. By the Leibniz rule for
  Sobolev functions we can write
  \begin{equation*}
    -\div(a\nabla \Uu)=-a\Delta \Uu -\nabla a\cdot \nabla \Uu
  \end{equation*}
  in the sense of $\Kk_{\tilde\varkappa-1}^0(\D)$: (i) it holds
  $\Delta\Uu\in \Kk_{\tilde\varkappa-1}^0(\D)$ and
  $$a\in W^{1}_\infty(\D)\hookrightarrow L^\infty(\D)$$ which implies
  $a\Delta\Uu\in \Kk_{\tilde\varkappa-1}^0(\D)$ (ii) it holds
  $$\nabla \Uu \in \Kk_{\tilde\varkappa}^1(\D)\hookrightarrow\Kk_{\tilde\varkappa}^0(\D),$$ and $\nabla a \in L^\infty(\D)$
  which implies
  $\nabla a \cdot\nabla \Uu \in \Kk_{\tilde\varkappa-1}^0(\D)$.
  Hence,
  \begin{equation*}
    -\div(a\nabla \Uu) = -a\Delta \Uu -\nabla a\cdot \nabla \Uu
    = f,
  \end{equation*}
  and further
  \begin{equation*}
    -\Delta \Uu = \frac{1}{a}\Big(f+\nabla a\cdot \nabla \Uu\Big)=:\tilde f\in \Kk_{\tilde \varkappa-1}^0(\D)
  \end{equation*}
  since $\frac{1}{a}\in L^\infty(\D)$ due to $\rho(a)>0$. Our goal is
  to show that in fact $\tilde f\in \Kk_{\varkappa-1}^0(\D)$.
  Because of $-\Delta \Uu=\tilde f$ and $\Uu|_{\partial\D}\equiv 0$,
  Theorem~\ref{thm:bacuta} then implies
  \begin{equation}\label{eq:Kkkappa}
    \norm[\Kk_{\varkappa+1}^2]{\Uu}\le
    C \norm[\Kk_{\varkappa-1}^0]{\tilde f}
  \end{equation}
  for a constant $C$ solely depending on $\D$.

  Denote by $C_H$ %
  a constant (solely depending on $\D$) such that
  \begin{equation*}
    \norm[L^2]{r_\D^{-1}v}\le C_H\norm[L^2]{\nabla v}\qquad
    \forall v\in H_0^1(\D).
  \end{equation*}
  This constant exists as a consequence of Hardy's inequality, see
  e.g.\ \cite{MR1544414} and \cite{MR664599,MR163054} for the
  statement and proof of the inequality on bounded Lipschitz domains.
  Then due to
  \begin{align*}
    \rho(a) \norm[L^2]{\nabla \Uu}^2&\le \Re\left(\int_\D a\nabla \Uu\cdot
                                      \overline{\nabla \Uu} \dd \bx\right)%
                                        = \Re\left(\int_\D f \overline{\Uu} \dd \bx \right)\nonumber\\
                                    &\le \norm[L^2]{r_\D^{1-\varkappa}f}
                                      \norm[L^2]{r_\D^{\varkappa-1}\Uu}%
                                          \le \norm[\Kk_{\varkappa-1}^{0}]{f}\norm[L^2]{r_\D^{-1}\Uu}\nonumber\\
                                    &
                                      \le C_H \norm[\Kk_{\varkappa-1}^{0}]{f}\norm[L^2]{\nabla \Uu}
  \end{align*}
  it holds
  $$\norm[L^2]{\nabla \Uu}\le \frac{C_H\norm[\Kk_{\varkappa-1}^0]{f}}{\rho(a)}.$$
  Hence, using $r_\D^{1-\varkappa}\le 1$, we have that
  \begin{align*}
    \norm[\Kk_{\varkappa-1}^0]{\tilde f}&=\normc[L^2]{\frac{r_\D^{1-\varkappa}}{a}\Big(f+\nabla a\cdot \nabla \Uu\Big)}\nonumber\\
                                        &\le \normc[L^\infty]{\frac{1}{a}}\left(\norm[L^2]{r_\D^{1-\varkappa}f}
                                          + \norm[L^\infty]{\nabla a}\norm[L^2]{\nabla \Uu}\right)
                                          \nonumber\\
                                        &\le \frac{1}{\rho(a)}\left(\norm[\Kk_{\varkappa-1}^0]{f}
                                          + \norm[W^{1}_\infty]{a}\frac{C_H\norm[\Kk_{\varkappa-1}^0]{f}}{\rho(a)}\right)\nonumber\\
                                        &= \frac{\norm[\Kk_{\varkappa-1}^0]{f}}{\rho(a)}\left(1+\frac{C_H\norm[W^{1}_\infty]{a}}{\rho(a)}\right)\nonumber\\
                                        &\le (1+C_H)\frac{1}{\rho(a)}\frac{\norm[W^{1}_\infty]{a}}{\rho(a)} \norm[\Kk_{\varkappa-1}^0]{f}.
  \end{align*}
  The statement follows by \eqref{eq:Kkkappa}.

  {\bf Step 2.} For general $s\in\N$, $s\ge 2$, we proceed by
  induction. Assume the theorem holds for $s-1\ge 2$. Then for
  $$f\in \Kk_{\varkappa-1}^{s-2}(\D)\hookrightarrow \Kk_{\varkappa-1}^{s-3}(\D)$$ and $$a\in W^{1}_\infty(\D)\cap \Ww^{s-1}_{\infty}(\D)\hookrightarrow W^{1}_\infty(\D)\cap \Ww^{s-2}_{\infty}(\D),$$
  we get
  \begin{equation}\label{eq:fKkbound}
    \norm[\Kk_{\varkappa+1}^{s-1}]{\Uu}\le \frac{C_{s-1}}{\rho(a)}
    \left(\frac{\norm[W^{1}_\infty]{a}+\norm[\Ww^{s-2}_{\infty}]{a}}{\rho(a)}\right)^{N_{s-1}}\norm[\Kk_{\varkappa-1}^{s-3}]{f}.
  \end{equation}
  As in Step 1, it holds
  \begin{equation*}
    -\Delta \Uu = \frac{1}{a}\Big(f+\nabla a\cdot \nabla \Uu \Big)=:\tilde f.
  \end{equation*}
  By %
  Lemma \ref{lemma:fgWm} and \eqref{eq:fKkbound}, for some constant
  $C$ (which can change in each line, but solely depends on $\D$ and
  $s$) we have that
  \begin{align*}
    \norm[\Kk_{\varkappa-1}^{s-2}]{\tilde f}
    &\le C \normc[\Ww^{s-2,\infty}]{\frac{1}{a}}
      \norm[\Kk_{\varkappa-1}^{s-2}]{f+\nabla a \cdot \nabla\Uu}\nonumber\\
    &\le C \frac{\norm[\Ww^{s-2}]{a}^{s-2}}{\rho(a)^{s-1}}
      \left(\norm[\Kk_{\varkappa-1}^{s-2}]{f}+\norm[\Ww^{s-1}_\infty]{a}
      \norm[\Kk_{\varkappa+1}^{s-1}]{\Uu}
      \right)\nonumber\\
    &\le C \frac{\norm[\Ww^{s-2}]{a}^{s-2}}{\rho(a)^{s-1}}
      \left(\norm[\Kk_{\varkappa-1}^{s-2}]{f}+
      C_{s-1}\frac{\norm[\Ww^{s-1}_\infty]{a}}{\rho(a)}\left(\frac{\norm[W^{1}_\infty]{a}+\norm[\Ww_{\infty}^{s-2}]{a}}{\rho(a)}\right)^{N_{s-1}}\norm[\Kk_{\varkappa-1}^{s-3}]{f}
      \right)\nonumber\\
    &\le C \frac{1}{\rho(a)}
      \left(\frac{\norm[W^{1}_\infty]{a}+\norm[\Ww_{\infty}^{s-1}]{a}}{\rho(a)}\right)^{N_{s-1}+1+(s-2)}
      \norm[\Kk_{\varkappa-1}^{s-2}]{f}.
  \end{align*}
  Note that
  $$N_{s-1}+(s-1)=\frac{(s-1)(s-2)}{2}+(s-1)=\frac{s(s-1)}{2}=N_s.$$
  We now use \eqref{eq:fKkbound} and the fact that \eqref{eq:DeltaIso}
  is a boundedly invertible isomorphism to conclude that there exist
  $C_s$ such that \eqref{eq:Uuapriori} holds.
\end{proof}
Throughout the rest of this section $\D$ is assumed a bounded
polygonal domain and $\varkappa>0$ the constant from
Theorem~\ref{thm:bacuta2}.

\begin{assumption}\label{ass:FEM}
  For some fixed $s\in\N$, $s\ge 2$, there exist constants $C>0$ and
  $\alpha>0$, and a sequence $(X_l)_{l\in\N}$ of subspaces of
  $X=H_0^1(\D;\C)=: H_0^1$, such that
  \begin{enumerate}
  \item $\sw{l}:={\rm dim}(X_l)$, $l\in\N$, satisfies Assumption
    \ref{ass:SW} (for some $K_\SW>0$),
  \item for all $l\in\N$
    \begin{equation}\label{eq:conv}
      \sup_{0\neq u\in \Kk_{\varkappa+1}^{s}}\frac{\inf_{v\in X_l}\norm[H_0^1]{u-v}}{\norm[\Kk_{\varkappa+1}^{s}]{u}}
      \le C \sw{l}^{-\conv}.
    \end{equation}
  \end{enumerate}
\end{assumption}
The constant $\conv$ in Assumption \ref{ass:FEM} can be interpreted as
the convergence rate of the finite element method. For the Kondrat'ev
space $\Kk_{\varkappa+1}^{s}(\D)$, finite element spaces $X_l$ of
piecewise polynomials of degree $s-1$ have been constructed in
\cite[Theorem~4.4]{BNZPolygon}, which achieve the optimal (in space
dimension $2$) convergence rate
\begin{equation}\label{eq:optimalrate}
  \conv=\frac{s-1}{2}
\end{equation}
in \eqref{eq:conv}.  For these spaces, Assumption \ref{ass:FEM} holds
with this $\alpha$, which consequently allows us to retain optimal
convergence rates. Nonetheless we keep the discussion general in the
following, and assume arbitrary positive $\alpha>0$.

We next introduce the finite element solutions of \eqref{eq:elliptic2}
in the spaces $X_l$, and provide the basic error estimate.
\begin{lemma}\label{lemma:femapprox}
  Let Assumption \ref{ass:FEM} be satisfied for some $s\ge 2$. Let
  $f\in \Kk_{\varkappa-1}^{s-2}(\D)$ and
  $$a\in W^{1}_\infty(\D)\cap \Ww^{s-1}_{\infty}(\D)\subseteq L^\infty(\D)$$
  with $\rho(a)>0$ and denote for $l\in\N$ by $\Uu^l(a)\in X_l$ the
  unique solution of
  \begin{equation*}
    \int_\D a(\nabla \Uu^l)^\top\overline{\nabla v}\dd \bx
    =\dup{f}{v}\qquad \forall v\in X_l,
  \end{equation*}
  where the right hand side denotes the (sesquilinear) dual pairing
  between $H^{-1}(\D)$ and $H_0^1(\D)$.  Then for the solution
  $\Uu(a)\in H_0^1(\D)$ it holds with the constants $N_s$, $C_s$ from
  Theorem.~\ref{thm:bacuta2},
  \begin{equation*}
    \norm[H_0^1]{\Uu(a)-\Uu^l(a)}
    \le \sw{l}^{-\conv} C \frac{\norm[L^\infty]{a}}{\rho(a)}
    \norm[\Kk_{\varkappa+1}^{s}]{\Uu(a)}
    \le \sw{l}^{-\conv}
    C C_s
    \frac{(\norm[W^{1}_\infty]{a}+\norm[\Ww^{s-1}_{\infty}]{a})^{N_s+1}}{\rho(a)^{N_s+2}}\norm[\Kk_{\varkappa-1}^{s-2}]{f}\;.
  \end{equation*}
  Here $C>0$ is the constant from Assumption \ref{ass:FEM}.
\end{lemma}
\begin{proof}
  By C\'ea's lemma in complex form we derive that
  \begin{align*}
    \norm[H_0^1]{\Uu(a)-\Uu^l(a)}
    \le \frac{\norm[L^\infty]{a}}{\rho(a)}
    \inf_{v\in X_l}  \norm[H_0^1]{\Uu(a)-v}.
  \end{align*}
  Hence the assertion follows by Assumption \ref{ass:FEM} and
  \eqref{eq:Uuapriori}.
\end{proof}
  
Throughout the rest of this section, as earlier we expand the
logarithm of the diffusion coefficient
\begin{equation*}\label{eq:diffusionml}
  a(\by)=\exp\Bigg(\sum_{j\in\N}y_j\psi_j\Bigg)
\end{equation*}
in terms of a sequence
$\psi_j\in W^{1}_\infty(\D)\cap\Ww_\infty^{s-1}(\D)$, $j\in\N$. Denote
\begin{equation}\label{eq:b1b2ml}
  b_{1,j}:=\norm[L^\infty]{\psi_j},\quad
  b_{2,j}:=\max\big\{\norm[W^{1}_\infty]{\psi_j},\norm[\Ww_\infty^{s-1}]{\psi_j}\big\}
\end{equation}
and $\bb_1:=(b_{1,j})_{j\in\N}$, $\bb_2:=(b_{2,j})_{j\in\N}$.

  \begin{example}\label{ex:sin}
    Let $\D=[0,1]$ and $\psi_j(x)=\sin(jx)j^{-r}$ for some $r>2$. Then
    $\bb_1\in\ell^{p_1}(\NN)$ for every $p_1>\frac{1}{r}$ and
    $\bb_2\in\ell^{p_2}(\NN)$ for every $p_2>\frac{1}{r-(s-1)}$.
  \end{example}

  In the next proposition we verify Assumption \ref{ass:ml}. This will
  yield validity of the multilevel convergence rates proved in
  Theorems~\ref{thm:mlint} and \ref{thm:mlquad} 
  in the present setting as we discuss subsequently.

  \begin{proposition}\label{prop:u-ul}
    Let Assumption \ref{ass:FEM} be satisfied for some $s\ge 2$ and
    $\alpha>0$. Let $\bb_1\in\ell^{p_1}(\NN)$,
    $\bb_2\in\ell^{p_2}(\NN)$ with $p_1$, $p_2\in (0,1)$.

    Then there exist $\xi>0$ and $\delta>0$ such that
    \begin{equation}\label{eq:uby}
      u(\by):=\Uu\Bigg(\exp\Bigg(\sum_{j\in\N}y_j\psi_j\Bigg)\Bigg)
    \end{equation}
    is $(\bb_1,\xi,\delta,H_0^1)$-holomorphic, and for every $l\in\N$
    \begin{enumerate}
    \item\label{item:ulhol}
      $u^l(\by):=\Uu^l(\exp(\sum_{j\in\N}y_j\psi_j))$ is
      $(\bb_1,\xi,\delta,H_0^1)$-holomorphic,
    \item\label{item:u-ulhol1} $u-u^l$ is
      $(\bb_1,\xi,\delta,H_0^1)$-holomorphic,
    \item\label{item:u-ulhol2} $u-u^l$ is
      $(\bb_2,\xi,\delta\sw{l}^{-\conv},H_0^1)$-holomorphic.
    \end{enumerate}
  \end{proposition}
  \begin{proof}
    {\bf Step 1.} We show \ref{item:ulhol} and \ref{item:u-ulhol1}.
    The argument to show that $u^l$ is
    $(\bb_1,\xi,\delta,H_0^1)$-holomorphic (for some constants
    $\xi>0$, $\delta>0$ independent of $l$) is essentially the same as
    in Section \ref{sec:pdc}.

    We wish to apply Theorem \ref{thm:bdX} with
    $\data=L^\infty(\D)$ and $X=H^1_0$. 
    To this end let
    $$O_1=\set{a\in L^\infty(\D;\C)}{\rho(a)>0} \subset L^\infty(\D; \C).$$  
    By assumption, $b_{1,j}=\norm[L^\infty]{\psi_j}$ satisfies
    $\bb_1=(b_{1,j})_{j\in\N}\in\ell^{p_1}(\NN )\subseteq
    \ell^1(\NN)$, which corresponds to assumption \ref{item:psi} of
    Theorem \ref{thm:bdX}. It remains to verify assumptions
    \ref{item:uhol}, \ref{item:norma} and \ref{item:loclip} of
    Theorem \ref{thm:bdX}:
    \begin{enumerate}
    \item $\Uu^l:O_1\to H_0^1$ is holomorphic: This is satisfied
      because the operation of inversion of linear operators is
      holomorphic on the set of boundedly invertible linear
      operators. Denote by $A_l:X_l\to X_l'$ the differential
      operator $$A_lu=-\div(a\nabla u)\in X_l'$$ via
      $$\dup{A_lu}{v}=\int_{\D}a\nabla u^\top \overline{\nabla v}\dd \bx \quad \forall v\in X_l.$$ 
      Observe that $A_l$ depends boundedly and linearly (thus
      holomorphically) on $a$, and therefore, the map
      $a\mapsto A_l(a)^{-1}f=\Uu^l(a)$ is a composition of holomorphic
      functions. We refer once more to \cite[Example 1.2.38]{JZdiss}
      for more details.
    \item It holds for all $a\in O$
      \begin{equation*}
        \norm[H_0^1]{\Uu^l(a)}\le \frac{\norm[X_l']{f}}{\rho(a)}
        \le \frac{\norm[H^{-1}]{f}}{\rho(a)}.
      \end{equation*}
      The first inequality follows by the same calculation as
      \eqref{eq:apriori} (but with $X$ replaced by $X_l$), and the
      second inequality follows by the definition of the dual norm,
      viz
      \begin{equation*}
        \norm[X_l']{f}=\sup_{0\neq v\in X_l}\frac{|\dup{f}{v}|}{\norm[H_0^1]{v}}
        \le\sup_{0\neq v\in H_0^1}\frac{|\dup{f}{v}|}{\norm[H_0^1]{v}}=
        \norm[H^{-1}]{f}.
      \end{equation*}
    \item For all $a$, $b\in O$ we have
      \begin{equation*}
        \norm[H_0^1]{\Uu^l(a)-\Uu^l(b)}\le \norm[H^{-1}]{f}
        \frac{1}{\min\{\rho(a),\rho(b)\}^2}
        \norm[L^\infty]{a-b},
      \end{equation*}
      which follows again by the same calculation as in in the proof
      of \eqref{eq:lipschitz}.
    \end{enumerate}
    According to Theorem
    \ref{thm:bdX}  
    the map
    $$\Uu^l\in L^2(U,X_l;\gamma)\subseteq L^2(U,H_0^1;\gamma)$$ is
    $(\bb_1,\xi_1,\tilde C_1,H_0^1)$-holomorphic, for some fixed
    constants $\xi_1>0$ and $\tilde C_1>0$ depending on $O_1$ but
    independent of $l$.  In fact the argument also works with $H_0^1$
    instead of $X_l$, i.e.\ also $u$ is
    $(\bb_1,\xi_1,\tilde C_1,H_0^1)$-holomorphic (with the same
    constants $\xi_1$ and $\tilde C_1$).

    Finally, it follows directly from the definition that the
    difference $u-u^l$ is $(\bb_1,\xi,2\delta,H_0^1)$-holomorphic.
    
    {\bf Step 2.} To show \ref{item:u-ulhol2}, we set
    $$O_2=\set{a\in W^{1}_\infty(\D)\cap\Ww^{s-1}_\infty(\D)}{\rho(a)>0},$$ and verify again
    assumptions \ref{item:uhol}, \ref{item:norma} and
    \ref{item:loclip} of Theorem \ref{thm:bdX}, but now with
    ``$\data$'' in this lemma being
    $W^{1}_\infty(\D)\cap\Ww^{s-1}_\infty(\D)$. First, observe
    that with
    $$b_{2,j}:=\max\big\{\norm[\Ww^{s-1}_\infty]{\psi_j},\norm[W^{1}_\infty]{\psi_j}\big\},$$
    by assumption
    $$\bb_2=(b_{2,j})_{j\in\N}\in\ell^{p_2}(\NN)\hookrightarrow \ell^1(\NN)$$
    which corresponds to the assumption \ref{item:psi} of Theorem
    \ref{thm:bdX}.

    For every $l\in\N$:
    \begin{enumerate}
    \item $\Uu-\Uu^l:O_2\to H_0^1(\D)$ is holomorphic: Since $O_2$ can
      be considered a subset of $O_1$ (and $O_2$ is equipped with a
      stronger topology than $O_1$), Fr\'echet differentiability
      follows by Fr\'echet differentiability of
      $$\Uu-\Uu^l:O_1\to H_0^1(\D),$$ which holds by Step 1.
    \item For every $a\in O_2$
      \begin{equation*}
        \norm[H_0^1]{(\Uu-\Uu^l)(a)}
        \le \underbrace{\sw{l}^{-\conv}       C C_s
          \norm[\Kk_{\varkappa-1}^{s-2}]{f}}_{=:\delta_l}
        \frac{(\norm[W^{1}_\infty]{a}+\norm[\Ww^{s-1}_{\infty}]{a})^{N_s+1}}{\rho(a)^{N_s+2}}
      \end{equation*}
      by Lemma \ref{lemma:femapprox}.
    \item For every $a$, $b\in O_2\subseteq O_1$, by Step 1 and
      \eqref{eq:lipschitz}, \begin{align*}
        \norm[H_0^1]{(\Uu-\Uu^l)(a)-(\Uu-\Uu^l)(b)} &\le
        \norm[H_0^1]{\Uu(a)-\Uu(b)}
                                                      +\norm[H_0^1]{\Uu^l(a)-\Uu^l(b)}\nonumber\\
                                                    &\le
                                                      \norm[H^{-1}]{f}
                                                      \frac{2}{\min\{\rho(a),\rho(b)\}^2}
                                                      \norm[L^\infty]{a-b}. 
                            \end{align*}
                            We conclude with Theorem
                            \ref{thm:bdX}  
                            that there exist $\xi_2$ and $\tilde C_2$
                            depending on $O_2$, $\D$ but independent
                            of $l$ such that $u-u^l$ is
                            $(\bb_2,\xi_2,\tilde
                            C_2\delta_l,H_0^1)$-holomorphic.
                          \end{enumerate}

                          In all, the proposition holds with
                          \begin{equation*}
                            \xi:=\min\{\xi_1,\xi_2\}\qquad\text{and}\qquad
                            \delta:=\max\{\tilde C_1,\tilde C_2 C C_s\norm[\Kk_{\varkappa-1}^{s-2}]{f}\}.\qedhere
                          \end{equation*}
                        \end{proof}

                        Items \ref{item:u-ulhol1} and
                        \ref{item:u-ulhol2} of
                        Proposition~\ref{prop:u-ul} show that
                        Assumption~\ref{ass:FEM} implies validity of
                        Assumption~\ref{ass:ml}. This in turn allows
                        us to apply Theorems~\ref{thm:mlint} and
                        \ref{thm:mlquad}. Specifically,
                        assuming the optimal convergence rate
                        $\alpha=\frac{s-1}{2}$ in
                        \eqref{eq:optimalrate}, we obtain that for $u$
                        in \eqref{eq:uby} and every $n\in \N$ there is
                        $\eps:=\eps_n>0$ such that
                        $\work(\blev_{\eps})\leq n$ and the multilevel
                        interpolant $\VIml_\blev$ defined in
                        \eqref{eq:VIml} satisfies
                        \begin{equation*}
                          \norm[L^2(U,{H^1_0};\gamma)]{u-\VI_{\blev_\eps}^{\rm ML}u}\le C(1+\log n)n^{-R_I},\quad
                          R_I = \min\left\{\frac{s-1}{2},\frac{\frac{s-1}{2}(\frac{1}{p_1}-\frac{3}{2})}{\frac{s-1}{2}+\frac{1}{p_1}-\frac{1}{p_2}}\right\},
                        \end{equation*}  
                        and the multilevel quadrature operator
                        $\VQml_\blev$ defined in \eqref{eq:VQml}
                        satisfies
                        \begin{equation*}
                          \normc[H^1_0]{\int_Uu(\by)\dd\gamma(\by)-\VQ_{\blev_\eps}^{\rm ML}u}\le C(1+\log n)n^{-R_Q},\quad
                          R_Q=\min\left\{\frac{s-1}{2},\frac{\frac{s-1}{2}(\frac{2}{p_1}-\frac{5}{2})}{\frac{s-1}{2}+\frac{2}{p_1}-\frac{2}{p_2}}\right\}.
                        \end{equation*}
                        Let us consider these convergence rates in the
                        case where the $\psi_j$ are algebraically
                        decreasing, with this decrease encoded by some
                        $r>1$: if for fixed but arbitrarily small
                        $\varepsilon>0$ holds
                        $\norm[L^\infty]{\psi_j}\sim
                        j^{-r-\varepsilon}$, and we assume
                        (cp.~Ex.~\ref{ex:sin})
                        \begin{equation*}
                          \max\big\{\norm[W^{1}_\infty]{\psi_j},\norm[\Ww^{s-1}_{\infty}]{\psi_j}\big\}\sim
                          j^{-r+(s-1)-\varepsilon},
                        \end{equation*}
                        then setting $s:=r$ we can choose
                        $p_1=\frac{1}{r}$ and $p_2=1$. Inserting those
                        numbers, the convergence rates become
                        \begin{equation*}
                          R_I = \min\left\{\frac{r-1}{2},\frac{\frac{r-1}{2}(r-\frac{3}{2})}{\frac{r-1}{2}+r-1}\right\}=\frac{r}{3}-\frac{1}{2}
                          \quad\text{and}\quad
                          R_Q = \min\left\{\frac{r-1}{2},\frac{\frac{r-1}{2}(2r-\frac{5}{2})}{\frac{r-1}{2}+2r-2}\right\}=\frac{2r}{5}-\frac{1}{2}.
                        \end{equation*}
  
        \subsubsection{Parametric holomorphy of the posterior density in Bayesian PDE inversion}
         \label{sec:PosDnsBayInv}
                        Throughout this section we assume that
                        $\D\subseteq\R^2$ is a polygonal Lipschitz
                        domain and that
                        $f\in \Kk_{\varkappa-1}^{s-2}(\D)$ with
                        $\varkappa$ as in Theorem~\ref{thm:bacuta2}.
  
                        As in Section \ref{sec:BIP}, to treat the
                        approximation of the (unnormalized) posterior
                        density or its integral, we need an upper
                        bound on $\norm[H_0^1]{u(\by)}$ for all
                        $\by$. This is achieved by considering
                        \eqref{eq:elliptic2}  
                        with diffusion coefficient $a_0+a$ where
$$\rho(a_0):=\underset{\bx\in\D}{ \essinf }\Re(a_0) > 0.$$ 
The shift of the diffusion coefficient by $a_0$ ensures uniform
ellipticity for all
  $$a\in \set{a\in L^\infty(\D,\CC)}{\rho(a)\ge 0}.$$  As a consequence,
  solutions $\Uu(a_0+a)\in X=H_0^1(\D;\C)=:H^1_0$ of \eqref{eq:elliptic2}
  satisfy the apriori bound (cp.~\eqref{eq:apriori})
  \begin{subequations}\label{eq:Uulabound}
    \begin{equation*}
      \norm[H^1_0]{\Uu(a_0+a)}\le \frac{\norm[H^{-1}]{f}}{\rho(a_0)}.
    \end{equation*}
    As before, for a sequence of subspaces $(X_l)_{l\in\N}$ of
    $H_0^1(\D,\CC)$, for $a\in O$ we denote by $\Uu^l(a)\in X_l$ the
    finite element approximation to $\Uu(a)$.
    By the same calculation as for $\Uu$ it also holds
    \begin{equation*}
      \norm[H^1_0]{\Uu^l(a_0+a)}\le
      \frac{\norm[H^{-1}]{f}}{\rho(a_0)}
    \end{equation*}
    independent of $l$.
  \end{subequations}

  Assuming that $b_j=\norm[L^\infty]{\psi_j}$ satisfies
  $(b_j)_{j\in\N}\in\ell^1(\NN)$, the function
  $u(\by)=\Uu(a_0+a(\by))$
  with $$a(\by)=\exp\bigg(\sum_{j\in\N}y_j\psi_j\bigg),$$ is
  well-defined.  For a fixed \emph{observation} $\bobs\in\R^m$
  consider again the (unnormalized) posterior density given in
  \eqref{eq:TheParm},
  \begin{equation*}
    \tilde\pi(\by|\bobs):=\exp\left( -(\bobs-\bcalO(u(\by)))^\top \bGamma^{-1} (\bobs-\bcalO(u(\by))) \right).
  \end{equation*}
  Recall that $\bcalO:X\to\C^m$ (the observation operator) is assumed
  to be a bounded linear map, and $\bGamma\in\R^{m\times m}$ (the
  noise covariance matrix) is symmetric positive definite.
  For $l\in\N$ (tagging discretization level of the PDE), and with
  $u^l(\by)=\Uu^l(a_0+a(\by))$, we introduce approximations
  \begin{equation*}
    \tilde\pi^l(\by|\bobs):=\exp\left( -(\bobs-\bcalO(u^l(\by)))^\top \bGamma^{-1} (\bobs-\bcalO(u^l(\by))) \right)
  \end{equation*}
  to $\tilde\pi(\by|\bobs)$. In the following we show the analog of
  Proposition~\ref{prop:u-ul}, that is we show validity of the
  assumptions required for the multilevel convergence results.

  \begin{lemma}\label{lemma:PhiLip}
    Let $\bcalO:H_0^1(\D;\C)\to\C^m$ be a bounded linear operator,
    $\bobs\in\C^m$ and $\bGamma\in\R^{m\times m}$ symmetric positive
    definite.  Set
    \begin{equation*}
      \Phi:=\begin{cases}
        H_0^1(\D;\C)\to\C\\
        u\mapsto \exp(-(\bobs-\bcalO(u))^\top \bGamma^{-1}(\bobs-\bcalO(u))).
      \end{cases}
    \end{equation*}
    Then the function $\Phi$ is
    continuously differentiable 
    and  for every $r>0$ has a Lipschitz constant $K$
    solely depending on $\norm[]{\bGamma^{-1}}$,
    $\norm[L(H_0^1;\C^m)]{\bcalO}$, $\norm[]{\bobs}$ and $r$, on the set
    $$
    \set{u\in H_0^1(\D;\C)}{\norm[H_0^1]{u}<r}.
    $$
  \end{lemma}
  \begin{proof}
    The function $\Phi$ is continuously differentiable 
    as a composition of continuously differentiable  functions. 
    Hence
    for $u$, $v$ with $w:=u-v$ and with the derivative
    $D\Phi:H_0^1\to L(H_0^1;\C)$ of $\Phi$,
    \begin{equation}\label{eq:diffPhi}
      \Phi(u)-\Phi(v)=\int_{0}^1 D\Phi(v+tw)w \dd t.
    \end{equation}
    Due to the symmetry of $\bGamma$ it holds
    \begin{equation*}
      D\Phi(u+tw)w = 2
      \bcalO(w)^\top \bGamma^{-1}(\bobs-\bcalO(u+tw))
      \exp\Big(-(\bobs-\bcalO(u+tw))^\top \bGamma^{-1}(\bobs-\bcalO(u+tw))\Big).
    \end{equation*}
    If $\norm[H_0^1]{u}$, $\norm[H_0^1]{v}<r$ then also
    $\norm[H_0^1]{u+tw}<r$ for all $t\in[0,1]$ and we can bound
    \begin{equation*}
      |D\Phi(u+tw)w|\le
      K \norm[H_0^1]{w},
    \end{equation*}
    where
    \begin{equation} \label{ConstantK} K:=
      2\norm[L(H_0^1;\C^m)]{\bcalO}
      \norm[]{\bGamma^{-1}}(\norm[]{\bobs}+r
      \norm[L(H_0^1;\C^m)]{\bcalO})
      \exp(\norm[]{\bGamma^{-1}}(\norm[]{\bobs}+\norm[L(H_0^1;\C^m)]{\bcalO}
      r)^2).
    \end{equation}
    The statement follows by \eqref{eq:diffPhi}.
  \end{proof}

  \begin{remark} {\rm
    The reason why we require the additional positive $a_0$ term in
    \eqref{eq:Uulabound}, is to guarantee boundedness of the solution
    $\Uu(a)$ and Lipschitz continuity of $\Phi$.
  } \end{remark} 

  \begin{proposition}\label{prop:pi-pil}
    Let Assumption \ref{ass:FEM} be satisfied for some $s\ge 2$ and
    $\alpha>0$.  Let $a_0$,
    $(\psi_j)_{j\in\N}\subseteq W^{1}_\infty(\D)\cap
    \Ww^{s-1}_{\infty}(\D)$ and $\bb_1\in\ell^{p_1}(\NN)$,
    $\bb_2\in\ell^{p_2}(\NN)$ with $p_1$, $p_2\in (0,1)$ (see
    \eqref{eq:b1b2ml} for the definition of $\bb_1$, $\bb_2$).  Fix
    $\bobs\in\C^m$.

    Then there exist $\xi>0$ and $\delta>0$ such that
    $\tilde\pi(\by|\bobs)$ is $(\bb_1,\xi,\delta,\C)$-holomorphic, and
    for every $l\in\N$
    \begin{enumerate}
    \item\label{item:pilhol} $\tilde\pi^l(\by|\bobs)$ is
      $(\bb_1,\xi,\delta,\C)$-holomorphic,
    \item\label{item:pi-pilhol1}
      $\tilde\pi(\by|\bobs)-\tilde\pi^l(\by|\bobs)$ is
      $(\bb_1,\xi,\delta,H_0^1)$-holomorphic,
    \item\label{item:pi-pilhol2}
      $\tilde\pi(\by|\bobs)-\tilde\pi^l(\by|\bobs)$ is
      $(\bb_2,\xi,\delta\sw{l}^{-\conv},H_0^1)$-holomorphic.
    \end{enumerate}
  \end{proposition}
  \begin{proof}
    {\bf Step 1.} We show \ref{item:pilhol} and \ref{item:pi-pilhol1}.
    Set $$O_1:=\set{a\in L^\infty(\D;\C)}{\rho(a)>0}.$$ By
    \eqref{eq:Uulabound} %
    for all $a\in O_1$ and all $l\in\N$ with
    $r:=\frac{\norm[H^{-1}]{f}}{\rho(a_0)}$
    \begin{equation}\label{eq:Uuboundr}
      \norm[H_0^1]{\Uu^l(a_0+a)}\le r\qquad\text{and}\qquad
      \norm[H_0^1]{\Uu(a_0+a)}\le r.
    \end{equation}
    As in Step 1 of the proof of Proposition~\ref{prop:u-ul}, one can
    show that $u(\by)=\Uu(a_0+a(\by))$ and $u^l(\by)=\Uu(a_0+a(\by))$
    where $a(\by)=\exp(\sum_{j\in\N}y_j\psi_j)$ are
    $(\bb_1,\xi_1,\tilde C_1,H_0^1)$-holomorphic for certain $\xi_1>0$
    and $\tilde C_1>0$ (the only difference to
    Proposition~\ref{prop:u-ul} is the affine offset $a_0$ in \eqref{eq:elliptic2},
    which ensures a positive lower bound for $a+a_0$).
    In the following $\Phi$ is as in Lemma
    \ref{lemma:PhiLip} and $T_{a_0}(a):=a_0+a$ so that
    \begin{equation}\label{eq:posteriorbdX}
      \tilde \pi (\by|\bobs) = \Phi(\Uu^l(T_{a_0}(a(\by)))).
    \end{equation}
    With $b_{1,j}=\norm[L^\infty]{\psi_j}$, by the assumption
    $$\bb_1=(b_{1,j})_{j\in\N}\in\ell^{p_1}(\NN)\hookrightarrow \ell^1(\NN)$$
    which corresponds to assumption \ref{item:psi} of Theorem \ref{thm:bdX}.  
    We now verify assumptions \ref{item:uhol}, \ref{item:norma} and \ref{item:loclip} 
    of Theorem \ref{thm:bdX} for \eqref{eq:posteriorbdX}.
    
    For every $l\in\N$:
    \begin{enumerate}
    \item\label{item:abc} %
      The map
      \begin{equation*}
        \Phi\circ \Uu^l\circ T_{a_0}:\begin{cases}
          O_1 \to \C\\
          a\mapsto \Phi(\Uu(T_{a_0}(a)))
        \end{cases}
      \end{equation*}
      is holomorphic as a composition of holomorphic functions.
    \item for all $a\in O_1$, since
      $\norm[H_0^1]{\Uu^l(T_{a_0}(a))}\le r$
      \begin{equation*}
        |\Phi(\Uu^l(T_{a_0}(a)))|\le \exp((\norm[]{\bobs}+\norm[L(H_0^1(\D;\C);\C^m)]{\bcalO}r)^2\norm[]{\bGamma^{-1}})
      \end{equation*}
      and thus assumption \ref{item:norma} of Theorem \ref{thm:bdX} is
      trivially satisfied for some $\delta>0$ independent of $l$,
    \item for all $a$, $b\in O_1$ by Lemma \ref{lemma:PhiLip} and the
      same calculation as in \eqref{eq:lipschitz}
      \begin{equation}\label{eq:uula-uulb}
        \begin{aligned}
          |\Phi(\Uu^l(T_{a_0}(a)))-\Phi(\Uu^l(T_{a_0}(b)))|&\le
          K\norm[H_0^1]{\Uu^l(T_{a_0}(a))-\Uu^l(T_{a_0}(b))}
          \\
          & \le K
          \frac{\norm[H^{-1}]{f}}{\rho(a_0)}\norm[L^\infty]{a-b},
        \end{aligned}
      \end{equation}
      where $K$ is the constant given as in \eqref{ConstantK}.
    \end{enumerate}

    Now we can apply Theorem \ref{thm:bdX} to conclude that there
    exist $\xi_1$, $\delta_1$ (independent of $l$) such that
    $\tilde\pi^l(\cdot|\bobs)$ is
    $(\bb_1,\xi_1,\delta_1,H_0^1)$-holomorphic for every
    $l\in\N$. Similarly one shows that $\tilde\pi(\cdot|\bobs)$ is
    $(\bb_1,\xi_1,\delta_1,H_0^1)$-holomorphic, and in particular
    $\tilde\pi(\cdot|\bobs)-\tilde\pi^l(\cdot|\bobs)$ is
    $(\bb_1,\xi_1,2\delta_1,H_0^1)$-holomorphic.

    {\bf Step 2.}  Set
    $$O_2=\set{a\in W^{1}_\infty(\D)\cap\Ww^{s-1}_\infty(\D)}{\rho(a)>0}.$$
    We verify once more assumptions \ref{item:uhol}, \ref{item:norma}
    and \ref{item:loclip} of Theorem \ref{thm:bdX} with ``$\data$'' in
    this lemma being $W^{1}_\infty(\D)\cap\Ww^{s-1}_\infty(\D)$. With
    $b_{2,j}=\max\{\norm[W^{1}_\infty]{\psi_j},\norm[\Ww^{s-1}_\infty]{\psi_j}\}$,
    by the assumption
    $$4\bb_2=(b_{2,j})_{j\in\N}\in\ell^{p_2}(\NN)\hookrightarrow \ell^1(\NN),$$
    which corresponds to assumption \ref{item:psi} of Theorem \ref{thm:bdX}.

    We will apply Theorem \ref{thm:bdX} with the function
    \begin{equation}\label{eq:pit-pitl}
      \tilde \pi(\by|\bobs)-
      \tilde \pi^l(\by|\bobs)
      =\Phi(\Uu(T_{a_0}(a(\by))))-\Phi(\Uu^l(T_{a_0}(a(\by)))).
    \end{equation}
    For every $l\in\N$:
    \begin{enumerate}
    \item By item \ref{item:abc}
      in %
      Step 1 (and because $O_2\subseteq O_1$)
      \begin{equation*}
        \Phi\circ \Uu\circ T_{a_0}-\Phi\circ\Uu^l\circ T_{a_0}:\begin{cases}
          O_2 \to \C\\
          a\mapsto \Phi(\Uu(T_{a_0}(a)))-\Phi(\Uu^l(T_{a_0}(a)))
        \end{cases}
      \end{equation*}
      is holomorphic,
    \item for every $a\in O_2$, by Lemma \ref{lemma:femapprox}
      \begin{equation*}
        \norm[H_0^1]{\Uu(T_{a_0}(a))-\Uu^l(T_{a_0}(a))}\le \sw{l}^{-\conv} CC_s%
        \frac{(\norm[W^{1}_\infty]{a_0+a}+\norm[\Ww^{s-1}_{\infty}]{a_0+a})^{N_s+1}}{\rho(a_0+a)^{N_s+2}}
        \norm[\Kk^{s-2}_{\varkappa-1}]{f}.
      \end{equation*}
      Thus by \eqref{eq:Uuboundr} and Lemma \ref{lemma:PhiLip}
      \begin{equation*}\label{eq:PhiUuSa}
        |\Phi(\Uu(T_{a_0}(a)))-\Phi(\Uu^l(T_{a_0}(a)))|
        \le \sw{l}^{-\conv} KCC_s
        \frac{(\norm[W^{1}_\infty]{a_0+a}+\norm[\Ww^{s-1}_{\infty}]{a_0+a})^{N_s+1}}{\rho(a_0+a)^{N_s+2}}\norm[\Kk^{s-2}_{\varkappa-1}]{f},
      \end{equation*}
    \item for all $a$, $b\in O_2\subseteq O_1$ by \eqref{eq:uula-uulb}
      (which also holds for $\Uu^l$ replaced by $\Uu$):
      \begin{equation*}
        |\Phi(\Uu(T_{a_0}(a)))-\Phi(\Uu^l(T_{a_0}(a)))-(\Phi(\Uu(b))-\Phi(\Uu^l(b)))|\le 2 K \frac{\norm[H^{-1}]{f}}{\rho(a_0)}\norm[L^\infty]{a-b}. %
      \end{equation*}
    \end{enumerate}
    By Theorem \ref{thm:bdX} and \eqref{eq:pit-pitl} we conclude that
    there exists $\delta>0$ and $\xi_2$ independent of $l$ such that
    $\tilde \pi(\by|\bobs)-\tilde \pi^l(\by|\bobs)$ is
    $(\bb_2,\delta \sw{l}^{-\alpha},\xi_2,H_0^1)$-holomorphic.
  \end{proof}

  Items \ref{item:u-ulhol1} and \ref{item:u-ulhol2} of
  Proposition~\ref{prop:u-ul} show that Assumption \ref{ass:FEM}
  implies validity of Assumption \ref{ass:ml}. This in turn allows us
  to apply Theorems~\ref{thm:mlint} and \ref{thm:mlquad}.
  Specifically, assuming the optimal convergence rate
  \eqref{eq:optimalrate}, we obtain that for every $n\in \NN$ there is
  $\eps:=\eps_n>0$ such that $\work(\blev_{\eps}) \le n$ and the
  multilevel interpolant $\VIml_{\blev}$ defined in \eqref{eq:VIml}
  satisfies
  \begin{equation*}
    \norm[L^2(U,{H^1_0};\gamma)]{\tilde \pi(\cdot|\bobs)-\VI_{\blev_\eps}^{\rm ML}
      \tilde \pi(\cdot|\bobs)}\le C(1+\log n) n^{-R_I},\quad
    R_I = \min\left\{\frac{s-1}{2},
      \frac{\frac{s-1}{2}(\frac{1}{p_1}-\frac{3}{2})}{\frac{s-1}{2}+\frac{1}{p_1}-\frac{1}{p_2}}
    \right\}.
  \end{equation*}

  Of higher practical interest is the application of the multilevel
  quadrature operator $\VQml$ defined in \eqref{eq:VQml}.  In case the
  prior is chosen as $\gamma$, then
  $$\int_U\tilde \pi(\by|\bobs)\dd\gamma(\by)$$ equals the
  normalization constant in \eqref{eq:BayesDens}.  It can be
  approximated with the error converging like
  \begin{equation}\label{eq:posteriorconv}
    \left|\int_U\tilde \pi(\by|\bobs)\dd\gamma(\by)-\VQ_{\blev_\eps}^{\rm ML}u\right|
    \le C(1+\log n)n^{-R_Q},\quad
    R_Q:= 
    \min\left\{\frac{s-1}{2},
      \frac{\frac{s-1}{2}(\frac{2}{p_1}-\frac{5}{2})}{\frac{s-1}{2}+\frac{2}{p_1}-\frac{2}{p_2}}
    \right\}.
  \end{equation}

  Typically, one is not merely interested in the
  normalization constant 
$$
 Z = \int_U\tilde \pi(\by|\bobs)\dd\gamma(\by),
$$ 
  but for example also in an estimate of the $j$th parameter $y_j$ 
  given as the conditional expectation, which up to multiplying with the
  normalization constant $\frac{1}{Z}$, corresponds to
$$
\int_U y_j \tilde \pi(\by|\bobs)\dd\gamma(\by).
$$  
Since
$\by\mapsto y_j$ is analytic, one can show the same convergence rate
as in \eqref{eq:posteriorconv} for the multilevel quadrature applied
with the approximations $\by\mapsto y_j\tilde\pi^l(\by|\bobs)$ for
$l\in\N$.  Moreover, for example if $\phi:H_0^1(\D;\C)\to \C$ is a
bounded linear functional representing some quantity of interest, then
we can show the same error convergence for the approximation of
$$\int_U \phi(u(\by))\tilde\pi(\by|\bobs)\dd\gamma(\by)$$ with the
multilevel quadrature applied with the approximations
$\phi(u^l(\by))\tilde\pi^l(\by|\bobs)\dd\gamma(\by)$ to the integrand
for $l\in\N$.
\subsection{Linear  multilevel interpolation and quadrature approximation}
\label{Multilevel approximation and quadrature in Bochner spaces}
In this section, we briefly recall some results from \cite{dD21}
  (see also \cite{dD-Erratum22} for some corrections).
  The difference with Sections
  \ref{sec:SetNot} - \ref{sec:Approx} is, that the interpolation and
  quadrature operators presented in this section are \emph{linear}
  operators; in contrast, the operators $\VIml$, $\VQml$ in \eqref{eq:VIml},
  \eqref{eq:VQml} are in general nonlinear, 
  since they build on the
  approximations $u^n$ of $u$ from Assumption \ref{ass:ml}. 
  These approximations are not assumed to be linear (and, in general, are not linear) 
  in $u$.

  In this section we proceed similarly, but with $u^n:=P_nu$
  \emph{ for a linear operator $P_n$}; 
  if $u$ denotes the solution of an elliptic PDE in $H^1(\D)$, $P_n$ 
  could for instance be the orthogonal
  projection from $H^1(\D)$ into some fixed finite dimensional
  subspace. We emphasize, that such operators are not available in
  practice, and many widely used implementable algorithms 
  (such as the finite element method, boundary element method, finite differences) 
  realize projections that are \emph{not of this type}. 
  We will discuss this in more detail in Remark \ref{rmk:comparison}. 
  Therefore the present results are mainly of theoretical rather than of practical importance.
  On a positive note, the convergence rates for both, Smolyak sparse-grid
  interpolation and  quadrature obtained in this section
  via thresholding (see \eqref{G(xi)} ahead) improve the rates shown
  in the previous sections for the discretization levels allocated via
  Algorithm \ref{alg:levels} by a logarithmic factor, cp.~Theorems
  \ref{thm:mlint} and \ref{thm:mlquad}. Yet we emphasize that the
  latter are computable (in linear complexity, see
  Sec.~\ref{sec:mleps}). 
  
\subsubsection{Multilevel Smolyak sparse-grid interpolation}
\label {Interpolation}
In this section, 
we recall some results in \cite{dD21} 
(see also, \cite{dD-Erratum22} for some corrections) 
on  linear multilevel polynomial interpolation approximation in Bochner spaces.

In order to have a correct definition of interpolation operator let us
impose some necessary restrictions on $v \in L^2(U,X;\gamma)$.
Let $\Ee$ be a $\gamma$-measurable subset in $U$ such that $\gamma(\Ee) =1$ and $\Ee$ contains  all $\by \in U$ with $|\by|_0 < \infty$, where $|\by|_0$ denotes the number of nonzero components $y_j$ of $\by$. For a given $\Ee$ and separable Hilbert space $X$, let $C_\Ee(U)$ the set of  all functions $v$ on $U$ taking values in $X$ such that $v$ are continuous on $\Ee$  w.r. to the local convex topology of  $U:= \RRi$ (see Example~\ref{ex:R^infty}). We define $L^2_\Ee(U,X, \gamma):= L^2(U,X; \gamma)\cap C_\Ee(U)$. We will treat  all elements 
	$v \in L^2_\Ee(U,X, \gamma)$ as their representative belonging to  $C_\Ee(U)$. Throughout this  and next sections, we fix  a set $\Ee$. 
	
We define the univariate operator $\Delta^{{\rm I}}_m$ for $m \in \NN_0$ by
\begin{equation} \nonumber
	\Delta^{{\rm I}}_m
	:= \
	I_m - I_{m-1},
\end{equation} 
with the convention $I_{-1} = 0$, where $I_m$ is defined in Section \ref{sec:int}.

For  $v \in L^2_\Ee(U,X;\gamma)$, we introduce the tensor product operator $\Delta^{{\rm I}}_\bnu$ for $\bnu \in \cF$ by
\begin{equation} \nonumber
	\Delta^{{\rm I}}_\bnu(v)
	:= \
	\bigotimes_{j \in \NN} \Delta^{{\rm I}}_{\nu_j}(v),
\end{equation}
where the univariate operator
$\Delta^{{\rm I}}_{\nu_j}$ is applied to the univariate function 
$\bigotimes_{j' =1 }^{j-1} \Delta^{{\rm I}}_{\nu_{j'}}(v)$ by considering this function 
as a function of  variable $y_j$ with all remaining variables held fixed. 
From the definition of $L^2_\Ee(U,X;\gamma)$ one infers
that the operators  $\Delta^{{\rm I}}_\bnu$ are well-defined for all $\bnu \in \cF$. 

Let us recall a setting from \cite{dD21} of 
linear fully discrete polynomial interpolation of functions in the Bochner space $L^2(U,X^2;\gamma)$,
with the approximation error measured by the norm of the  Bochner space 
$L^2(U,X^1;\gamma)$ for separable Hilbert spaces $X^1$ and $X^2$.   
To construct  linear fully discrete methods of polynomial interpolation, 
besides weighted $\ell^2$-summabilities with respect to $X^1$ and $X^2$ 
we need an approximation property on the spaces $X^1$ and $X^2$ combined in the following assumption.

\begin{assumption} \label{Assumption II} 
	For the Hilbert spaces $X^1$
  and $X^2$ and $v \in L^2_\Ee(U,X^2; \gamma)$ represented by the
  series
	\begin{equation} \label{HermiteSeriesV}
		v = \sum_{\bnu\in\cF} v_\bnu H_\bnu, \quad v_\bnu \in X^2,
	\end{equation}	
	there holds the following.
	\begin{itemize}
		\item[{\rm (i)}] 
		$X^2$ is  a linear subspace of  $X^1$ and $\|\cdot\|_{X^1} \le C\, \|\cdot\|_{X^2}$.
		\item[{\rm (ii)}] 
		For $i=1,2$,  there exist numbers $q_i$ with $0<q_1\leq q_2 <\infty$ and $q_1 < 2$, and families $(\sigma_{i;\bnu})_{\bnu \in \cF}$  of numbers strictly larger than $1$ such that 
		$\sigma_{i;\bee_{j'}} \le \sigma_{i;\bee_j}$ if $j' < j$, and
		\begin{equation} \nonumber
			\sum_{\bnu\in\cF} (\sigma_{i;\bnu} \|v_\bnu\|_{X^i})^2 \le M_i <\infty \quad  \text{and} \quad 
		\left(p_{\bnu}(\tau,\lambda)\sigma_{i;\bnu}^{-1}\right)_{\bnu \in \cF} \in \ell^{q_i}(\cF)
		\end{equation}		
	 for every $\tau > \frac{17}{6}$ and  $\lambda \ge 0$, where we recall that $(\be_j)_{j\in \NN}$ is  the standard basis of $\ell^2(\NN)$. 
		\item[{\rm (iii)}]  	 
		There are a sequence $(V_n)_{n \in \NN_0}$ of subspaces $V_n \subset X^1 $ of dimension $\le n$, and a sequence $(P_n)_{n \in \NN_0}$ of linear operators from $X^1$ into $V_n$, 
		and a number $\alpha>0$ such that
		\begin{equation} \label{spatialappnX}
			\|P_n(v)\|_{X^1} \leq 
			C\|v\|_{X^1} , \quad
			\|v-P_n(v)\|_{X^1} \leq 
			Cn^{-\alpha} \|v\|_{X^2}, \quad \forall n \in \NN_0, \quad \forall v \in X^2.
		\end{equation}
	\end{itemize}
\end{assumption}

Let Assumption \ref{Assumption II} hold for Hilbert spaces $X^1$ and $X^2$   and $v \in L^2_\Ee(U,X^2; \gamma)$. Then we are able to construct a linear fully discrete polynomial interpolation approximation. 
We introduce the interpolation operator 
	$$
	\Ii_G: L^2_\Ee(U,X^2; \gamma) \to \Vv(G)
	$$ 
	for a given finite set $G \subset \NN_0 \times \cF$ by
\begin{equation} \nonumber
	\Ii_G v
	:= \
	\sum_{(k,\bnu) \in G} \delta_k \Delta^{{\rm I}}_\bnu (v),
\end{equation}
where $\Vv(G)$ denotes the subspace in $L^2(U,X^1; \gamma)$ of  all functions $v$
of the form
\begin{equation} \nonumber
	v
	\ = \
	\sum_{(k,\bnu) \in G} v_k  H_\bnu, \quad v_k \in V_{2^k}.
\end{equation}
Notice that interpolation $v \mapsto \Ii_G v$
is a linear  method of fully discrete polynomial interpolation approximation,
which is the sum taken over the (finite) index set $G$, 
of mixed tensor products of dyadic scale successive differences of ``spatial" approximations to $v$, 
and of successive differences of their parametric Lagrange interpolation polynomials. 

Define for $\xi>0$
\begin{equation} \label{G(xi)}
	G(\xi)
	:= \ 
	\begin{cases}
		\big\{(k,\bnu) \in \NN_0 \times\cF: \, 2^k \sigma_{2;\bnu}^{q_2} \leq \xi\big\} \quad &{\rm if }  \ \alpha \le 1/q_2 - 1/2;\\
		\big\{(k,\bnu) \in \NN_0 \times\cF: \, \sigma_{1;\bnu}^{q_1} \le \xi, \  
		2^{(\alpha+1/2)  k} \sigma_{2;\bnu}\leq \xi^\vartheta \big\} \quad  & {\rm if }  \ \alpha > 1/q_2 - 1/2,
	\end{cases}
\end{equation}
where
\begin{equation} \label{theta}
	\vartheta:= \frac{1}{q_1} + \frac{1}{2\alpha}\bigg(\frac{1}{q_1} - \frac{1}{q_2}\bigg).
\end{equation}
For any $\xi  > 1$ we have that $G(\xi) \subset F(\xi)$ where 
$$
F(\xi):= \{(k,\bnu) \in \NN_0 \times \Ff: k \le \log \xi, \ \bnu \in \Lambda(\xi) \}
$$
 and 
\begin{equation*} \label{Lambda(xi)}
	\Lambda(\xi)
	:= \ 
	\begin{cases}
		\big\{\bnu \in\cF: \,  \sigma_{2;\bnu}^{q_2} \leq \xi\big\} 
		\quad &{\rm if }  \ \alpha \le 1/q_2 - 1/2;\\
		\big\{\bnu \in\cF:  \, \sigma_{1;\bnu}^{q_1} \le \xi \big\} \  
	\quad  & {\rm if }  \ \alpha > 1/q_2 - 1/2.
	\end{cases}
\end{equation*}
From \cite[Lemma 3.3]{dung2021deep} it follows that 
\begin{equation} \label{supp(nu)}
\bigcup_{\bnu \in \Lambda(\xi)} \operatorname{supp} (\bnu) \subset \{1,..., \lfloor C\xi \rfloor\}
\end{equation}
for some positive constant $C$ that is independent of $\xi>1$. 
Denote by $\Gamma_\bnu$ and $\Gamma(\Lambda)$,
the set of interpolation points in the operators $\Delta^{{\rm I}}_\bnu$ and $\VI_\Lambda$, respectively. 
We have that 
$$
\Gamma_\bnu = \{\by_{\bnu - \be;\bm}:  \be \in \EE_\bnu; \ m_j = 0,\ldots,s_j - e_j, \ j \in \NN \},
$$
and 
$$
\Gamma(\Lambda) = \bigcup_{\bnu \in \Lambda} \Gamma_\bnu,
$$
	where
 $\EE_\bnu$ is the subset in $\cF$ of all $\be$ such that $e_j$ is $1$ or $0$ if $\nu_j > 0$, and $e_j$ is $0$ if $\nu_j = 0$, and $\by_{\bnu;\bm}:= (y_{\nu_j;m_j})_{j \in \NN}$. 	
Hence, by \eqref{supp(nu)}
	$$
\Gamma(\Lambda(\xi)) \subset \RR^{\lfloor C\xi \rfloor} \subset U,
$$
and therefore,
the operator $\Ii_{G(\xi)}$ is well-defined for any $v \in L^2_\Ee(U,X^2; \gamma)$ since $v$ is continuous on $\RR^{\lfloor C\xi \rfloor}$.

\begin{theorem}
	\label{thm[coll-approx]} 
	Let Assumption \ref{Assumption II} hold for Hilbert spaces $X^1$ and $X^2$  and 
	$v \in L^2_\Ee(U,X^2; \gamma)$. 
	Then for each $n \in \NN$, there exists a number $\xi_n$  such that  
	for the interpolation operator 
        $$
        \Ii_{G(\xi_n)}: L^2_\Ee(U,X^2; \gamma) \to \Vv(G(\xi_n)),
        $$
        we have $\dim \Vv(G(\xi_n)) \le  n$ and 
	\begin{equation} \label{u-I_Gu, p le 2}
		\|v -\Ii_{G(\xi_n)}v\|_{L^2(U,X^1;\gamma)} \leq Cn^{-\min(\alpha,\beta)}.
	\end{equation}
	The rate $\alpha$  is as in \eqref{spatialappnX} and
	the rate $\beta$ is given by 
	\begin{equation} 	\label{[beta]1}
		\beta := \left(\frac 1 {q_1} - \frac 1 2 \right)\frac{\alpha}{\alpha + \delta}, \quad 
		\delta := \frac 1 {q_1} - \frac 1 {q_2}.
	\end{equation}	
	The constant $C$ in \eqref{u-I_Gu, p le 2}  is independent of $v$ and $n$.
\end{theorem}

\begin{remark} \label{multilevel-I}
		{\rm
Observe that the operator $\Ii_{G(\xi_n)}$ can be represented 
in the form of a multilevel Smolyak sparse-grid interpolation with $k_n$ levels:
\begin{equation}  \nonumber
	\Ii_{G (\xi_n)} 
	\ = \ 
	\sum_{k=0}^{k_n} \delta_k \VI_{\Lambda_k(\xi_n)},
\end{equation}
where  $k_n:= \lfloor \log_2 \xi_n \rfloor$, 
the operator $\VI_\Lambda$ is defined as in \eqref{eq:VILambda},
and for $k \in \NN_0$ and  $\xi > 1$, 
\begin{equation} \nonumber
	\Lambda_k(\xi)
	:= \ 
	\begin{cases}
		\big\{\bs \in \Ff: \,\sigma_{2;\bs}^{q_2} \leq 2^{-k}\xi\big\} \quad &{\rm if }  \ 
		\alpha \le 1/q_2 - 1/2;\\
		\big\{\bs \in \Ff: \, \sigma_{1;\bs}^{q_1} \le \xi, \  
		\sigma_{2;\bs} \leq 2^{- (\alpha +1/2) k}\xi^\vartheta \big\} \quad  & {\rm if }  \ 
		\alpha > 1/q_2 - 1/2.
	\end{cases}
\end{equation}
	In Theorem \ref{thm[coll-approx]}, the multilevel polynomial interpolation  of $v \in L^2_\Ee(U,X^2; \gamma)$ by operators $\Ii_{G(\xi_n)}$ is a collocation method. It is based on the finite point-wise information in $\by$, more precisely, on   
	$|\Gamma(\Lambda_0(\xi_n))|= \Oo(n)$ of particular values of   $v$ at the interpolation points $\by \in \Gamma(\Lambda_0(\xi_n))$ and the approximations of   $v(\by)$, $\by \in \Gamma(\Lambda_0(\xi_n))$, by  $P_{2^k}v(\by)$ for $k=0,\ldots, \lfloor \log_2 \xi_n \rfloor$ with $\lfloor \log_2 \xi_n \rfloor = \Oo(\log_2 n)$.  
} \end{remark}

\subsubsection{Multilevel Smolyak sparse-grid quadrature}
\label{Integration}
In this section, we recall results of \cite{dD21} (see also \cite{dD-Erratum22})
on linear methods for numerical integration of functions from Bochner spaces as well as their linear functionals. 
We define the univariate operator $\Delta^{{\rm Q}}_m$ for even $m \in \NN_0$ by \index{quadrature!multilevel $\sim$}
\begin{equation} \nonumber
	\Delta^{{\rm Q}}_m
	:= \
	Q_m - Q_{m-2},
\end{equation} 
with the convention $Q_{-2} := 0$. 
We make use of the notation: 
$$
\cF_{\rev} := \{\bnu \in \cF: \nu_j \ {\rm  even}, \ j \in \NN \}.
$$
For  a function $v \in L^2_\Ee(U,X;\gamma)$, 
we introduce the operator $\Delta^{{\rm Q}}_\bnu$ defined for $\bnu \in \cF_{\rev}$ by
\begin{equation} \nonumber
	\Delta^{{\rm Q}}_\bnu(v)
	:= \
	\bigotimes_{j \in \NN} \Delta^{{\rm Q}}_{\nu_j}(v),
\end{equation} 
where the univariate operator
$\Delta^{{\rm Q}}_{\nu_j}$ is applied to the univariate function 
$\bigotimes_{j' =1 }^{j-1} \Delta^{{\rm Q}}_{\nu_{j'}}(v)$ by considering this function
as a univariate function of $y_j$, with all other variables held fixed.
As $\Delta^{{\rm I}}_\bnu$, 
the operators  $\Delta^{{\rm Q}}_\bnu$ are well-defined for all $\bnu \in \cF_{\rev}$. 

Letting Assumption \ref{Assumption II} hold for Hilbert spaces $X^1$ and $X^2$, 
we can construct linear  fully discrete quadrature operators. 
For a finite set $G \subset \NN_0 \times \cF_{\rev}$, 
we introduce the quadrature operator $\Qq_G$ which is defined for $v$ 
by
\begin{equation}  \label{Qq=int}
	\Qq_G v
	:= \
	\sum_{(k,\bnu) \in G} \delta_k \Delta^{{\rm Q}}_\bnu (v).
\end{equation}

If $\phi \in (X^1)'$  is a bounded linear functional on $X^1$, 
for a finite set $G \subset \NN_0 \times \cF_{\rev}$, 
the quadrature formula $\Qq_G v$ generates the quadrature formula $\Qq_G \langle \phi, v \rangle$ 
for integration of $\langle \phi, v \rangle$ 
by
\begin{equation} \nonumber
	\Qq_G \langle \phi, v \rangle
	:= \
	\langle \phi, \Qq_G v  \rangle.
\end{equation}

Define for $\xi>0$, 
\begin{equation} \label{G_rev(xi)}
	G_{\rev}(\xi)
	:= \ 
	\begin{cases}
		\big\{(k,\bnu) \in \NN_0 \times\cF_{\rev}: \, 2^k \sigma_{2;\bnu}^{q_2} 
                \leq \xi\big\} \ &{\rm if }  \ \alpha \le 1/q_2 - 1/2;
                \\
		\big\{(k,\bnu) \in \NN_0 \times\cF_{\rev}: \, \sigma_{1;\bnu}^{q_1} \le \xi, \  
		2^{(\alpha+1/2)  k} \sigma_{2;\bnu}
                \leq \xi^\vartheta \big\} \  & {\rm if }  \ \alpha > 1/q_2 - 1/2,
	\end{cases}
\end{equation}
where $\vartheta$ is as in \eqref{theta}.	

\begin{theorem}\label{thm[quadrature]} 
	Let the hypothesis of Theorem \ref{thm[coll-approx]} hold.  
	Then  we have the following.
	
	\begin{itemize}
		\item[{\rm (i)}]
		For each $n \in \NN$ there exists a number $\xi_n$  such that  $\dim\Vv(G_{\rev}(\xi_n))\le n$ and 
		\begin{equation} \label{u-Q_Gu}
			\left\|\int_{U}v(\by)\, \rd \gamma(\by ) - \Qq_{G_{\rev}(\xi_n)}v\right\|_{X^1} \leq Cn^{-\min(\alpha,\beta)}.
		\end{equation}
		\item[{\rm (ii)}] 
		Let  $\phi \in (X^1)'$ be a bounded linear functional on $X^1$. Then
		for each $n \in \NN$ there exists a number $\xi_n$  such that  $\dim\Vv(G_{\rev}(\xi_n))\le n$ and 
		\begin{equation} \label{u-Q_Gu_phi}
			\left|\int_{U} \langle \phi,  v (\by) \rangle\, \rd \gamma(\by ) - \Qq_{G_{\rev}(\xi_n)} \langle \phi,  v \rangle\right| \leq C \|\phi\|_{(X^1)'} n^{-\min(\alpha,\beta)}.
		\end{equation}
	\end{itemize}
	
	The rate $\alpha$ is as in \eqref{spatialappnX} and the rate $\beta$ is given by \eqref{[beta]1}.
	The constants $C$ in \eqref{u-Q_Gu} and \eqref{u-Q_Gu_phi}  are independent of $v$ and $n$.
\end{theorem}

The proof Theorem \ref{thm[quadrature]} are related to approximations in the norm of $L^1(U,X;\gamma)$ 
by special polynomial interpolation operators which generate the corresponding quadrature operators. 
Let us briefly describe this connection, for details see \cite{dD21,dD-Erratum22}.

\begin{remark} \label{I-to-Q}
	{\rm
We define  the univariate interpolation operator $\Delta^{{\rm I}*}_m$ for even $m \in \NN_0$ by
\begin{equation} \nonumber
	\Delta^{{\rm I}*}_m
	:= \
	I_m - I_{m-2},
\end{equation} 
with the convention $I_{-2} = 0$. 
The interpolation operators $\Delta^{{\rm I}*}_\bnu$ for $\bnu \in \cF_{\rev}$, $I^*_\Lambda$ for a finite set $\Lambda \subset \cF_{\rev}$, and $\Ii^*_G$ for a  finite set $G \subset \NN_0 \times \FF_{\rev}$, are defined in a similar way as the corresponding quadrature  operators $\Delta^{{\rm Q}}_\bnu$, $Q_\Lambda$  and $\Qq_G$ by  replacing 
$\Delta^{{\rm Q}}_{\nu_j}$ with $\Delta^{{\rm I}*}_{\nu_j}$, $j \in \NN$.

From the definitions it follows the equalities expressing the relationship between the interpolation and quadrature operators
\begin{equation} \nonumber
	Q_\Lambda v
	\ = \
	\int_{U} I^*_\Lambda v (\by)\, \rd \gamma(\by),
	\quad
	Q_\Lambda \langle \phi, v \rangle
	\ = \
	\int_{U} \langle \phi, I^*_\Lambda v (\by) \rangle\, \rd \gamma(\by),
\end{equation} 
and
\begin{equation}  \nonumber
	\Qq_G v
	\ = \
	\int_{U} \Ii^*_G v (\by)\, \rd \gamma(\by),
	\quad
	\Qq_G \langle \phi, v \rangle
	\ = \
	\int_{U} \langle \phi, \Ii^*_G v (\by) \rangle\, \rd \gamma(\by).
\end{equation}
} \end{remark}

\begin{remark} \label{multilevel-}
	{\rm
Similarly to  $\Ii_{G(\xi_n)}$,  the operator $\Qq_{G_{\rev}(\xi_n)}$ can be represented in the form of a multilevel 
Smolyak sparse-grid quadrature 
with $k_n$ levels:
\begin{equation}  \nonumber
	\Qq_{G_{\rev} (\xi_n)} 
	\ = \ 
	\sum_{k=0}^{k_n} \delta_k Q_{\Lambda_{\rev,k}(\xi_n)},
\end{equation}
where  $k_n:= \lfloor \log_2 \xi_n \rfloor$,
\begin{equation} \label{Q_Lambda}
	Q_\Lambda
	:= \
	\sum_{\bnu \in \Lambda} \Delta^{{\rm I}}_\bnu, \ \ \Lambda \subset \FF_{\rev},
\end{equation}
 and  for $k \in \NN_0$ and $\xi >0$,
\begin{equation} \nonumber
	\Lambda_{\rev,k}(\xi)
	:= \ 
	\begin{cases}
		\big\{\bs \in \Ff_{\rev}: \,\sigma_{2;\bs}^{q_2} \leq 2^{-k}\xi\big\} \quad &{\rm if }  \ 
		\alpha \le 1/q_2 - 1/2;\\
		\big\{\bs \in \Ff_{\rev}: \, \sigma_{1;\bs}^{q_1} \le \xi, \  
		\sigma_{2;\bs} \leq 2^{- (\alpha +1/2) k}\xi^\vartheta \big\} \quad  & {\rm if }  \ 
		\alpha > 1/q_2 - 1/2.
	\end{cases}
\end{equation}		
} 
\end{remark}

\begin{remark} \label{work-I}
	{\rm The convergence rates established in Theorems \ref{thm[coll-approx]} and \ref{thm[quadrature]}  
and in Theorems \ref{thm:mlint} and \ref{thm:mlquad} are proven with respect to  different parameters $n$ 
as the dimension of the approximation space and the work \eqref{work}, respectively. 
However, we could define the work of the operators $\Ii_{G(\xi_n)}$ and $\Qq_{G_{\rev}(\xi_n)}$ similarly as
\begin{equation}  \nonumber
	\sum_{k=0}^{k_n} 2^k |\Gamma(\Lambda_k(\xi_n))|,
\end{equation}
and 
\begin{equation}  \nonumber
	\sum_{k=0}^{k_n}  2^k |\Gamma(\Lambda_{\rev,k}(\xi_n))|,
\end{equation}
respectively, and 
prove the same convergence rates with respect to this work measure
as in Theorems \ref{thm[coll-approx]} and \ref{thm[quadrature]}. 		
} \end{remark}

\subsubsection{Applications to parametric divergence-form elliptic PDEs}	
\label{Applications to parametric divergence-form elliptic PDEs}
In this section, we apply the results in Sections \ref{Interpolation} and \ref{Integration} to parametric divergence-form elliptic PDEs \eqref{SPDE}. The spaces $V$ and $W$ are as in Section \ref{Some related results on sparsity}.
\begin{assumption} \label{Assumption2}
	There are a sequence $(V_n)_{n \in \NN_0}$ of subspaces $V_n \subset V $ of dimension $\le m$, and a sequence $(P_n)_{n \in \NN_0}$ of linear operators from $V$ into $V_n$, 
	and a number $\alpha>0$ such that
	\begin{equation} \label{spatialappn}
		\|P_n(v)\|_V \leq 
		C\|v\|_V , \quad
		\|v-P_n(v)\|_V \leq 
		Cn^{-\alpha} \|v\|_W, \quad \forall n \in \NN_0, \quad \forall v \in W.
	\end{equation}
\end{assumption}

If Assumption \ref{Assumption2} and the assumptions of Theorem~\ref {thm[ell_2summability]} hold  for the spaces $W^1=V$ and $W^2=W$
with some $0<q_1\leq q_2 <\infty$, then Assumption \ref{Assumption II} holds for the  spaces $X^i=W^i$, $i= 1, 2$, and the solution $u \in L^2(U,X^2;\gamma)$ to \eqref{SPDE}--\eqref{eq:CoeffAffin}. Hence we obtain the following results on multilevel (fully discrete) approximations.

\begin{theorem}\label{thm[L^2-approx]pde}
	Let Assumption \ref{Assumption2} hold. Let the hypothesis of Theorem~\ref {thm[ell_2summability]} hold  for the spaces $W^1=V$ and $W^2=W$
	with some $0<q_1\leq q_2 <\infty$ and $q_1 < 2$.	For $\xi >0$, let $G(\xi)$ be the set defined by \eqref{G(xi)} for $\sigma_{i;\bnu}$ as in \eqref{sigma_r,s^2}, $i=1,2$. Let $\alpha$ be as in \eqref{spatialappn}.
	Then for every $n \in \NN$ there exists a number $\xi_n$  such that  
		$\dim \Vv(G(\xi_n)) \le  n$ and 
		\begin{equation} \label{u-I_Gu, p le 2-pde}
			\|u -\Ii_{G(\xi_n)}u\|_{L^2(U,V; \gamma)} \leq Cn^{-\min(\alpha,\beta)},
		\end{equation}
		where $\beta$ is given by \eqref{[beta]1}.
		The constant $C$ in \eqref{u-I_Gu, p le 2-pde} is independent of $u$ and $n$.
\end{theorem}

\begin{theorem}\label{thm[quadrature]pde}
	Let Assumption \ref{Assumption2} hold. Let the assumptions of Theorem~\ref {thm[ell_2summability]} hold  for the spaces $W^1=V$ and $W^2=W$
	for some $0<q_1\leq q_2 <\infty$ with $q_1 < 4$. Let $\alpha$ be the rate  as given by \eqref{spatialappn}. For $\xi >0$, let $G_{\rev}(\xi)$ be the set defined by \eqref{G_rev(xi)}  for  $\sigma_{i;\bnu}$ as in \eqref{sigma_r,s^2}, $i=1,2$.  
	Then we have the following.
	\begin{itemize}
		\item[{\rm (i)}]
		For each $n \in \NN$ there exists a number $\xi_n$  such that  $\dim\Vv(G_{\rev}(\xi_n))\le n$ and 
		\begin{equation} \label{u-Q_Gu-quadrature}
			\left\|\int_{U}v(\by)\, \rd \gamma(\by ) - \Qq_{G_{\rev}(\xi_n)}v\right\|_V \leq Cn^{-\min(\alpha,\beta)}.
		\end{equation}
		\item[{\rm (ii)}] 
		Let  $\phi \in V'$ be a bounded linear functional on $V$. Then
		for each $n \in \NN$ there exists a number $\xi_n$  such that  $\dim\Vv(G_{\rev}(\xi_n))\le n$ and 
		\begin{equation} \label{u-Q_Gu_phi-pde}
			\left|\int_{U} \langle \phi,  v (\by) \rangle\, \rd \gamma(\by ) - \Qq_{G_{\rev}(\xi_n)} \langle \phi,  v \rangle\right| \leq C \|\phi\|_{V'} n^{-\min(\alpha,\beta)}.
		\end{equation}
	\end{itemize}
	The rate $\beta$ is given by 
	\begin{equation} 	\nonumber
		\beta := \left(\frac 2 {q_1} - \frac{1}{2}\right)\frac{\alpha}{\alpha + \delta}, \quad 
		\delta := \frac 2 {q_1} - \frac 2 {q_2}.
	\end{equation}	 
	The constants $C$ in \eqref{u-Q_Gu-quadrature} and \eqref{u-Q_Gu_phi-pde} are independent of $u$ and $n$.
\end{theorem}

\begin{proof}
	From Theorem \ref {thm[ell_2summability]}, Lemma \ref{lemma[bcdm]} and Assumption \ref{Assumption2} we can see that 	the assumptions of Theorem~\ref{thm[coll-approx]} hold for $X^1=V$	and $X^2 = W$ with $0 < q_1/2 \le q_2/2 < \infty$ and $q_1/2 < 2$. Hence, by applying Theorem~\ref{thm[quadrature]} we prove the theorem.
	\hfill
\end{proof}

\subsubsection{Applications to holomorphic functions}
\label{Applications to holomorphic functions}

As noticed, the proof of the weighted $\ell_2$-summability result formulated in Theorem~\ref {thm[ell_2summability]}  employs bootstrap arguments and induction on the differentiation order of derivatives with respect to the parametric variables, for details see \cite{BCDS,BCDM}. In the log-Gaussian case, this approach and technique are too complicated and difficult for extension to more general parametric PDE problems, in particular, of higher regularity. As it has been seen in the previous sections, the approach to a unified summability analysis of Wiener-Hermite PC expansions of various scales of function spaces based on  parametric holomorphy, covers a wide range of parametric PDE problems. 
In this section, we apply the results in 
Sections~\ref{Interpolation} and \ref{Integration} on linear approximations and integration 
in Bochner spaces to approximation and numerical integration of parametric holomorphic functions  
based on weighted  $\ell^2$-summabilities of the coefficient sequences of the Wiener-Hermite PC expansion.   

The following theorem on weighted $\ell_2$-summability for  
$(\bb,\xi,\delta,X)$-holomorphic functions can be derived 
from Theorem \ref{thm:bdXSum} and Lemma \ref{lemma[bcdm]}.  
\begin{theorem} \label{thm:Holom-AssumpA} 
        Let $v$ be
	$(\bb,\xi,\delta,X)$-holomorphic for some $\bb\in \ell^p(\N)$ with $0< p <1$. 
        Let  $s =1,2 $ and $\tau, \lambda \ge 0$.  
        Let  further the sequence $\bvarrho=(\varrho_j)_{j \in \NN}$ 
        be defined by 
	$$
	\varrho_j:=b_j^{p-1}\frac{\xi}{4\sqrt{r!}} \|\bb\|_{\ell^p}.
	$$
	Then, 	for any $r > \frac{2s (\tau + 1)}{q}$,
		\begin{equation} \nonumber
		\sum_{\bnu\in\cF_s} (\sigma_{\bnu} \|v_\bnu\|_{X})^2 \le M <\infty \quad  \text{and} \quad 
		\left(p_{\bnu}(\tau,\lambda)\sigma_{\bnu}^{-1}\right)_{\bnu \in \cF_s} \in \ell^{q/s}(\cF_s),
	\end{equation}		
where	
$q := \frac{p}{1-p}$, $M:= \delta^2C(\bb)$ and $(\sigma_\bnu)_{\bnu \in \cF}$
	with $\sigma_\bnu:= \beta_\bnu(r,\bvarrho)^{1/2}$.
\end{theorem}

To treat multilevel approximations and integration of parametric, holomorphic functions, 
it is appropriate to replace  Assumption \ref{Assumption II} by its modification.

\begin{assumption} \label{Assumption3}
	Assumption \ref{Assumption II} holds with item {\rm (ii)} replaced with item 

\begin{itemize}
	\item [(ii')] 	For $i=1,2$,   $v$ is
	$(\bb_i,\xi,\delta,X^i)$-holomorphic for some $\bb_i\in \ell^{p_i}(\N)$ with $0< p_1 \le p_2 <1$. 
\end{itemize}
\end{assumption}

Assumption \ref{Assumption3} is a condition for fully discrete approximation 
of $(\bb,\xi,\delta,X)$-holomorphic functions.
This is formalized in the following corollary of Theorem \ref{thm:Holom-AssumpA}. 

\begin{corollary} \label{corollary[AssumptionII]}
	Assumption \ref{Assumption3}  implies Assumption \ref{Assumption II} 
	for
	$q_i := \frac{p_i}{1-p_i}$ and $(\sigma_{i;\bnu})_{\bnu \in \cF}$, $i=1,2$, 
        where
	$$
	\sigma_{i;\bnu}:= \beta_{i;\bnu}(r,\bvarrho_i)^{1/2}, \quad \varrho_{i;j}:=b_{i;j}^{p_i-1}\frac{\xi}{4\sqrt{r!}} \|\bb_i\|_{\ell^{p_i}}.
	$$
\end{corollary}

We formulate results on multilevel quadrature of parametric holomorphic functions 
as consequences of Corollary \ref{corollary[AssumptionII]} and 
Theorems \ref{thm[coll-approx]} and \ref{thm[quadrature]}.

\begin{theorem}\label{thm[L^2-approx]X}
	Let Assumption \ref{Assumption3} hold for the Hilbert spaces $X^1$ and $X^2$ with $p_1<2/3$, and
	$v \in L^2(U,X^2; \gamma)$. 
	For $\xi >0$, 
        let $G(\xi)$ be the set defined by \eqref{G(xi)} for $\sigma_{i;\bnu}$, $i=1,2$ 
        as given in Corollary \ref{corollary[AssumptionII]}. 
	Then for every $n \in \NN$ there exists a number $\xi_n$  such that  
		$\dim \Vv(G(\xi_n)) \le  n$ and 
		\begin{equation} \label{u-I_Gu, p le 2-X}
			\|v -\Ii_{G(\xi_n)}v\|_{L^2(U,X^1; \gamma)} \leq Cn^{-R},
		\end{equation}
		where $R$ is given by the formula \eqref{R-I} and 
	the constant $C$ in\eqref{u-I_Gu, p le 2-X} is independent of $v$ and $n$.
\end{theorem}

\begin{theorem}\label{thm[quadrature]X}
	Let Assumption \ref{Assumption3} hold for the Hilbert spaces $X^1$ and $X^2$ with $p_1 < 4/5$, 
	and $v \in L^2(U,X^2;\gamma)$. 
For $\xi >0$, let $G_{\rev}(\xi)$ be the set defined by \eqref{G_rev(xi)} 
for $\sigma_{i;\bnu}$, $i=1,2$, as given in Corollary \ref{corollary[AssumptionII]}.
Then we have the following.
	\begin{itemize}
		\item[{\rm (i)}]
		For each $n \in \NN$ there exists a number $\xi_n$  such that  $\dim\Vv(G_{\rev}(\xi_n))\le n$ and 
		\begin{equation} \label{u-Q_Gu-quadratureX}
			\left\|\int_{U}v(\by)\, \rd \gamma(\by ) - \Qq_{G_{\rev}(\xi_n)}v\right\|_{X^1}\leq Cn^{-R}.
		\end{equation}
		\item[{\rm (ii)}] 
		Let  $\phi \in (X^1)'$ be a bounded linear functional on $X^1$. Then
		for each $n \in \NN$ there exists a number $\xi_n$  such that  $\dim\Vv(G_{\rev}(\xi_n))\le n$ and 
		\begin{equation} \label{u-Q_Gu_phi-X}
			\left|\int_{U} \langle \phi,  v (\by) \rangle\, \rd \gamma(\by ) - \Qq_{G_{\rev}(\xi_n)} \langle \phi,  v \rangle\right| \leq C\|\phi\|_{(X^1)'} n^{-R},
		\end{equation}
	\end{itemize}
	where the convergence rate $R$ is given by the formula \eqref{R-Q} and 
the constants $C$ in \eqref{u-Q_Gu-quadratureX} and \eqref{u-Q_Gu_phi-X} are independent of $v$ and $n$.
\end{theorem}

\begin{remark} {\rm\label{rmk:comparison}
	{\rm
		We comment on the relation of the results of Theorems \ref{thm:mlint} and \ref{thm:mlquad}  
        to the results of \cite{dD21} which are presented in Theorems \ref{thm[L^2-approx]pde} and \ref{thm[quadrature]pde}, 
        on multilevel approximation of solutions to parametric divergence-form elliptic PDEs with log-Gaussian inputs.  
	
	Specifically, in \cite{dD21}, by combining spatial and parametric
	approximability in the spatial domain and weighted
	$\ell^2$-summability of the $V:=H^1_0(\domain)$
	and $W$ norms of Wiener-Hermite PC
	expansion coefficients obtained in \cite{BCDM,BCDS}, the author
	constructed linear non-adaptive methods of fully discrete
	approximation by truncated Wiener-Hermite PC expansion and
	polynomial interpolation approximation as well as fully discrete
	weighted quadrature for parametric and stochastic elliptic PDEs with
	log-Gaussian inputs, and proved the convergence rates of approximation
	by them.  
        The results in \cite{dD21} are based on Assumption \ref{Assumption2} 
        that requires the \emph{existence of a sequence 
        $(P_n)_{n \in \NN_0}$ of linear operators independent of $\by$}, 
        from $H^1_0(\D)$ into $n$-dimensional subspaces $V_n \subset H^1_0(\D)$
        such that
	$$
	\|P_n(v)\|_{H^1_0} \leq C_1\|v\|_{H^1_0} \ \ \text{and} \ \ 
	\|v-P_n(v)\|_{H^1_0} \leq C_2n^{-\alpha} \|v\|_{W}
	$$ 
	for all $n \in \NN_0$ and for all $v \in W$, 
        where the constants $C_1, C_2$ are independent of $n$.
	\emph{The assumption of $P_n$ being independent of $\by$ is however
	typically not satisfied if $P_n(u(\by)) = u^n(\by)$ is a numerical
	approximation to $u(\by)$} 
        (such as, e.g., a Finite-Element or a Finite-Difference discretization).
	
	In contrast, the present approximation rate analysis is based
	on quantified, parametric holomorphy of the discrete approximations
	$u^l$ to $u$ as in Assumption \ref{ass:ml}.  For example, assume
	that $u:U\to H_0^1(\D)$ is the solution of the parametric PDE
	$$
	-\div(a(\by)\nabla u(\by))=f
	$$ 
	for some $f\in L^2(\D)$ and a
	parametric diffusion coefficient $a(\by)\in L^\infty(\D)$ such that
	$$
	\underset{\bx\in\D}{\essinf}\, a(\by,\bx)>0 \ \ \ \forall \by\in U.
	$$
	Then
	$u^\lev:U\to H_0^1(\D)$ could be a numerical approximation to $u$,
	such as the FEM solution: for every $\lev\in\N$ there is a finite
	dimensional discretization space $X_\lev\subseteq H_0^1(\D)$, and
	$$
	\int_{\D}\nabla u^\lev(\by)^\top a(\by)\nabla v\dd \bx =\int_{\D}f
	v\dd \bx
	$$ 
	for every $v\in X_\lev$ and for every $\by\in U$. Hence
	$u^\lev(\by)$ is the orthogonal projection of $u(\by)$ onto $X_\lev$
	w.r.t.\ the inner product
	$$
	\langle v,w \rangle_{a(\by)}:= \int_\D \nabla v^\top a(\by)\nabla
	w\dd \bx
	$$ 
	on $H_0^1(\D)$. We may write this as
	$u^\lev(\by)=P_\lev(\by) u(\by)$, for a $\by$-dependent projector
	$$
	P_\lev(\by):H_0^1(\D)\to X_\lev.
	$$
	 This situation is covered by
	Assumption \ref{ass:ml}.
	
	The preceding comments can be extended to the results 
        on multilevel approximation of holomorphic functions in 
        Theorems \ref{thm:mlint} and \ref{thm:mlquad} 
	to the results in Theorems \ref{thm[L^2-approx]X} and \ref{thm[quadrature]X}. 
        On the other hand, as noticed above, the convergence rates in Theorems \ref{thm[L^2-approx]X} 
        and \ref{thm[quadrature]X} are slightly better 
        than those obtained in Theorems \ref{thm:mlint} and \ref{thm:mlquad}. 
}
} \end{remark} 

\newpage
\
\newpage
\section{Conclusions}
\label{sec:Concl}
We established holomorphy of parameter-to-solution maps
$$\data \ni a \mapsto u = \Uu(a) \in X$$ 
for linear, elliptic, parabolic, and other PDEs in various scales of
function spaces $\data$ and $X$, including in particular standard
and corner-weighted Sobolev spaces.  
Our discussion focused on non-compact parameter domains which arise
from uncertain inputs from function spaces expressed in a suitable
basis with Gaussian distributed coefficients. 
We introduced and used a form of quantified, parametric holomorphy 
in products of strips to show that this implies summability results of coefficients
of the Wiener-Hermite PC expansion
of such infinite parametric functions.
Specifically,
we proved weighted $\ell^2$-summability and
$\ell^p$-summability results for 
Wiener-Hermite PC expansions
of certain parametric, deterministic solution families
$\{ u(\by): \by\in U \} \subset X$, for a given ``log-affine''
parametrization \eqref{eq:CoeffAffin} of admissible random input data
$a\in \data$.

We introduced and analyzed constructive, deterministic, sparse-grid
(``stochastic collocation'') algorithms based on univariate
Gauss-Hermite points, to efficiently sample the parametric,
deterministic solutions in the possibly infinite-dimensional parameter
domain $U = \RRi$.  
The sparsity of the coefficients of Wiener-Hermite  PC expansion 
was
shown to entail corresponding convergence rates of the presently
developed sparse-grid sampling schemes.  In combination with suitable
Finite Element discretizations in the physical, space(-time) domain
(which include proper mesh-refinements to account for singularities in
the physical domain) we proved convergence rates for abstract,
multilevel algorithms which employ different combinations of
sparse-grid interpolants in the parametric domain with space(-time)
discretizations at different levels of accuracy in the physical domain.

The presently developed, holomorphic setting was also shown
to apply to the corresponding Bayesian inverse problems subject to
PDE constraints: here, the density of the Bayesian posterior with
respect to a Gaussian random field prior was shown to generically
inherit quantified holomorphy from the parametric forward problem,
thereby facilitating the use of the developed sparse-grid collocation
and integration algorithms also for the efficient deterministic
computation of Bayesian estimates of PDEs with uncertain inputs,
subject to noisy observation data.

Our 
approximation rate bounds are free from the curse-of-dimensionality
and only limited by the PC coefficient summability.  They will
therefore also be relevant for convergence rate analyses of other
approximation schemes, such as Gaussian process emulators or neural
networks (see, e.g., \cite{StTeckGPAppr2018,
  dung2021deep,dung2021collocation,SZ21_2946} and references there).

\bigskip
\noindent {\bf Acknowledgments.}  The work of Dinh D\~ung and Van Kien
Nguyen is funded by Vietnam National Foundation for Science and
Technology Development (NAFOSTED) under Grant No. 102.01-2020.03. A
part of this work was done when Dinh D\~ung was visiting the
Forschungsinstitut f\"ur Mathematik (FIM) of the ETH Z\"urich invited
by Christoph Schwab, and when Dinh D\~ung and Van Kien Nguyen were
at the Vietnam Institute for Advanced Study in Mathematics
(VIASM). They would like to thank the FIM and VIASM for providing a
fruitful research environment and working condition.  Dinh D\~ung
thanks Christoph Schwab for invitation to visit the FIM and for his
hospitality.

\newpage
\
\newpage

\bibliographystyle{amsplain} 
\bibliography{pdes_log.bib}
\newpage

\section*{List of symbols}
\renewcommand{\arraystretch}{1.25}
\begin{tabular}{>{$}l<{$} >{\hspace{2mm}}l}
    C^s(\domain) & Space of $s$-H\"older continuous functions on $\domain$ \\
	\Delta & Laplace operator \\  
    \div & Divergence operator \\
  \domain & Domain in $\R^d$ \\
  \partial \domain & Boundary of the domain $\domain$ \\
  \be_j & The multiindex $(\delta_{ij})_{j\in\N}\in\N_0^\infty$\\
    \gamma & Gaussian measure \\
  \gamma_d & 
             Gaussian measure on $\R^d$ \\
    \nabla & Gradient operator \\  
  H(\gamma) & Cameron-Martin space  of the Gaussian measure $\gamma$ \\
    H_k & $k$th normalized probabilistic Hermite polynomial \\  
    H^s(\domain) & $W^{s,2}(\domain)$ \\  
  H^1_0(\domain) & Space of functions $u\in W^{1,2}(\domain)$ such that $u|_{\partial\domain}=0$\\
  H^s_0(\D) & $H^s(\D)\cap H^1_0(\domain)$ \\
  H^{-1}(\D) & Dual space of $H^1_0(\domain)$ \\
  \VI_{\Lambda} & Smolyak interpolation operator\\
  \VIml_\blev & Multilevel Smolyak interpolation operator\\
    \Im(z) & Imaginary part of the complex number $z$ \\  
  \Kk^s_\varkappa(\domain) & Space of functions $u: \domain \to \CC$ such that $r_\domain^{|\balpha|-\varkappa}D^\balpha u\in L^2(\domain)$ for all $|\balpha|\leq s$\\
  \lambda_d & Lebesgue measure on $\R^d$ \\
  \Lambda & Set of multiindices\\  
    \hat{\mu} & Fourier transform of the measure $\mu$ \\  
    L^p(\Omega) & Space of  Lebesgue measurable, $p$-integrable functions on $\Omega$ \\
    L^p(\Omega,\mu) & Space of $\mu$-measurable, $p$-integrable functions on $\Omega$ \\
  L^p(\Omega,X;\mu) & Space of functions $u:\Omega\to X$ such that $\norm[X]{u}\in L^p(\Omega,\mu)$\\
    L^\infty(\Omega) & Space of  Lebesgue measurable, essentially bounded functions on $\Omega$\\
  \ell^p(I) & Space of sequences $(y_j)_{j \in I}$ such that $(\sum_{j \in I}|y_j|^p)^{1/p} < \infty$\\
  \|f\|_X & Norm of $f$ in the space $X$ \\
  \VQ_{\Lambda} & Smolyak quadrature operator\\
  \VQml_\blev & Multilevel Smolyak quadrature operator\\  
  r_\domain & Smooth function $\domain \to \RR_+$ which equals
 $| \bx - \bc |$ in the vicinity of each corner of $D$\\
    \Re(z) & Real part of the complex number $z$ \\
    U & $\R^\infty$ \\
    W^{s,q}(\domain) & Sobolev spaces of integer order $s$ and integrability $q$ on $\domain$ \\
  \Ww^s_\infty(\domain) & Space of functions $u: \domain\to \CC$ such that
                          $r_\domain^{|\balpha|}D^\balpha u\in L^\infty(\domain)$ for all $|\balpha|\leq s$
\end{tabular}

\newpage
\section*{List of abbreviations}
\renewcommand{\arraystretch}{1.25}
\begin{tabular}{>{$}l<{$} >{\hspace{2mm}}l}
	BIP & Bayesian inverse problems  \\
	BVP & boundary value problem \\
	FE & finite element  \\  
	FEM & finite element method \\	
	GM & Gaussian measure \\
	GRF &  Gaussian random fields \\	
	IBVP & initial boundary value problem\\
	KL & Karhunen-Lo\`eve\\ 
	MC & Monte-Carlo \\
	ONB & orthonormal basis  \\  
	PC & polynomial chaos \\
	PDE & partial differential equations \\	
	QMC & quasi-Monte Carlo \\
	RKHS &reproducing kernel Hilbert space \\  
	RV & random variable \\
	UQ & uncertainty quantification \\  
\end{tabular}
\newpage
\printindex
\end{document}